\documentclass[11pt,letterpaper,twoside,reqno,nosumlimits]{amsart}

\allowdisplaybreaks

\usepackage[usenames,dvipsnames]{xcolor}
\usepackage{fancyhdr}
\usepackage{amsmath,amsfonts,amsbsy,amsgen,amscd,amssymb,amsthm}
\usepackage{url}

\usepackage{mathtools}

\usepackage[font=small,margin=0.25in,labelfont={bf},labelsep={colon}]{caption}
\usepackage{subcaption}

\usepackage{tikz}
\usepackage{microtype}
\usepackage{enumitem}

\definecolor{dark-gray}{gray}{0.3}
\definecolor{dkgray}{rgb}{.4,.4,.4}
\definecolor{dkblue}{rgb}{0,0,.5}
\definecolor{medblue}{rgb}{0,0,.75}
\definecolor{rust}{rgb}{0.5,0.1,0.1}

\usepackage[colorlinks=true]{hyperref}

\hypersetup{urlcolor=rust}
\hypersetup{citecolor=dkblue}
\hypersetup{linkcolor=dkblue}

\usepackage{setspace}
\usepackage{graphicx}
\usepackage{booktabs,longtable,tabu} % Nice tables
\usepackage{multirow} % More control over tables

\usepackage{float}

\usepackage[full]{textcomp}

\usepackage[scaled=.98,sups]{XCharter}% lining figures in math, osf in text
\usepackage[scaled=1.04,varqu,varl]{inconsolata}% inconsolata typewriter
\usepackage[type1]{cabin}% sans serif
\usepackage[charter,vvarbb,scaled=1.07]{newtxmath}
\usepackage[cal=boondoxo]{mathalfa}
\usepackage[T1]{fontenc}

\usepackage{bm} % boldmath must be called after the package

\graphicspath{{figures/}}

\newtheorem{theorem}{Theorem}[section]
\newtheorem{lemma}[theorem]{Lemma}

\newtheorem{conjecture}[theorem]{Conjecture}

\theoremstyle{definition}

\newtheorem{remark}[theorem]{Remark}

\numberwithin{equation}{section} 
\numberwithin{figure}{section}
\numberwithin{table}{section}

\floatstyle{plaintop}
\newfloat{recipe}{thp}{lor}
\floatname{recipe}{Recipe}
\numberwithin{recipe}{section}

\providecommand{\mathbold}[1]{\bm{#1}}  % Must be after 'fourier'
\renewcommand{\phi}{\varphi}

\providecommand{\mathbbm}{\mathbb} % In case we don't load bbm

\newcommand{\R}{\mathbbm{R}}

\newcommand{\N}{\mathbbm{N}}

\newcommand{\diff}[1]{\mathrm{d}{#1}}
\newcommand{\idiff}[1]{\, \diff{#1}}

\newcommand{\mtx}[1]{\mathbold{#1}}

\newcommand{\triplenorm}[1]{{\left\vert\kern-0.25ex\left\vert\kern-0.25ex\left\vert #1
    \right\vert\kern-0.25ex\right\vert\kern-0.25ex\right\vert}}

\graphicspath{
  {Figures}
}

\usepackage[numbers]{natbib}

\newcommand{\om}{\omega}

\newcommand{\la}{\langle}
\newcommand{\ra}{\rangle}

\mathtoolsset{showonlyrefs}

\begin{document}

\title[Self-similar Blowups with Singular Profiles of Hou--Luo Model and Boussinesq equations]{Novel Self-similar Finite-time Blowups with Singular Profiles\\ of the 1D Hou--Luo Model and the 2D Boussinesq equations:\\ A Numerical Investigation}
\author[B. Chen, D. Huang, X. Li]{Bojin Chen$^1$, De Huang$^2$, and Xiangyuan Li$^3$}
\thanks{$^1$School of Mathematical Sciences, Peking University. E-mail: bj\_chen@stu.pku.edu.cn}
\thanks{$^2$School of Mathematical Sciences, Peking University. E-mail: dhuang@math.pku.edu.cn}
\thanks{$^3$School of Mathematical Sciences, Peking University. E-mail: lixiangyuan23@stu.pku.edu.cn}

\begin{abstract}
We present novel self-similar finite-time blowup scenarios for the 1D Hou--Luo model. We numerically demonstrate that solutions that initially satisfy certain derivative degeneracy condition can develop asymptotically self-similar finite-time blowups with singular self-similar profiles that are unbounded at some point. Moreover, this blowup phenomenon exhibits a two-stage feature: the solution first undergoes a local $L^{\infty}$ blowup at some time $\tilde{T}$, then continues in the weak sense beyond $\tilde{T}$ and develops a local $L^p$ blowup at a later time $T>\tilde{T}$ for some $p>0$. A further numerical investigation indicates that both stages are asymptotically self-similar. Finally, we extend our numerical study to the 2D Boussinesq equations and discover similar self-similar finite-time blowups with singular profiles that also exhibit a two-stage feature.
\end{abstract}

\maketitle

\section{Introduction}
In this paper, we present new self-similar blowup phenomena of the 1D Hou--Luo (HL) model:
\begin{equation}\label{eqt:1Dhouluo}
    \begin{aligned}
    &\om_t+u\om_x=\theta_x,\\
    &\theta_t+u\theta_x=0,\\
    & u_x=\mtx{H}(\om),
    \end{aligned}
\end{equation}
for $x\in \R$, where $\mtx{H}$ denotes the Hilbert transform on the real line. This model was first proposed by Luo and Hou \cite{luo2014potentially,luo2014toward} to acquire understanding of the numerically observed self-similar singularity formation of the 3D axisymmetric Euler equations on the solid boundary of an infinitely long cylinder. 

Whether initially smooth solutions to the 3D incompressible Euler equations can develop a singularity in finite time has been one of the most fundamental problems in mathematical fluid dynamics. So far, the Hou--Luo scenario \cite{luo2014potentially,luo2014toward} stands as the only case where self-similar finite-time blowup of the 3D Euler equations from smooth initial data has been rigorously established, though the presence of a solid boundary is critically necessary in this scenario. Ever since the report of the convincing numerical evidence of the Hou--Luo scenario, vast amounts of effort have been made to try to rigorously prove the existence of such boundary singularity for the 3D Euler equations. Recently, Chen and Hou \cite{chen2022stable,chen2025singularity} used a powerful computer-assisted approach to prove the nearly self-similar finite-time blowup of the 2D Boussinesq/3D Euler equations with boundary, thus finally settling the conjecture on the Hou--Luo scenario.

As a simplified model of the boundary behavior of the Hou--Luo scenario, the 1D HL model has attracted significant attention regarding its singularity formation. In view of the scaling property of the equations, we are particularly interested in the self-similar finite-time blowups of \eqref{eqt:1Dhouluo}. Specifically, we say \eqref{eqt:1Dhouluo} exhibits an exact self-similar blowup if it admits a solution $(\om,\theta)$ that takes the form
\begin{equation}\label{eqt:HL_exact_blowup}
    \om(x,t)=(T-t)^{\lambda}\bar{\Omega}\left(\frac{x}{(T-t)^{\gamma}}\right), \quad \theta(x,t)=(T-t)^{\mu}\bar{\Theta}\left(\frac{x}{(T-t)^{\gamma}}\right),
\end{equation}
 where $\bar{\Omega}$ and $\bar{\Theta}$ are the self-similar profiles, $\lambda$, $\mu$, $\gamma$ are scaling factors, and $T$ is the finite blowup time. Plugging this ansatz into \eqref{eqt:1Dhouluo} and balancing the equations as $t\rightarrow T$ yields $\lambda=-1$, $\mu=\gamma+2\lambda=\gamma-2$.  Moreover, the undetermined factor $\gamma$ is related to the far-field decay rates of $\bar{\Omega}$ and $\bar{\Theta}_X$ \cite{huang2025exact}. In contrast to an exact self-similar solution, an asymptotically self-similar finite-time blowup refers to a solution that exhibits clear self-similarity only as $t$ approaches the finite blowup time $T$:   
\begin{equation}\label{eqt:asymptotic_selfsimilar_HL}
    \om(x,t)=(T-t)^{\lambda}\Omega\left(\frac{x}{(T-t)^{\gamma}},t\right), \quad \theta(x,t)=(T-t)^{\mu}\Theta\left(\frac{x}{(T-t)^{\gamma}},t\right),
\end{equation}
where $\lambda=-1$, $\mu=\gamma-2$, and the time-dependent profiles $\Omega(\cdot,t)$, $\Theta(\cdot,t)$ will converge to some non-trivial steady states $\bar{\Omega}$, $\bar{\Theta}$ as $t\to T$. Sometimes \eqref{eqt:asymptotic_selfsimilar_HL} is more preferred as it characterizes a more robust blowup mechanism.

Next, we review some existing results on the singularity formation of the HL model \eqref{eqt:1Dhouluo}. Shortly after the original work of Luo and Hou \cite{luo2014toward}, Choi et al. \cite{choi2017finite} used a functional argument to prove the finite-time blowup of the HL model \eqref{eqt:1Dhouluo} and another 1D model known as the CKY model \cite{choi2015finite}. However, their approach was not able to capture the self-similar nature of the blowup. Years later, Chen, Hou, and Huang \cite{chen2022asymptotically} developed a novel analysis framework based on rigorous computer-assisted proofs to establish asymptotically self-similar blowups of \eqref{eqt:1Dhouluo} from smooth initial data. Instead of working with $(\omega,\theta)$, they established finite-time blowups for the variables $(\omega,\theta_x)$. Specifically, they first constructed an approximate self-similar profile $(\tilde{\Omega},\tilde{\Theta}_X)$ using numerical computation, and then reformulated the problem as a stability analysis of the approximate steady state $(\tilde{\Omega},\tilde{\Theta}_X)$ by analyzing the evolution of the rescaled variables $(\Omega,\Theta_X)$ rather than the physical variables $(\omega,\theta_x)$. Through an energy argument, they demonstrated that, under a proper choice of normalization conditions and initial data, there exist $\gamma$ and $T$ such that \eqref{eqt:1Dhouluo} exhibits a blowup of the form \eqref{eqt:asymptotic_selfsimilar_HL}. Specifically, if the initial data of \eqref{eqt:1Dhouluo} is sufficiently close to $(\tilde{\Omega},\tilde{\Theta}_X)$ in some suitable energy norm, then the rescaled variables $(\Omega,\Theta_X)$ will stay in a small neighborhood of $(\tilde{\Omega},\tilde{\Theta}_X)$ in the same norm. Furthermore, they also used a limit argument to show that, near the approximate steady state $(\tilde{\Omega},\tilde{\Theta}_X)$, there lies a true stable steady state that corresponds to exact self-similar profiles $(\bar{\Omega},\bar{\Theta}_X)$ of \eqref{eqt:1Dhouluo}. The existence of exact self-similar profiles of the HL model was later established alternatively by Huang, Qin, Wang, and Wei \cite{huang2025exact} via a fixed-point method which is purely analytic. By performing a detailed study of the fixed-point map, the authors provided finer characterizations of the self-similar profiles such as their smoothness and precise far-field decay rates.

Moreover, the numerical simulations performed in \cite{huang2025exact} suggested that the self-similar profiles obtained via the fixed-point iteration coincide with those constructed in \cite{chen2022asymptotically} (up to scaling). Thus, it is expected that these two results characterize the same self-similar blowup mechanism of the HL model. Figure \ref{fig:HL_physicalspace_profiles_nondegenerate} illustrates the evolution of $\omega$ in this scenario. Importantly, the initial data $(\omega^0,\theta_x^0)$ considered in \cite{chen2022asymptotically} are odd-symmetric smooth functions that are \textit{non-degenerate} at the origin in the sense that $\omega^0_x(0)\ne 0$, $\theta_{xx}^0(0)\ne 0$. Starting from such initial data, $\omega$ develops a finite-time singularity at the symmetry point $x=0$. Furthermore, when the profiles of $\omega$ are dynamically rescaled near the origin with appropriate scaling factors, the resulting variable $\Omega$ converges to a regular profile which is non-degenerate at the origin (as shown in Figure \ref{fig:HL_rescaling_profiles_nondegenerate}). Here, by ``regular'' we mean that the profile is a bounded continuous function. These observations strongly suggest that \cite{chen2022asymptotically} and \cite{huang2025exact} characterized a robust blowup scenario for the HL model: starting from odd-symmetric non-degenerate initial data, \eqref{eqt:1Dhouluo} can develop self-similar blowups with non-degenerate regular profiles.

In contrast, the understanding of the HL model starting from more degenerate initial data (in the sense that $\omega^0_x(0)=\theta_{xx}^0(0)=0$) remains very limited. To date, the only existing effort in this direction is a preliminary numerical study in the thesis of Liu \cite{liu2017spatial}, which suggested that in the degenerate case, \eqref{eqt:1Dhouluo} may blow up in a way distinct from the non-degenerate case. In this paper, we conduct a thorough numerical investigation of \eqref{eqt:1Dhouluo} from degenerate initial data and report a novel blowup scenario for the HL model. Unlike the non-degenerate case, we find that, starting from smooth initial data that are degenerate at the origin, \eqref{eqt:1Dhouluo} can develop self-similar finite-time blowups with singular profiles. Specifically, as $t$ approaches the blowup time, the rescaled variable $\Omega$ converges to a profile that is locally unbounded at some point (see Figure \ref{fig:HL_rescaling_profiles_degenerate}), which partially confirms Liu's early observations.

Similar to the non-degenerate scenario, our numerical results support the existence of a finite time $T>0$ such that the location of the maximum of $\omega$ converges to the origin as $t \to T$. However, refined numerical fitting reveals that the $L^{\infty}$ norm of $\om$ actually blows up at an earlier time $\tilde{T} < T$. Thus, our results suggest a two-stage finite-time blowup mechanism for the HL model from degenerate initial data (as illustrated in Figure \ref{fig:HL_physicalspace_profiles_degenerate}): the solution first undergoes a local $L^{\infty}$ blowup (while the peak is away from the origin), then continues in the weak sense and develops a local $L^p$ blowup at the origin for some $p>0$. We want to remark that our result characterizes how the HL model can behave after a blowup in the $L^{\infty}$ norm. Moreover, further investigation demonstrates that both stages are asymptotically self-similar: the first stage develops regular self-similar profiles that do not change sign on the real line, whereas the second develops singular self-similar profiles. To the best of our knowledge, this represents a previously unreported blowup phenomenon.

\begin{figure}[!htbp]
\centering
    \begin{subfigure}[b]{0.95\textwidth}
        \centering
        \includegraphics[width=0.4\textwidth]{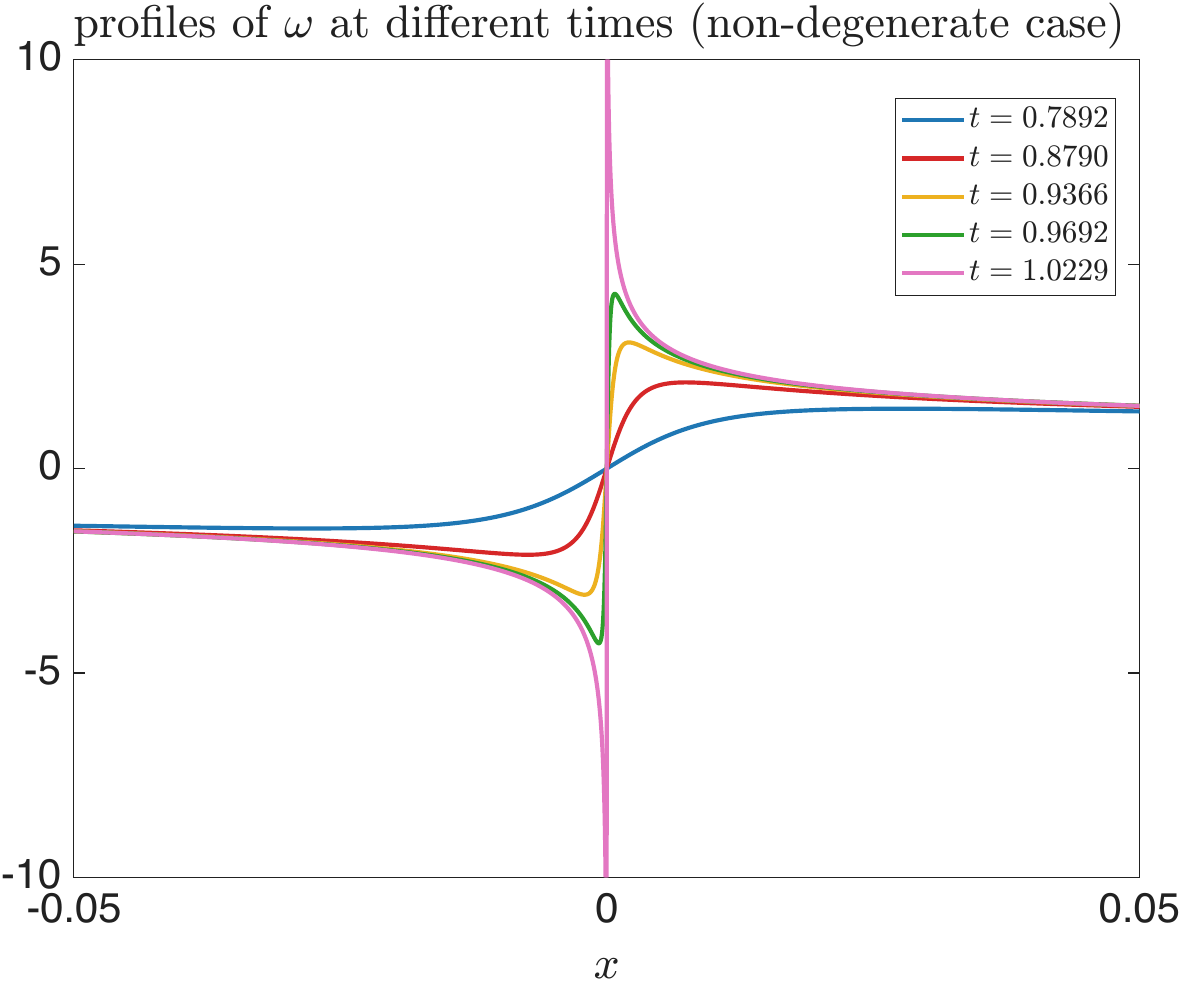}
        \caption{Evolution of $\om(\cdot,t)$ starting from odd-symmetric non-degenerate initial data. As shown, the spatial profile of $\omega$ concentrates towards the origin. During this process, the location of the maximum converges to $x=0$, accompanied by a rapid growth in the $L^{\infty}$ norm. Eventually, $\omega$ develops a finite-time singularity at $x=0$.}
            \label{fig:HL_physicalspace_profiles_nondegenerate}
    \end{subfigure}

    \begin{subfigure}[b]{0.95\textwidth}
        \centering
        \includegraphics[width=0.4\textwidth]{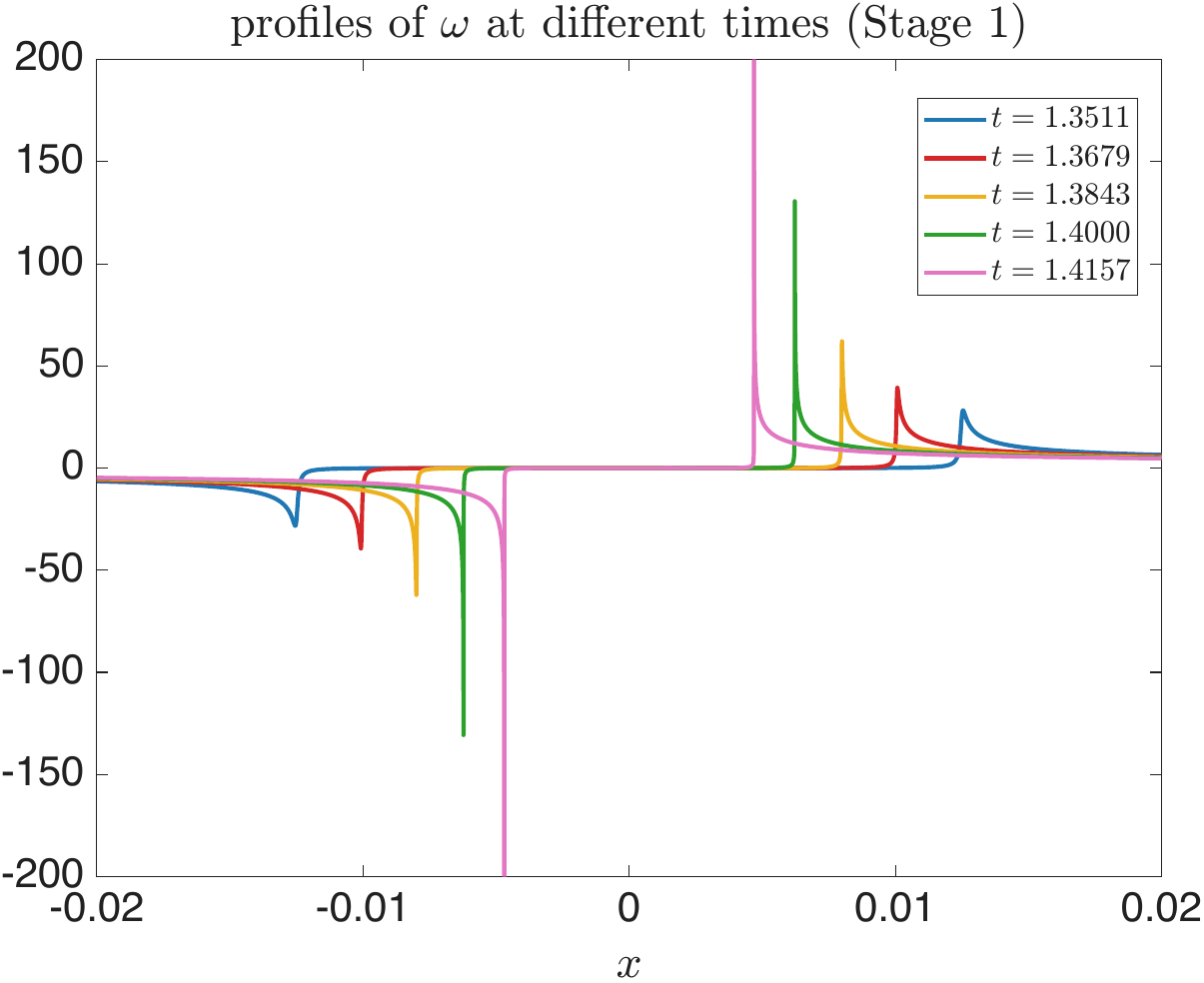}
        \includegraphics[width=0.4\textwidth]{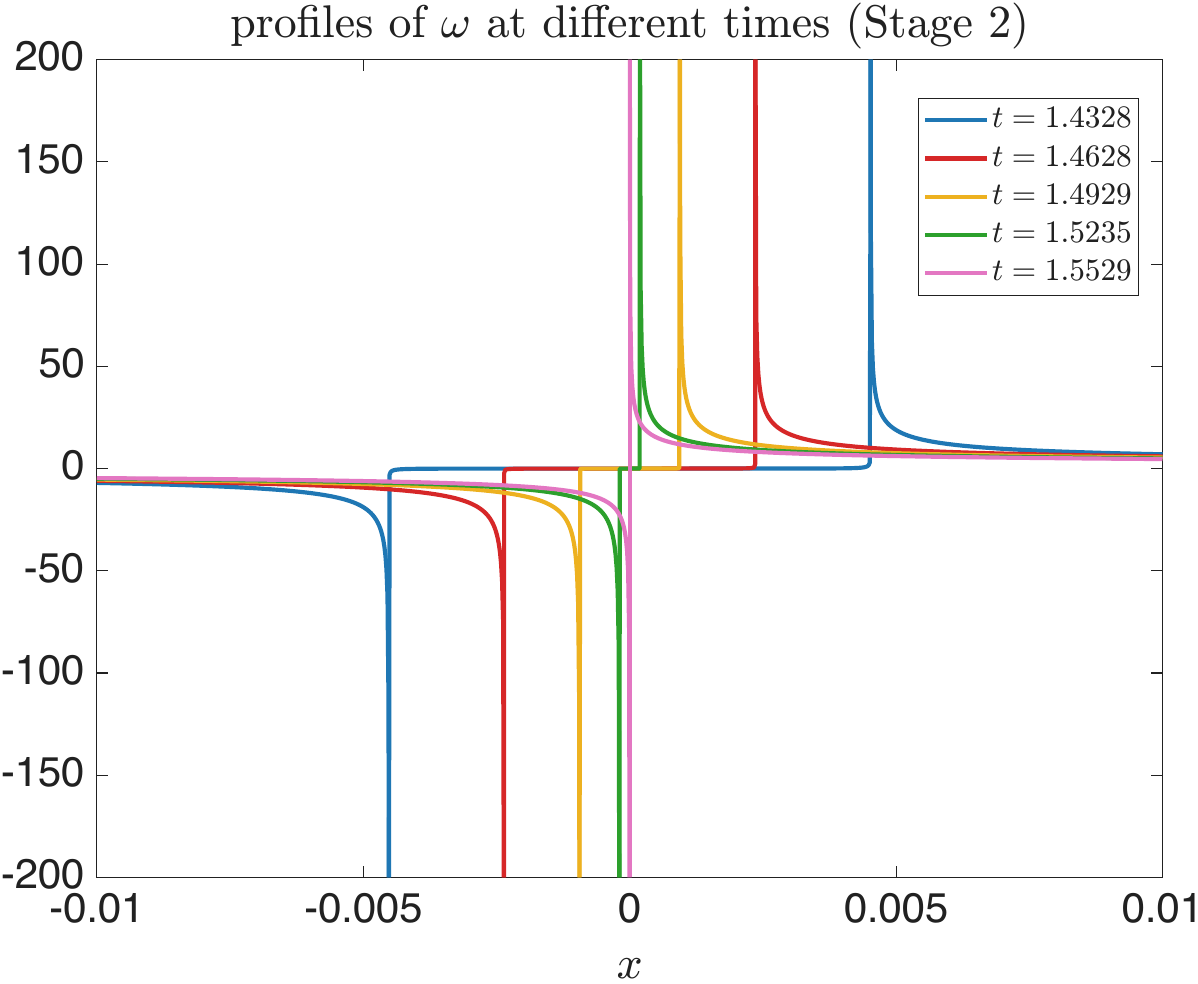}
        \caption{Evolution of $\om(\cdot,t)$ starting from odd-symmetric degenerate initial data. Left (Stage 1): $\om$ undergoes a finite-time $L^{\infty}$ blowup at two symmetric points with respect to $x=0$. This contrasts with the non-degenerate case, where the blowup occurs at the origin. Right (Stage 2): Beyond the first blowup time, the solution continues in the weak sense. The singularities of the weak solution are transported to the origin in finite time, leading to a local $L^p$ blowup near the origin for some $p>0$, where $p$ is determined by the corresponding limiting self-similar profile.}
          \label{fig:HL_physicalspace_profiles_degenerate}
    \end{subfigure}
 \caption[Evolution of $\om(\cdot,t)$.]{Evolution of $\om(\cdot,t)$ with different settings of initial data.}  
     \label{fig:HL_physicalspace_profiles}
\end{figure}

\begin{figure}[!htbp]
\centering
    \begin{subfigure}[b]{0.4\textwidth}
        \includegraphics[width=0.95\textwidth]{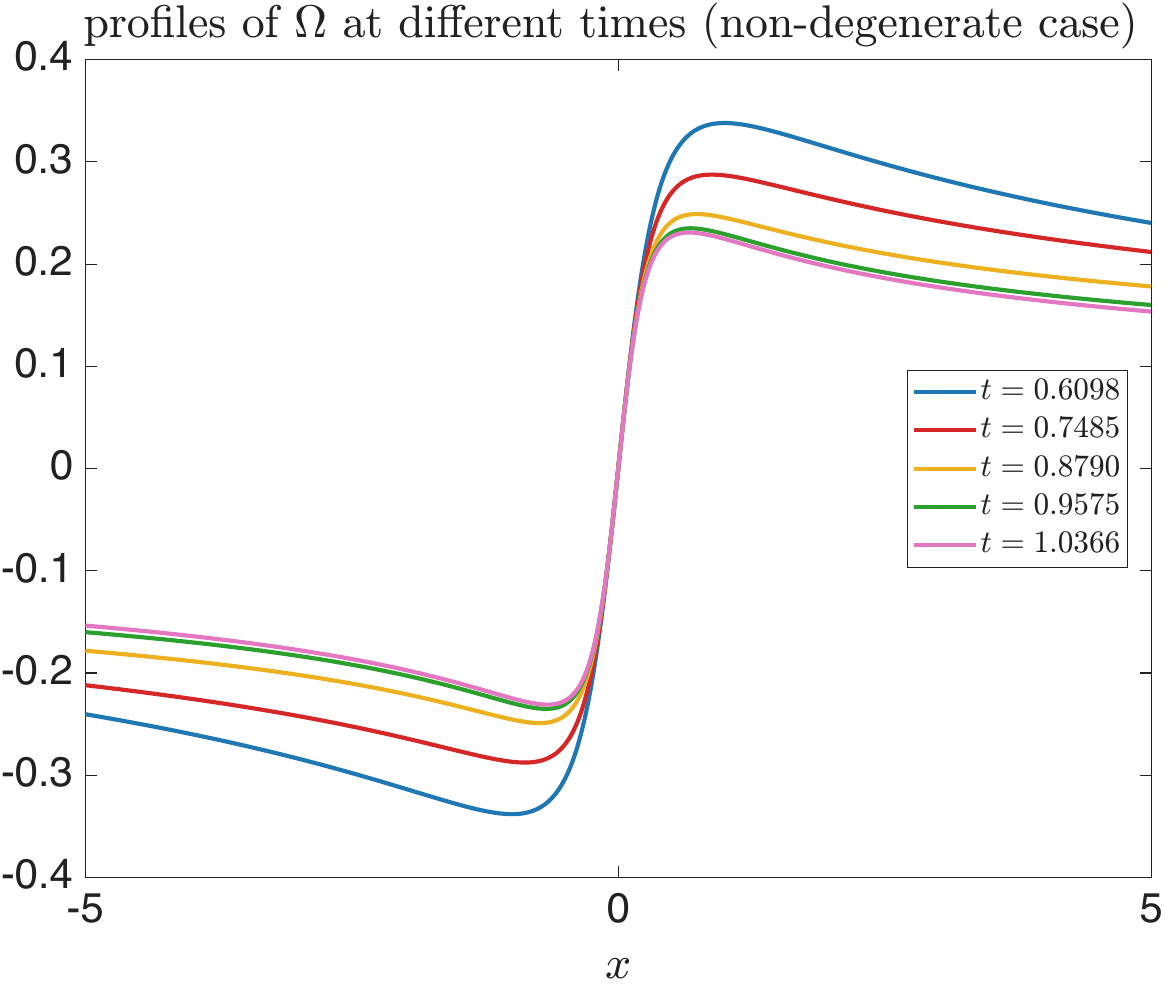}
          \caption{Non-degenerate case.}
         \label{fig:HL_rescaling_profiles_nondegenerate}
    \end{subfigure}
    \begin{subfigure}[b]{0.4\textwidth}
        \includegraphics[width=0.95\textwidth]{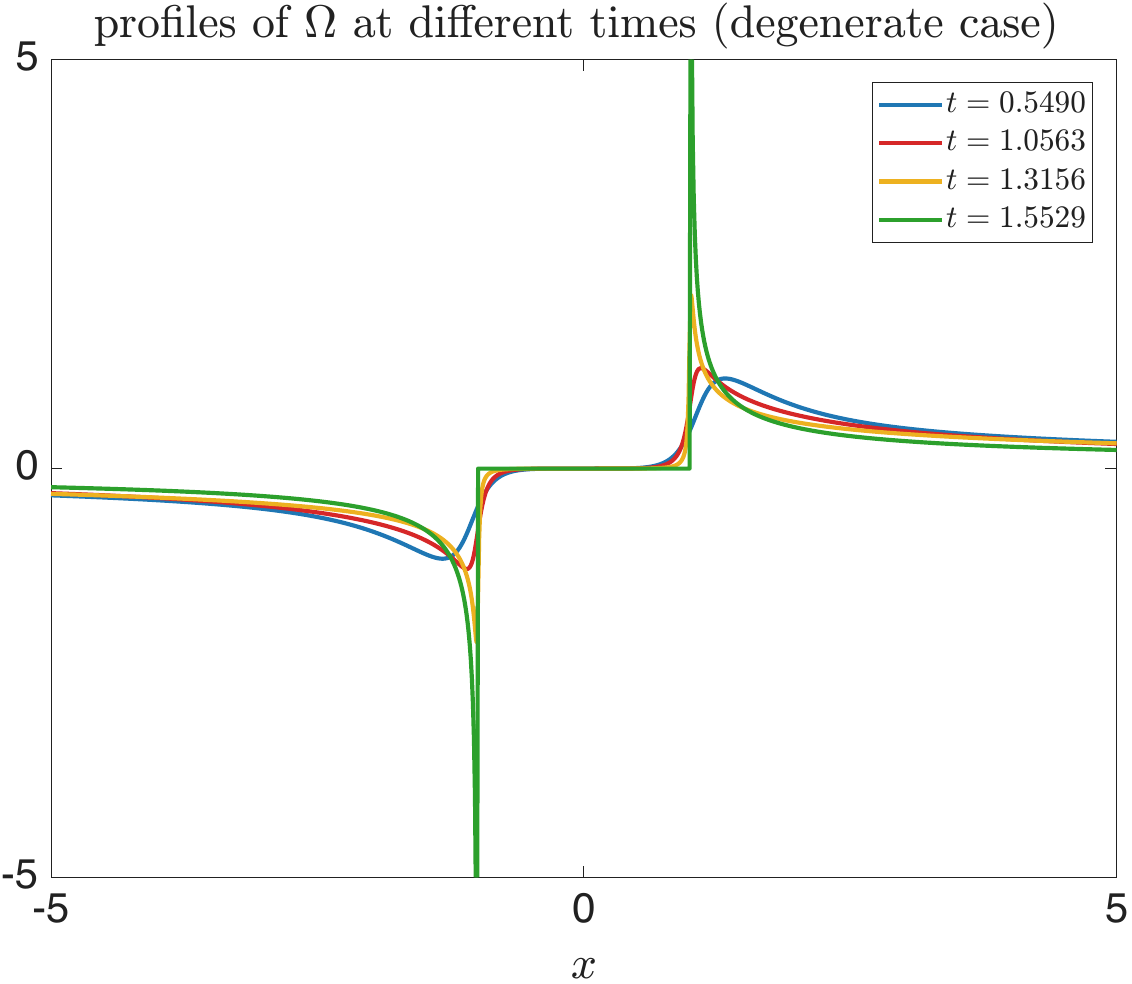}
         \caption{Degenerate case.}
        \label{fig:HL_rescaling_profiles_degenerate}
    \end{subfigure}
 \caption[Evolution of $\Omega(\cdot,t)$.]{Evolution of the profile $\Omega(\cdot,t)$. In the non-degenerate case (left), $\Omega$ converges to a regular profile, while in the degenerate case (right), $\Omega$ converges to a singular profile.}  
     \label{fig:HL_rescaling_profiles}
\end{figure}

Interestingly, we find that this two-stage self-similar blowup persists across a wider class of degenerate initial data. Specifically, if the odd symmetry restriction (which is a common assumption in previous studies) is removed, we can construct exact singular self-similar profiles for the HL model that admit explicit expressions. Furthermore, our numerical results suggest that these explicit singular profiles are asymptotically stable.
That is, starting from suitable smooth degenerate initial data, the rescaled variables $\Omega$ and $\Theta_X$ converge to the explicit singular profiles as $t$ approaches the Stage 2 blowup time $T$. This corroborates the robustness of our novel blowup mechanism, showing that it is not a specialized phenomenon tied to specific symmetric constraints.

It is worth noting that the HL model \eqref{eqt:1Dhouluo} also models the boundary induced singularity formation of the 2D Boussinesq equations \cite{luo2014toward,chen2022asymptotically} in the half-space $(x_1,x_2)\in\R\times\R_+$, 
\begin{equation}\label{eqt:Boussinesq}
\begin{split}
&\om_t + u_1\om_{x_1} + u_2\om_{x_2} = \theta_{x_1},\\
&\theta_t + u_1\theta_{x_1} + u_2\theta_{x_2} = 0,\\
&(u_1,u_2) = \nabla^{\perp}(-\Delta)^{-1}\om,
\end{split}
\end{equation}
for the 2D Boussinesq equations behave similarly to the 3D axisymmetric Euler equations away from the symmetry axis; see e.g. \cite{majda2002vorticity}. 
In fact, the boundary finite-time blowup of the 3D axisymmetric Euler equations can be approximated by the boundary finite-time blowup of the 2D Boussinesq equations up to an asymptotically small perturbation \cite{luo2014toward,liu2017spatial,chen2022stable}. In this paper, we also extend our numerical investigation to the 2D Boussinesq equations starting from degenerate initial data. Our results demonstrate that the 2D Boussinesq equations exhibit blowup behaviors similar to those of the HL model: starting from degenerate initial data, \eqref{eqt:Boussinesq} can develop self-similar finite-time blowups with singular profiles. Furthermore, this blowup phenomenon also exhibits a two-stage structure: the solution first develops an asymptotically self-similar $L^{\infty}$ blowup with a regular profile, and subsequently continues in the weak sense to develop an asymptotically self-similar $L^p$ blowup with a singular profile (Figure \ref{fig:BS_rescaling_profiles}(b)) for some $p>0$. 

Our blowup scenario differs significantly from existing theoretical and numerical results on the self-similar singularity formation of the 2D Boussinesq equations \cite{liu2017spatial,chen2022stable,wang2023asymptotic,wang2025discovery,chen2025singularity} (see Figure \ref{fig:BS_rescaling_profiles}(a)).
Notably, Chen and Hou \cite{chen2022stable,chen2025singularity} rigorously established, via a computer-assisted proof, a nearly self-similar blowup solution of \eqref{eqt:Boussinesq} with non-degenerate regular self-similar profiles that are odd in the $x_1$ direction. In contrast, our scenario describes a two stage self-similar blowup of \eqref{eqt:Boussinesq} starting from degenerate initial data, where the first stage admits regular self-similar profiles that do not possess any symmetry.
We remark that, though Liu \cite{liu2017spatial} also conducted a preliminary numerical study of \eqref{eqt:Boussinesq} using degenerate initial data, his results did not capture the two-stage blowup mechanism reported in this paper.
Very recently, Alp{\"o}ge and Buckmaster \cite{alpogebuckmasterboussinesq} constructed finite-time blowup from smooth initial data for the 2D Boussinesq equations with a smooth external force. Their blowup mechanism roots in the amplification of localized oscillations at successively smaller spatial scales under an external force, which produces a non-self-similar blowup in $\R^2$ (in the absence of boundary). In contrast, our numerical results suggest a boundary-driven, two-stage self-similar blowup for the unforced 2D Boussinesq equations in the half-plane $\R\times\R_+$.
\begin{figure}[!htbp]
\centering
    \begin{subfigure}[t]{0.42\textwidth}
        \includegraphics[width=1\textwidth]{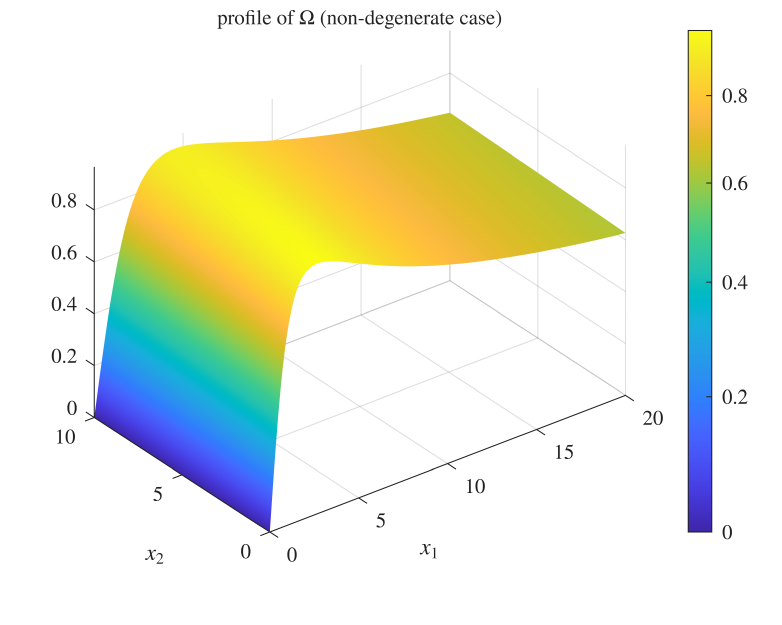}
        \caption{Non-degenerate case.}
         \label{fig:BS_rescaling_profiles_nondegenerate}
    \end{subfigure}
     \begin{subfigure}[t]{0.45\textwidth}
    \centering
    \raisebox{0ex}{%   
        \includegraphics[width=\linewidth]{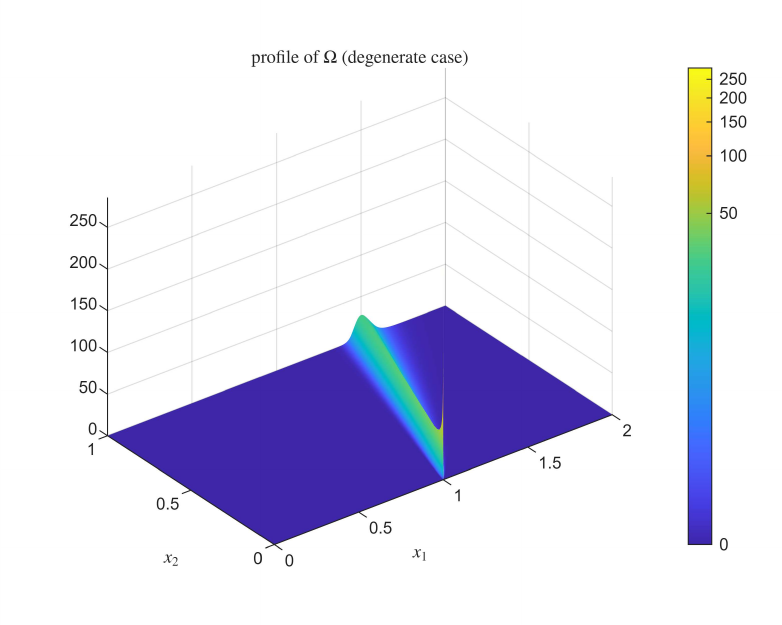}
    }
    \caption{Degenerate case.}
    \label{fig:BS_rescaling_profiles_degenerate}
\end{subfigure}
 \caption[Self-similar profiles of the 2D Boussinesq equations.]{Self-similar profiles of the 2D Boussinesq equations in different settings. The left panel illustrates a regular profile arising from non-degenerate initial data, as studied in the existing literature. The right panel presents a novel singular profile evolved from degenerate initial data.}  
     \label{fig:BS_rescaling_profiles}
\end{figure}

Considering the aforementioned connection between the 2D Boussinesq equations and the 3D axisymmetric Euler equations, let us also briefly review recent progress on the study of finite-time singularity formation for the 3D incompressible Euler in the whole $\R^3$. Elgindi \cite{elgindi2021finite,elgindi2022stability} and C{\'o}rdoba, Mart\'inez-Zoroa, and Zheng \cite{cordoba2025finite} constructed finite-time blowups (of vorticity) in the axisymmetric no-swirl setting from $C^{1,\alpha}$ initial velocity for sufficiently small $\alpha>0$. Elgindi's construction is self-similar, whereas the construction of C{\'o}rdoba, Mart\'inez-Zoroa, and Zheng builds on a cascading of vortexes and therefore is not self-similar. Later, independent works by Shkoller \cite{shkoller2026threshold}, Chen \cite{chen2026threshold}, and Shao, Wei, Zhang, and Zhang \cite{shao2026selfsimilar} obtained three different scenarios of finite-time blowup (with the latter two being typically self-similar) in the axisymmetric no-swirl setting with $C^{1,\alpha}\cap L^2$ initial velocity for $\alpha$ arbitrarily close to the regularity threshold $1/3$. Indeed, finite-time blowup has been excluded for $\alpha\geq 1/3$ (e.g. see \cite{shao2026global}). Astonishingly, OpenAI \cite{openai2026euler} recently claimed the existence of finite-time blowup from $C_c^{\infty}$ initial velocity for the 3D incompressible Euler equations on the whole $\R^3$. Their construction builds on a multiscale-amplification mechanism, an idea that was previously developed in the constructions of blowup by C{\'o}rdoba, Mart\'inez-Zoroa, and Zheng \cite{cordoba2025finite} and by C{\'o}rdoba and Mart{\'i}nez-Zoroa \cite{cordoba2023forced}. Nevertheless, the physical relevance of such a construction remains unclear. Therefore, it is still worthwhile to investigate potential finite-time blowup solutions to the 3D incompressible Euler equations that have clearer structural features such as self-similarity.

Given the ongoing efforts to understand the singularity formation of the 3D Euler equations, we hope that our numerical findings of the HL model and the 2D Boussinesq equations can provide some inspiration for future investigations into self-similar singularity formation of the 3D Euler equations on $\R^3$ without boundary. It is worth mentioning that the singular profiles reported in this paper share qualitative similarities with those observed in recent numerical simulations by Hou and Huang \cite{hou2023potential}, which suggested a potential two-scale self-similar blowup of 3D axisymmetric Euler equations from smooth initial data in the absence of a boundary. Whether these phenomena are fundamentally linked remains an intriguing question for further exploration.

The rest of this paper is organized as follows. In Section \ref{sec:dynamic_rescaling_of_HL} we introduce the dynamic rescaling formulation of the HL model \eqref{eqt:1Dhouluo}, review previous results on singularity formation of the HL model and propose novel blowup scenarios of the HL model. Section \ref{sec:HL_scenario1} presents numerical evidence demonstrating that the solutions of the HL model starting from degenerate initial data can develop self-similar blowups with singular profiles that exhibit a two-stage structure. We then conduct a deeper investigation into the Stage 1 blowup in Section \ref{sec:HL_scenario2}, providing numerical evidence that this scenario is related to the self-similar blowup of the HL model with positive regular profiles. A family of singular self-similar profiles of the HL model is constructed in Section \ref{sec:HL_steady_state}. Section \ref{sec:dynamic_rescaling_of_BS} extends the numerical study to the 2D Boussinesq equations \eqref{eqt:Boussinesq} and presents similar blowup scenarios as those of the HL model. Finally, the implementation details of the numerical computations are provided in Appendix \ref{sec:HL_numerical method}.

\section{Dynamic Rescaling and Self-similar Blowups of 1D Hou--Luo Model}\label{sec:dynamic_rescaling_of_HL}
In this section, we introduce the dynamic rescaling formulation for the HL model \eqref{eqt:1Dhouluo}, which transforms the challenge of investigating finite-time blowup into the study of the convergence and stability of some steady state. We then review established results concerning singularity formation of \eqref{eqt:1Dhouluo} and introduce the novel blowup scenarios discovered in this work.

Before we proceed, let us first clarify how to recover the velocity $u$ from the vorticity $\om$ in \eqref{eqt:1Dhouluo}. Note that $u_x$ is determined by the Hilbert transform of $\om$:
\[ u_x(x,t) = \mtx{H}(\om(\cdot,t))(x)= \frac{1}{\pi}\mathrm{P.V.}\int_{\R}\frac{\om(y,t)}{x-y}\idiff y,\]
where $\mathrm{P.V.}$ denotes the Cauchy principal value. To recover $u$ from the relation $u_x=\mtx{H}(\om)$, a direct approach is to compute 
\begin{equation}\label{eqt:1DHL_compute_u}
u(x,t)=-(-\Delta)^{-1/2}(\om(\cdot,t) )(x)=\frac{1}{\pi}\int_{\R}\om(y,t)\ln|x-y|\idiff y.
\end{equation}
When $\omega$ possesses odd symmetry, $u$ can be further simplified to
\[u(x,t)=\frac{1}{\pi}\int_{0}^{+\infty}\om(y,t)\ln\left|\frac{x-y}{x+y}\right|\idiff y. \]
In this case, $u$ is well-defined as long as $\om$ exhibits a slow decay since the kernel $\ln|(x-y)/(x+y)|$ behaves like $O(1/|y|)$ in the far field. However, this is not true for the general case as the kernel $\ln|x-y|$ diverges at infinity. Thus, \eqref{eqt:1DHL_compute_u} is only valid when $\om$ decays sufficiently fast in the far field. To resolve this for general data, we may choose a reference point $x_0\in \R$ where we impose $u(x_0)=0$ and integrate from $x_0$ to $x$ to obtain
\[u(x,t)=\frac{1}{\pi}\int_{\R}\om(y,t)\ln\left|\frac{x-y}{x_0-y}\right|\idiff y. \]
This integral remains well-defined provided that $\om$ satisfies mild decay conditions. Furthermore, the following lemma shows that solutions with the same initial data but different choices of $x_0$ are equivalent modulo a proper time-dependent spatial translation.
\begin{lemma}\label{lem:translation_invariance}
Let $(\om_0,\theta_0)$ be some suitable initial data, and let $x_1(t),x_2(t)$ be two continuous trajectories. Suppose that $(\om_i,\theta_i,u_i)_{i=1,2}$ is the solution of \eqref{eqt:1Dhouluo} that satisfies $u_i(x_i(t))=0$. Let $c(t)$ be a time-dependent quantity that satisfies the following ODE:
\[c'(t)=-u_1(x_2(t)-c(t),t)=-\frac{1}{\pi}\int_{\R}\om_1(y,t)\ln\left|\frac{x_2(t)-c(t)-y}{x_1(t)-y} \right|\idiff y, \quad c(0)=0.\]
Then it holds that
\[(\om_2(x,t),\theta_2(x,t),u_2(x,t))=(\om_1(x-c(t),t),\theta_1(x-c(t),t),u_1(x-c(t),t)-u_1(x_2(t)-c(t),t)).\]
\end{lemma}
The proof of this lemma will be deferred to Appendix \ref{sec:HL_translation_invariance}.
\subsection{Dynamic rescaling formulation}
With Lemma \ref{lem:translation_invariance}, we can assume that $u(0)=0$ and apply the change of variables  
\begin{equation}\label{eqt:dynamic_rescaling_of_HL_change_of_variables1}
\begin{aligned}
    &\omega(x,t)=C_{\om}(\tau)^{-1}\Omega\left(C_l(\tau)x,\tau(t)\right),\quad \theta(x,t)=C_{\theta}(\tau)^{-1}\Theta\left(C_l(\tau)x,\tau(t)\right),\\ &u(x,t)=C_{\om}(\tau)^{-1}C_{l}(\tau)^{-1}U\left(C_l(\tau)x,\tau(t)\right)
\end{aligned}
\end{equation}
with
\begin{equation}\label{eqt:dynamic_rescaling_of_HL_change_of_variables2}
\begin{aligned}
    C_{\om}(\tau)&=\exp\left(\int_0^\tau c_{\om}(s)\idiff s\right) ,\quad  C_{\theta}(\tau)=\exp\left(\int_0^\tau c_{\theta}(s) \idiff s\right),\\  C_{l}(\tau)&=\exp\left(\int_0^\tau c_l(s)\idiff s\right), \quad  t(\tau)=\int_0^{\tau}C_{\om}(s)\idiff s, \quad c_{\theta}=c_l+2c_{\om}
\end{aligned}
\end{equation}
to reformulate the original equations \eqref{eqt:1Dhouluo} into the following equivalent form: 
\begin{equation}\label{eqt:dynamic_rescaling_of_HL}
    \begin{aligned}
    &\Omega_\tau + (U+c_lX)\Omega_X=c_{\om}\Omega+\Theta_X,\\
    &\Theta_\tau+ (U+c_lX)\Theta_X=(c_l+2c_{\om})\Theta,\\
    &U_X =\mtx{H}(\Omega),\,\ U(0)=0,
    \end{aligned}
\end{equation}
where $X=C_l x$.
We refer to \eqref{eqt:dynamic_rescaling_of_HL} as the dynamic rescaling formulation of \eqref{eqt:1Dhouluo}. We remark that \eqref{eqt:dynamic_rescaling_of_HL} is completely equivalent to \eqref{eqt:1Dhouluo} under the change of variables given by \eqref{eqt:dynamic_rescaling_of_HL_change_of_variables1} and \eqref{eqt:dynamic_rescaling_of_HL_change_of_variables2} with the same initial condition.

We are also particularly interested in steady state solutions to \eqref{eqt:dynamic_rescaling_of_HL} that satisfy the self-similar profile equation
\begin{equation}\label{eqt:HL_steady_state}
    \begin{aligned}
    &(U+c_lX){\Omega}_X=c_{\om}{\Omega}+{\Theta}_X,\\
    &(U+c_lX){\Theta}_X=(c_l+2c_{\om}){\Theta},\\
    &U_X =\mtx{H}({\Omega}),\,\ U(0)=0.
    \end{aligned}
\end{equation}
The reader can easily verify that if $(\bar{\Omega},\bar{\Theta},\bar c_l,\bar c_\omega)$ is a solution to \eqref{eqt:HL_steady_state}, then it corresponds to exact self-similar finite-time blowups of the HL model in form of \eqref{eqt:HL_exact_blowup} with $\gamma=-\bar{c}_l/\bar{c}_\omega$, $\lambda=-1$ and $\mu=\gamma-2$.
Note that if $(\bar{\Omega}(X),\bar{\Theta}(X),\bar{c}_l,\bar{c}_{\om})$ is a steady state solution of \eqref{eqt:dynamic_rescaling_of_HL}, then
\begin{equation}\label{eqt:scaling_invariance}
(\alpha \bar{\Omega}(\beta X),\alpha^2 \bar{\Theta}(\beta X)/\beta,\alpha \bar{c}_l,\alpha \bar{c}_{\om})
\end{equation} is also a steady state solution for any $\alpha\in \R$, $\beta>0$, which indicates that this equation admits two degrees of freedom. Hence, in order to uniquely identify the profile up to rescaling, we need to impose two normalization conditions to fix these two degrees of freedom.

Moreover, if under some normalization conditions the solution tuple $(\Omega,\Theta,c_l,c_{\omega})$ converges to a steady state $(\bar{\Omega},\bar{\Theta},\bar{c_l},\bar{c}_{\om})$ of \eqref{eqt:dynamic_rescaling_of_HL} as $\tau\to\infty$, then the solution would blow up in an asymptotically self-similar manner \eqref{eqt:asymptotic_selfsimilar_HL}. We demonstrate this more precisely by the following lemma.

\begin{lemma}\label{lem:DLstability_implies_blowup}
    Suppose \eqref{eqt:dynamic_rescaling_of_HL} converges to some steady state $(\bar{\Omega},\bar{\Theta},\bar{c_l},\bar{c}_{\om})$ with $\bar{c}_{\om}<0$. Furthermore, assume the parameters $c_l$ and $c_{\omega}$ converge sufficiently fast in the sense that
   \begin{equation}\label{eqt:convergence_of_clcomega}
    \int_{0}^{+\infty}|c_l(\tau)-\bar{c}_l|\idiff \tau<\infty, \quad \int_{0}^{+\infty}|c_{\om}(\tau)-\bar{c}_{\om}|\idiff \tau<\infty.
   \end{equation}
    Then $\om(x,t)$ develops a self-similar finite-time blowup asymptotically. More specifically, there exists a constant $C_0>0$ such that for any $X\in \R$ near which $\Omega$ converges to $\bar{\Omega}$ uniformly, it holds that
    \begin{equation}\label{eqt:DLstability_implies_blowup}
    \lim_{t\to T}(T-t)\om\left(\frac{(T-t)^\gamma X}{C_0},t\right)=-\frac{1}{\bar{c}_{\om}}\bar{\Omega}(X),
    \end{equation}
where \[T=\int_0^{+\infty}C_{\om}(s)\idiff s, \quad \gamma= -\frac{\bar{c}_{l}}{\bar{c}_{\omega}}.\]
\end{lemma}
\begin{proof}
Let $C_0>0$ be a constant that will be determined later. According to the change of variables \eqref{eqt:dynamic_rescaling_of_HL_change_of_variables1} and \eqref{eqt:dynamic_rescaling_of_HL_change_of_variables2}, we have 
\[(T-t)\om\left(\frac{(T-t)^\gamma X}{C_0},t\right)=\frac{T-t}{C_{\om}(\tau)}\Omega\left(\frac{C_l(\tau)(T-t)^{\gamma}X}{C_0},\tau\right).\]
We first compute that 
\[\lim_{t\to T}\frac{T-t}{C_{\om}(\tau)}= \lim_{\tau\to +\infty}\frac{\int_{\tau}^{+\infty}C_{\om}(s)\idiff s}{C_{\om}(\tau)}=\lim_{\tau\to +\infty}\frac{-C_{\om}(\tau)}{C_{\om}(\tau)c_{\om}(\tau)}=-\frac{1}{\bar{c}_{\om}},\]
where we have used L'Hospital's rule in the second equality. Furthermore, we have 
\begin{align*}\lim_{t\to T}\left(T-t\right)^{\gamma}{C_{l}} &=\lim_{t\to T}\left(\frac{-1}{\bar{c}_{\om}}\right)^{\gamma}C_{\om}^{\gamma}C_l=\left(\frac{-1}{\bar{c}_{\om}}\right)^{\gamma}\lim_{\tau\to +\infty}\exp\left(\int_{0}^{\tau}\left(c_l(s)+\gamma c_{\om}(s)\right) \idiff s\right)\\
& =\left(\frac{-1}{\bar{c}_{\om}}\right)^{\gamma}\lim_{\tau\to +\infty}\exp\left(\int_{0}^{\tau}\left(c_l(s)-\bar{c}_l+\gamma (c_{\om}(s)-\bar{c}_{\om})\right) \idiff s\right)\\
& =\left(\frac{-1}{\bar{c}_{\om}}\right)^{\gamma}\exp\left(\int_{0}^{+\infty}\left(c_l(s)-\bar{c}_l+\gamma (c_{\om}(s)-\bar{c}_{\om})\right) \idiff s\right),
 \end{align*}
where the convergence of the last integral is due to the assumption that $c_l(\tau)$ and $c_{\om}(\tau)$ converge sufficiently fast as in \eqref{eqt:convergence_of_clcomega}. Moreover, by defining the constant 
 \[ C_0:=\left(\frac{-1}{\bar{c}_{\om}}\right)^{\gamma}\exp\left(\int_{0}^{+\infty}\left(c_l(s)-\bar{c}_l+\gamma (c_{\om}(s)-\bar{c}_{\om})\right) \idiff s\right),\]
 we obtain the claimed result.
\end{proof}

\begin{remark}\label{rem:HL_critical_Lp_exponent}
The conclusion of Lemma \ref{lem:DLstability_implies_blowup} describes an asymptotically self-similar blowup of $\om$ at some finite time $T$ as
\[\om(x,t)\sim (T-t)^{-1}\bar{\Omega}\left(\frac{x}{(T-t)^{\gamma}}\right).\]
This, if happens, immediately implies the $L^{\infty}$ blowup of $\om$. Moreover, one can also study local $L^p$ blowup of $\om$. Suppose that $(\Omega,\Theta,c_l,c_{\omega})$ converges to $(\bar{\Omega},\bar{\Theta},\bar{c_l},\bar{c}_{\om})$ sufficiently fast, then for any $p>0$, 
\begin{equation}\label{eqt:HL_asymptotoicallyblowup}
\|\om(\cdot,t)\|_{L^{p}(-1,1)}\sim  (T-t)^{-1} \left\|\bar{\Omega}\left(\frac{\cdot}{(T-t)^{\gamma}}\right)\right\|_{L^p(-1,1)}= (T-t)^{-1+\gamma/p}\left\|\bar{\Omega}\right\|_{L^p(-1/(T-t)^{\gamma},1/(T-t)^{\gamma})}.   
\end{equation}
   It is not hard to see that \eqref{eqt:HL_asymptotoicallyblowup} implies a local $L^p$ blowup of $\om$ near the origin for any $p\in (\gamma,+\infty]$. Moreover, the steady-state equations \eqref{eqt:HL_steady_state} suggest that $|\bar{\Omega}(X)|$ decays like $|X|^{-1/\gamma}$ in the far field. If we assume that $\bar{\Omega}$ is a regular profile with this far-field decay rate, then 
   \[\left\|\bar{\Omega}\right\|_{L^p(-1/(T-t)^{\gamma},1/(T-t)^{\gamma})}\sim 
   \begin{cases}
      1, \quad  & p\in (\gamma,+\infty),\\
      (-\ln (T-t))^{1/p},\quad & p=\gamma,\\
      (T-t)^{1-\gamma/p},\quad & p\in (0,\gamma),
       \end{cases}
    \]
    and therefore 
    \[\|\om(\cdot,t)\|_{L^{p}(-1,1)}\sim
   \begin{cases}
      (T-t)^{-1+\gamma/p}, \quad  & p\in (\gamma,+\infty),\\
      (-\ln (T-t))^{1/p},\quad & p=\gamma,\\
      1,\quad & p\in (0,\gamma).
       \end{cases}
    \]
In this case, $\gamma$ is a critical exponent in the sense that $\om$ develops a local $L^p$ blowup at $T$ for any $p\in [\gamma,+\infty]$, while its local $L^p$ quantity remains bounded up to $T$ for any $p\in (0,\gamma)$. We remark that the $L^p$ blowup discussed above is due to scaling, yet there can be other reasons that account for $L^p$ blowups of $\om$. For example, if the profile $\bar{\Omega}$ itself has a spatial singularity at some point, then $\om$ may develop a local $L^p$ blowup at $T$ for a smaller $p<\gamma$ due to the local unintegrability of $\bar{\Omega}$.
\end{remark}

\subsection{Finite-time singularity of the HL model: existing results} Next, we review existing results on singularity formation of the HL model \eqref{eqt:1Dhouluo}.
\begin{itemize}
    \item Shortly after the original work of Luo and Hou \cite{luo2014toward}, Choi et al. \cite{choi2017finite} used a functional argument to prove the finite-time blowup of the HL model \eqref{eqt:1Dhouluo}. Under mild assumptions on the initial data, they showed that the functional 
    \[\int_0^{+\infty}\theta_x \ln x \idiff x\] blows up in finite time, which implies the finite-time blowup of \eqref{eqt:1Dhouluo}. However, their approach was not able to capture the self-similar nature of the blowup.
    \item Years later, Chen, Hou, and Huang \cite{chen2022asymptotically} established asymptotically self-similar blowups of \eqref{eqt:1Dhouluo} from smooth initial data by proving the asymptotic stability of the dynamic rescaling equations \eqref{eqt:dynamic_rescaling_of_HL} around some steady state. The initial data $(\Omega^0, \Theta_X^0)$ considered in their work are smooth functions with odd symmetry that are non-degenerate at the origin in the sense that $\Omega^0_X(0) \neq 0$ and $\Theta_{XX}^0(0) \neq 0$. To fix the degrees of freedom in the rescaling formulation, the normalization conditions were chosen to ensure that
\[ \partial_{\tau}\Omega_X(0)=\partial_{\tau}\Theta_{XX}(0)\equiv 0.\]
Imposing the above normalization conditions, they first performed numerical computation of \eqref{eqt:dynamic_rescaling_of_HL} to construct an approximate steady state $(\tilde{\Omega},\tilde{\Theta}_X,\tilde{c}_{l},\tilde{c}_{\om})$. They subsequently constructed a weighted $H^1$ norm $\|\cdot\|$ and demonstrated that the solution $(\Omega, \Theta_X, c_l, c_{\omega})$ remains within a small neighborhood of $(\tilde{\Omega},\tilde{\Theta}_X,\tilde{c}_{l},\tilde{c}_{\om})$ in this norm, provided the initial data is sufficiently close to the approximate steady state. Furthermore, they also showed that $(\Omega, \Theta_X, c_l, c_{\omega})$ converges to a true stable steady state $(\bar{\Omega},\bar{\Theta}_X,\bar{c}_l,\bar{c}_{\om})$ that corresponds to exact self-similar profiles of \eqref{eqt:1Dhouluo}, thus establishing the asymptotically self-similar blowup of \eqref{eqt:1Dhouluo} from odd-symmetric non-degenerate initial data.

\item  Recently, the existence of exact self-similar profiles of the HL model was established alternatively by Huang, Qin, Wang, and Wei \cite{huang2025exact} via a fixed-point method which is purely analytic. By performing a detailed study of the fixed-point map, the authors provided finer characterizations of the self-similar profiles. Specifically, they proved the existence of a steady-state solution $(\bar{\Omega}, \bar{\Theta}_X, \bar{c}_l, \bar{c}_{\omega})$ that satisfies the steady-state equations \eqref{eqt:HL_steady_state}. Furthermore, the following properties of self-similar profiles were established:
\begin{enumerate}
    \item Symmetry and non-degeneracy: $\bar{\Omega}$ and $\bar{\Theta}_X$ are odd-symmetric functions that are non-degenerate at the origin.
    \item Regularity: $\bar{\Omega}$ and $\bar{\Theta}_X$ are both infinitely smooth on $\R$. Moreover, $\bar{\Omega}\in L^{\infty}(\R)\cap L^q(\R)\cap \dot{H}^p(\R)$ for any $q>\gamma$ and any $p\geq 1$, and $\bar{\Theta}_X\in L^{\infty}(\R)\cap L^q(\R)\cap \dot{H}^p(\R)$ for any $q>\gamma/2$ and any $p\geq 1$, where $\gamma=-\bar{c}_l/\bar{c}_{\om}$. 
    \item Far-field decay: Both the limits $\lim_{X\to +\infty}X^{1/\gamma}\bar{\Omega}(X)$ and $\lim_{X\to +\infty}X^{2/\gamma}\bar{\Theta}_X(X)$ exist and are finite.
\end{enumerate}
Moreover, the authors also conducted numerical simulations and demonstrated that the steady state obtained via their fixed-point iteration scheme matches perfectly with that constructed in \cite{chen2022asymptotically}. This strongly suggested that both works characterized a robust blowup scenario for the HL model: starting from odd-symmetric non-degenerate initial data, \eqref{eqt:1Dhouluo} can develop self-similar blowups with non-degenerate regular profiles. Here, by ``regular'' we mean that the profile is a bounded continuous function.
\end{itemize}

\subsection{Singular profiles and multi-scale self-similar blowups}
Note that in the non-degenerate case, the self-similar blowups of the HL model have been well understood through the aforementioned works \cite{chen2022asymptotically,huang2025exact}. In contrast, the understanding of the HL model starting from more degenerate (in the sense that $\omega^0_x(0)=\theta_{xx}^0(0)=0$) initial data remains very limited. Early numerical evidence presented in \cite{liu2017spatial} indicated that solutions of \eqref{eqt:1Dhouluo} with degenerate initial data will blow up in a very different way. More specifically, under suitable normalization conditions, Liu's numerical results suggested that the solutions of \eqref{eqt:dynamic_rescaling_of_HL} starting from smooth degenerate initial data converge to singular profiles as $\tau\rightarrow +\infty$. Here, ``singular'' denotes a profile that is locally unbounded at some specific point. In the context of the original HL model \eqref{eqt:1Dhouluo}, this corresponds to finite-time blowups at $x=0$ with singular self-similar profiles.  

It is worth mentioning that a mechanism for the formation of self-similar blowups with singular profiles was recently proposed in \cite{huang2025multiscale}, where the authors established asymptotic self-similar blowups of the Constantin-Lax-Majda (CLM) model for degenerate initial data that exhibit a two-scale feature:
\begin{equation}\label{eqt:multi_scale_blowup}
    \om(x,t)=(T-t)^{\hat{\lambda}}\left(\Omega\left(\frac{x-r(t)(T-t)^\gamma}{(T-t)^{\hat{\gamma}}}\right)+o(1)\right).
\end{equation}
In this formulation, $\Omega$ represents a regular profile. The larger spatial scale $(T-t)^{\gamma}$ (with a smaller power $\gamma>0$) measures the distance between the origin and the center location of the bulk part of the solution, while the smaller spatial scale $(T-t)^{\hat{\gamma}}$ (with a larger power $\hat{\gamma}>0$) measures the size of the bulk part of the solution. $r(t)$ is a continuous function that captures the precise location of the bulk part in the larger scale. It is important to note that for the CLM model, it holds that $\hat{\lambda}<-1$ in \eqref{eqt:multi_scale_blowup}. As a result, if one observes the solution at the larger spatial scale $(T-t)^{\gamma}$ with the usual magnitude scaling $(T-t)^{-1}$ (which is equivalent to $C_{\om}^{-1}$ in the change of variables \eqref{eqt:dynamic_rescaling_of_HL_change_of_variables2}), a singular profile emerges.
We refer the readers to \cite{huang2025multiscale} for a comprehensive introduction to multi-scale self-similar blowups and their connection to traveling wave solutions of the original equations. This two-scale blowup phenomenon was also conjectured by Liu \cite{liu2017spatial} for the HL model with degenerate initial data.

Therefore, based on these observations and conjectures, it is natural to pose the following questions regarding the HL model with degenerate initial data:
\begin{itemize}
    \item Will the HL model develop finite-time blowups in a way that is distinct from the non-degenerate case?
    \item If so, will the HL model develop blowups with singular self-similar profiles?
    \item If singular profiles are observed, is their formation related to a two-scale  blowup mechanism?
\end{itemize}

\subsection{Scenario 1: novel asymptotically self-similar blowups with singular profiles.}

To address the questions raised above, in Section \ref{sec:HL_scenario1} we conduct a detailed numerical investigation of the dynamic rescaling equations \eqref{eqt:dynamic_rescaling_of_HL}. Specifically, we consider two distinct classes of initial data exhibiting degeneracy at the origin:
\begin{itemize}
    \item Odd symmetry case: $\Omega_0$ and $\Theta_{0,x}$ are odd symmetric functions which are non-negative on $[0,+\infty)$.
    \item One-sided case: $\Omega_0$ and $\Theta_0$ are non-negative functions such that $\operatorname{supp}(\Omega_0)\subset[0,+\infty)$, $\operatorname{supp}(\Theta_0)\subset[0,+\infty)$. 
\end{itemize}
\begin{figure}[!htbp]
    \centering
        \includegraphics[width=0.4\textwidth]{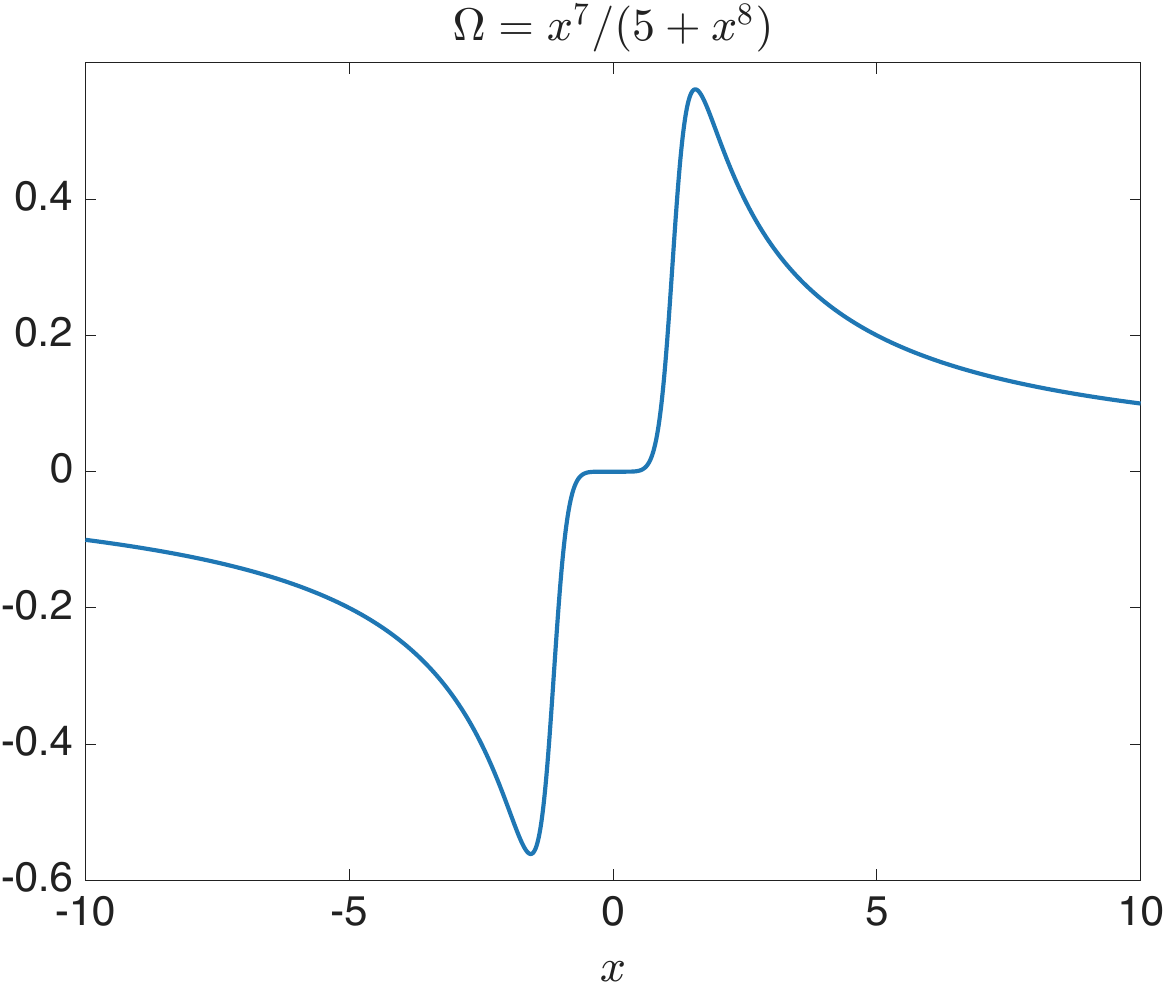}
        \includegraphics[width=0.4\textwidth]{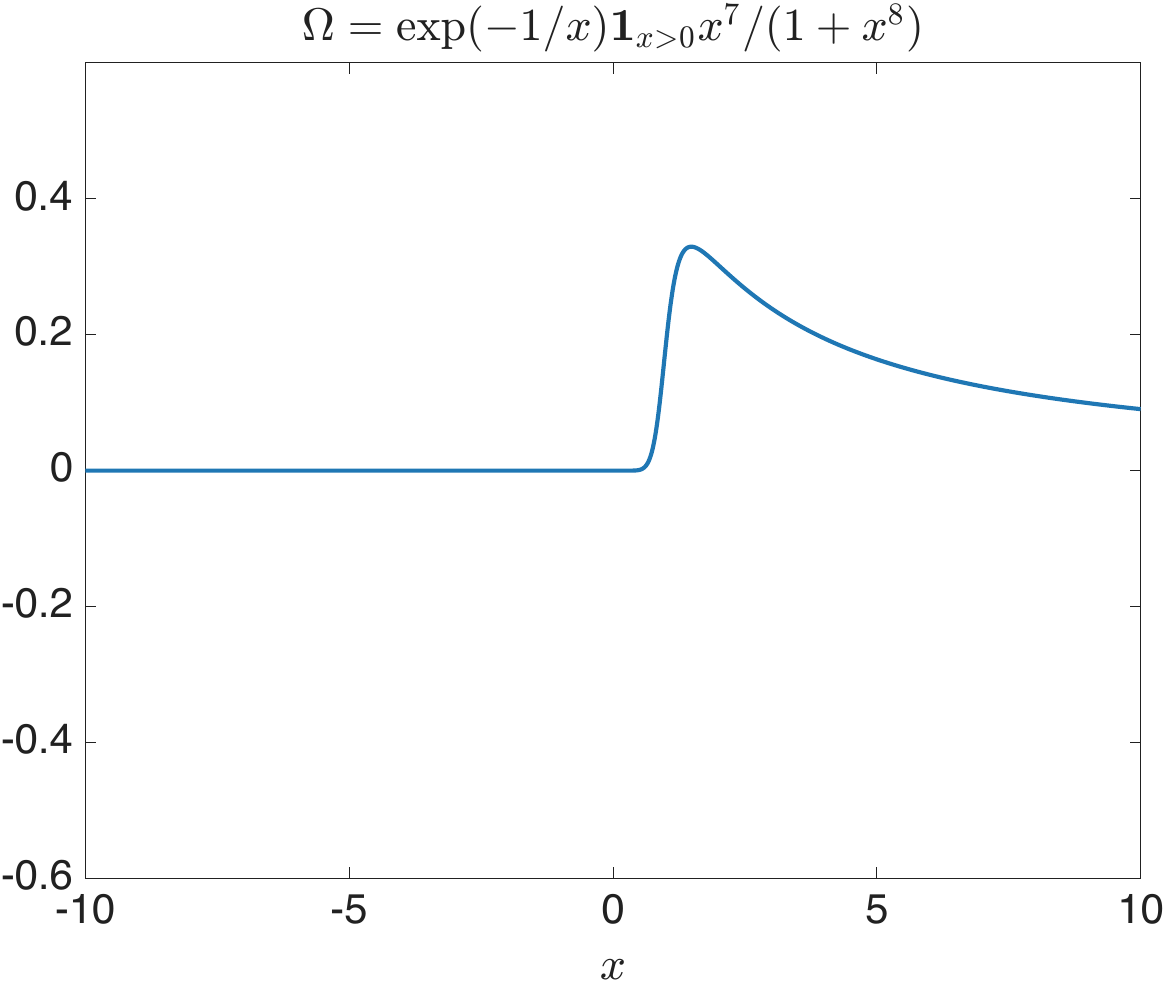}
    \caption[Spatial profiles of the initial data.]{Spatial profiles of the initial data. Left: The odd-symmetric case. Right: The one-sided case.} 
     \label{fig:HL_initiual_data}
\end{figure}
The profiles of the initial data are shown in Figure \ref{fig:HL_initiual_data}. 
Our interest in the one-sided case is motivated by the fact that we can explicitly construct one-sided singular steady solutions for both the HL model \eqref{eqt:1Dhouluo} and its dynamic rescaling counterpart \eqref{eqt:dynamic_rescaling_of_HL}.
\begin{theorem}\label{thm:HL_weak_steady_state_simplified_version}
In the weak sense,
\[(\bar{\om},\bar{\theta})=\left(\frac{\mtx{1}_{\{x>0\}}}{\sqrt{x}},\frac{\pi \mtx{1}_{\{x>0\}}}{2}\right)\] solves the steady state equation of \eqref{eqt:1Dhouluo}. Moreover, \[\left(\bar{\Omega},\bar{\Theta},\bar{c}_l,\bar{c}_{\om}\right)=\left(\frac{\mtx{1}_{\{X>1\}}}{\sqrt{X-1}},\frac{\pi \mtx{1}_{\{X>1\}}}{2},2,-1\right)\] solves \eqref{eqt:HL_steady_state} in the weak sense. 
\end{theorem}
The rigorous statement of Theorem \ref{thm:HL_weak_steady_state_simplified_version} will be deferred to Section \ref{sec:HL_steady_state}, where we clarify the precise definition of ``weak sense'' and provide the detailed proof. Similar to \eqref{eqt:scaling_invariance}, the steady state equations of \eqref{eqt:1Dhouluo} also exhibit scaling invariance: if $(\bar{\om}(x),\bar{\theta}(x))$ is a steady solution of \eqref{eqt:1Dhouluo}, then $(\alpha \bar{\om}(\beta x),\alpha^2 \bar{\theta}(\beta x)/\beta)$ is also a steady solution for any $\alpha\in \R$, $\beta>0$. Therefore, Theorem \ref{thm:HL_weak_steady_state_simplified_version} actually characterizes two families of one-sided singular steady solutions parameterized by $\alpha$, $\beta$.

Furthermore, it is not hard to see that there exists a direct and readily verifiable correspondence between these two steady state solutions. On one hand, $(\bar{\Omega},\bar{\Theta},\bar{c}_l,\bar{c}_{\om})$ represents exact self-similar profiles of the HL model; physically, this corresponds to the scenario where $\omega$ and $\theta$ evolve into the exact steady state of \eqref{eqt:1Dhouluo}—becoming strictly equal to a scaled version of $(\mtx{1}_{x>0}/\sqrt{x},\pi \mtx{1}_{x>0}/2)$ globally—within a finite time and remaining invariant thereafter. On the other hand, from an asymptotic perspective, if the solution of \eqref{eqt:dynamic_rescaling_of_HL} converges to the specific steady state $(\bar{\Omega},\bar{\Theta},\bar{c}_l,\bar{c}_{\om})$ as $\tau\to \infty$, then the corresponding $\om$ in the physical space will asymptotically exhibit a singular behavior of order $x^{-1/2}$ in a right neighborhood of the origin in finite time.   

In our numerical simulations, we first perform a preliminary computation of \eqref{eqt:dynamic_rescaling_of_HL} imposing suitable normalization conditions. Our results indicate that $\Omega$ demonstrates a distinct trend of developing a singularity, which agrees with the early numerical results reported in \cite{liu2017spatial}. To accurately resolve this emerging singularity, we impose suitable normalization conditions to pin the maximizer of $\Omega$ at the rescaled coordinate $X=1$ and proceed with an adaptive mesh strategy that dynamically increases grid density near the peak of $\Omega$. In contrast to the numerical results reported in \cite{liu2017spatial}, our computations provide strong numerical evidence that solutions of \eqref{eqt:dynamic_rescaling_of_HL} for both types of initial conditions develop a finite-time singularity at $X=1$ with respect to the rescaled time $\tau$. Furthermore, by focusing on the singular region at an appropriate spatial scale and normalizing the $L^{\infty}$ norms of $\Omega$ and $\Theta_X$, we observe that the inner profiles of $\Omega$ and $\Theta_X$ near $X=1$ locally settle into regular profiles as $\tau$ approaches the blowup time. Therefore, our results suggest that $\Omega$ and $\Theta_X$ develop asymptotically self-similar blowups with regular profiles at $X=1$. In the context of the original 1D HL model, this corresponds to asymptotically self-similar blowups of $\om$ and $\theta_x$ occurring at a location distinct from $x=0$. To the best of our knowledge, this represents a previously unreported blowup phenomenon.

Recall that we have explicitly constructed a one-sided singular steady solution for the HL model \eqref{eqt:1Dhouluo}, where $\omega$ exhibits a singularity of order $x^{-1/2}$ near the origin. A natural question, therefore, is whether a solution can dynamically evolve into such a singular structure from smooth initial data. As previously discussed, this question corresponds to determining the long-time behavior of the dynamic rescaling equations \eqref{eqt:dynamic_rescaling_of_HL}. Since solutions of \eqref{eqt:dynamic_rescaling_of_HL} are observed to develop finite-time blowups, their evolution beyond the blowup time can only be defined in the weak sense. To capture the post-blowup dynamics, we perform simulations on a locally refined fixed mesh employing suitable numerical regularization to maintain a feasible time step size. We remark that this approach aims to approximate a weak solution of \eqref{eqt:dynamic_rescaling_of_HL} obtained in the vanishing viscosity limit. We do not know whether the weak continuation of a solution after the first blowup is unique. In fact, weak solutions starting from singular initial data can be non-unique. For example, see \cite{hou2025nonuniqueness} for non-uniqueness of Leray--Hopf solutions to the 3D incompressible Navier--Stokes equations.

Our numerical results suggest that, for both initial conditions introduced above, $\Omega$ converges to a singular profile which behaves like $1/\sqrt{X-1}$ in a right neighborhood of $X=1$ and decays like $1/\sqrt{X}$ in the far field, and $\Theta$ converges to a constant multiple of the Heaviside step function. In particular, evolving from one-sided initial data, the solutions are observed to converge to profiles that closely match the steady-state solution $(\bar{\Omega},\bar{\Theta},\bar{c}_l,\bar{c}_{\om})=(\mtx{1}_{\{X>1\}}/\sqrt{X-1},\pi \mtx{1}_{\{X>1\}}/2,2,-1)$ (after proper rescaling) introduced in Theorem \ref{thm:HL_weak_steady_state_simplified_version}. Based on these observations, we propose the following conjecture regarding the stability of $(\bar{\Omega},\bar{\Theta},\bar{c}_l,\bar{c}_{\om})$.
\begin{conjecture}\label{conj:HL_singular_steady_state_stability}
The singular steady state $(\bar{\Omega},\bar{\Theta},\bar{c}_l,\bar{c}_{\om})$ is asymptotically stable. More specifically, under suitable normalization conditions, for any degenerate smooth initial data $(\Omega^0,\Theta^0)$ which is sufficiently close to $(\bar{\Omega},\bar{\Theta})$ in some norm $\|\cdot\|$, the solution of the Cauchy problem \eqref{eqt:dynamic_rescaling_of_HL} converges to $(\bar{\Omega},\bar{\Theta})$ in $\|\cdot\|$.
\end{conjecture}

\begin{remark}\label{rem:HL_singular_steady_state_stability}
The choice of the norm $\|\cdot\|$ is crucial in the above conjecture. In the existing literature \cite{chen2021finite,chen2022asymptotically}, the stability of the self-similar profiles is usually established in some weighted Sobolev norm. In particular, to obtain linear damping in the energy estimates, the weight functions are typically chosen to be singular at the origin. Note that in our case, the steady state is singular at $X=1$. Therefore, it is
reasonable to consider weight functions that are singular at the origin and vanish at $X=1$ so that the singularity of the profile at $X=1$ can be regularized. However, the precise choice of the weight functions and the corresponding norm $\|\cdot\|$ is also part of this conjecture.
\end{remark}

In summary, we observe that the solutions of the dynamic rescaling equations \eqref{eqt:dynamic_rescaling_of_HL} with suitable degenerate initial data eventually converge to singular profiles. Consequently, the HL model \eqref{eqt:1Dhouluo} can develop finite-time blowups at $x=0$ with singular self-similar profiles. However, in contrast to the preliminary numerical results reported in \cite{liu2017spatial}, we find no numerical evidence supporting a two-scale blowup mechanism for the HL model analogous to that described in \cite{huang2025multiscale}. Instead, the formation of singular profiles arises because $\Omega$ and $\Theta_X$ develop finite-time blowups with respect to the dynamic rescaling time $\tau$. We thus propose the following two-stage self-similar blowup scenario for the HL model:
\begin{itemize}\item \textbf{Stage 1: Asymptotically self-similar blowups with regular profiles.} Starting from $\tau=0$, the rescaled variables $\Omega$ and $\Theta_X$ develop asymptotically self-similar blowups with regular profiles at $(X,\tau)=(1,T_{\text{dr}})$. These profiles do not possess any symmetry with respect to the blowup point. Correspondingly, in the physical space, $\omega$ and $\theta_x$ develop self-similar singularities at $(x,t)=(\tilde{x},\tilde{T})$ with the same regular self-similar profiles, where $\tilde{x}=1/C_l(T_{\text{dr}})$ and $\tilde{T}=t(T_{\text{dr}})$ are determined via the change of variables \eqref{eqt:dynamic_rescaling_of_HL_change_of_variables2}. This corresponds to local $L^\infty$ blowups of $\omega$ and $\theta_x$ at $\tilde{x}$.
\item \textbf{Stage 2: Asymptotically self-similar blowups with singular profiles.} After the blowup time $T_{\text{dr}}$, $\Omega$ and $\Theta_X$ continue in the weak sense and eventually converge to singular profiles as $\tau\to +\infty$. Correspondingly, $\omega$ and $\theta_x$ develop asymptotically self-similar blowups at $(x,t)=(0,T)$ with singular profiles, where $T=t(+\infty)$. Moreover, our numerical results suggest that $\omega$ develops a local $L^p$ blowup around the origin for $p\geq 2$.
\end{itemize}

\begin{remark}\label{rem:HL_L2blowup}
Our numerical results show that the Stage 2 limiting values of $-c_l/c_{\om}$ in both cases are very close to $2$, suggesting that \eqref{eqt:1Dhouluo} develops asymptotically self-similar blowups that take the form \eqref{eqt:asymptotic_selfsimilar_HL} with $\gamma=2$. As discussed in Remark \ref{rem:HL_critical_Lp_exponent}, this implies local $L^2$ blowups of $\om$. In the meanwhile, the numerically computed profiles of $\Omega$ in both cases are observed to develop spatial singularities that behave like $1/\sqrt{X-1}$ as $X\to1^+$. This singular behavior would imply a local $L^2$ blowup of $\Omega$ near $X=1$ and hence, under the self-similar rescaling, a local $L^2$ blowup of $\om$ near the origin. These two observations both suggest that the critical exponent $p$ is $2$ for a local $L^p$ blowup of $\om$ in Stage 2.
\end{remark}

\subsection{Scenario 2: novel self-similar finite-time blowups with positive regular profiles.}\label{subsec:dynamic_rescaling_HLscenario2}
Motivated by the potential self-similar finite-time blowup of the HL model observed in the first stage of Scenario 1, in Section \ref{sec:HL_scenario2} we conduct a further numerical investigation of this phenomenon. As demonstrated in the previous subsection, in the physical space the Stage 1 blowup occurs at a point $\tilde{x}\neq 0$, where the velocity field satisfies $u(\tilde{x})\neq 0$. Therefore, as $t$ approaches the blowup time $\tilde{T}$, the singular region is expected to be transported toward $\tilde{x}$ by an approximately constant background flow, with displacement of order $O(\tilde{T}-t)$, while its spatial scale shrinks like $O((\tilde{T}-t)^{\gamma})$ for some $\gamma>1$. To resolve the self-similar profiles, we use the translation invariance of the HL model (Lemma \ref{lem:translation_invariance}) to shift the reference frame by adjusting the zero of the velocity field, thereby removing the effect of the constant background flow near the blowup point. We then introduce a modified dynamic rescaling formulation to zoom in on the singular region and resolve the corresponding inner profiles.

It is important to note that the self-similar blowup mechanisms considered in the existing literature for the 3D Euler equations and related simplified models are essentially odd-symmetric. By this we mean that the blowup point coincides with the symmetry point of the solution, namely the origin. As a result, by zooming in on the origin through dynamic rescaling equations similar to \eqref{eqt:dynamic_rescaling_of_HL}, one can resolve self-similar profiles that inherit the underlying odd symmetry. The situation in the present work is different. Although the solution may still preserve global odd symmetry, the Stage 1 blowup points are located away from the origin. Consequently, when zooming in on one of the blowup points, the resulting inner profile is not expected to be symmetric; indeed, our simulations indicate that the inner profiles observed in the Stage 1 blowup of Scenario 1 do not exhibit any symmetry with respect to the blowup point. Therefore, in the absence of a local symmetry, the blowup point cannot be predicted a priori. Our numerical simulations further suggest that the dynamic rescaling equations \eqref{eqt:dynamic_rescaling_of_HL} may fail to converge to a steady state when initialized with such non-symmetric local data. We therefore introduce an additional degree of freedom to track the location of the blowup point. A similar strategy was adopted in \cite{Hou2018PotentialSF} to numerically investigate potential singularity formation for a family of models of axisymmetric incompressible flow. In this paper, we implement this idea by introducing a time-dependent spatial shift $r(\tau)$ in the following change of variables:

\begin{equation}\label{eqt:HL_scenario2_change_of_variables}
    \begin{aligned}
    &\omega(x,t)=C_{\om}(\tau)^{-1}\Omega\big(C_l(\tau)(x-r(\tau)),\tau(t)\big),\quad \theta(x,t)=C_{\theta}(\tau)^{-1}\Theta\big(C_l(\tau)(x-r(\tau)),\tau(t)\big),\\ &u(x,t)=C_{\om}(\tau)^{-1}C_{l}(\tau)^{-1}U\big(C_l(\tau)(x-r(\tau)),\tau(t)\big),
    \end{aligned}
\end{equation}
with
\begin{equation}\label{eqt:HL_scenario2_change_of_variables2}
\begin{aligned}
    C_{\om}(\tau)&=\exp\left(\int_0^\tau c_{\om}(s)\idiff s\right) ,\quad  C_{\theta}(\tau)=\exp\left(\int_0^\tau c_{\theta}(s) \idiff s\right),\quad  C_{l}(\tau)=\exp\left(\int_0^\tau c_l(s)\idiff s\right),\\ r(\tau)&=-\int_{0}^{\tau}c_r(s)C_l(s)^{-1}\idiff s, \quad  t(\tau)=\int_0^{\tau}C_{\om}(s)\idiff s, \quad c_\theta=c_l+2c_{\om}.
\end{aligned}
\end{equation}
Then \eqref{eqt:1Dhouluo} is reformulated into the following equivalent form: 
\begin{equation}\label{eqt:dynamic_rescaling_of_HLscenario2}
    \begin{aligned}
&\Omega_\tau + (U+c_lX+c_r)\Omega_X=c_{\om}\Omega+\Theta_X, \\ & \Theta_\tau+ (U+c_lX+c_r)\Theta_X=(c_l+2c_{\omega})\Theta, \\ &U_X=\mtx{H}(\Omega),\quad  U(0)=0,
    \end{aligned}
\end{equation}
where $X=C_l(x-r)$. In particular, we simply impose the condition $U(0)=0$ in \eqref{eqt:dynamic_rescaling_of_HLscenario2}. In the physical space, this is equivalent to assuming that the velocity $u$ vanishes at a moving trajectory $r(\tau(t))$, which will not affect the solution's dynamics as guaranteed by the translation invariance established in Lemma \ref{lem:translation_invariance}. Moreover, it is not hard to see that the asymptotic stability of a steady state of the above equations implies the asymptotically self-similar blowup of the original model \eqref{eqt:1Dhouluo}, and the eventual blowup point in the physical space will be $x=r(+\infty)$. Indeed, this conclusion can be established via a shifted version of Lemma \ref{lem:DLstability_implies_blowup}, the detailed proof of which is omitted for brevity.

In Section \ref{sec:HL_scenario2} we perform numerical simulation of \eqref{eqt:dynamic_rescaling_of_HLscenario2} under suitable normalization conditions. Our numerical results suggest that, starting from appropriate initial data, the solution of \eqref{eqt:dynamic_rescaling_of_HLscenario2} eventually converges to a steady state $(\bar{\Omega},\bar{\Theta}_X,\bar{c}_l,\bar{c}_{\om},\bar{c}_r)$, where $\bar{c}_{\om}<0$, $\bar{c}_l>0$, $\bar{\Omega}$ and $\bar{\Theta}_X$ are positive regular profiles that do not possess symmetry. Moreover, after proper rescaling, $\bar{\Omega}$ and $\bar{\Theta}_X$ match perfectly with the inner self-similar profiles observed during the Stage 1 blowup of Scenario 1, providing strong numerical evidence that the modified dynamic rescaling equations \eqref{eqt:dynamic_rescaling_of_HLscenario2} accurately capture the mechanism underlying the Stage 1 self-similar blowup of Scenario 1. 

\section{Numerical Results of the HL Model: Scenario 1}\label{sec:HL_scenario1} 
In this section, we present numerical evidence supporting the two-stage self-similar blowups of the HL model with degenerate initial data. To simplify the notation in Section \ref{sec:HL_scenario1} and Section \ref{sec:HL_scenario2}, we denote the space-time variable $(X,\tau)$ by $(x,t)$ throughout the discussion of the dynamic rescaling equations \eqref{eqt:dynamic_rescaling_of_HL} and \eqref{eqt:dynamic_rescaling_of_HLscenario2}.

For computational purposes, it is preferable to work with quantities that exhibit sufficient decay in the far field. Accordingly, we introduce the change of variable $V:=\Theta/x$ and rewrite \eqref{eqt:dynamic_rescaling_of_HL} as 
\begin{equation}\label{eqt:numerical_dynamic_rescaling_of_HLscenario1}
    \begin{aligned}
    &\Omega_t+(U+c_lx)\Omega_x=c_{\om}\Omega+xV_x+V,\\
    &V_t+(U+c_lx)V_x=(2c_{\om}-U/x)V,\\
    & U_x=\mtx{H}(\Omega), \,\ U(0)=0.
    \end{aligned}
\end{equation}

In view of the scaling invariance property \eqref{eqt:scaling_invariance}, we have two degrees of freedom in choosing $c_l$, $c_{\om}$, and we do this by imposing two normalization conditions. To pin the singularity of $\Omega$ at $x=1$, a natural choice of the normalization conditions is to enforce that \[U(1,t)+c_l\equiv 0,\quad U_x(0,t)\equiv U_x(0,0).\]
If we assume $U_x(0,0)=-1$, then $c_l$ and $c_{\om}$ are given by
\begin{equation}\label{eqt:HLscenario1_normlization_1}
    \begin{cases}
     c_l =-U(1,t),\\
     c_{\om}=\mtx{H}(\Theta_x-(U+c_lx)\Omega_x )|_{x=0}=\mtx{H}(xV_x+V-(U+c_lx)\Omega_x )|_{x=0},
    \end{cases}
\end{equation}
Alternatively, when employing an adaptive mesh strategy, it is more advantageous to enforce that $\Omega$ attains its global maximum exactly at $x=1$. This ensures that the finest grid resolution is consistently concentrated near $x=1$. We therefore propose the following alternative choice of the normalization condition:
 \[\Omega_x(1,t)\equiv 0,\quad U_x(0,t)\equiv U_x(0,0).\]
Correspondingly, under the assumption that $U_x(0,0)=-1$ and $\Omega_x(1,0)=0$, the expressions for $c_l$ and $c_{\om}$ become
\begin{equation}\label{eqt:HLscenario1_normlization_2}
    \begin{cases}
     c_l =-U(1,t)+\Theta_{xx}(1,t)/\Omega_{xx}(1,t)=-U(1,t)+(V_{xx}(1,t)+2V_x(1,t))/\Omega_{xx}(1,t),\\
     c_{\om}=\mtx{H}(\Theta_x-(U+c_lx)\Omega_x )|_{x=0}=\mtx{H}(xV_x+V-(U+c_lx)\Omega_x )|_{x=0}.
    \end{cases}
\end{equation}

\subsection{Case 1: odd data}
In this subsection, we investigate the evolution of solution profiles starting from odd degenerate initial data. We begin by performing a preliminary numerical simulation of \eqref{eqt:numerical_dynamic_rescaling_of_HLscenario1} imposing the normalization conditions \eqref{eqt:HLscenario1_normlization_1}. Figure \ref{fig:HL_odd_initial_profile} illustrates the early-stage evolution of $(\Omega,\Theta)$ alongside the corresponding  physical space counterpart $(\om,\theta)$. As shown, in the physical space, the peak of $\omega$ moves towards the origin ($x=0$) while becoming increasingly sharp. Concurrently, the rescaled variable $\Omega$ demonstrates a distinct trend of developing a singularity at $x=1$, which agrees with the early numerical results reported in \cite{liu2017spatial}. Due to the rapid focusing of the profile, we terminate the preliminary computation when the characteristic width of the singular spike shrinks below a prescribed threshold. Specifically, let $x_1(t)<x_2(t)$ be the points satisfying $\Omega(x_1,t)=\Omega(x_2,t)=0.3\|\Omega(\cdot,t)\|_{L^{\infty}}$. Throughout the simulation, we numerically track the width $x_2(t)-x_1(t)$ and terminate the iteration once $x_2(t)-x_1(t)<0.1$.

\begin{figure}[!htbp]
\centering
    \begin{subfigure}[b]{1\textwidth}
        \centering
        \includegraphics[width=0.4\textwidth]{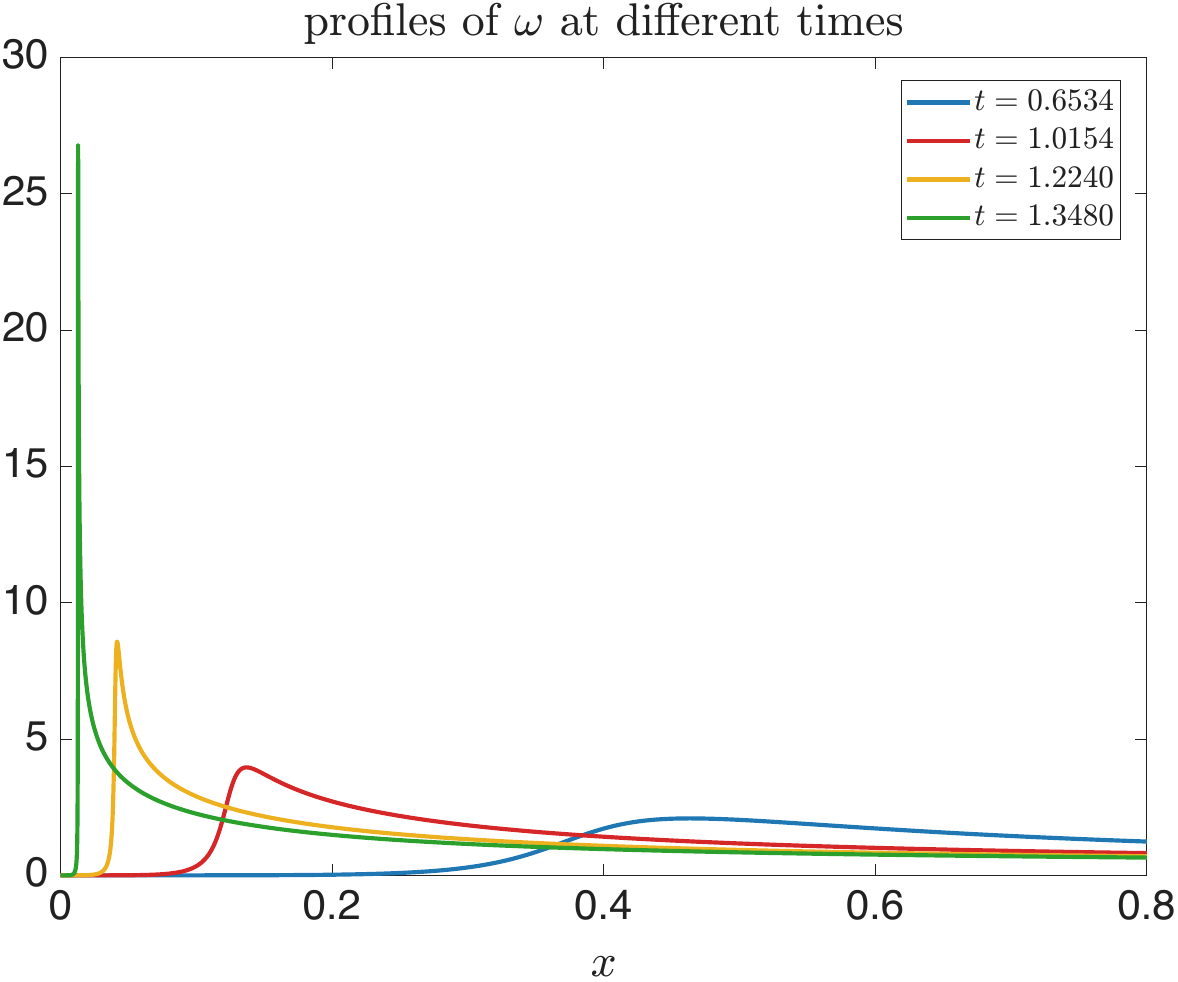}
        \includegraphics[width=0.4\textwidth]{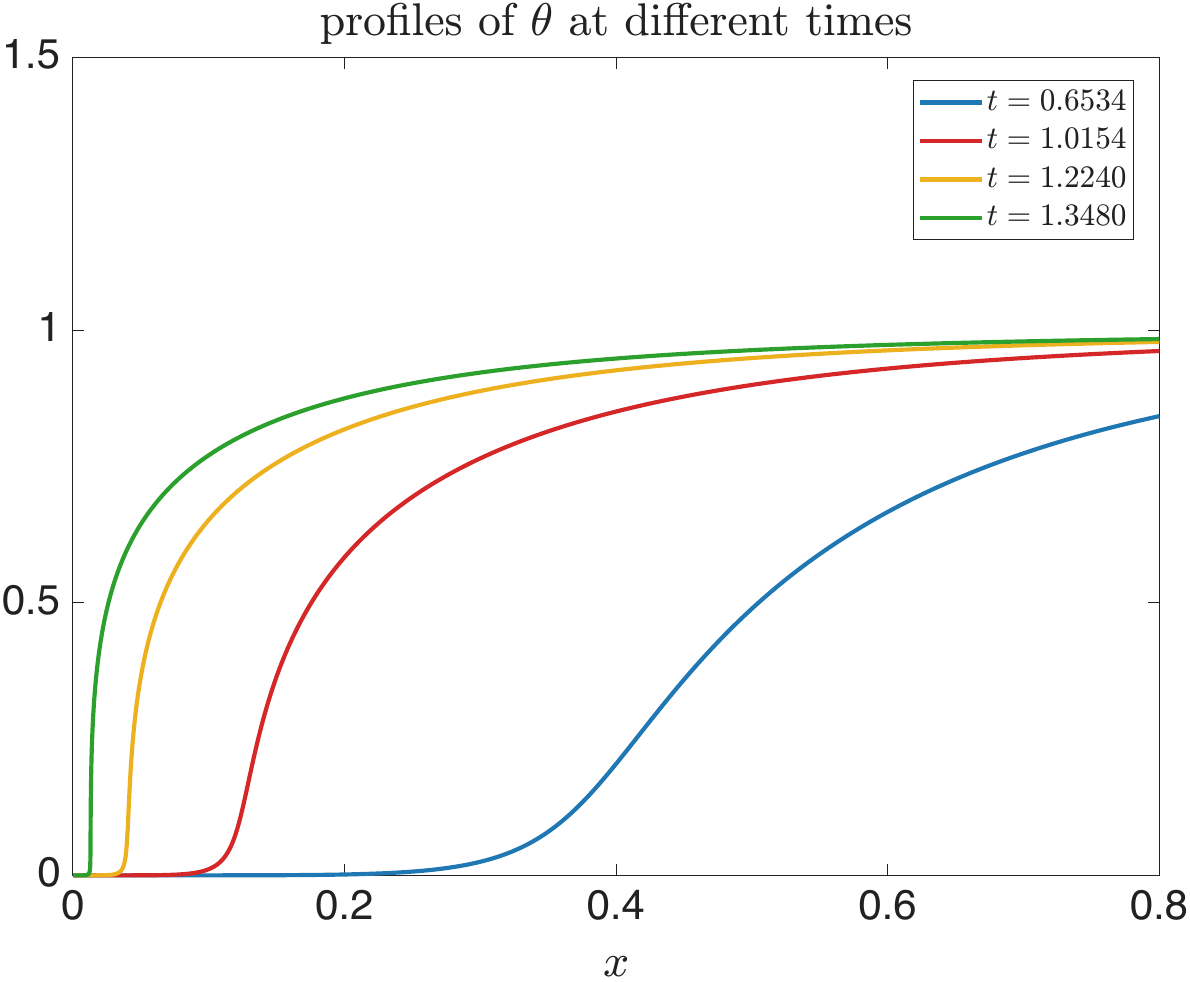}
        \caption{Spatial profiles of $\om$ (left) and $\theta$ (right) in the physical space.}
    \end{subfigure}
    \begin{subfigure}[b]{1\textwidth}
        \centering
        \includegraphics[width=0.4\textwidth]{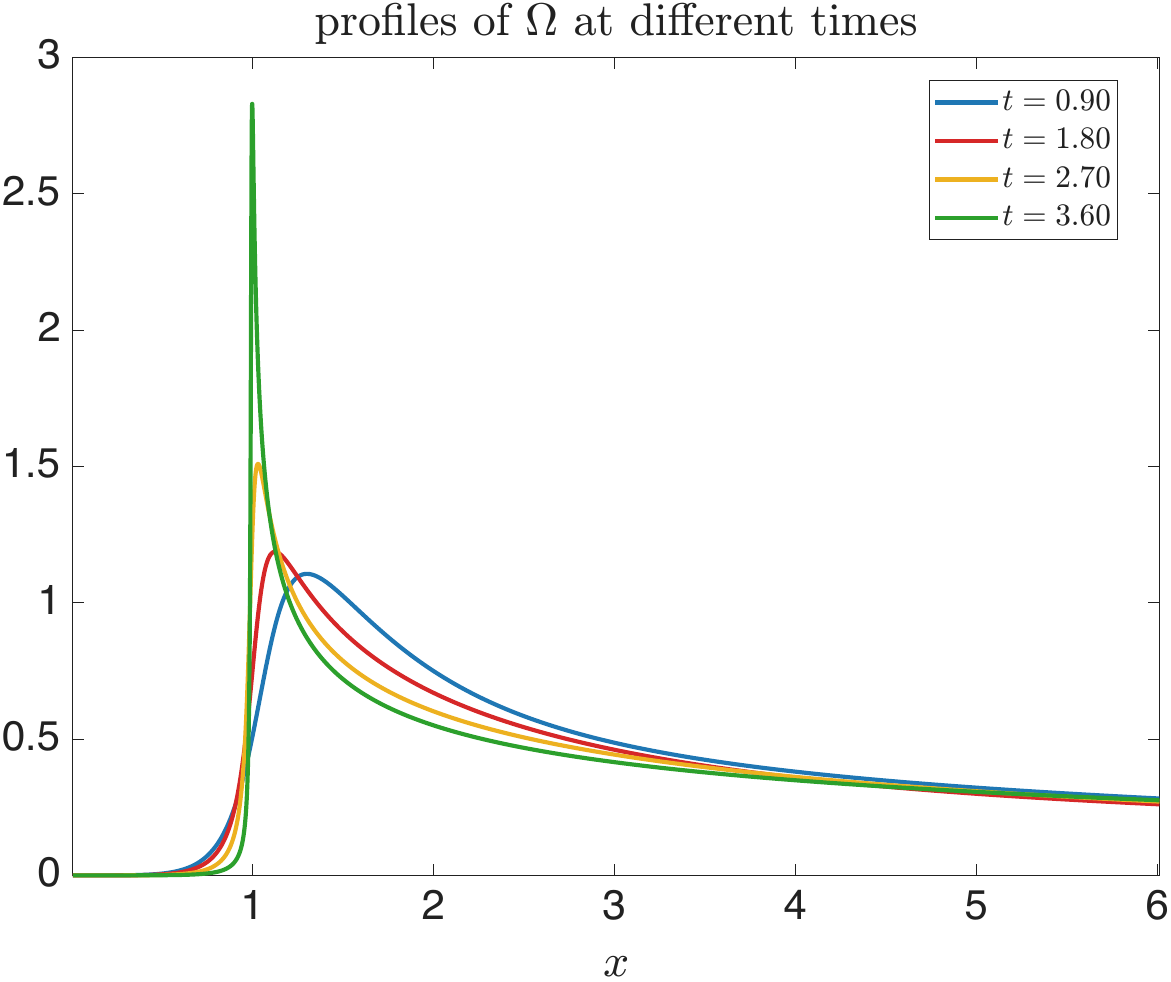}
        \includegraphics[width=0.4\textwidth]{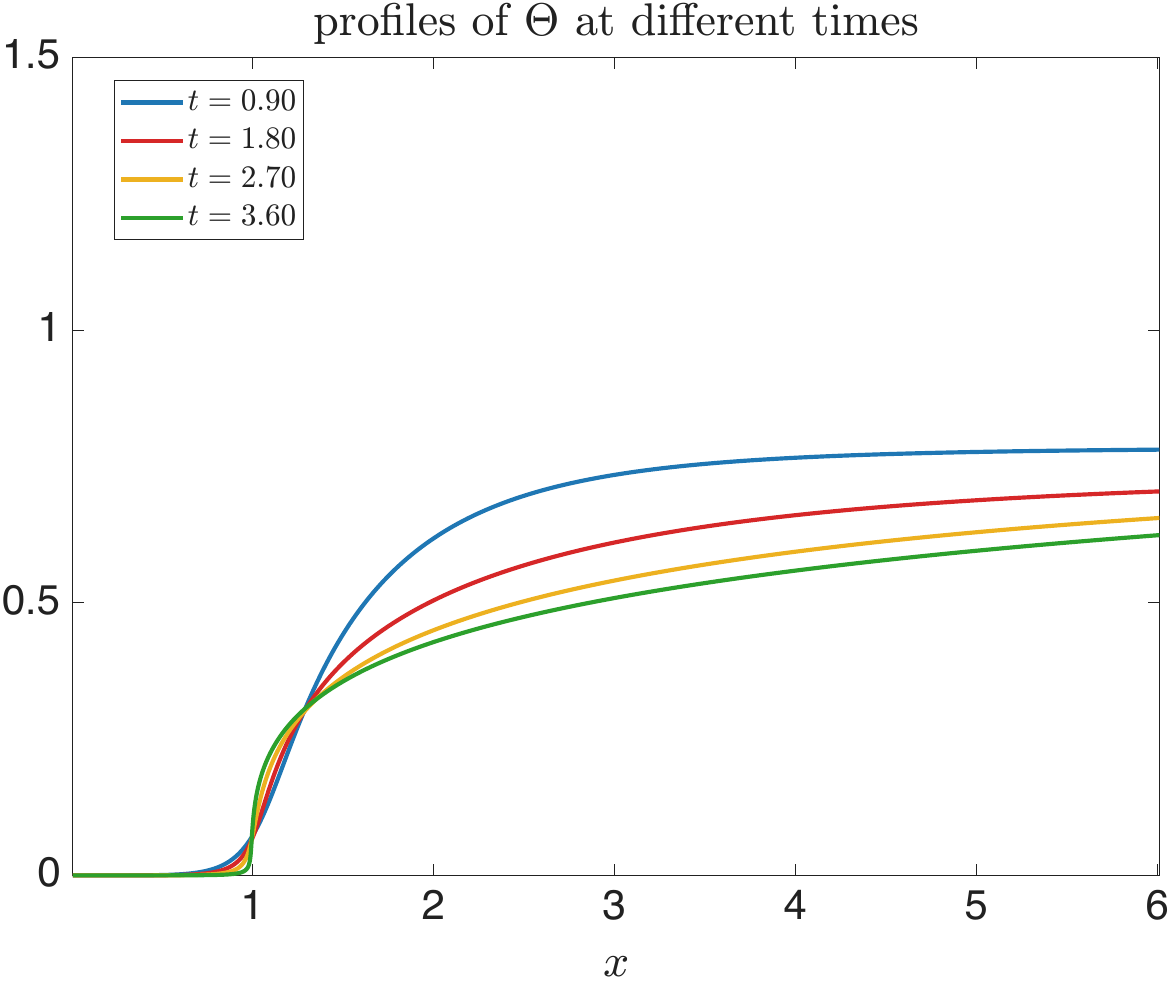}
        \caption{Spatial profiles of $\Omega$ (left) and $\Theta$ (right) in the dynamic rescaling space.}
    \end{subfigure}
 \caption[Early evolution of the solution in the odd symmetry case in Scenario 1.]{Early evolution of the solution in the odd symmetry case in Scenario 1. In the top row, $(x,t)$ represents the physical space variables, while in the bottom row, $(x,t)$ represents the dynamic rescaling space variables. The two coordinate systems are related via the change of variables introduced in Section \ref{sec:dynamic_rescaling_of_HL}.}  
     \label{fig:HL_odd_initial_profile}
\end{figure}

To study the singular behavior of the solutions, we proceed with an adaptive mesh strategy that dynamically increases grid density near the peak of $\Omega$. For computational convenience, we prefer to pin the location where $\Omega$ attains its global maximum during the simulation. Let $t^*$ denote the terminal time of the preliminary computation mentioned above. Before conducting further simulations, we rescale the terminal profile $(\Omega(\cdot,t^*),V(\cdot,t^*))$ obtained from the preliminary computation to $(\tilde{\Omega}(x),\tilde{V}(x))=(\Omega(X^*x,t^*),V(X^*x,t^*))$, where $X^*$ is the maximum location of $\Omega(\cdot,t^*)$. We then restart and continue the computation of the dynamic rescaling equations \eqref{eqt:numerical_dynamic_rescaling_of_HLscenario1} with the new initial data $(\tilde{\Omega},\tilde{V})$ and the new normalization conditions \eqref{eqt:HLscenario1_normlization_2} (instead of the previous ones \eqref{eqt:HLscenario1_normlization_1}). This ensures that the continued solution $\Omega(\cdot,t)$ (for $t>t^*$) attains its global maximum exactly at $x=1$ for all time onward. We want to remark that, using $(\tilde{\Omega},\tilde{V})$ as the initial data of our further computation for \eqref{eqt:numerical_dynamic_rescaling_of_HLscenario1} instead of $(\Omega,V)$ is equivalent to modifying the scaling parameter $C_l$ in the change of variables \eqref{eqt:dynamic_rescaling_of_HL_change_of_variables1}, which will not bring any essential change to the solution's dynamics.

In contrast to the numerical results reported in \cite{liu2017spatial}, our findings suggest that solutions of \eqref{eqt:numerical_dynamic_rescaling_of_HLscenario1} develop asymptotically self-similar finite-time blowups. This phenomenon is corroborated by Figures \ref{fig:HL_odd_Linfty} and \ref{fig:HL_odd_inner profile}. The first row of Figure \ref{fig:HL_odd_Linfty} illustrates the rapid growth of the maximums of $\Omega$ and $\Theta_x$ in logarithmic coordinates. In particular, in the singular region, $\Theta_x$ is observed to scale like $\Omega^2$, which implies that $\Omega$ blows up at the rate of $(T_{\text{dr}}-t)^{-1}$, where $T_{\text{dr}}$ denotes the blowup time of the dynamic rescaling equations. This blowup rate is further verified by the second row of the figure, which demonstrates the linear dependence of $\|\Omega\|_{L^{\infty}}^{-1}$ and $\|\Theta_x\|_{L^\infty}^{-1/2}$ on time $t$. Moreover, Figure \ref{fig:HL_odd_inner profile} displays the spatial profiles of $(\Omega,\Theta_x)$ in the singular region. As evident from the figure, the inner profiles of $\Omega$ and $\Theta_x$, after proper rescaling, are observed to approach regular limiting profiles, providing strong evidence for the existence of a self-similar blowup mechanism. In the context of the original HL model \eqref{eqt:1Dhouluo}, this corresponds to asymptotically self-similar blowups of $\omega$ and $\theta_x$ occurring at a location distinct from the origin. We refer to this phenomenon as the Stage 1 blowup.

\begin{figure}[!htbp]
    \centering
        \includegraphics[width=0.4\textwidth]{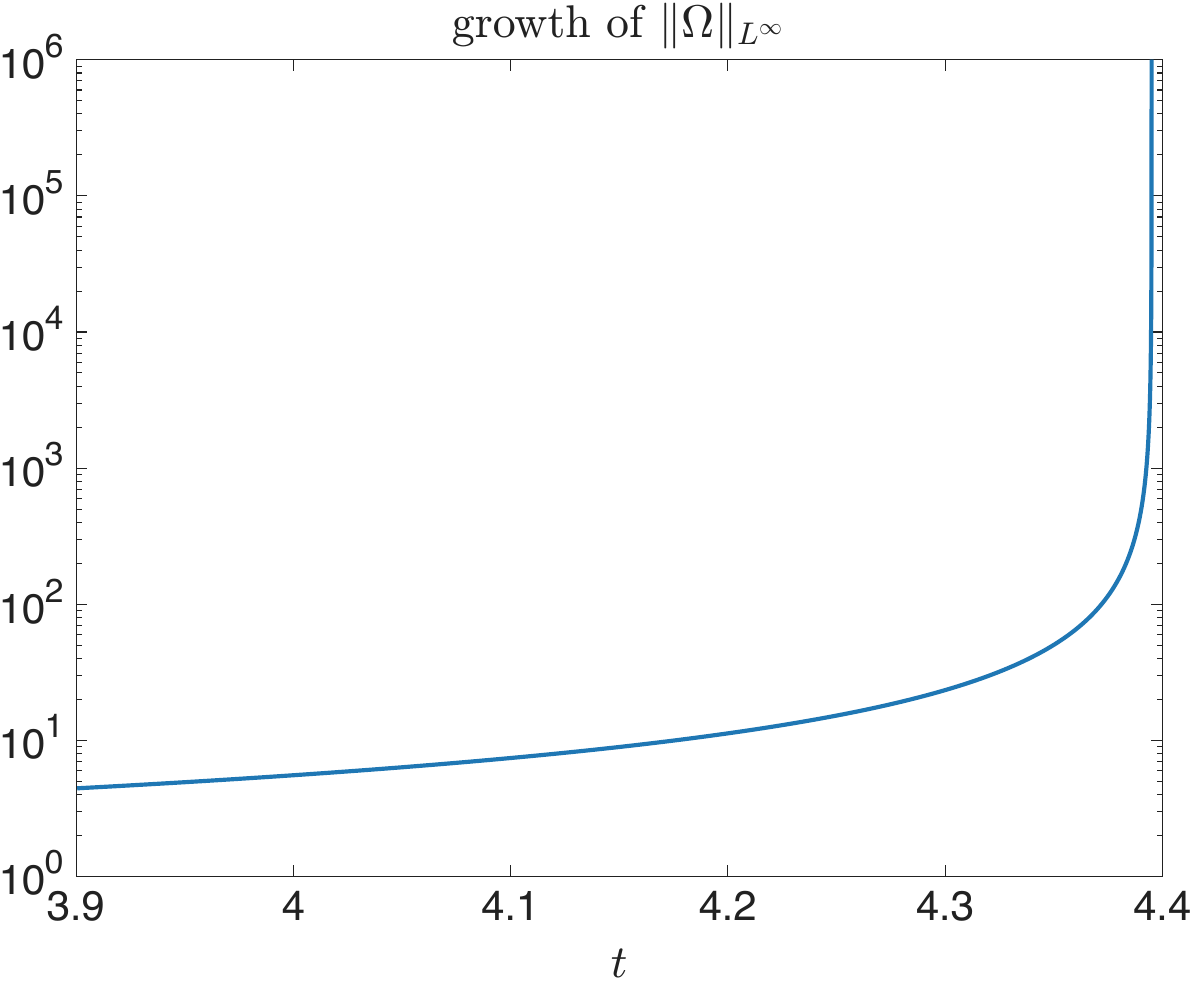}
        \includegraphics[width=0.4\textwidth]{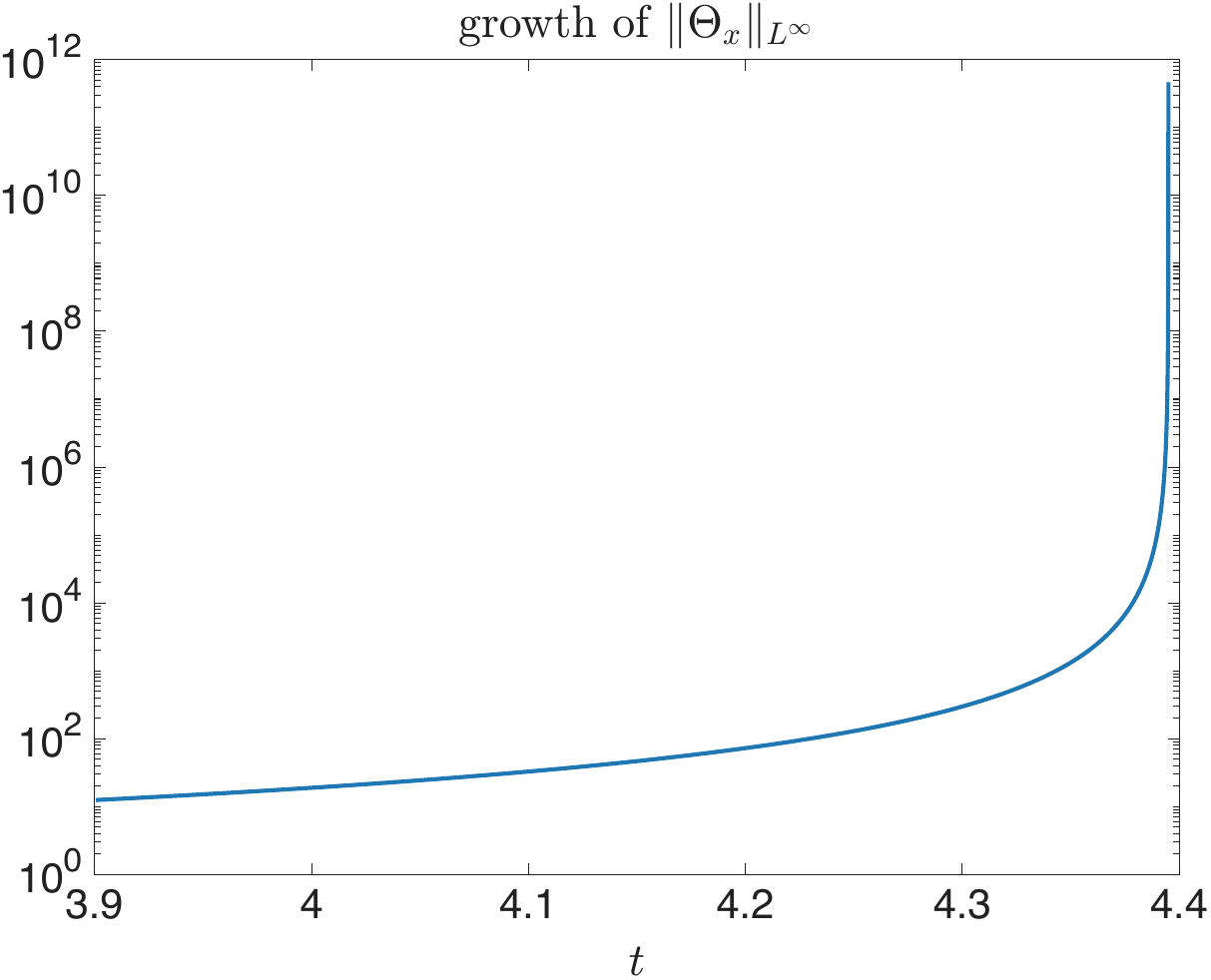}
        \includegraphics[width=0.4\textwidth]{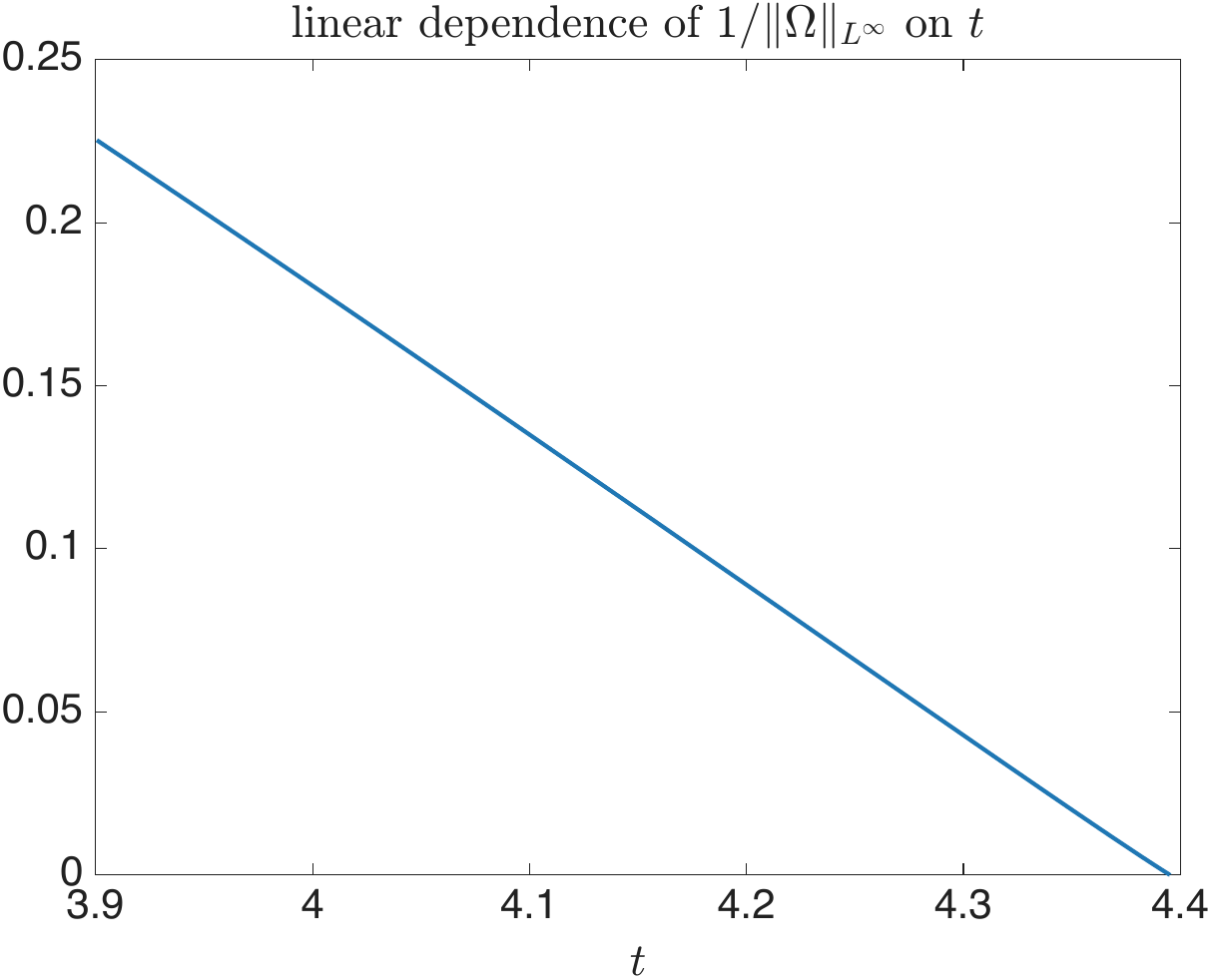}
        \includegraphics[width=0.4\textwidth]{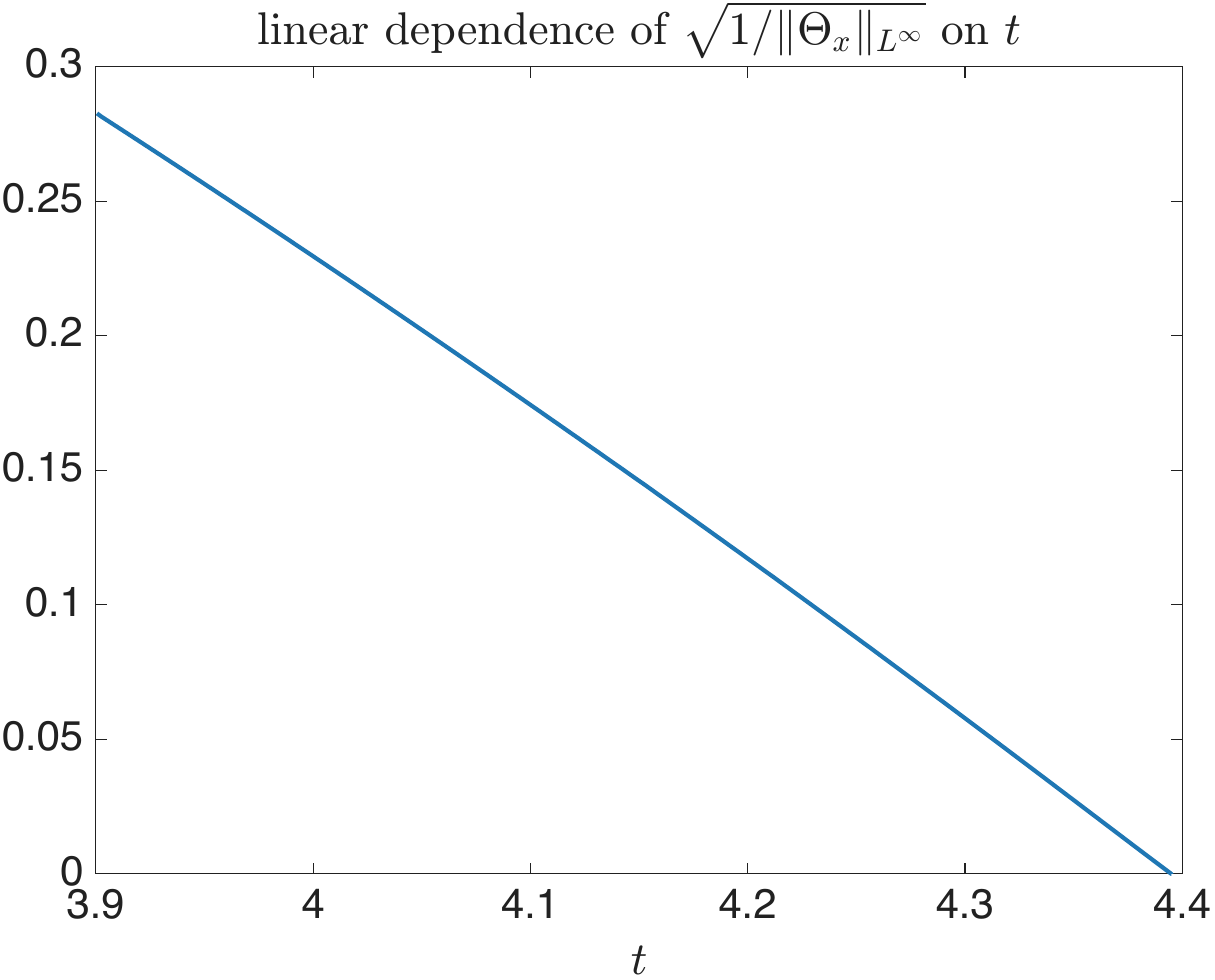}
    \caption[Time evolution of $\|\Omega\|_{L^{\infty}}$ and $\|\Theta_x\|_{L^\infty}$ in the odd symmetry case computed on an adaptive mesh in Scenario 1.]{Time evolution of $\|\Omega\|_{L^{\infty}}$ and $\|\Theta_x\|_{L^\infty}$ in the odd symmetry case computed on an adaptive mesh in Scenario 1. The top row displays the rapid growth of $\|\Omega\|_{L^{\infty}}$ and $\|\Theta_x\|_{L^\infty}$, while the bottom row illustrates that both $\|\Omega\|_{L^{\infty}}^{-1}$ and ${\|\Theta_x\|_{L^\infty}}^{-1/2}$ decay linearly with time $t$.} 
     \label{fig:HL_odd_Linfty}
\end{figure}

\begin{figure}[!htbp]
    \centering
        \includegraphics[width=0.9\textwidth]{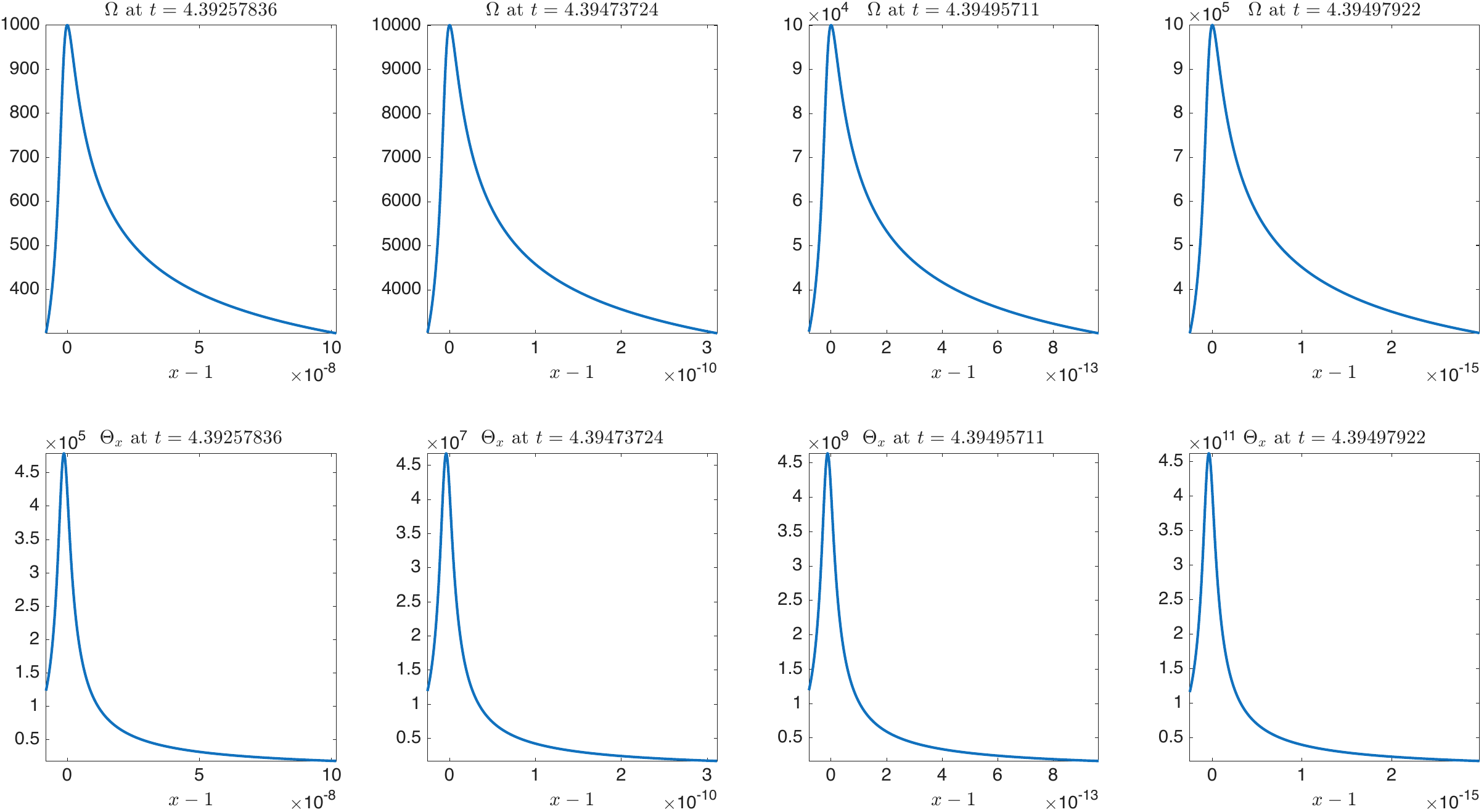}
    \caption[Inner profiles of $\Omega$ and $\Theta_x$ in the odd symmetry case computed on an adaptive mesh in Scenario 1.]{Inner profiles of $\Omega$ (top row) and $\Theta_x$ (bottom row) in the odd symmetry case computed on an adaptive mesh in Scenario 1. After proper rescaling, the inner profiles eventually stabilize at regular profiles as $t$ approaches the blowup time.} 
     \label{fig:HL_odd_inner profile}
\end{figure}

Since solutions of \eqref{eqt:numerical_dynamic_rescaling_of_HLscenario1} are observed to develop finite-time blowups, the evolution beyond the blowup time $T_{\text{dr}}$ can only be defined in the weak sense. Restricted by the CFL stability condition, it is computationally infeasible to extend the adaptive mesh simulations beyond $T_{\text{dr}}$. Therefore, to capture the post-blowup dynamics of the weak solution to \eqref{eqt:numerical_dynamic_rescaling_of_HLscenario1}, we perform simulations on a locally refined fixed mesh with a maximum resolution of $2\times 10^{-8}$ near $x=1$. During computation, we employ suitable numerical regularization to suppress the growth of $\|\Omega\|_{L^{\infty}}$ and $\|\Theta_x\|_{L^{\infty}}$, thereby ensuring that the simulation proceeds with an acceptable time step size. This approach can be interpreted as approximating the weak solution of \eqref{eqt:numerical_dynamic_rescaling_of_HLscenario1} via the method of vanishing viscosity.

Due to the CFL restriction, the discrete time step is small (approximately $10^{-6}$). To efficiently track the evolution of the solution, we therefore record the values of $\Omega$ and $V$ at relatively large time intervals (approximately $10^{-2}$) and use the rate of change between consecutive recordings to quantify the convergence of the numerical solution. More specifically, the computation is terminated once 
\begin{equation}\label{eqt:HLscenario1_stopping_criterion}
    \left\|\left(\frac{\Omega^{k_i}-\Omega^{k_{i-1}}}{t_{k_i}-t_{k_{i-1}}}\right)\phi\right\|_{L^{3}(0,+\infty)}<10^{-7},
\end{equation} 
where $\{\Omega^{k_i}\}_{i\in\N}$ and $\{t_{k_i}\}_{i\in\N}$ denote the recorded values of $\Omega$ and the corresponding times. The weight function $\phi$, given by
\begin{equation}\label{eqt:weight_function}
    \phi(x):=\frac{\sqrt{|x-1|}}{1+\sqrt{|x-1|}},
\end{equation}
is employed to suppress the singularity at $x=1$. 

Figures \ref{fig:HL_odd_omega} and \ref{fig:HL_odd_theta} illustrate the time evolution of the solution to \eqref{eqt:numerical_dynamic_rescaling_of_HLscenario1} under the normalization condition \eqref{eqt:HLscenario1_normlization_1}. As time progresses, the stopping criterion \eqref{eqt:HLscenario1_stopping_criterion} is eventually satisfied, indicating that the solution has stabilized at a steady state. Figure \ref{fig:HL_odd_omega} displays the evolution of $\Omega$ across different scales, suggesting that $\Omega$ converges to a singular limiting profile $\bar{\Omega}$ that satisfies the following properties:
\begin{itemize}
  \item $\bar{\Omega}$ vanishes on the interval $(0,1)$;
  \item $\bar{\Omega}$ exhibits a singularity of the type $1/\sqrt{x-1}$ in a right neighborhood of $x=1$;
  \item $\bar{\Omega}$ decays like $1/\sqrt{x}$ in the far field.
\end{itemize}
Figure \ref{fig:HL_odd_theta} displays the spatial profiles of $\Theta$ at different times. As shown, $\Theta$ finally stabilizes to a limiting profile $\bar{\Theta}$ that agrees with a constant multiple of the Heaviside step function. Thus, in the odd symmetry case, the limiting profile $(\bar{\Omega},\bar{\Theta})$ is observed to fully capture the structural features of the one-sided steady state $(\mtx{1}_{\{x>1\}}/\sqrt{x-1},\pi \mtx{1}_{\{x>1\}}/2)$ introduced in Theorem \ref{thm:HL_weak_steady_state_simplified_version}.

\begin{figure}[!htbp]
\centering
        \includegraphics[width=0.32\textwidth]{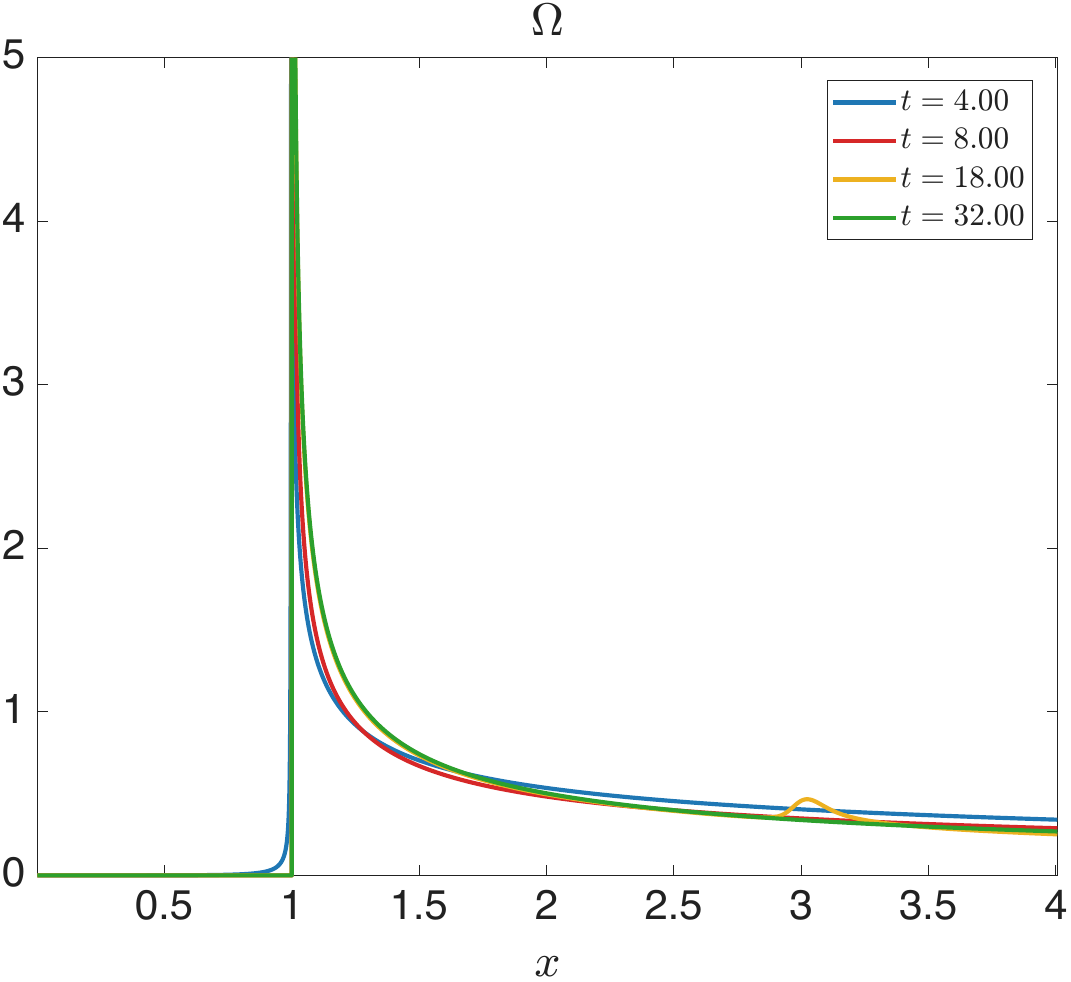}
        \includegraphics[width=0.32\textwidth]{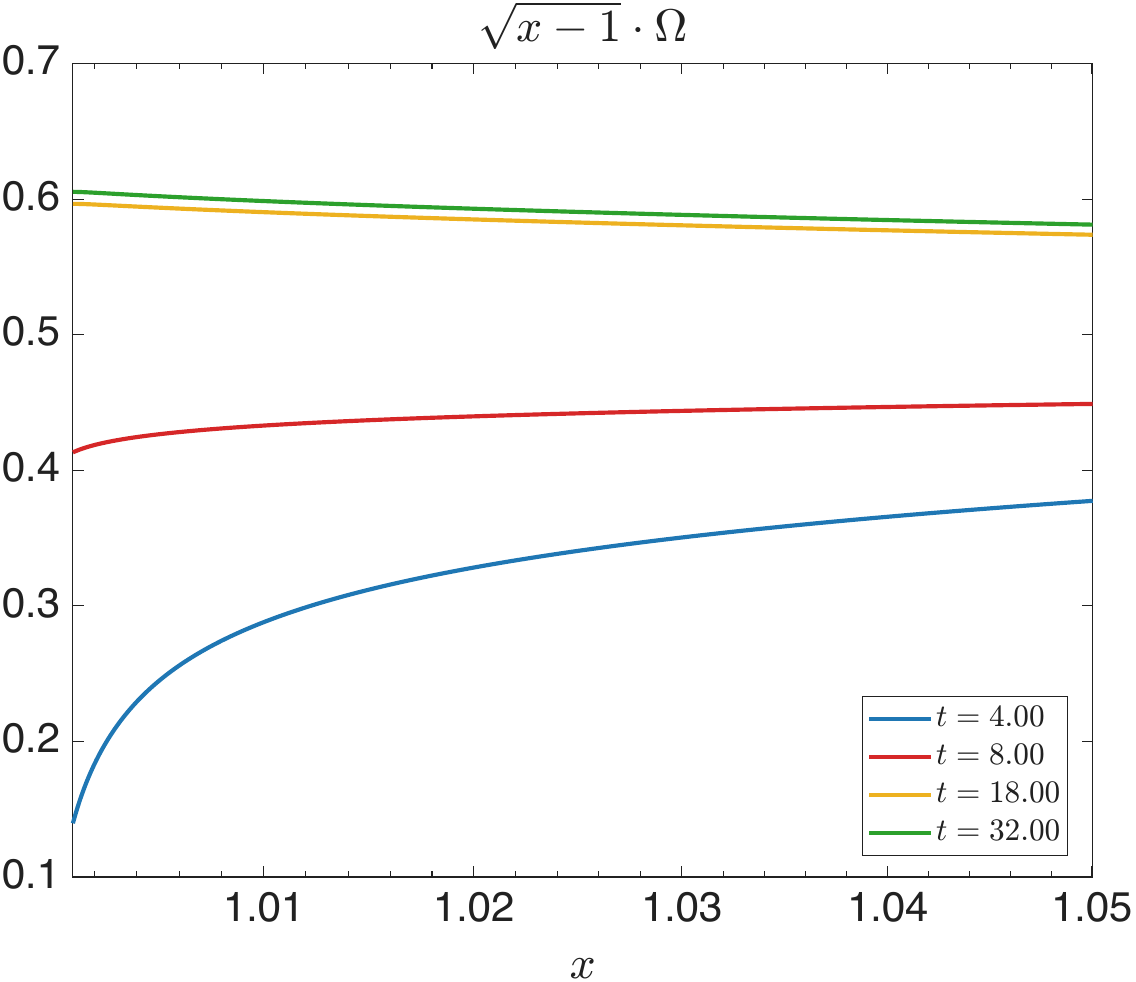}
        \includegraphics[width=0.32\textwidth]{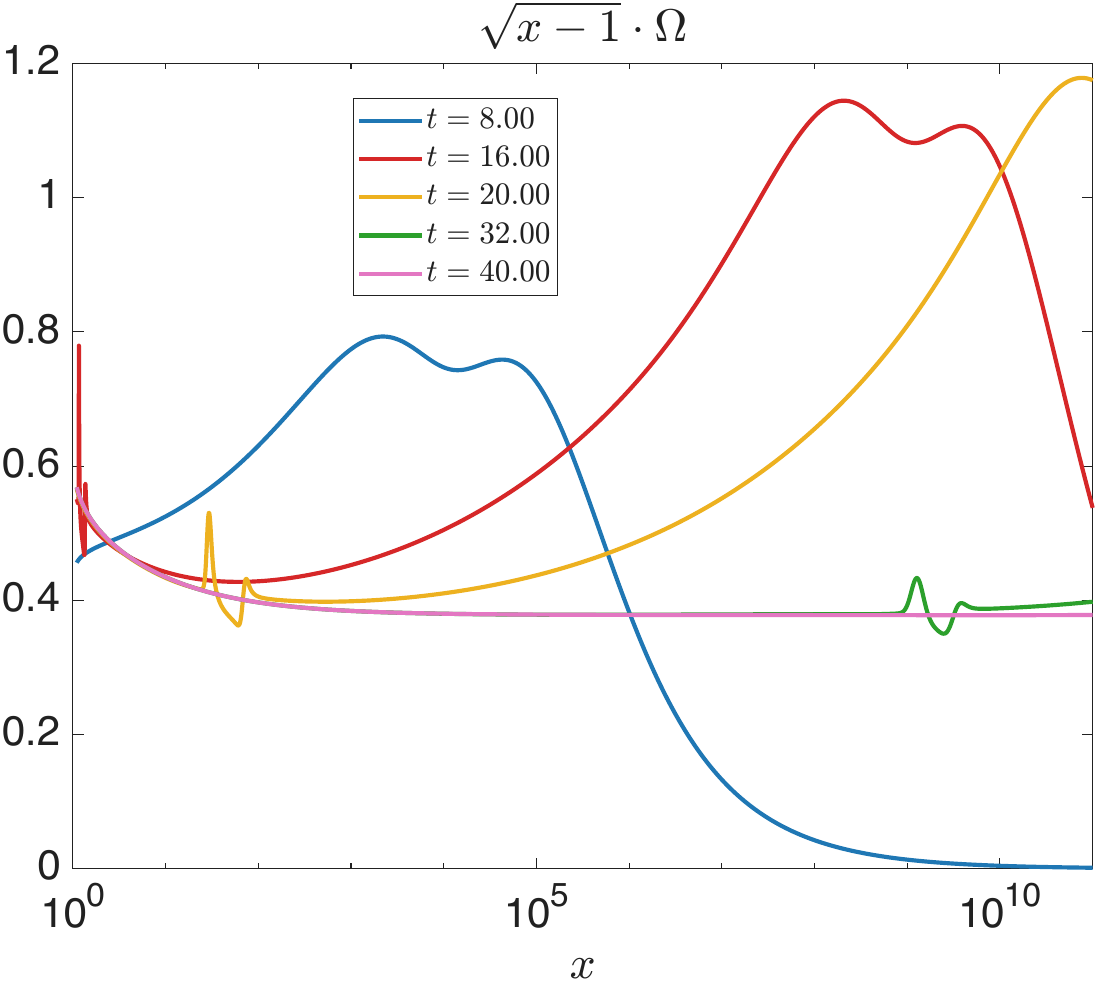}
    \caption[Evolution of spatial profiles of $\Omega$ in different regions (odd symmetry case, Scenario 1, computed on a fixed mesh).]{Evolution of spatial profiles of $\Omega$ in different regions (odd symmetry case, Scenario 1, computed on a fixed mesh). Left: $\Omega$ in the $O(1)$ region. Middle: $\sqrt{x-1}\cdot\Omega$ near the singularity at $x=1$. Right: $\sqrt{x-1}\cdot\Omega$ in the far field. As illustrated, $\Omega$ eventually approaches a singular limiting profile.}
     \label{fig:HL_odd_omega}
\end{figure}

\begin{figure}[!htbp]
\centering
        \includegraphics[width=0.4\textwidth]{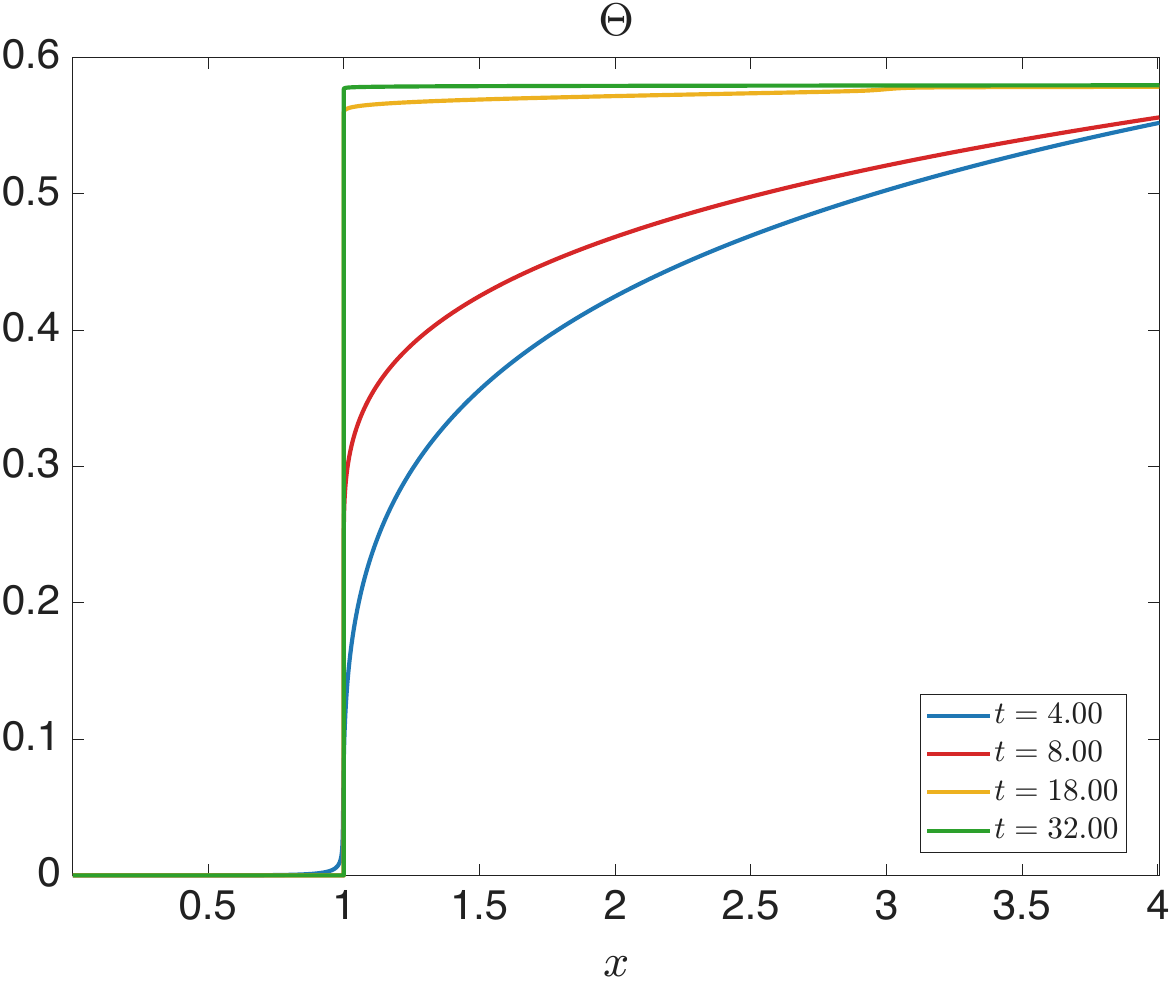}
        \includegraphics[width=0.4\textwidth]{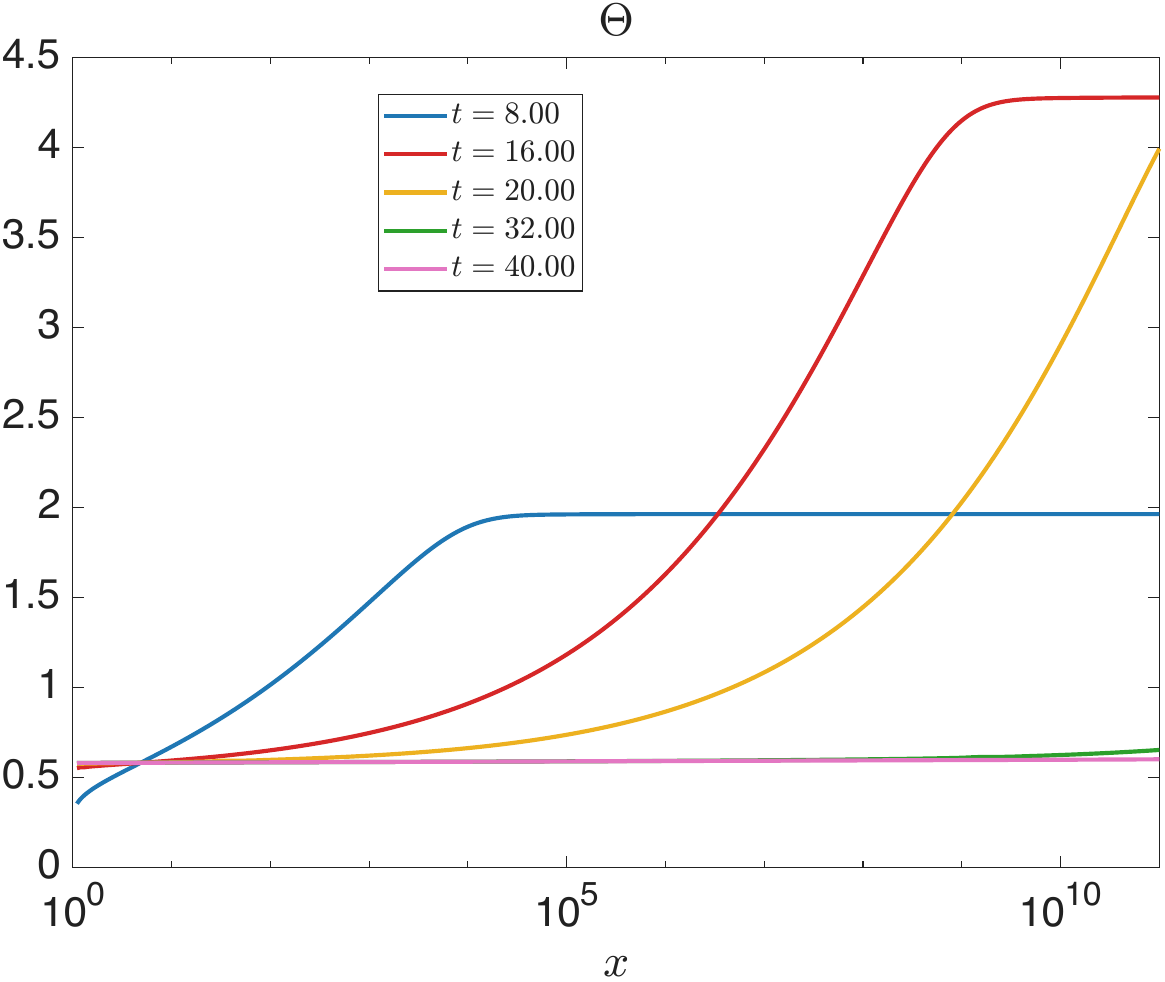}
    \caption[Evolution of spatial profiles of $\Theta$ in different regions (odd symmetry case, Scenario 1, computed on a fixed mesh).]{Evolution of spatial profiles of $\Theta$ in different regions (odd symmetry case, Scenario 1, computed on a fixed mesh). Left: The $O(1)$ region. Right: The far field. As illustrated, $\Theta$ eventually approaches a step-like profile.}  
      \label{fig:HL_odd_theta}
\end{figure}

The evolution of $U$ and the corresponding velocity field $U+c_lx$ is presented in Figure \ref{fig:HL_odd_u}. We observe that the velocity field $U+c_lx$ remains positive on $(0,1)$, while it vanishes at $x=1$ as enforced by the normalization condition \eqref{eqt:HLscenario1_normlization_1}. Furthermore, both $U$ and $U+c_lx$ develop cusps at $x=1$ due to the singularity formation of $U_x=\mtx{H}(\Omega)$. Heuristically, during the evolution of \eqref{eqt:numerical_dynamic_rescaling_of_HLscenario1}, the velocity field $U+c_lx$ is positive on $(0,1)$, and the advection terms in \eqref{eqt:numerical_dynamic_rescaling_of_HLscenario1} dominate the dynamics in this region. As a result, the degeneracy of the initial data at the origin is transported rightward, effectively suppressing both $\Omega$ and $\Theta$ to zero for $x<1$. Meanwhile, at the stagnation point $x=1$ where the advection velocity vanishes, $\Theta$ develops a shock-like structure, thereby leading to the singularity formation in $\Omega$ and $\Theta_x$.      

\begin{figure}[!htbp]
\centering
        \includegraphics[width=0.4\textwidth]{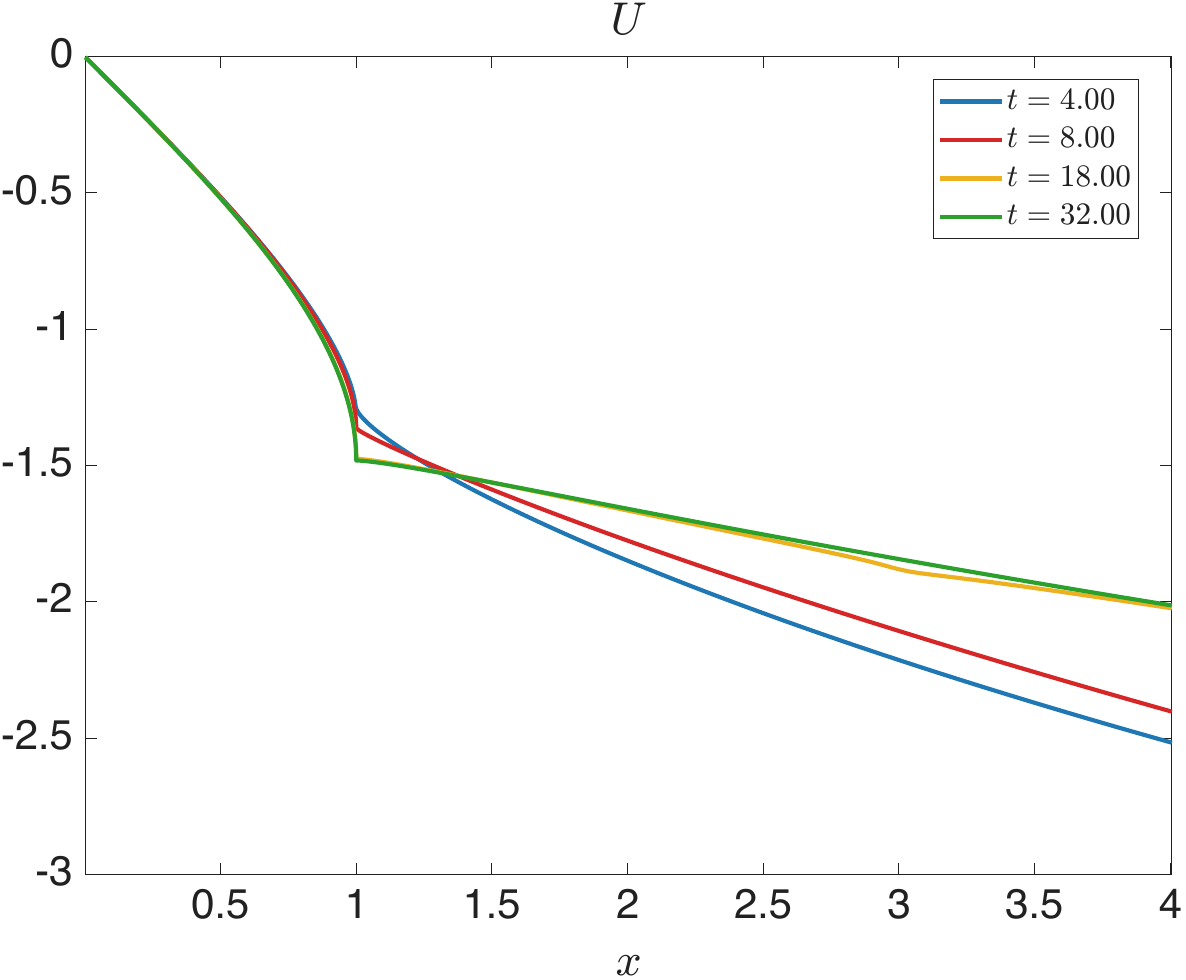}
        \includegraphics[width=0.4\textwidth]{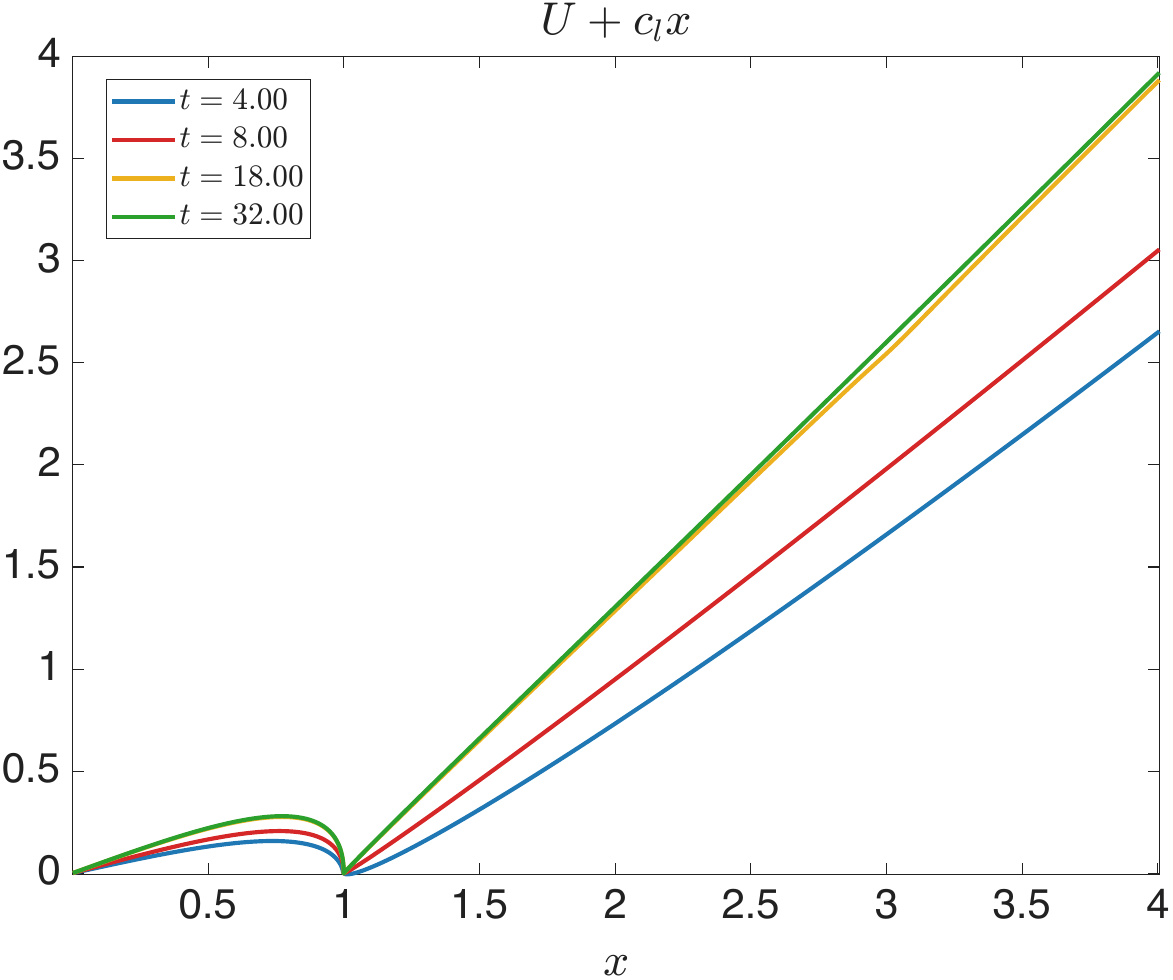}
    \caption[Evolution of spatial profiles of $U$ (left) and $U+c_l x$ (right) in the odd symmetry case computed on a fixed mesh in Scenario 1. ]{Evolution of spatial profiles of $U$ (left) and $U+c_l x$ (right) in the odd symmetry case computed on a fixed mesh in Scenario 1. As illustrated, both $U$ and $U+c_l x$ develop cusps at $x=1$.}  
      \label{fig:HL_odd_u}
\end{figure}

Figure \ref{fig:HL_odd_clcomega} displays the convergence of $c_l$, $c_{\om}$ and $c_l/c_{\om}$, respectively. The limiting value of the ratio obtained from our numerical computations is approximately $-2.00095$, which is remarkably close to $-2$. This result is consistent with the observation that the limiting profile $\bar{\Omega}$ decays as $x^{-1/2}$ in the far field. Therefore, our results strongly suggest that beyond the Stage 1 blowup, the weak solution of \eqref{eqt:dynamic_rescaling_of_HL} eventually converges to a steady state $(\bar{\Omega},{\bar{\Theta}},\bar{c}_{l},\bar{c}_{\om})$ where $\bar{\Omega}$ is singular at $x=1$ and $\bar{c}_{l}/\bar{c}_{\om}=-2$. As demonstrated in Remark \ref{rem:HL_critical_Lp_exponent} and Remark \ref{rem:HL_L2blowup}, this suggests a local $L^2$ blowup of the physical variable $\omega$ around the origin with a singular self-similar profile. We refer to this subsequent phenomenon as the Stage 2 blowup. 

\begin{figure}[!htbp]
\centering
        \includegraphics[width=0.32\textwidth]{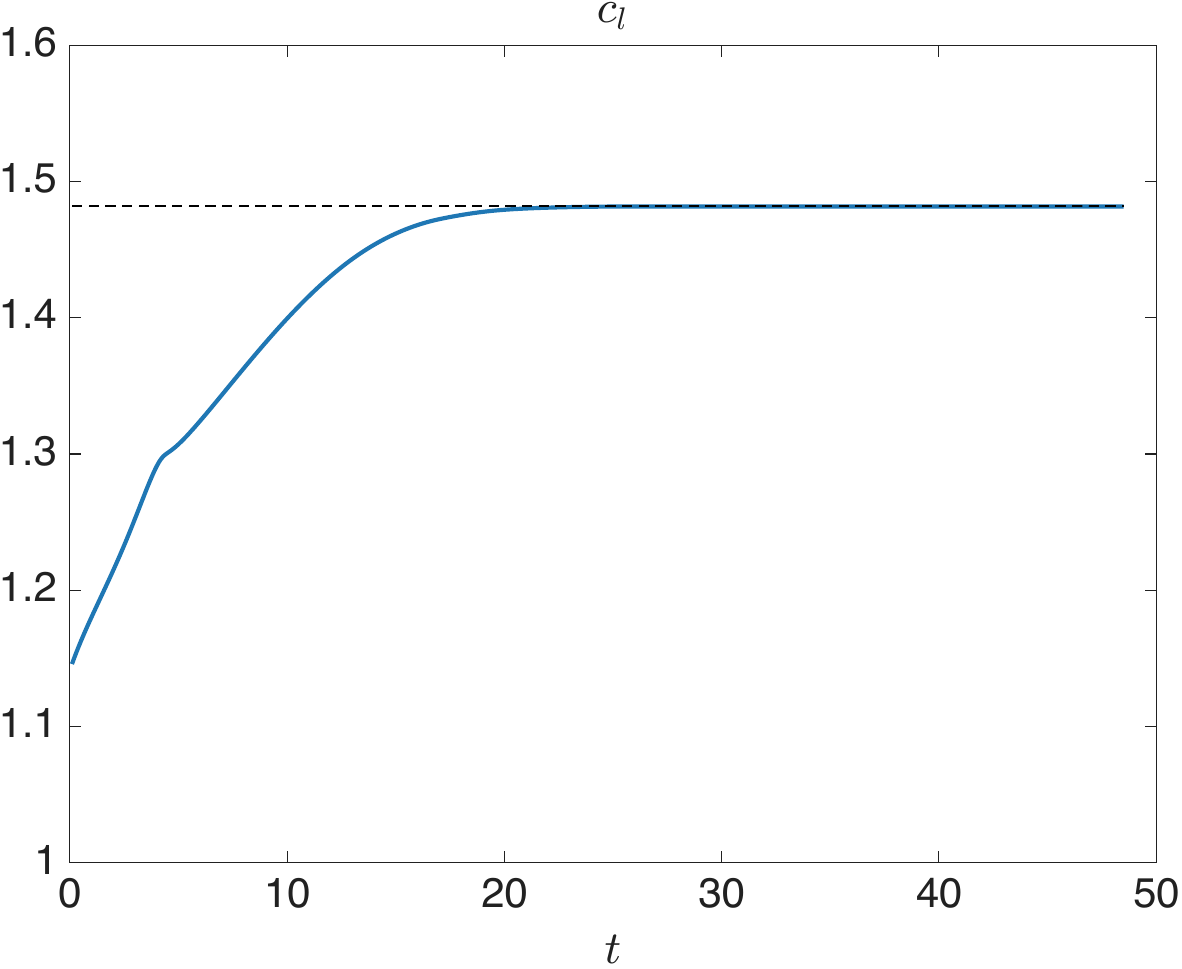}
        \includegraphics[width=0.32\textwidth]{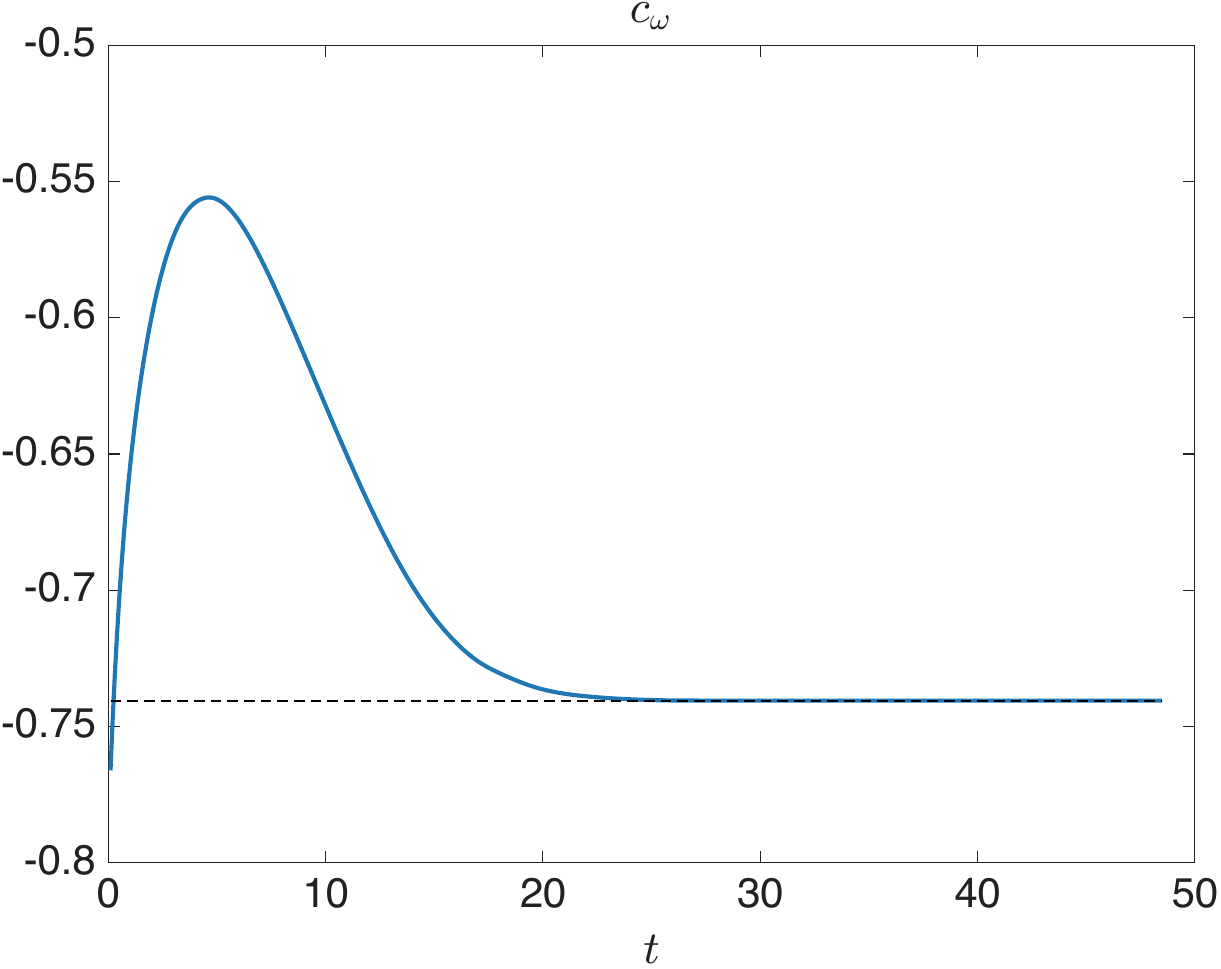}
        \includegraphics[width=0.32\textwidth]{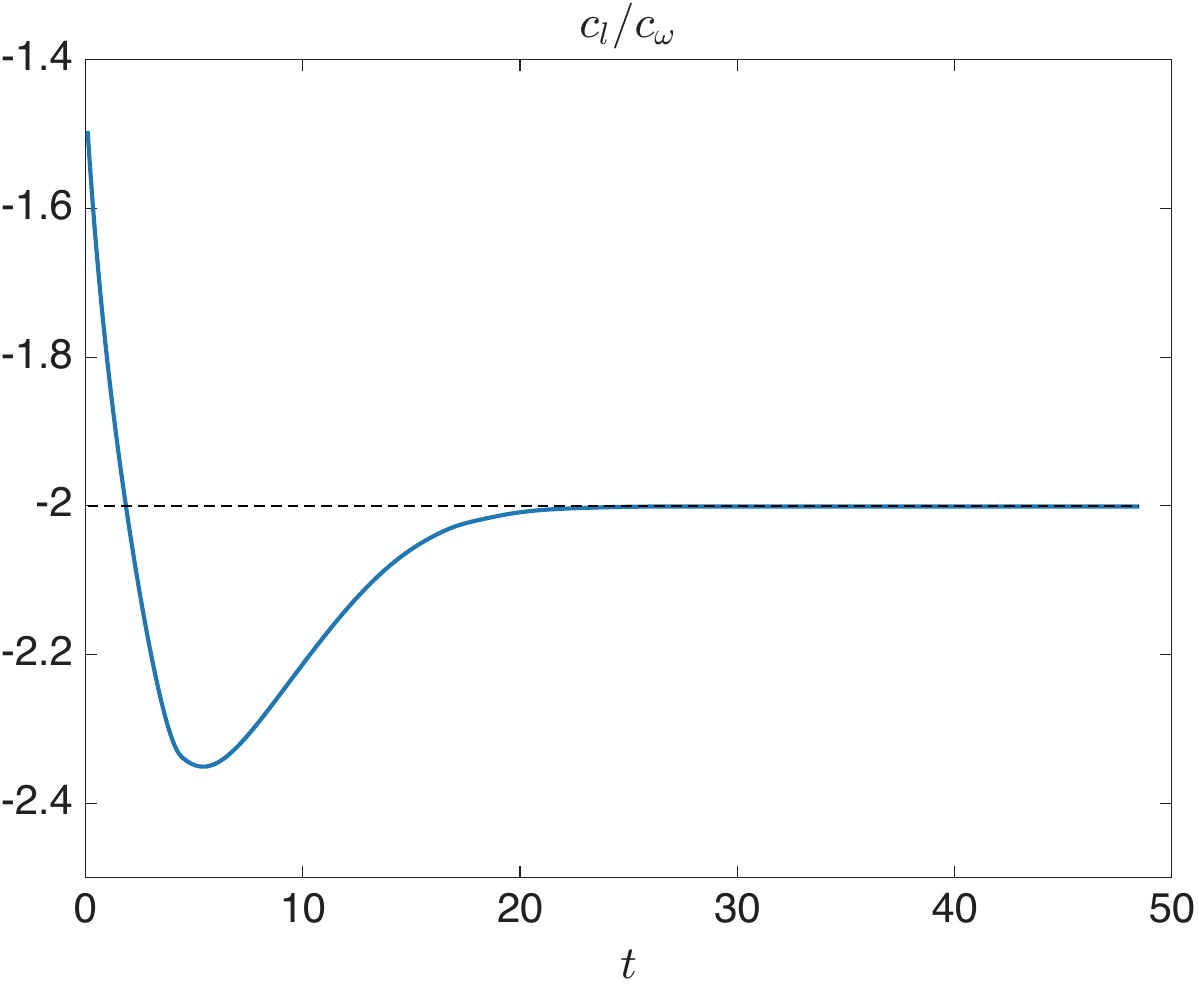}
    \caption[Time evolution of $c_l$, $c_{\om}$ and $c_l/c_{\om}$ in the odd symmetry case computed on a fixed mesh in Scenario 1.]{Time evolution of $c_l$ (left), $c_{\om}$ (middle) and $c_l/c_{\om}$ (right) in the odd symmetry case computed on a fixed mesh in Scenario 1. The black dashed lines represent the computed limiting values $(c_l,c_{\om},c_l/c_{\om})\approx(1.48166,-0.74048,-2.00095)$, respectively. }  
     \label{fig:HL_odd_clcomega}
\end{figure}

\subsection{Case 2: one-sided data}
In this subsection, we investigate the evolution of solution profiles starting from one-sided degenerate initial data. Analogous to the odd symmetry case, we begin by performing a preliminary numerical simulation of \eqref{eqt:numerical_dynamic_rescaling_of_HLscenario1} using the normalization condition \eqref{eqt:HLscenario1_normlization_1}. Subsequently, we implement an adaptive mesh strategy combined with the normalization condition \eqref{eqt:HLscenario1_normlization_2} to accurately resolve the singular behavior. Our results indicate that the solution develops an asymptotically self-similar finite-time blowup (Stage 1 blowup) analogous to the odd symmetry case. Figure \ref{fig:HL_singleside_initial_profile} illustrates the early-stage evolution of the profiles $(\Omega,\Theta)$ alongside their physical space counterparts $(\omega,\theta)$. As shown, within the dynamic rescaling variables, $\Omega$ exhibits a distinct tendency to form a singularity at $x=1$. Figure \ref{fig:HL_singleside_Linfty} tracks the evolution of $\|\Omega\|_{L^{\infty}}$ and $\|\Theta_x\|_{L^{\infty}}$, which corroborates the blowup rate of $(T_{\text{dr}}-t)^{-1}$. Furthermore, Figure \ref{fig:HL_singleside_inner profile} displays the spatial structures of $(\Omega,\Theta_x)$ in the singular region. It is evident that the inner profiles of $\Omega$ and $\Theta_x$, after proper rescaling, converge to regular limiting profiles. Moreover, these limiting profiles coincide with those observed in the odd symmetry case. These observations provide strong evidence for the existence of a robust self-similar blowup mechanism that does not depend on the symmetry of the initial data. We will numerically investigate this mechanism further in Section \ref{sec:HL_scenario2}.

\begin{figure}[!htbp]
\centering
    \begin{subfigure}[b]{1\textwidth}
        \centering
        \includegraphics[width=0.4\textwidth]{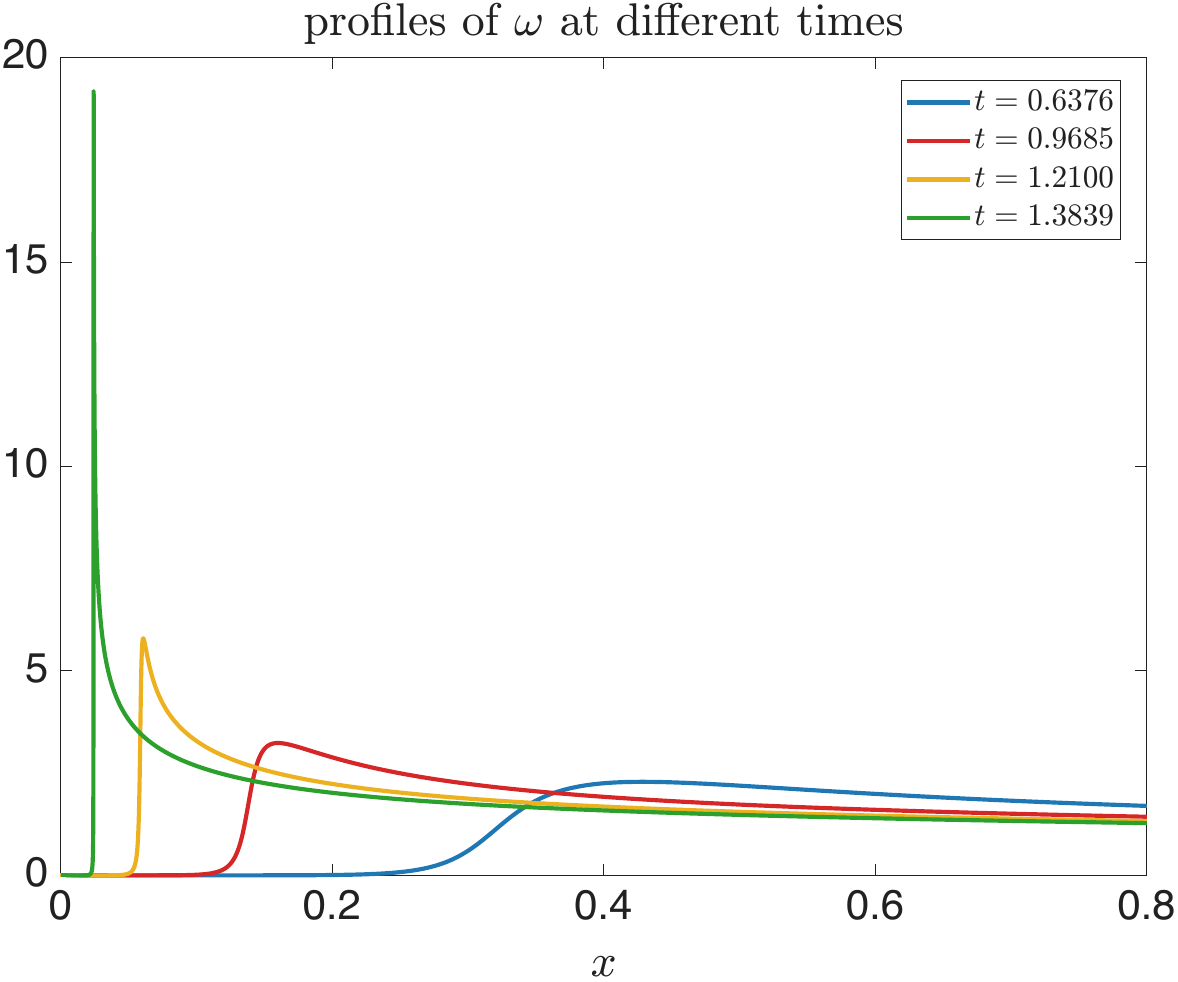}
        \includegraphics[width=0.4\textwidth]{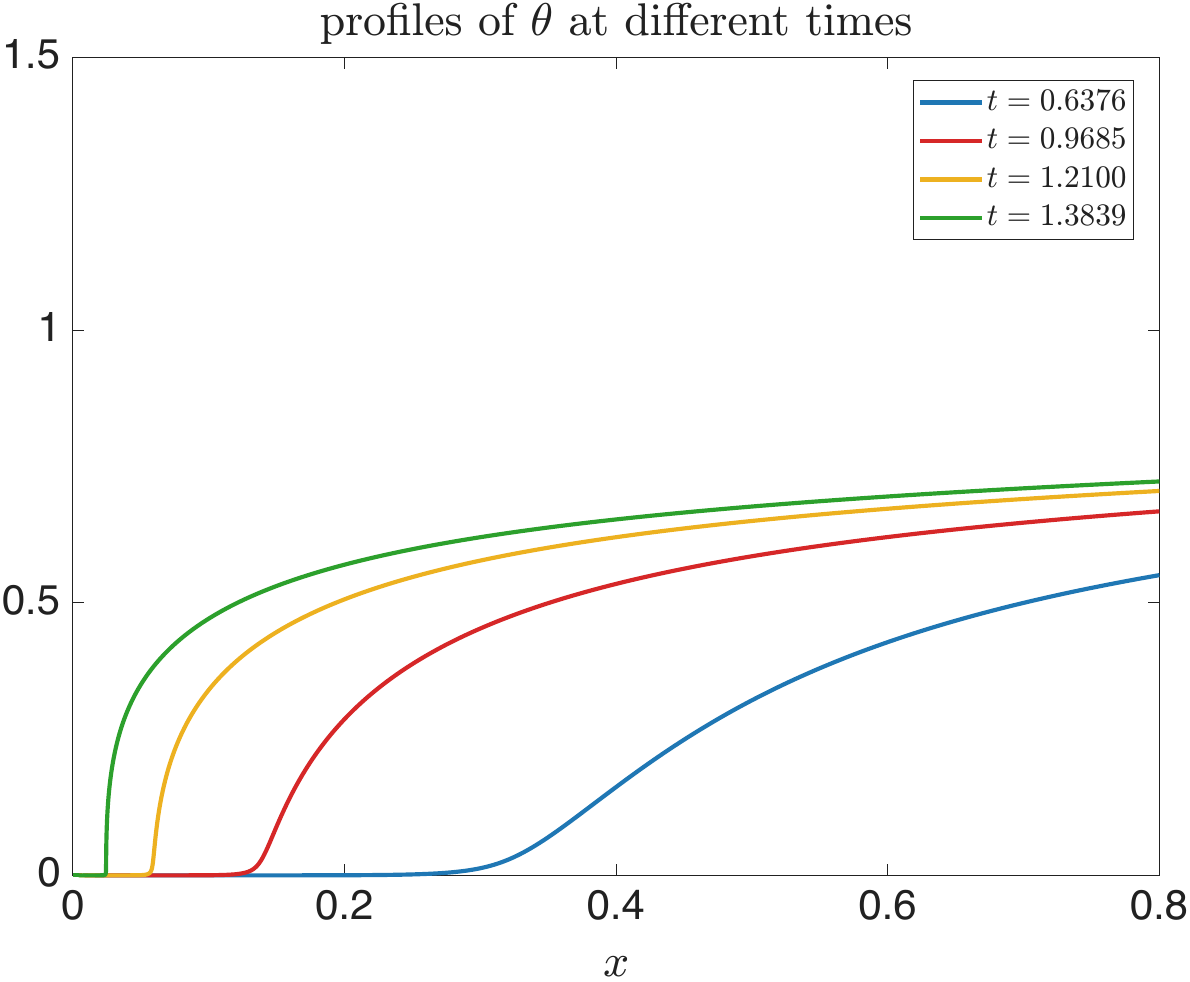}
        \caption{Spatial profiles of $\omega$ (left) and $\theta$ (right) in the physical space.}
    \end{subfigure}
    \begin{subfigure}[b]{1\textwidth}
        \centering
        \includegraphics[width=0.4\textwidth]{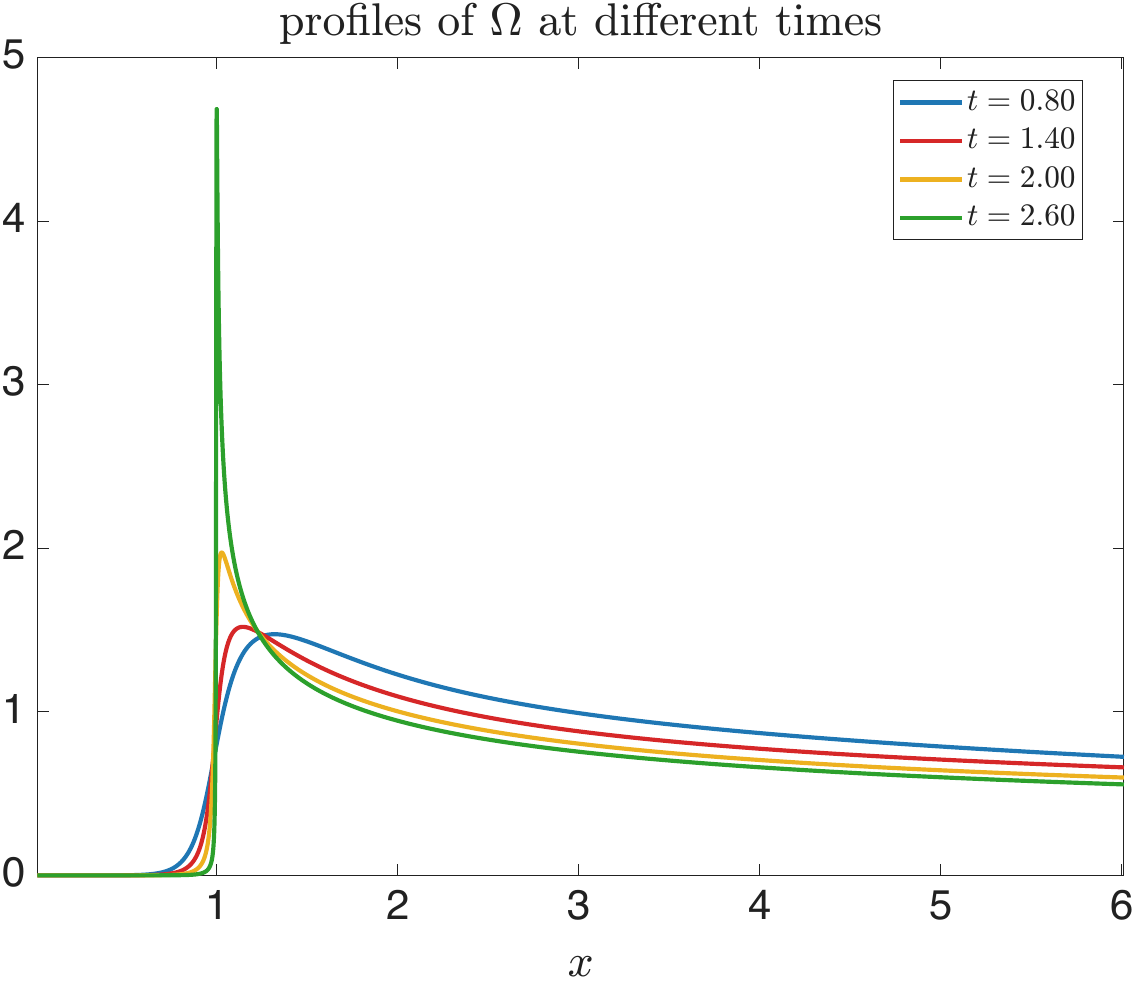}
        \includegraphics[width=0.4\textwidth]{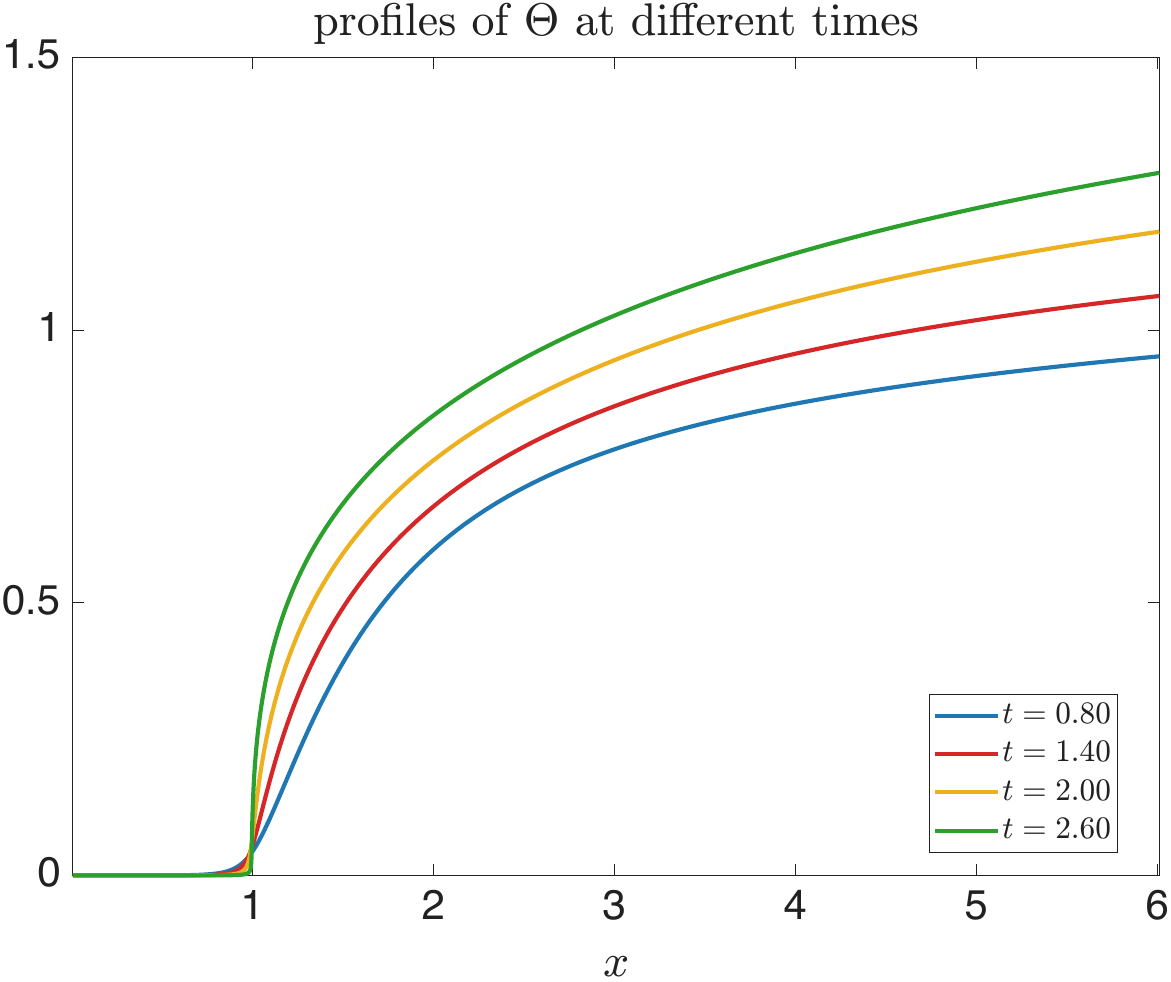}
        \caption{Spatial profiles of $\Omega$ (left) and $\Theta$ (right) in the dynamic rescaling space.}
    \end{subfigure}
 \caption[Early evolution of the solution in the one-sided case in Scenario 1.]{Early evolution of the solution in the one-sided case in Scenario 1. In the top row, $(x,t)$ represents the physical space variables, while in the bottom row, $(x,t)$ represents the dynamic rescaling space variables. The two coordinate systems are related via the change of variables introduced in Section \ref{sec:dynamic_rescaling_of_HL}.}  
     \label{fig:HL_singleside_initial_profile}
\end{figure}

\begin{figure}[!htbp]
    \centering
        \includegraphics[width=0.4\textwidth]{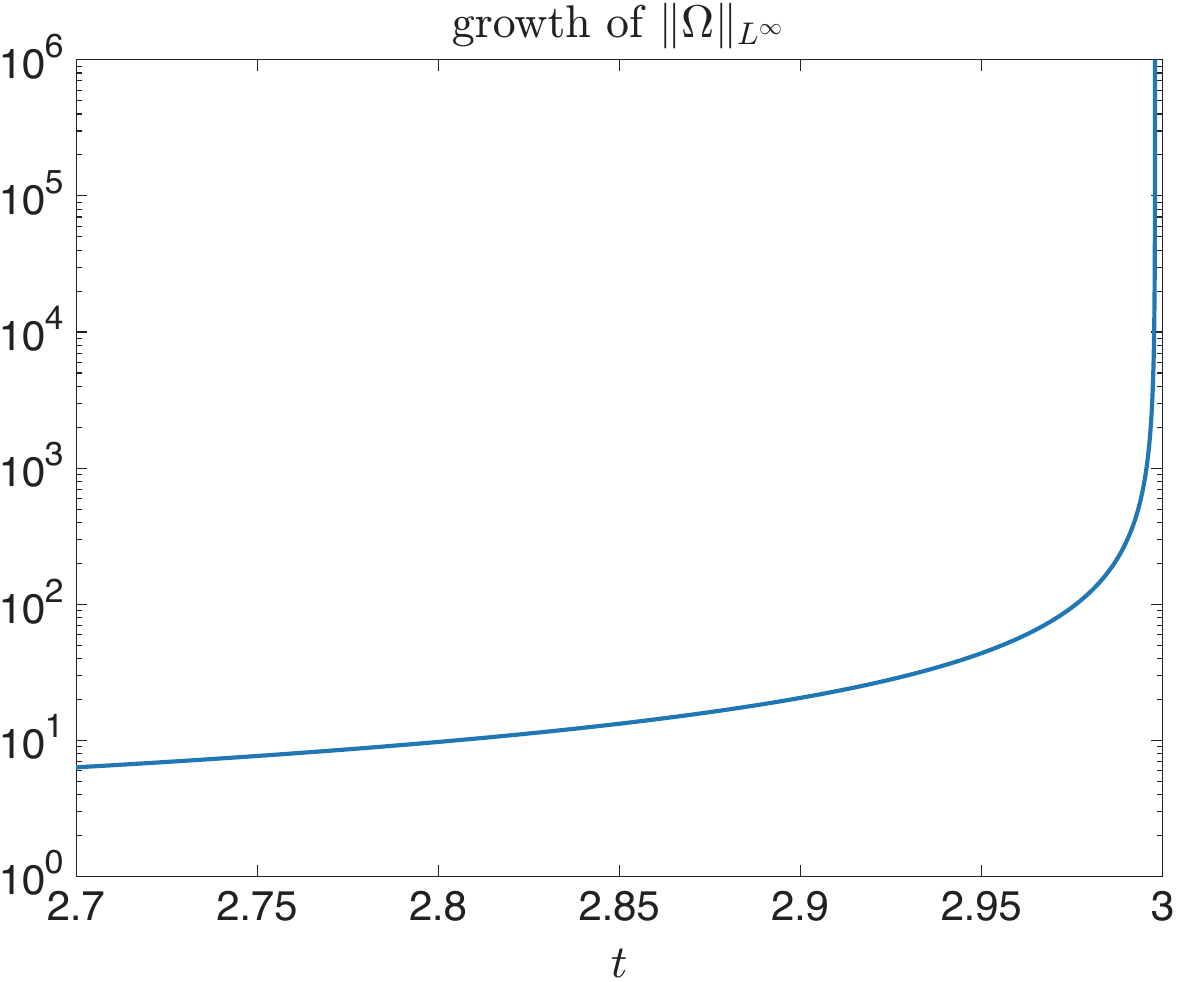}
        \includegraphics[width=0.4\textwidth]{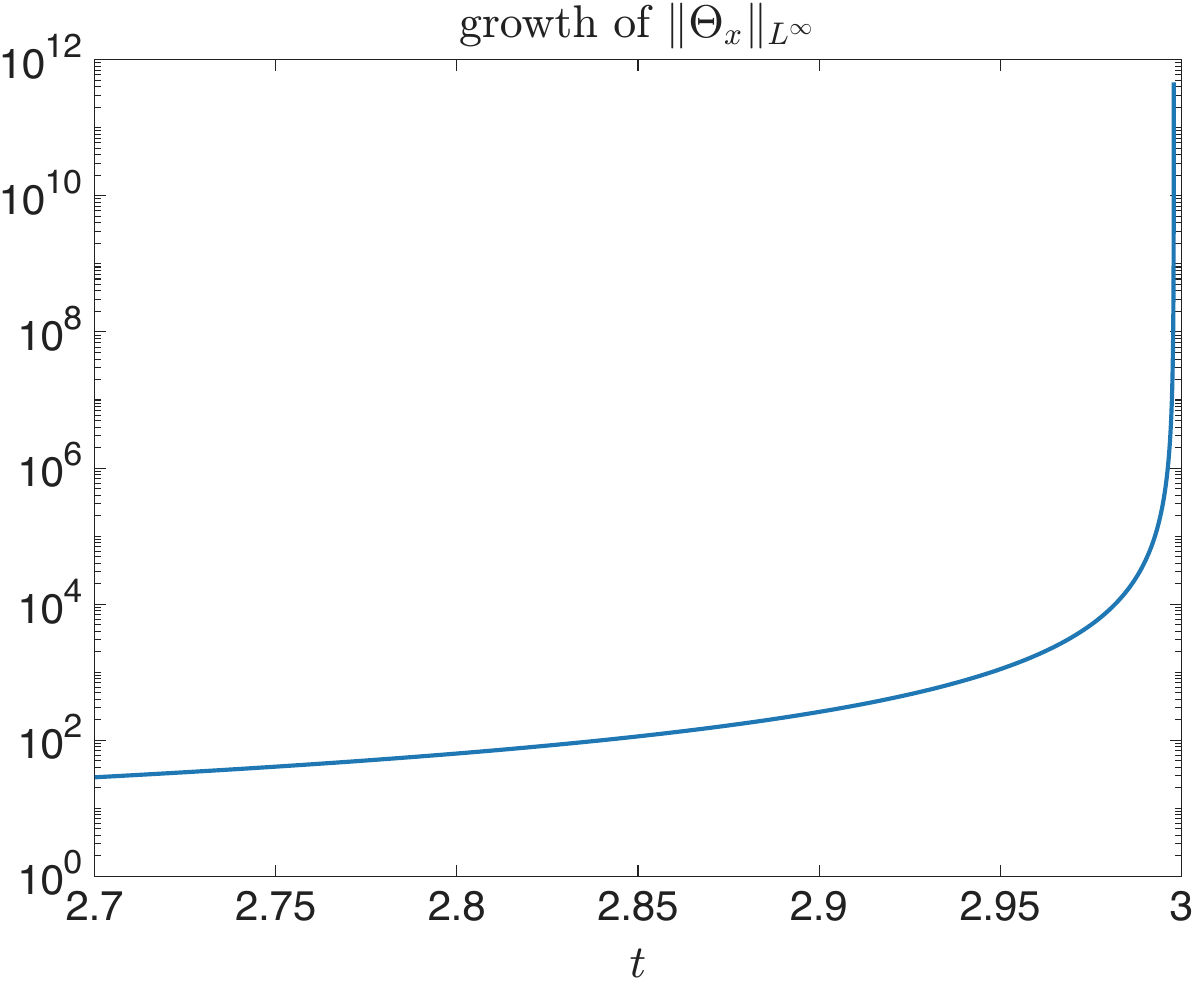}
        \includegraphics[width=0.4\textwidth]{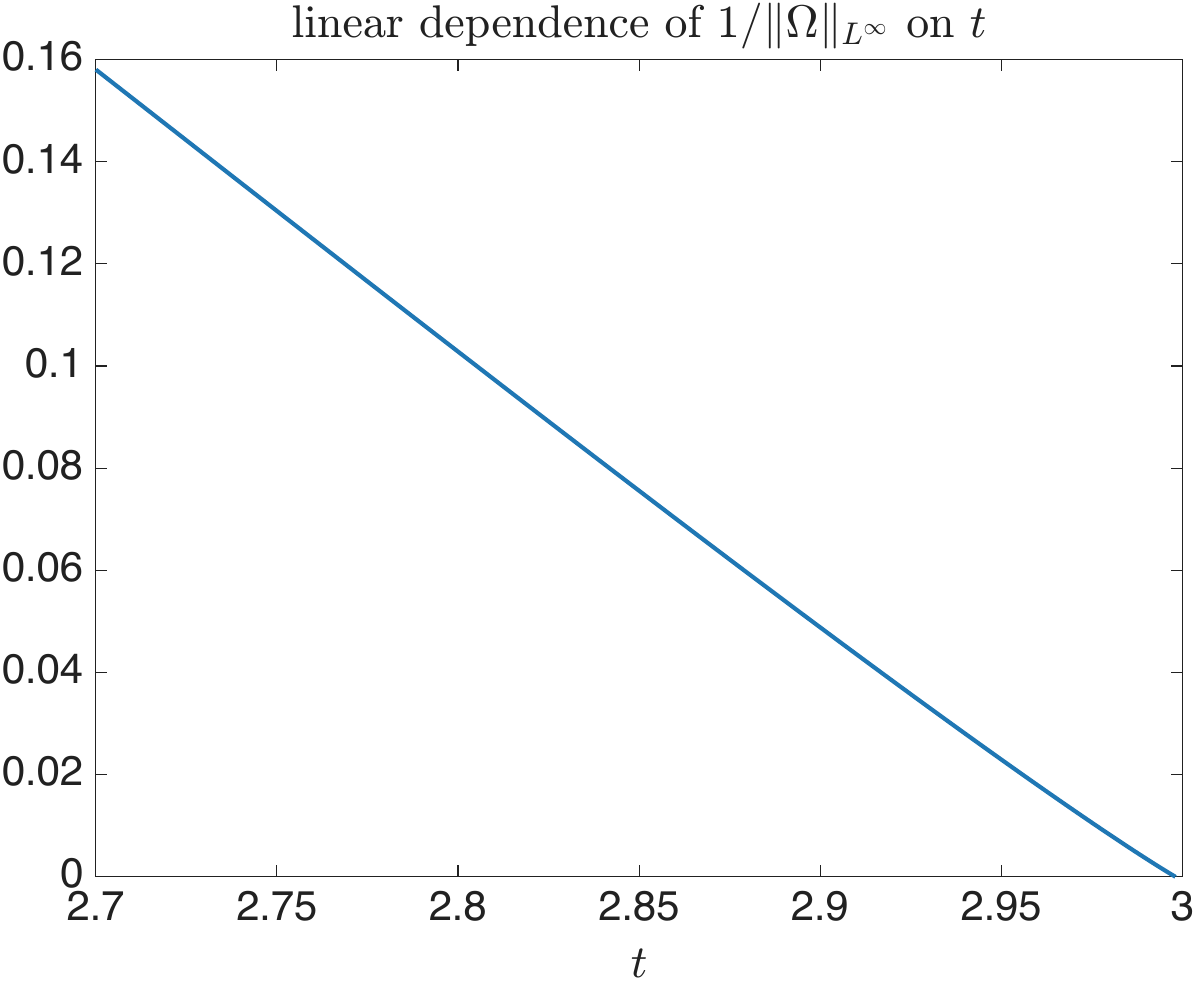}
        \includegraphics[width=0.4\textwidth]{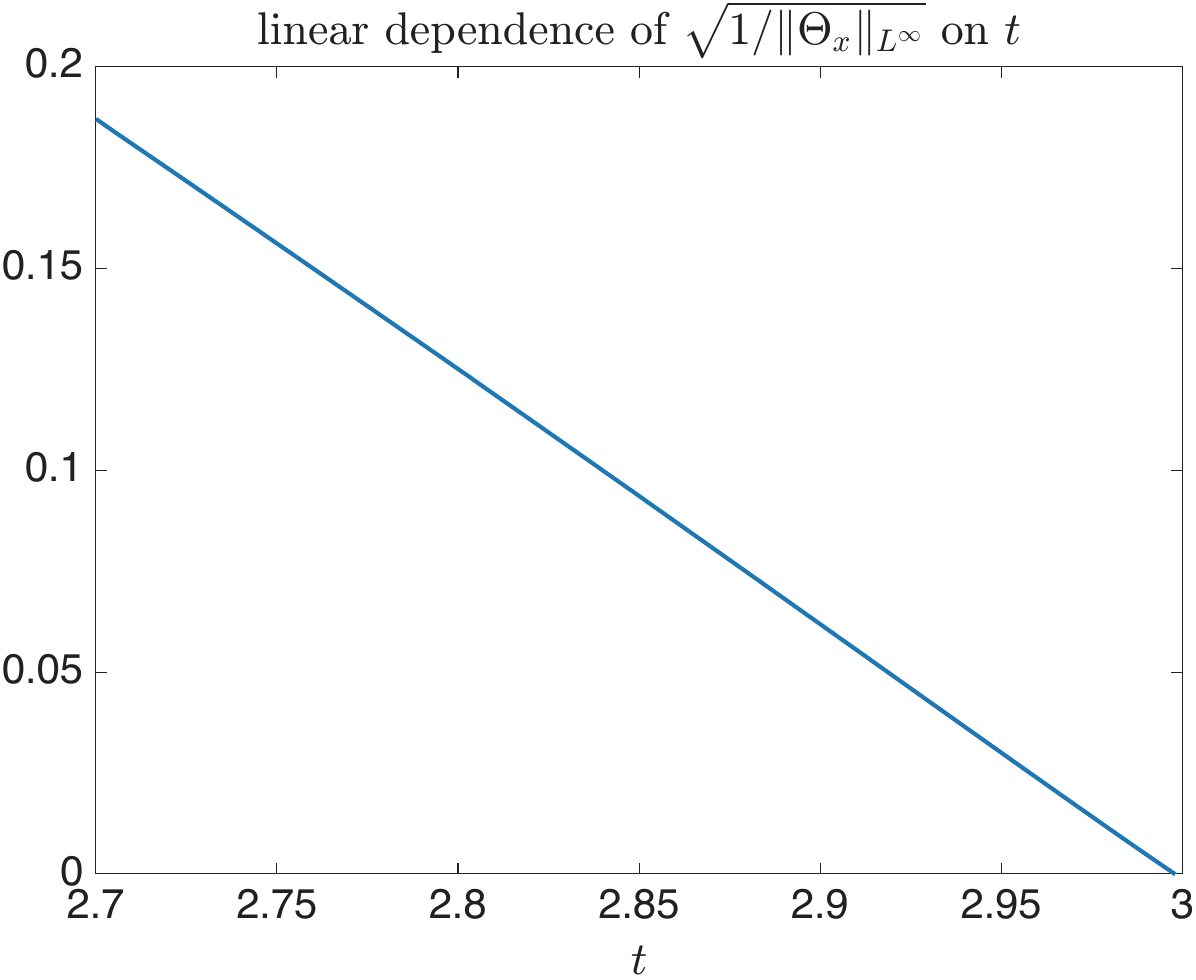}
    \caption[Time evolution of $\|\Omega\|_{L^{\infty}}$ and $\|\Theta_x\|_{L^\infty}$ in the one-sided case computed on an adaptive mesh in Scenario 1.]{Time evolution of $\|\Omega\|_{L^{\infty}}$ and $\|\Theta_x\|_{L^\infty}$ in the one-sided case computed on an adaptive mesh in Scenario 1. The top row displays the rapid growth of $\|\Omega\|_{L^{\infty}}$ and $\|\Theta_x\|_{L^\infty}$, while the bottom row illustrates that both $\|\Omega\|_{L^{\infty}}^{-1}$ and ${\|\Theta_x\|_{L^\infty}}^{-1/2}$ decay linearly with time $t$.} 
     \label{fig:HL_singleside_Linfty}
\end{figure}

\begin{figure}[!htbp]
    \centering
        \includegraphics[width=0.9\textwidth]{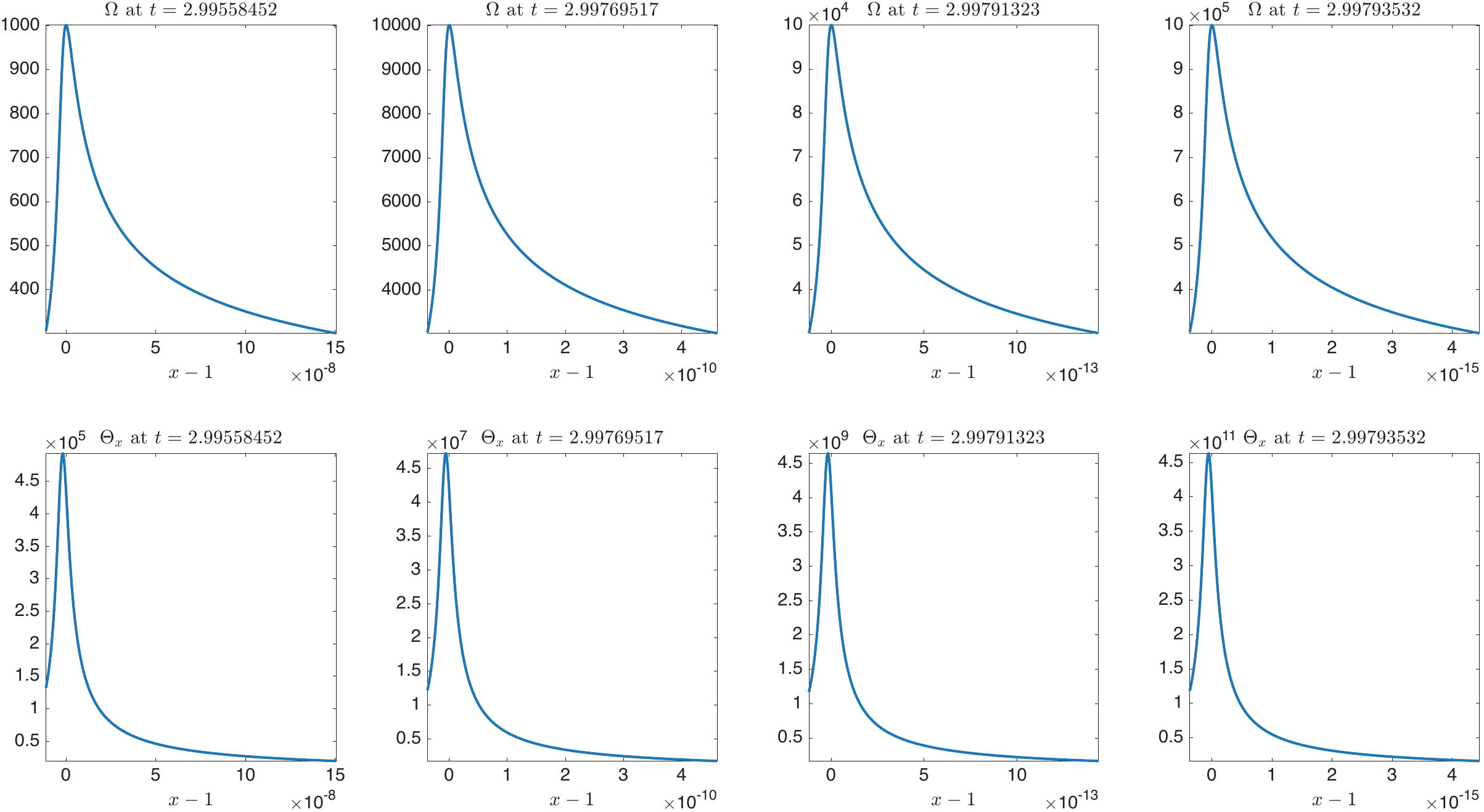}
    \caption[Inner profiles of $\Omega$ and $\Theta_x$ in the one-sided case computed on an adaptive mesh in Scenario 1.]{Inner profiles of $\Omega$ (top row) and $\Theta_x$ (bottom row) in the one-sided case computed on an adaptive mesh in Scenario 1. After proper rescaling, the inner profiles eventually stabilize at regular profiles as $t$ approaches the blowup time.} 
     \label{fig:HL_singleside_inner profile}
\end{figure}

We also investigate the long-time behavior of \eqref{eqt:numerical_dynamic_rescaling_of_HLscenario1} by conducting simulations using the same locally refined fixed mesh and stopping criterion. Recall that in the one-sided case, an explicit singular steady state solution was identified in Theorem \ref{thm:HL_weak_steady_state_simplified_version} as $(\bar{\Omega},\bar{\Theta},\bar{c}_l,\bar{c}_{\omega}) = (\mtx{1}_{\{x>1\}}/\sqrt{x-1}, \pi \mtx{1}_{\{x>1\}}/2, 2, -1)$, which serves as a rigorous reference for our study. Our numerical findings indicate that the solution to \eqref{eqt:numerical_dynamic_rescaling_of_HLscenario1} converges towards this steady state, as illustrated in Figures \ref{fig:HL_singleside_omega}, \ref{fig:HL_singleside_theta}, \ref{fig:HL_singleside_u}, and \ref{fig:HL_singleside_clcomega}.
Figures \ref{fig:HL_singleside_omega}, \ref{fig:HL_singleside_theta} and \ref{fig:HL_singleside_u} display the time evolution of the spatial profiles of $\Omega$, $\Theta$, $U$ and $U+c_lx$. Over time, the solutions are observed to approach profiles that closely match the steady state solution, which is plotted as a black dashed line in each figure for direct comparison. It is important to note that since the numerically computed $\Omega$ is bounded, it cannot perfectly resolve the singularity at $x=1$. However, outside a small neighborhood of the singularity, the agreement is remarkable. In fact, by the end of the simulation, it holds that 
\[\|(\Omega-\bar{\Omega})\phi\|_{L^{\infty}(\R\backslash[0.999,1.001])} < 10^{-4},\] 
where $\phi$ is the weight function defined by \eqref{eqt:weight_function}. Furthermore, Figure \ref{fig:HL_singleside_clcomega} tracks the time evolution of $c_l$, $c_{\om}$, and $c_l/c_{\om}$. The computed limiting values are $(c_l, c_{\om},c_l/c_{\om})=(1.9998,-0.9995,-2.0008)$, which are very close to the theoretical expectations $(\bar{c}_l, \bar{c}_{\om},\bar{c}_l/\bar{c}_{\om})=(2,-1,-2)$. In conclusion, these numerical observations suggest that beyond the Stage 1 blowup, the weak solution of \eqref{eqt:1Dhouluo} undergoes a local $L^p$ blowup with singular limiting self-similar profiles constructed in Theorem \ref{thm:HL_weak_steady_state_simplified_version}. Moreover, as discussed in Remark \ref{rem:HL_critical_Lp_exponent} and Remark \ref{rem:HL_L2blowup}, our numerical results suggest that $p=2$.

\begin{figure}[!htbp]
\centering
        \includegraphics[width=0.32\textwidth]{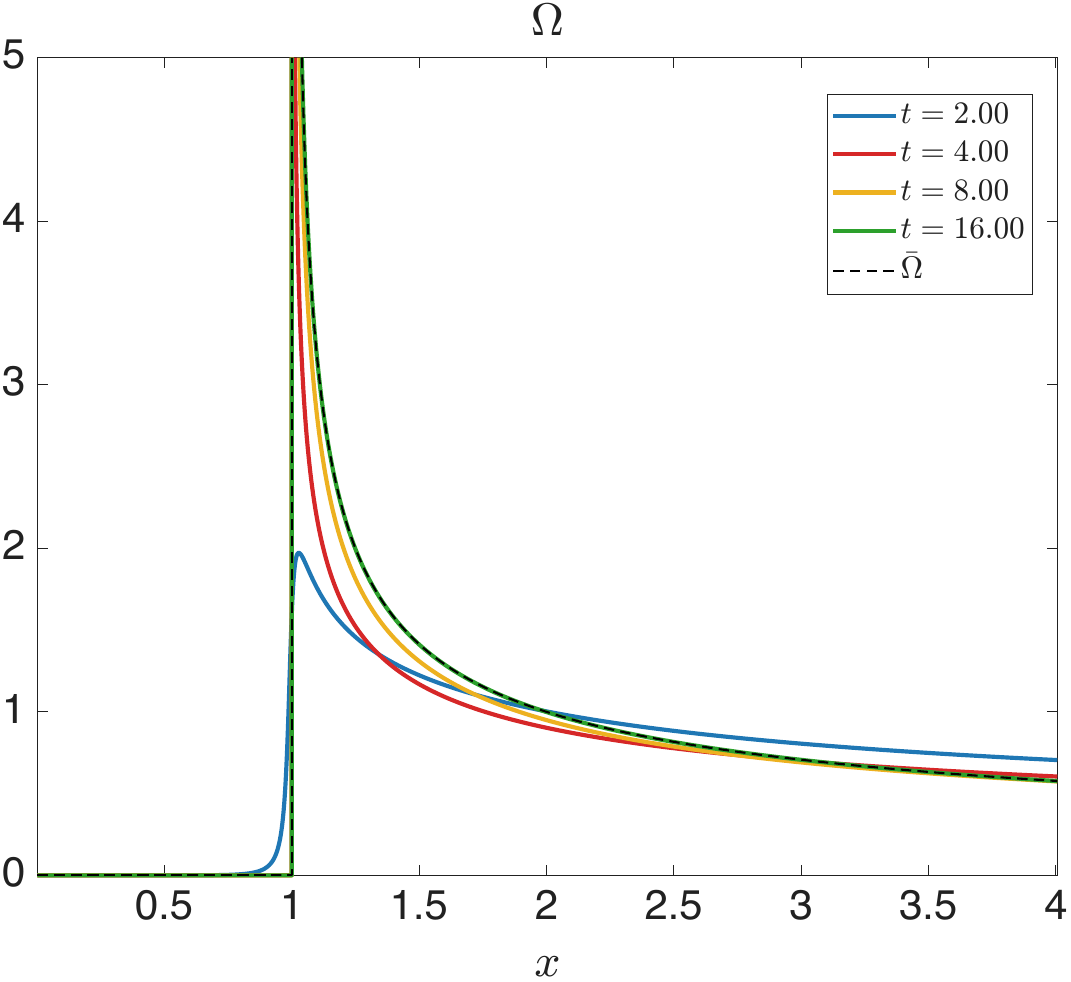}
        \includegraphics[width=0.32\textwidth]{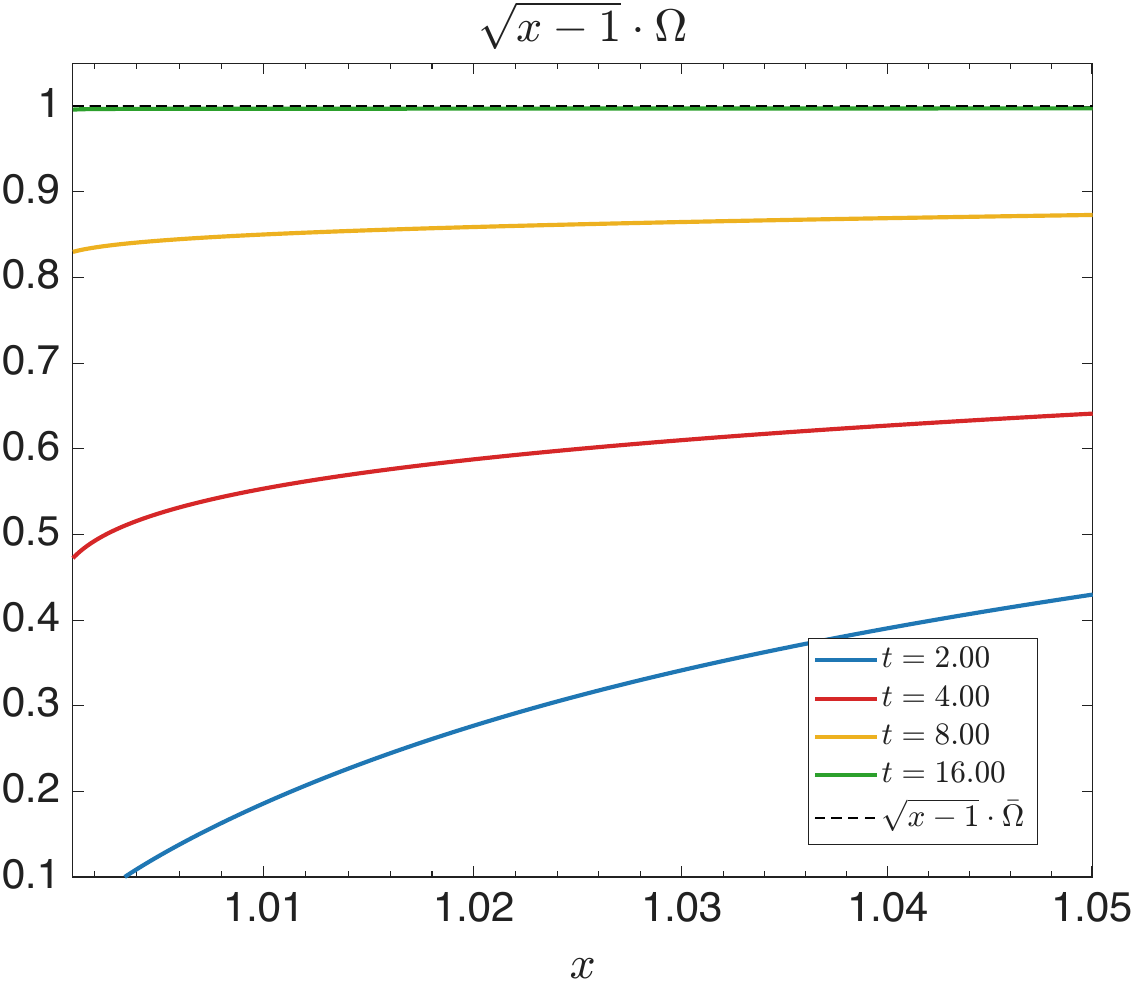}
        \includegraphics[width=0.32\textwidth]{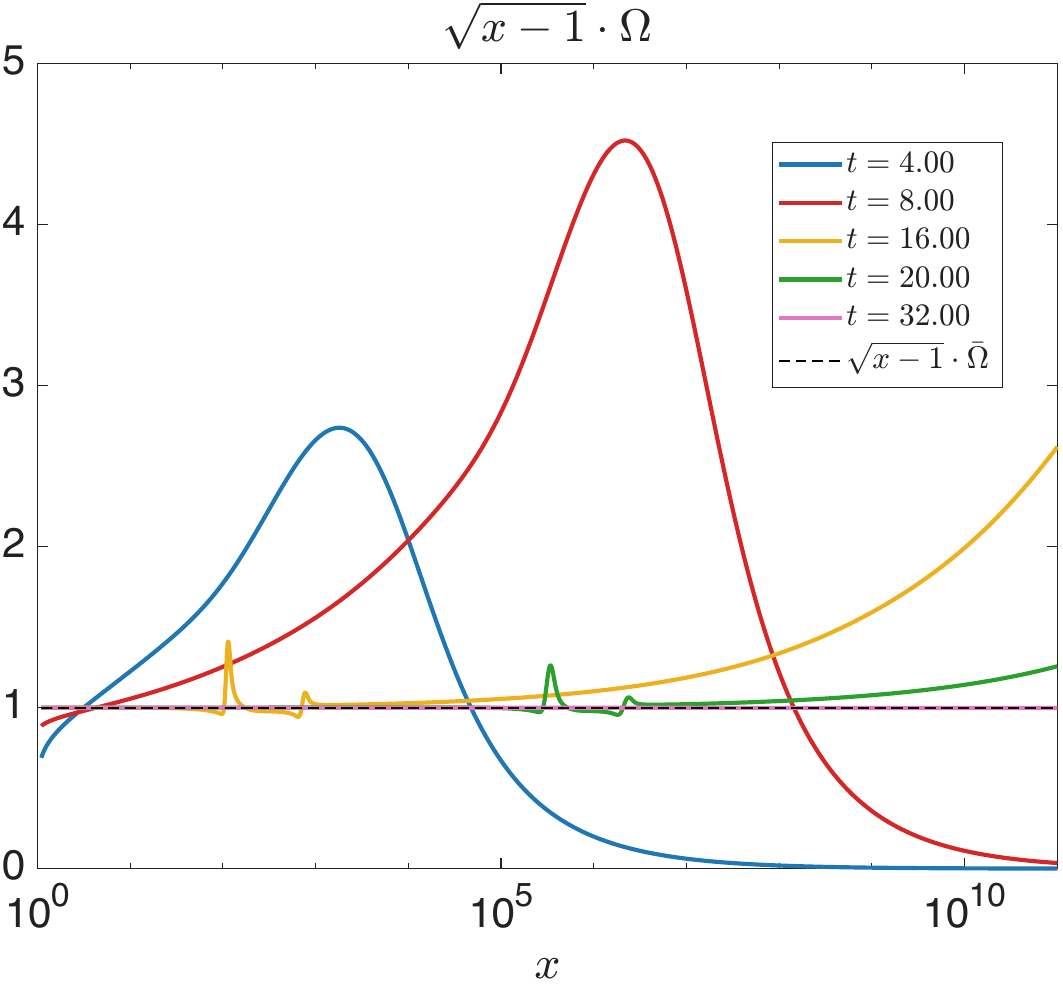}
    \caption[Evolution of spatial profiles of $\Omega$ in different regions (one-sided case, Scenario 1, computed on a fixed mesh).]{Evolution of spatial profiles of $\Omega$ in different regions (one-sided case, Scenario 1, computed on a fixed mesh). Left: $\Omega$ in the $O(1)$ region. Middle: $\sqrt{x-1}\cdot\Omega$ near the singularity at $x=1$. Right: $\sqrt{x-1}\cdot\Omega$ in the far field. The black dashed lines represent the steady state solution $\bar{\Omega}=\mtx{1}_{\{x>1\}}/\sqrt{x-1}$. As observed, $\Omega$ eventually shows excellent agreement with $\bar{\Omega}$. }  
     \label{fig:HL_singleside_omega}
\end{figure}

\begin{figure}[!htbp]
\centering
        \includegraphics[width=0.4\textwidth]{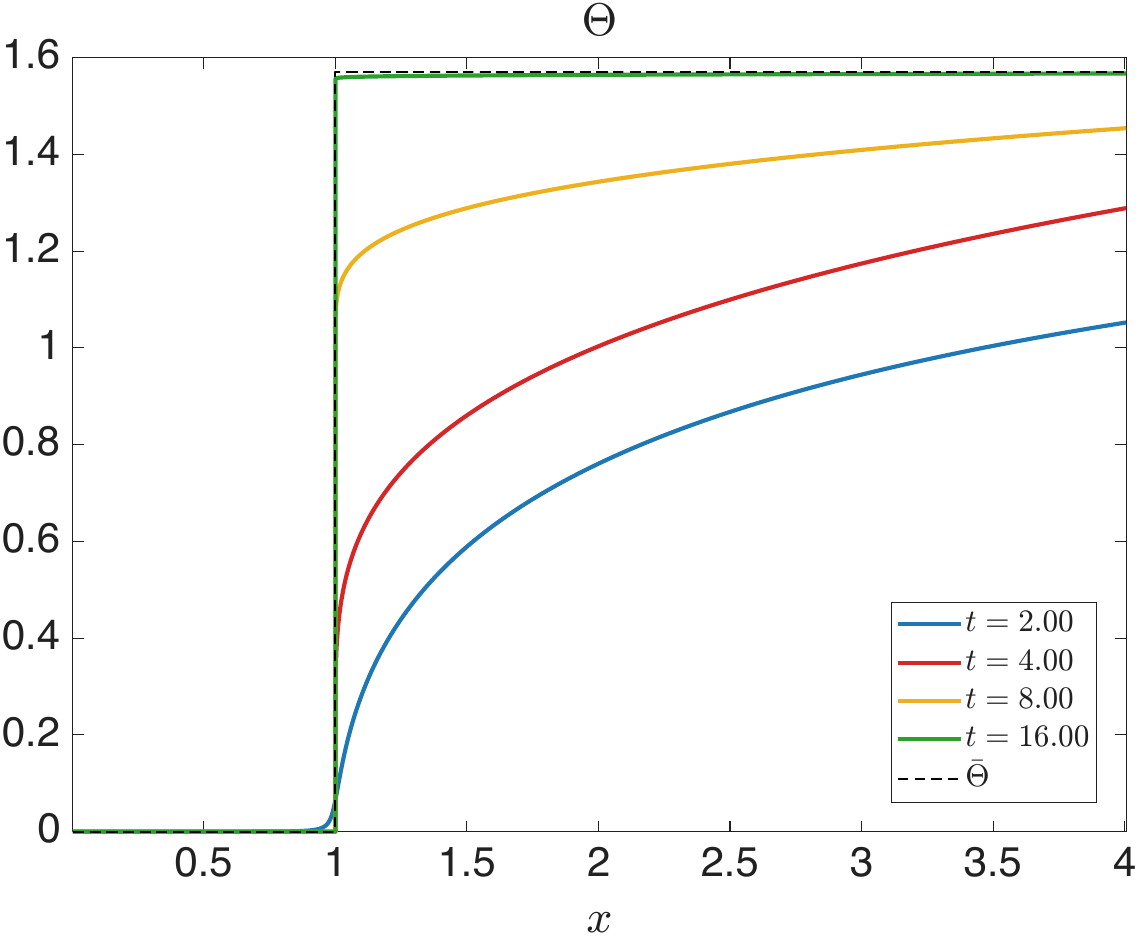}
        \includegraphics[width=0.4\textwidth]{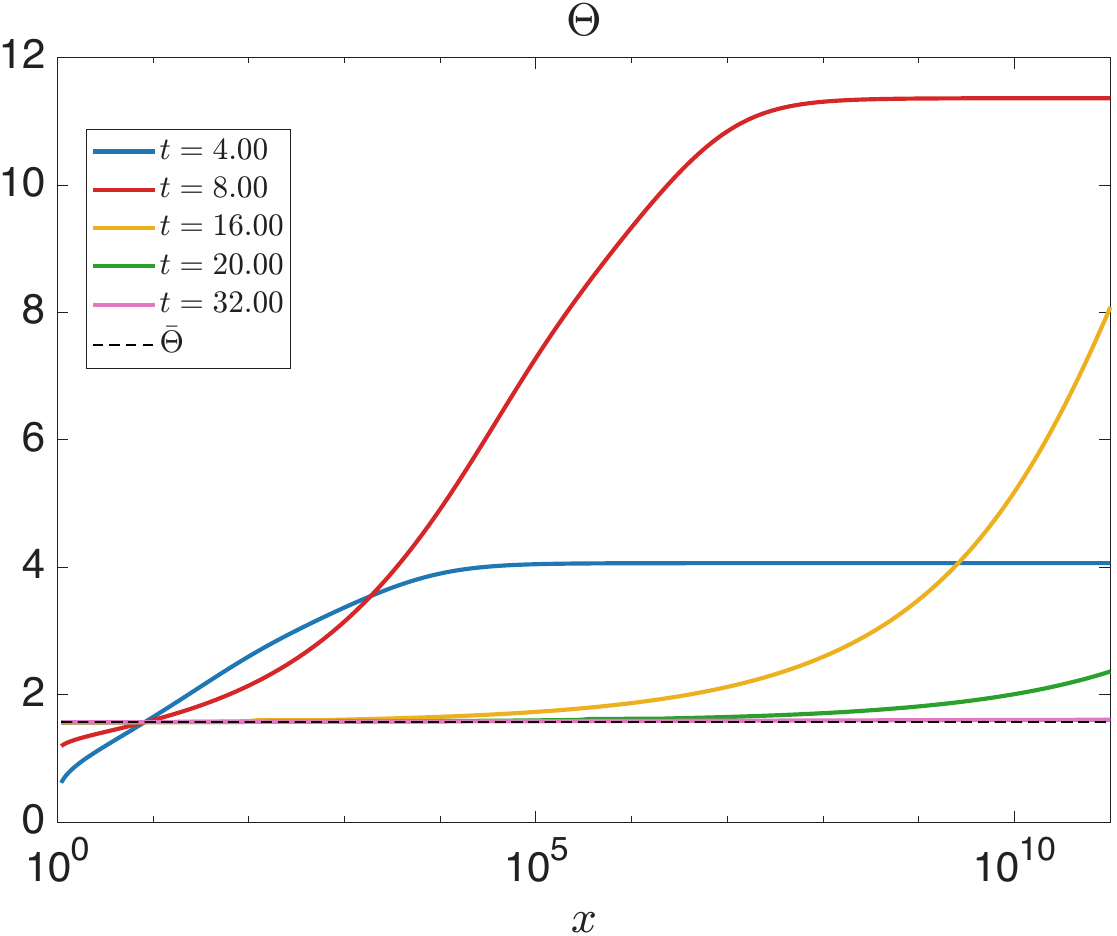}
    \caption[Evolution of spatial profiles of $\Theta$ in different regions (one-sided case, Scenario 1, computed on a fixed mesh).]{Evolution of spatial profiles of $\Theta$ in different regions (one-sided case, Scenario 1, computed on a fixed mesh). Left: The $O(1)$ region. Right: The far field. The steady state solution $\bar{\Theta}=\pi \mtx{1}_{\{x>1\}}/2$ is marked in black dashed lines. As shown, $\Theta$ closely matches $\bar{\Theta}$ at later times.}
      \label{fig:HL_singleside_theta}
\end{figure}

\begin{figure}[!htbp]
\centering
        \includegraphics[width=0.4\textwidth]{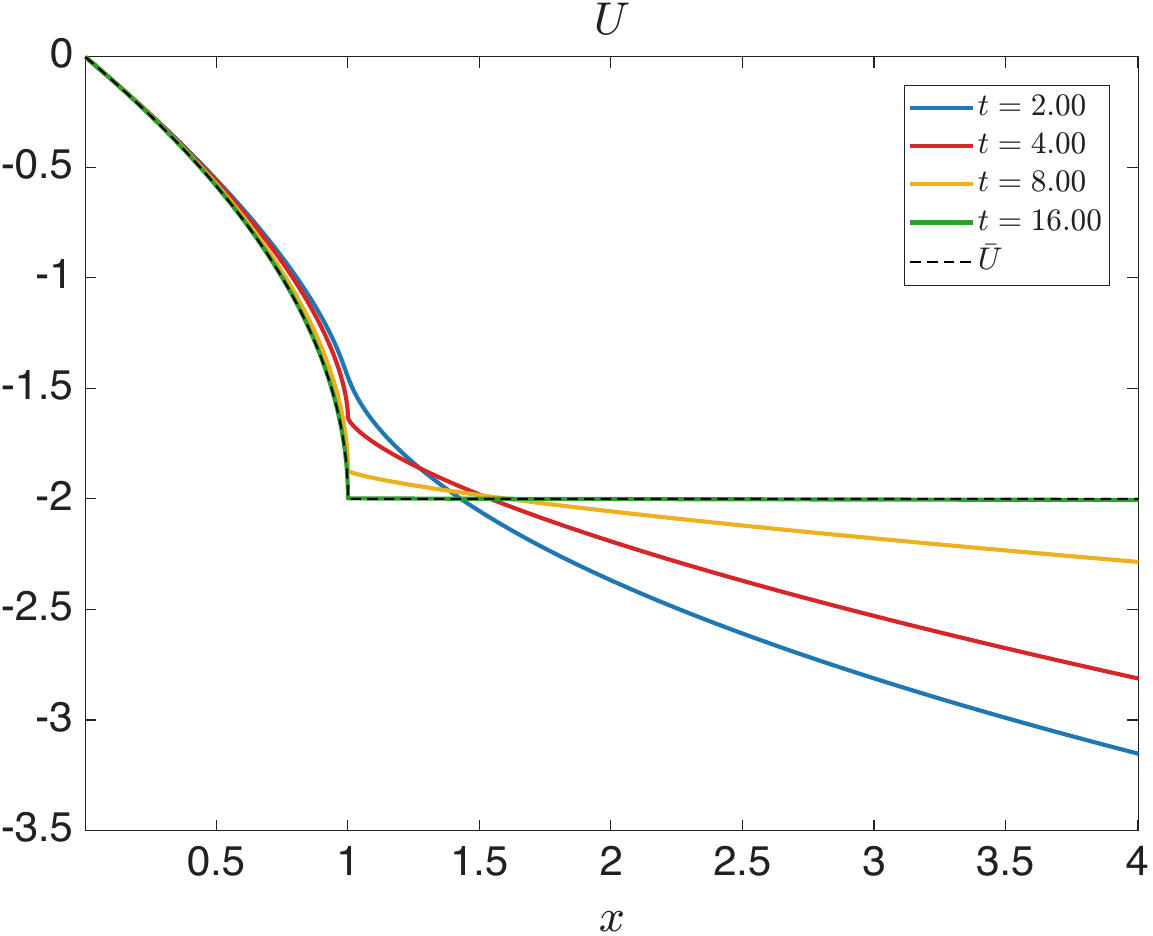}
        \includegraphics[width=0.4\textwidth]{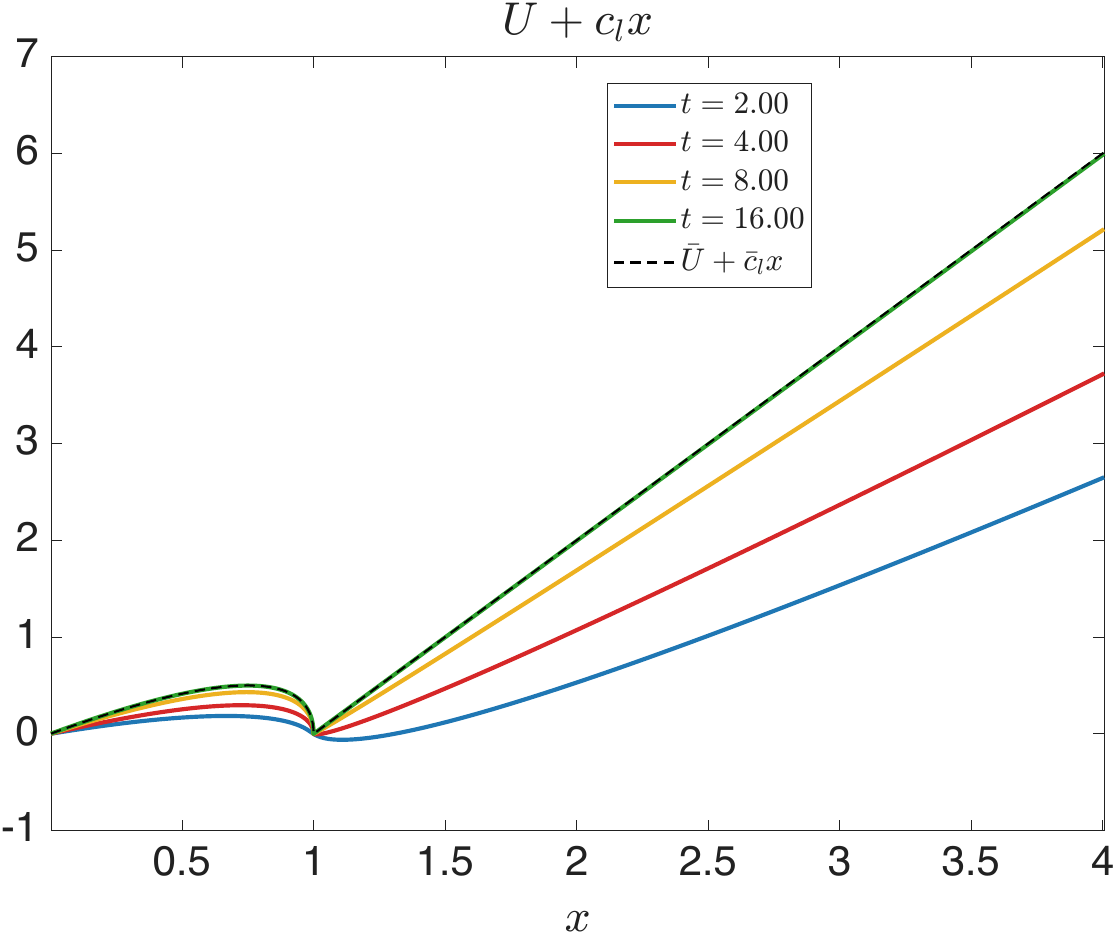}
    \caption[Evolution of spatial profiles of $U$ (left) and $U+c_l x$ (right) in the one-sided case computed on a fixed mesh in Scenario 1. ]{Evolution of spatial profiles of $U$ (left) and $U+c_l x$ (right) in the one-sided case computed on a fixed mesh in Scenario 1. The black dashed lines denote the steady state solutions $\bar{U}=2\sqrt{1-x}\mtx{1}_{\{x<1\}}-2$ and $\bar{U}+ \bar{c}_l x=2\sqrt{1-x}\mtx{1}_{\{x<1\}}+2x-2$. As illustrated, $U$ and $U+c_l x$ approach $\bar{U}$ and $\bar{U}+ \bar{c}_l x$, respectively.}
      \label{fig:HL_singleside_u}
\end{figure}

\begin{figure}[!htbp]
\centering
        \includegraphics[width=0.32\textwidth]{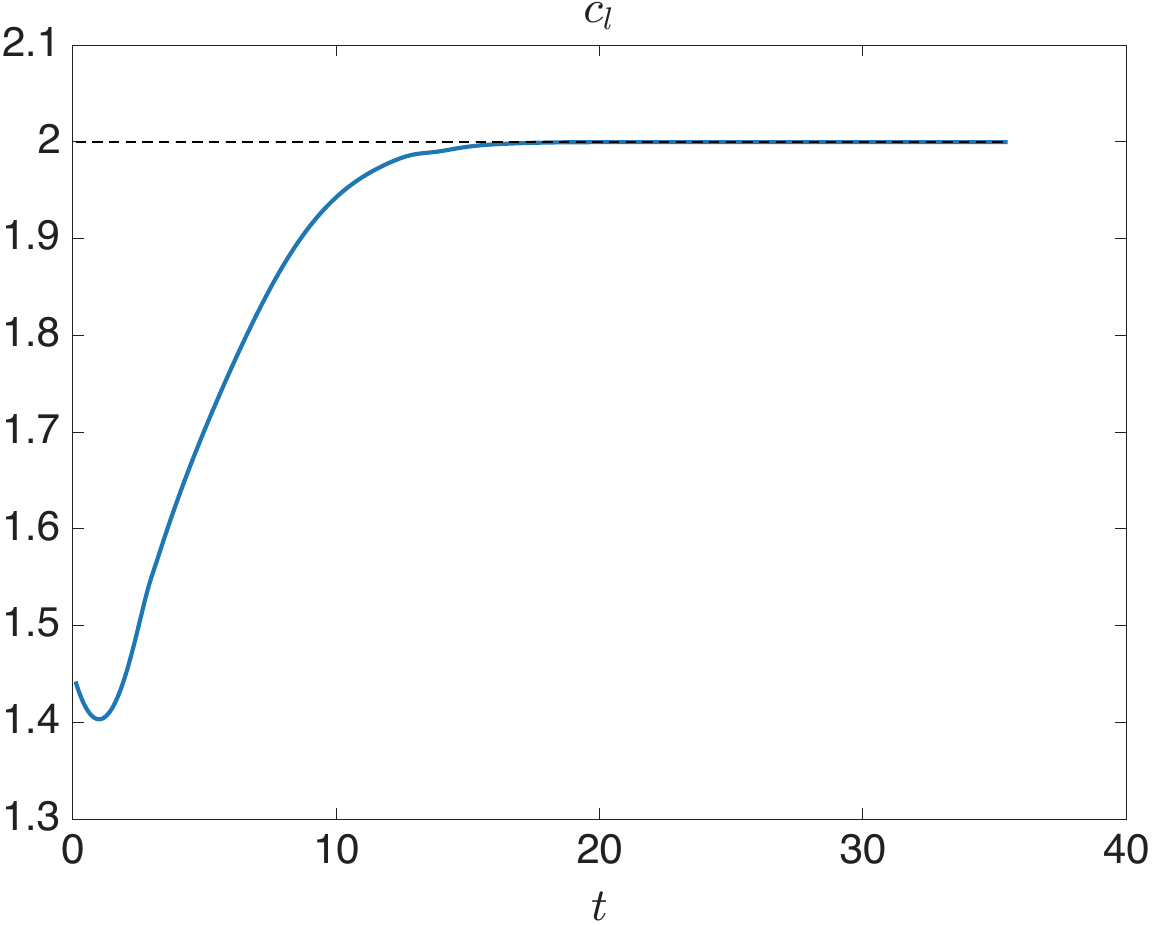}
        \includegraphics[width=0.32\textwidth]{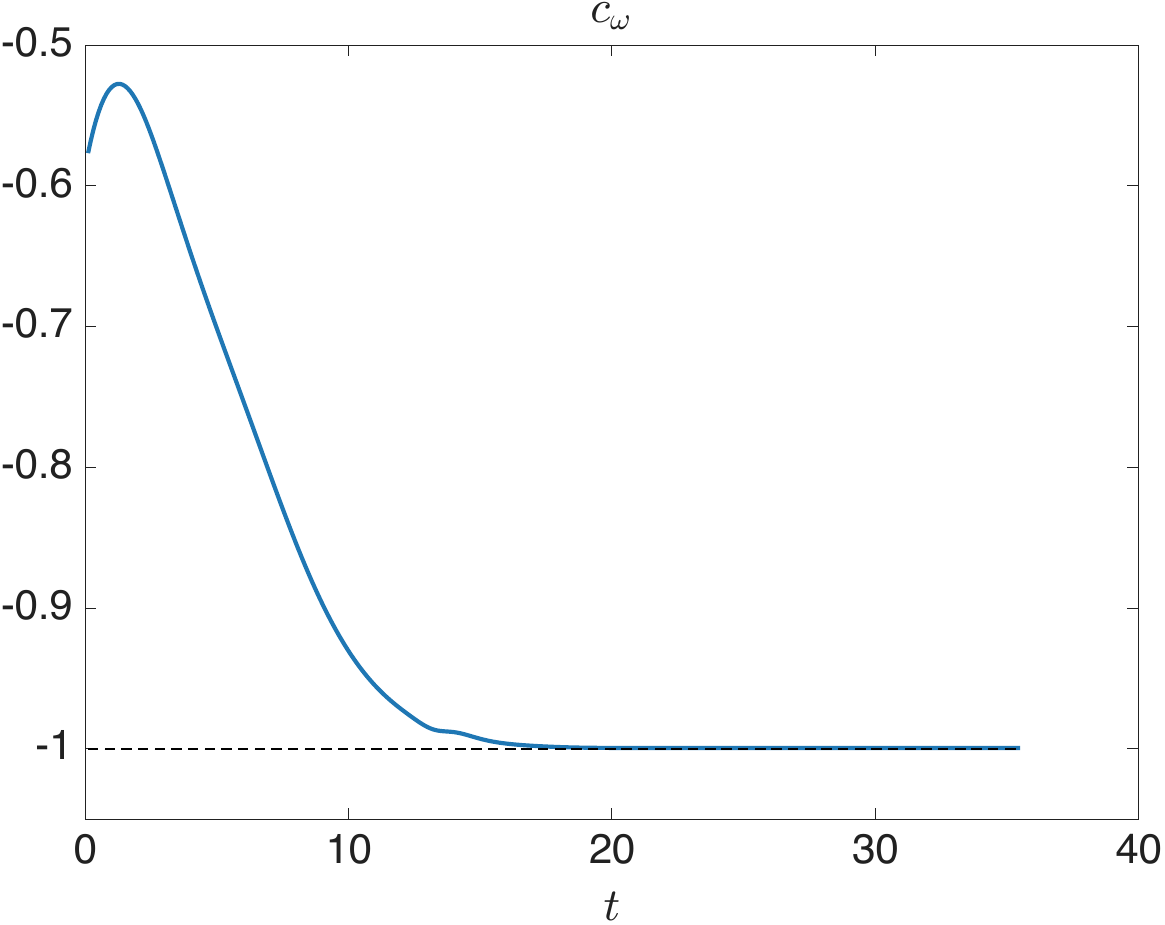}
        \includegraphics[width=0.32\textwidth]{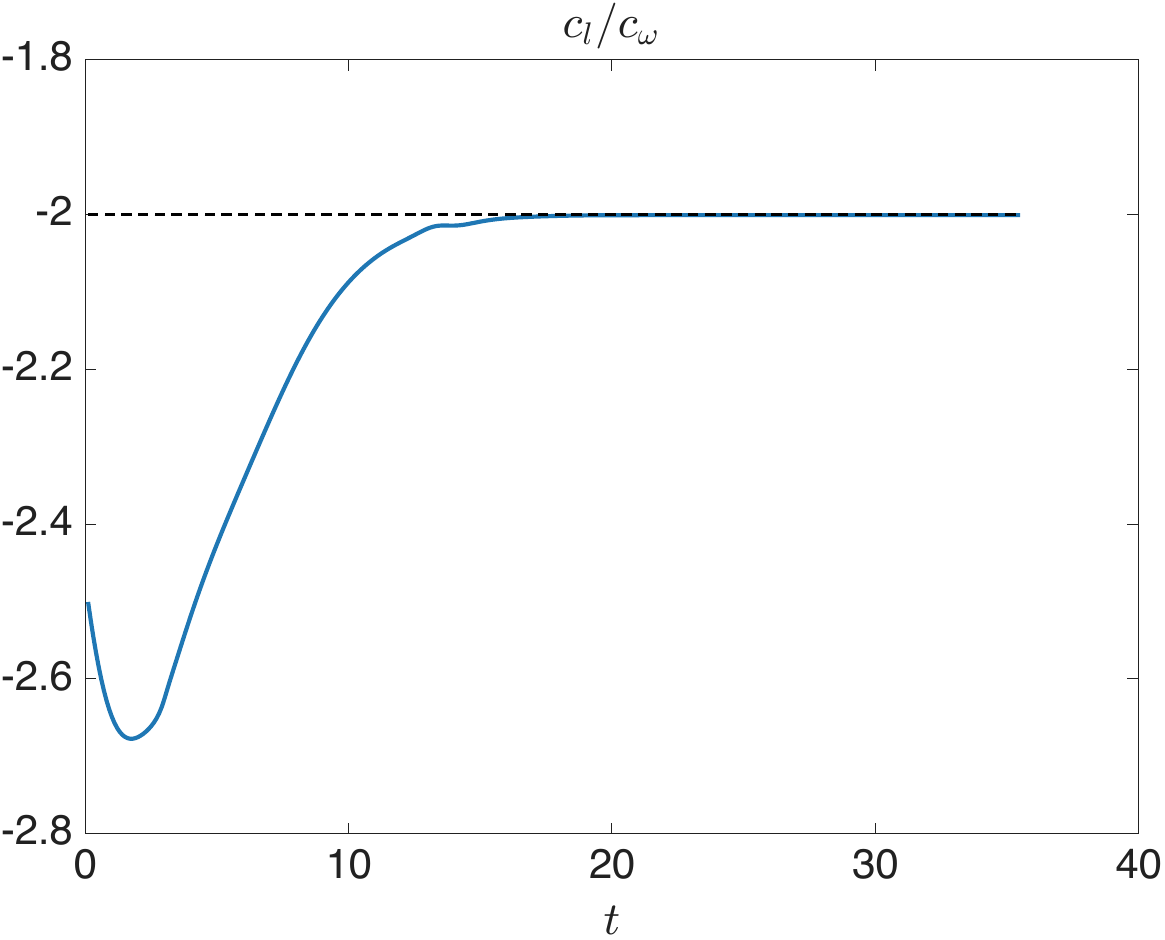}
    \caption[Time evolution of $c_l$, $c_{\om}$ and $c_l/c_{\om}$ in the one-sided case computed on a fixed mesh in Scenario 1.]{Time evolution of $c_l$ (left), $c_{\om}$ (middle) and $c_l/c_{\om}$ (right) in the one-sided case computed on a fixed mesh in Scenario 1. The black dashed lines represent the reference values $(\bar{c}_l,\bar{c}_{\om},\bar{c}_l/\bar{c}_{\om})=(2,-1,-2)$, respectively.}  
     \label{fig:HL_singleside_clcomega}
\end{figure}
\begin{remark}
    In this section, we have used a weighted $L^3$ norm and a weighted $L^{\infty}$ seminorm to quantify the convergence of the numerical solution. These measures are chosen for the convenience of comparison and visualization. One can choose other measures under which the singular steady state is bounded and quantify the convergence of the numerical solution.
    
    However, our numerical observations only suggest the convergence of the solution to a potential singular steady state in these measures (and for some particular degenerate initial data), not a general convergence result as described in Conjecture \ref{conj:HL_singular_steady_state_stability}.
    As discussed in Remark \ref{rem:HL_singular_steady_state_stability}, the measure of convergence for a rigorous stability analysis of the singular steady state needs to be designed carefully. In particular, it relies on the choice of a nice weight function which is usually singular at the origin, whereas the weight function used in our numerical computations does not have this property. Nevertheless, if convergence is established in some suitable norm, it also follows in any weaker norms such as those we use in this paper. Therefore, our numerical results provide necessary evidence that Conjecture \ref{conj:HL_singular_steady_state_stability} can hold in general.
\end{remark}

\section{Numerical Results of the HL Model: Scenario 2}\label{sec:HL_scenario2} 

In Section \ref{sec:HL_scenario1}, we numerically observed the Stage 1 blowup phenomenon for the HL model \eqref{eqt:1Dhouluo}. More precisely, under the assumption that the velocity field $u$ vanishes at the origin $x=0$, our simulations show that \eqref{eqt:1Dhouluo} may develop asymptotically self-similar blowup from degenerate initial data, with blowup points located away from the origin. Moreover, by zooming in on the most singular region of the solutions, we observe that the corresponding inner profiles exhibit no symmetry property. Our simulations suggest that the dynamic rescaling equations \eqref{eqt:dynamic_rescaling_of_HL} may fail to converge to such steady state solutions when no symmetry condition is imposed. To resolve the non-symmetric self-similar profiles, in Subsection \ref{subsec:dynamic_rescaling_HLscenario2} we introduce a modified dynamic rescaling formulation \eqref{eqt:dynamic_rescaling_of_HLscenario2}, in which an extra translation parameter $c_r$ is used to track the spatial translation of the blowup point. In this section, we conduct simulations of \eqref{eqt:dynamic_rescaling_of_HLscenario2} and report numerical evidence for the existence of non-symmetric self-similar profiles which agree remarkably well with the inner profiles observed in the Stage 1 blowup of Scenario 1.

Similar to the previous numerical simulations, instead of studying the evolution of $\Theta$ directly, we prefer to work with quantities about $\Theta$ that exhibit suitable decay in the far field. We therefore consider the change of variable $V:=\Theta_x$ and rewrite \eqref{eqt:dynamic_rescaling_of_HLscenario2} as 
 \begin{equation}\label{eqt:dynamic_rescaling_of_HL_scenario2_numerical}
    \begin{aligned}
    &\Omega_t+(U+c_lx+c_r)\Omega_x=c_{\om}\Omega+V,\\
    &V_t+(U+c_lx+c_r)V_x=(2c_{\omega}-U_x)V,\\
    & U_x=\mtx{H}(\Omega),\quad U(0)=0.
    \end{aligned}
\end{equation}

Again, we need to determine the values of $c_l$, $c_{\omega}$ and $c_r$ by imposing 3 normalization conditions. More specifically, we choose $c_l$, $c_{\omega}$ and $c_r$ to enforce that 
\[ \partial_t\Omega(0,t)=\partial_t\Omega_x(0,t)=\partial_tV(0,t)\equiv 0.\] Correspondingly, $c_l$, $c_{\omega}$ and $c_r$ are obtained by solving the following linear equations:
\begin{equation}\label{eqt:HLscenario2_normalization}
    \begin{cases} \Omega_x(0)c_r -\Omega(0)c_{\omega}=V(0),\\
                V_x(0)c_r-2V(0)c_{\omega}=-U_x(0)V(0),\\
                \Omega_{xx}(0)c_r-\Omega_x(0)c_{\omega}+\Omega_x(0)c_l=V_x(0)-U_x(0)\Omega_x(0).
                \end{cases}
    \end{equation}
We perform numerical simulations of \eqref{eqt:dynamic_rescaling_of_HL_scenario2_numerical} imposing the normalization condition \eqref{eqt:HLscenario2_normalization} and terminate the iteration when the residuals satisfy \[\max\{\|\Omega_t\|_{L^{\infty}},\|V_t\|_{L^{\infty}}\}<10^{-6},\] 
where the $L^{\infty}$ norm is evaluated over all computational grid points.  

By the end of our simulation, the above stopping criterion is eventually satisfied, suggesting that the solutions of \eqref{eqt:dynamic_rescaling_of_HL_scenario2_numerical} converge to regular limiting profiles. The evolution of the spatial profiles of $\Omega$ and $\Theta_x$ over time is illustrated in Figure \ref{fig:HL_scenario2_profiles_evolution}. As observed, $\Omega$ and $\Theta_x$ approach non-symmetric regular profiles that remain strictly positive throughout the computational domain. Figure \ref{fig:HL_scenario2_parameters_evolution} tracks the convergence of the parameters $c_l$, $c_{\om}$, and $c_r$, as well as the ratio $c_l/c_{\om}$. The computed limiting value of $c_l/c_{\om}$ is $-2.5114$. Notably, this value is distinct from the ratio of approximately $-2.9987$ reported in \cite{chen2022asymptotically} for the odd-symmetric non-degenerate case. Thus, our results confirm the existence of a self-similar blowup solution of the HL model with non-symmetric regular profiles, which possesses a spatial contraction rate distinct from the existing literature.

\begin{figure}[!htbp]
\centering
        \includegraphics[width=0.4\textwidth]{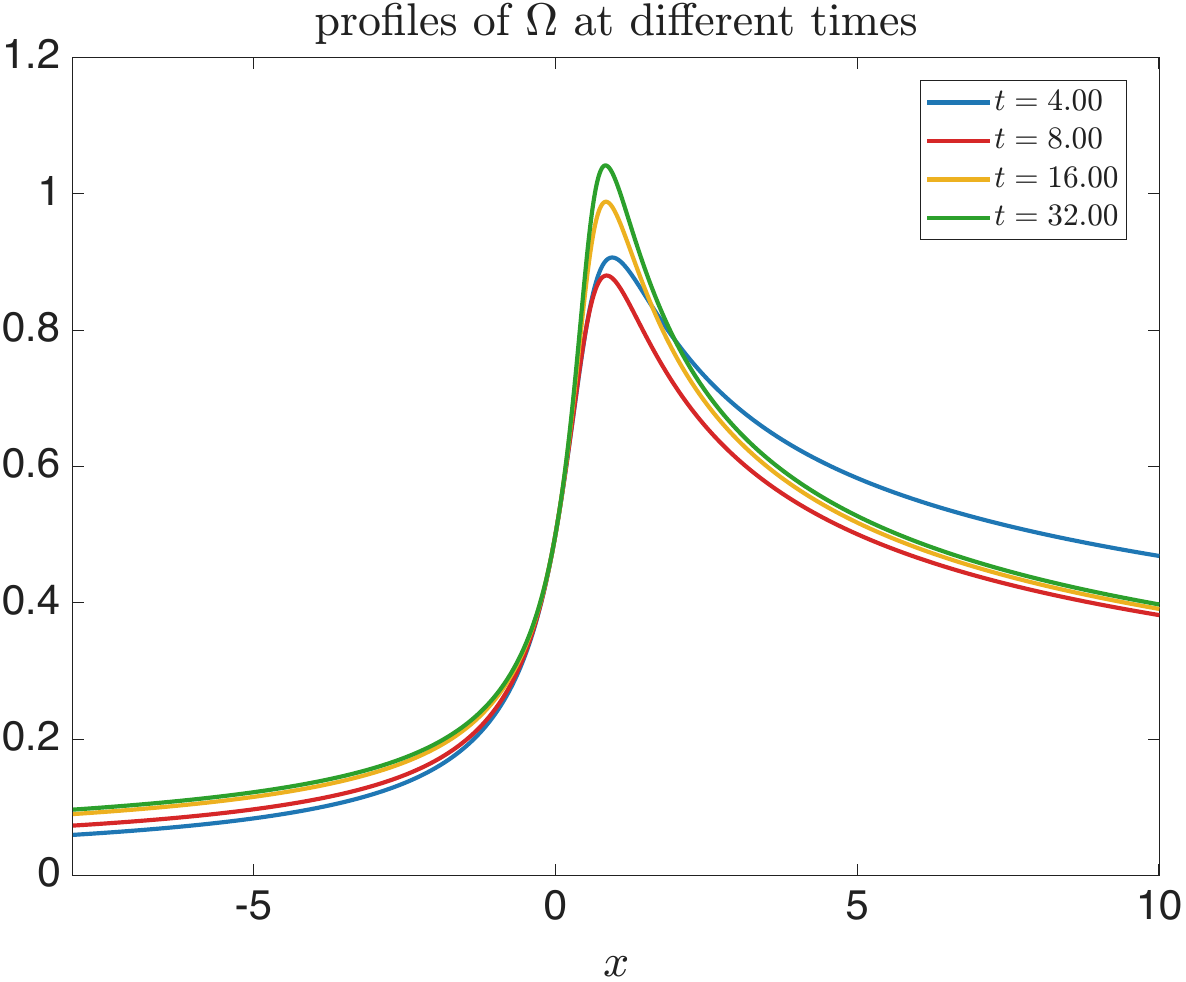}
        \includegraphics[width=0.4\textwidth]{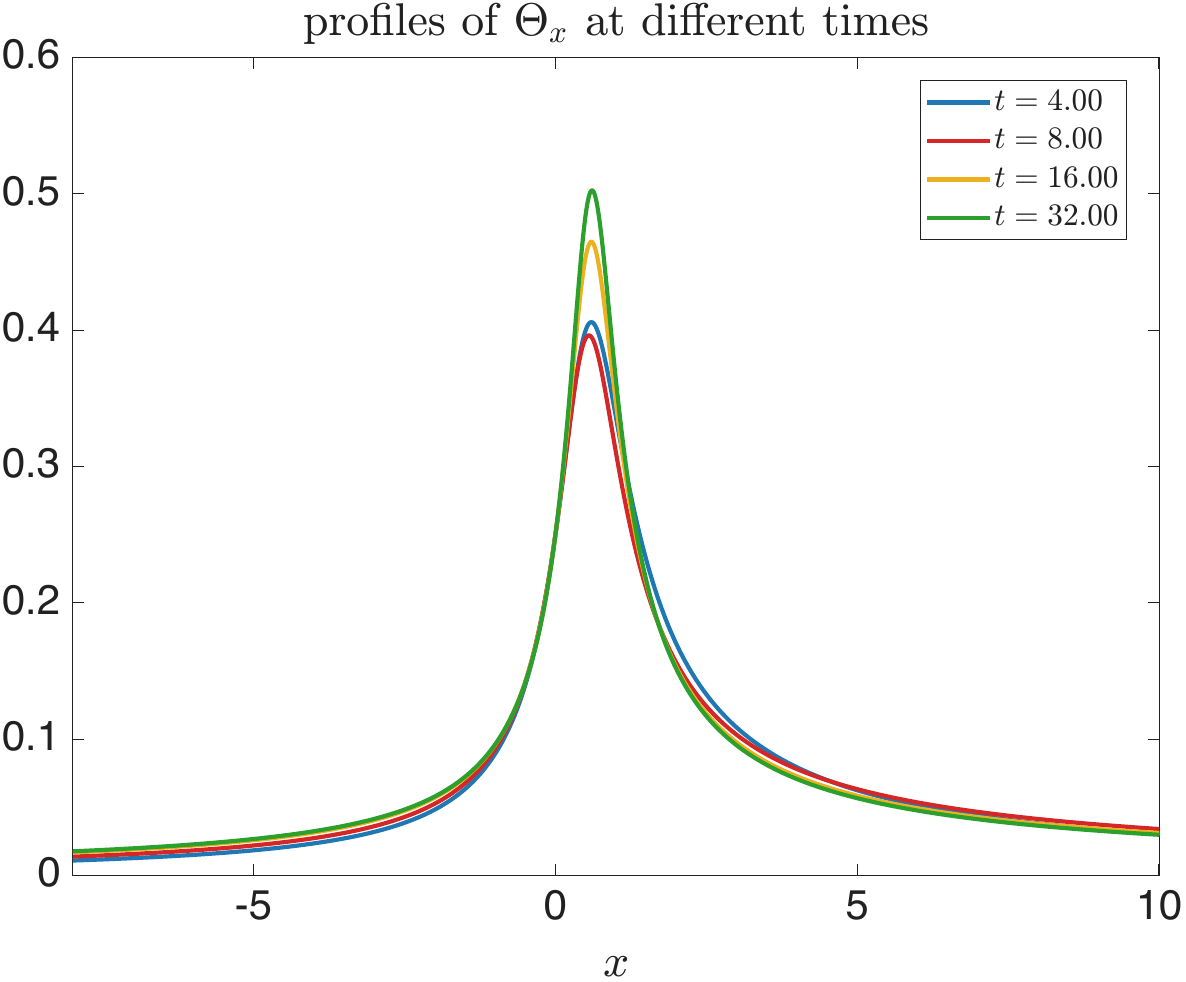}
    \caption[Evolution of spatial profiles of $\Omega$ and $\Theta_x$ in Scenario 2.]{Evolution of spatial profiles of $\Omega$ (left) and $\Theta_x$ (right) in Scenario 2. As illustrated, $\Omega$ and $\Theta_x$ eventually approach non-symmetric regular limiting profiles.}  
      \label{fig:HL_scenario2_profiles_evolution}
\end{figure}

\begin{figure}[!htbp]
\centering
        \includegraphics[width=0.38\textwidth]{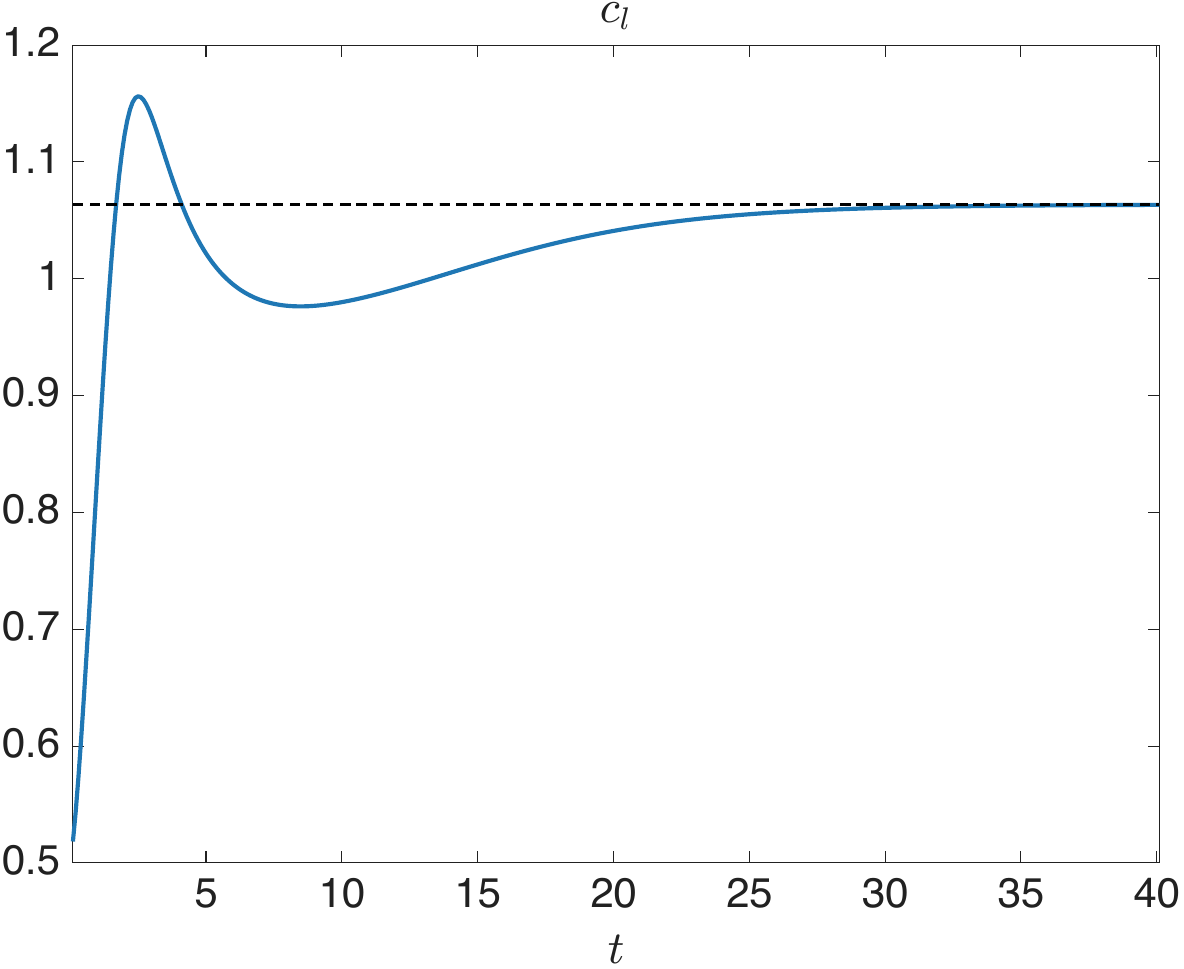}
        \includegraphics[width=0.38\textwidth]{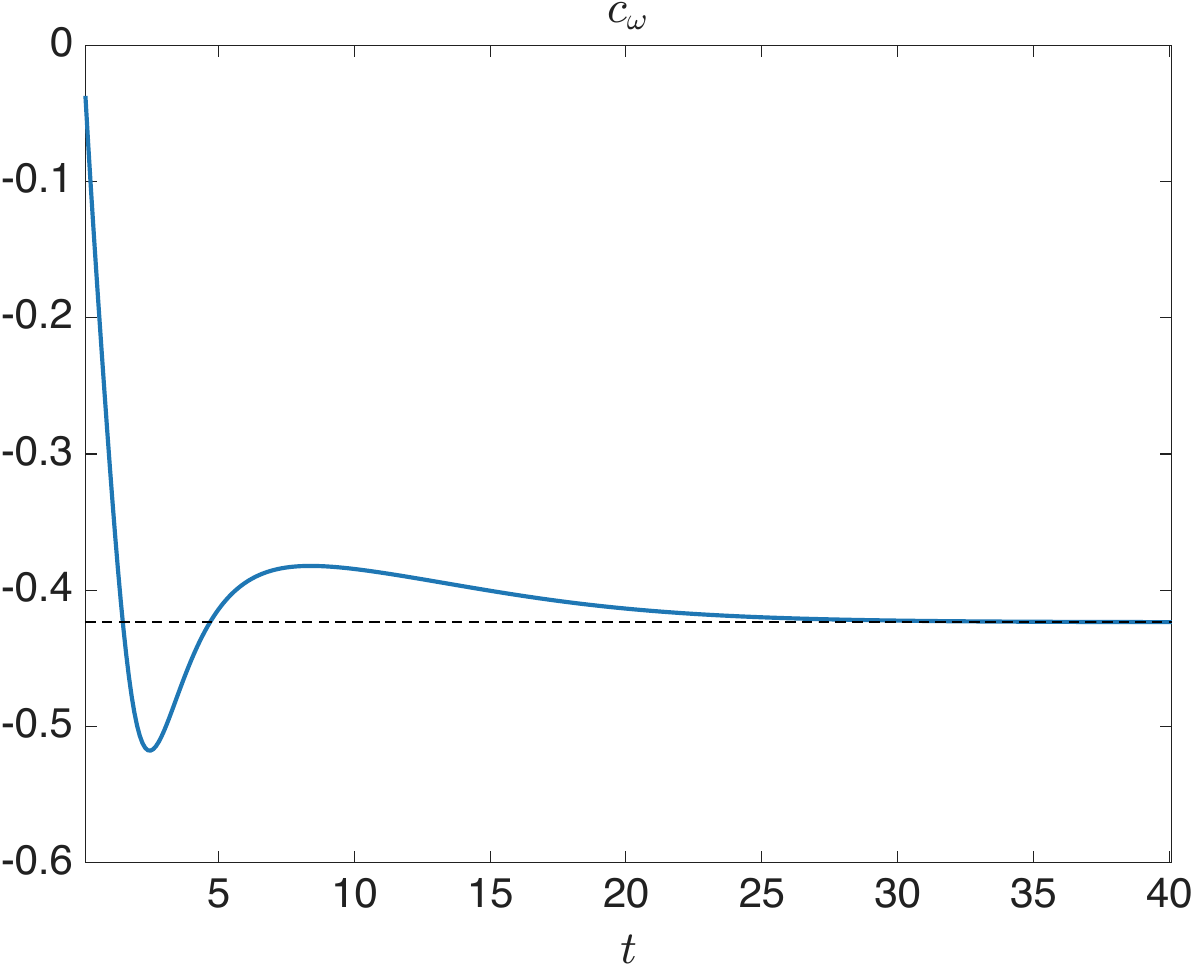}
         \includegraphics[width=0.38\textwidth]{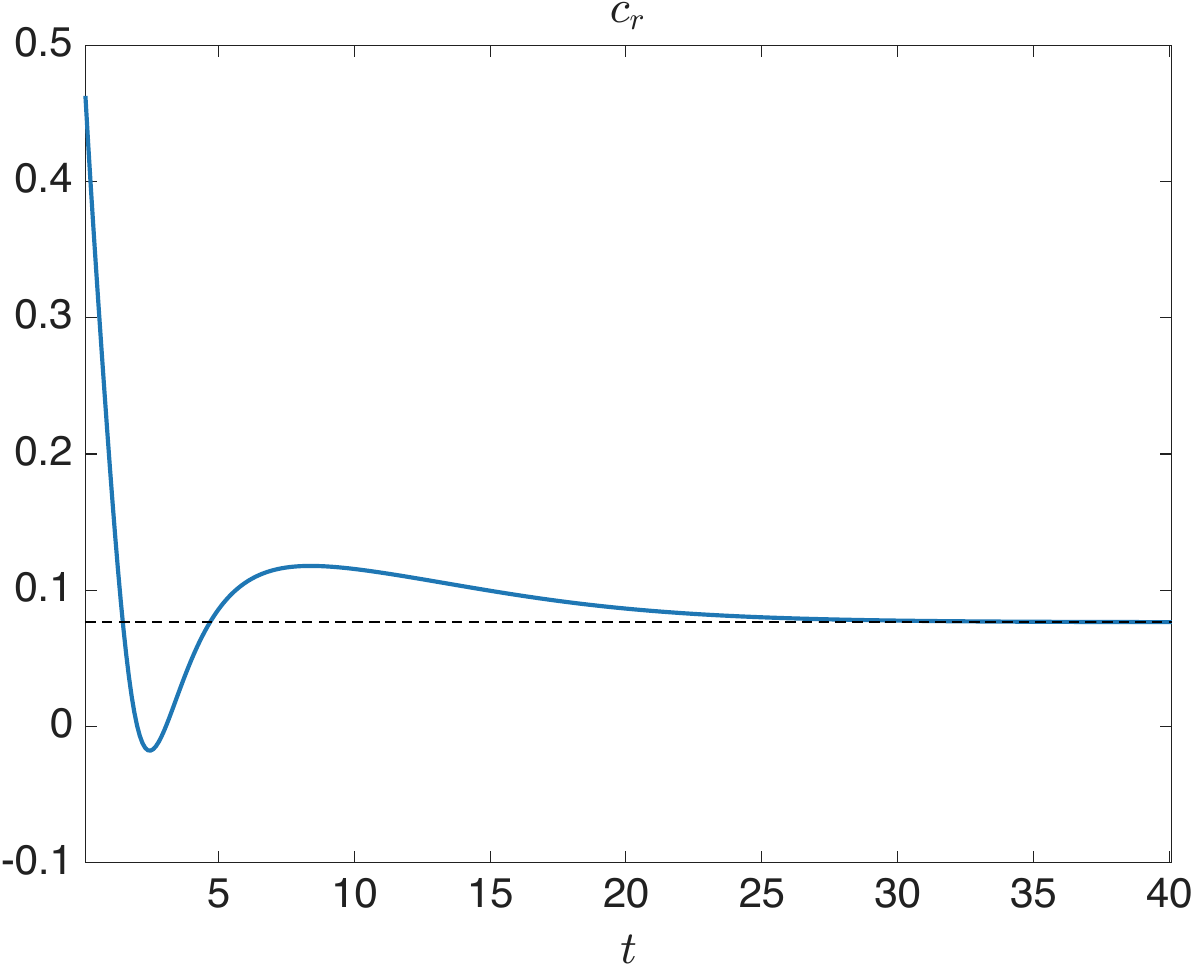}
        \includegraphics[width=0.38\textwidth]{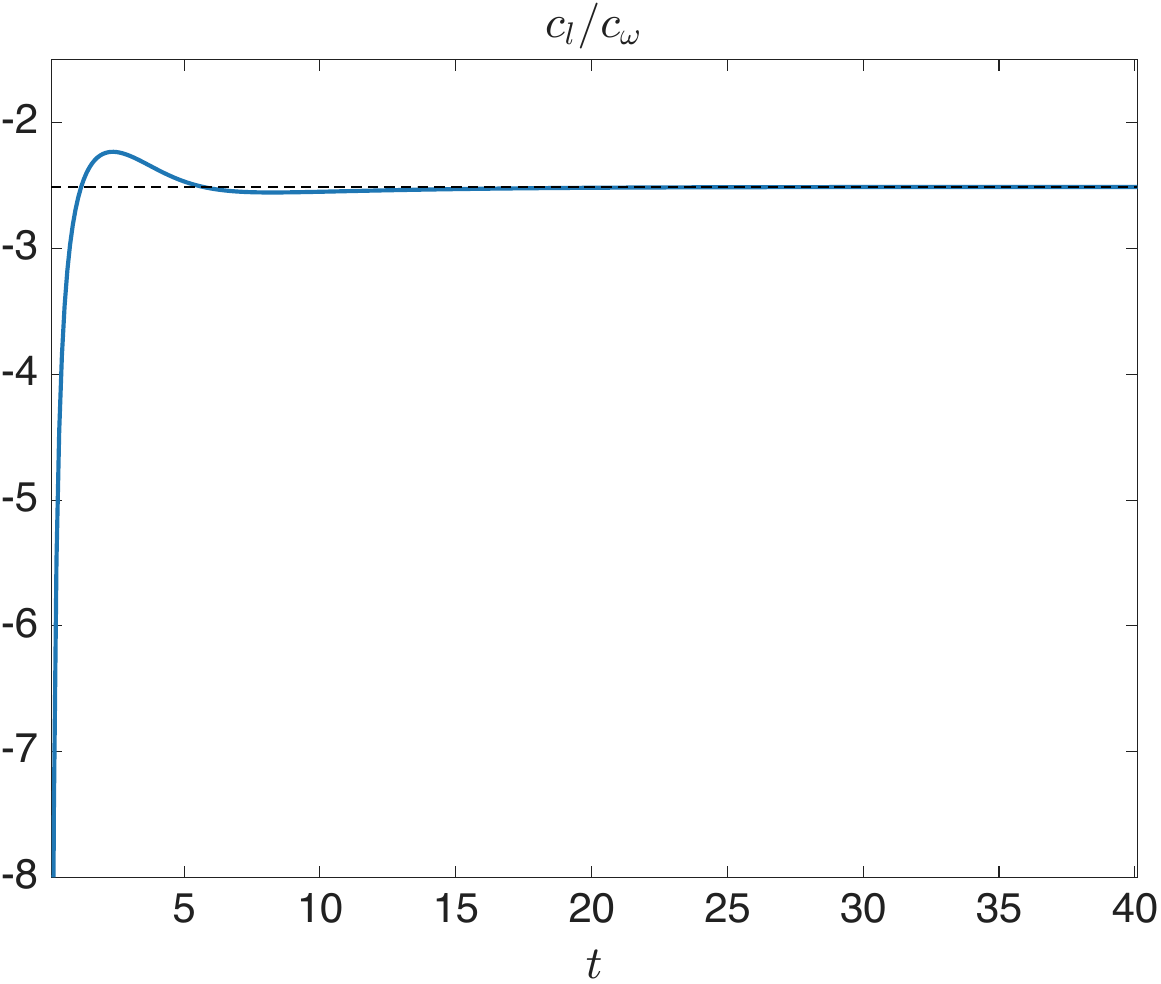}
    \caption[Evolution of $c_l$, $c_{\om}$, $c_r$ and $c_l/c_{\om}$ in Scenario 2.]{Evolution of $c_l$ (top, left), $c_{\om}$ (top, right), $c_r$ (bottom, left) and $c_l/c_{\om}$ (bottom, right) in Scenario 2. The black dashed lines represent the computed limiting values $(c_l,c_{\om},c_r,c_l/c_{\om})\approx(1.0636,-0.4235,0.0765,-2.5114)$, respectively.} 
      \label{fig:HL_scenario2_parameters_evolution}
\end{figure}

Since the solution of \eqref{eqt:dynamic_rescaling_of_HL_scenario2_numerical} converges to a steady state with $c_l>0$ and $c_{\omega}<0$, it follows from the change of variables \eqref{eqt:HL_scenario2_change_of_variables} and \eqref{eqt:HL_scenario2_change_of_variables2} that the original HL model \eqref{eqt:1Dhouluo} develops an asymptotically self-similar blowup in the physical variables. Moreover, the corresponding blowup point in the physical space is given by
\[
x=r(+\infty)
=-\int_0^{\infty} c_r(\tau)\exp\left(-\int_0^{\tau}c_l(s)\idiff{s}\right)\idiff{\tau}.
\]
Since our simulations show that both $c_r$ and $c_l$ converge to positive constants, the above integral is finite, and hence the blowup point $r(+\infty)$ is a finite negative number. Figure \ref{fig:HL_scenario2_physicalspace} illustrates the blowup of $\omega$ and $\theta_x$ in the physical coordinate. As shown in the figure, the blowup occurs at $x=r(+\infty)\approx -0.2675$, which is away from the origin.

\begin{figure}[!htbp]
\centering
        \includegraphics[width=0.4\textwidth]{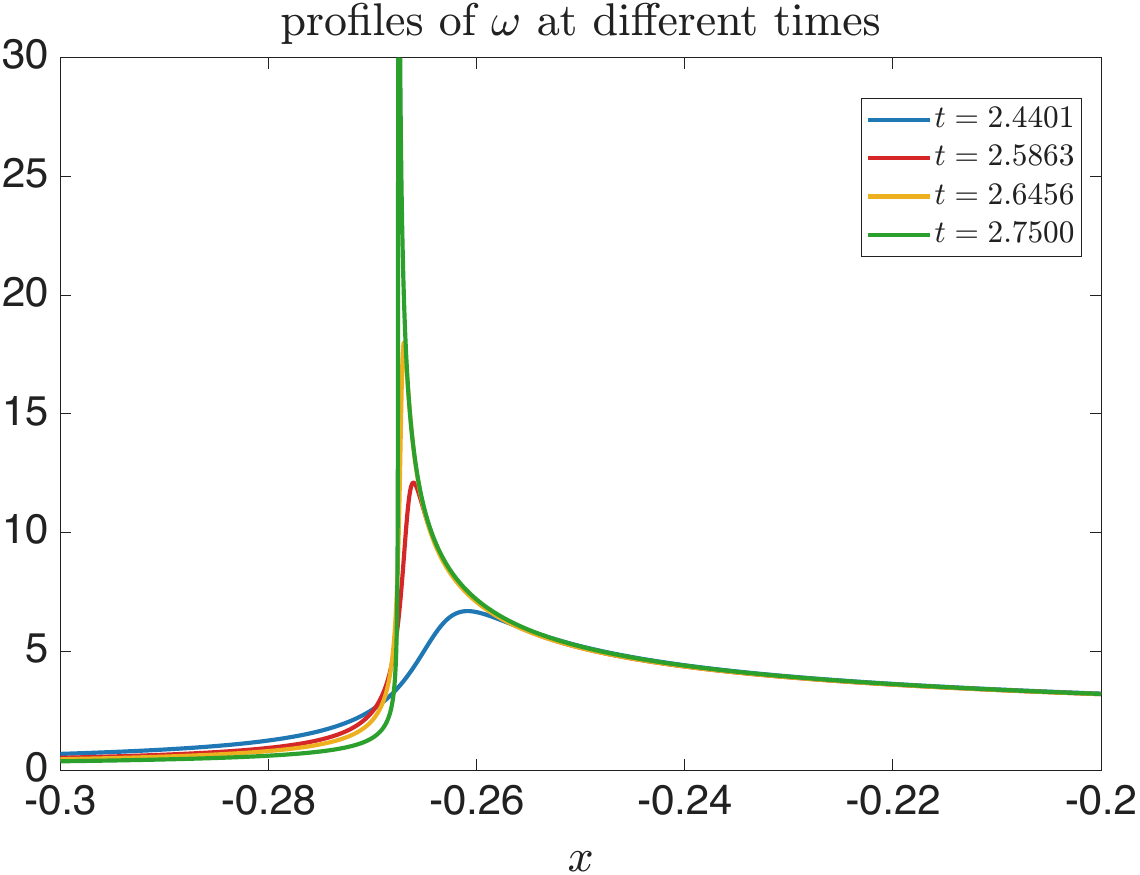}
        \includegraphics[width=0.4\textwidth]{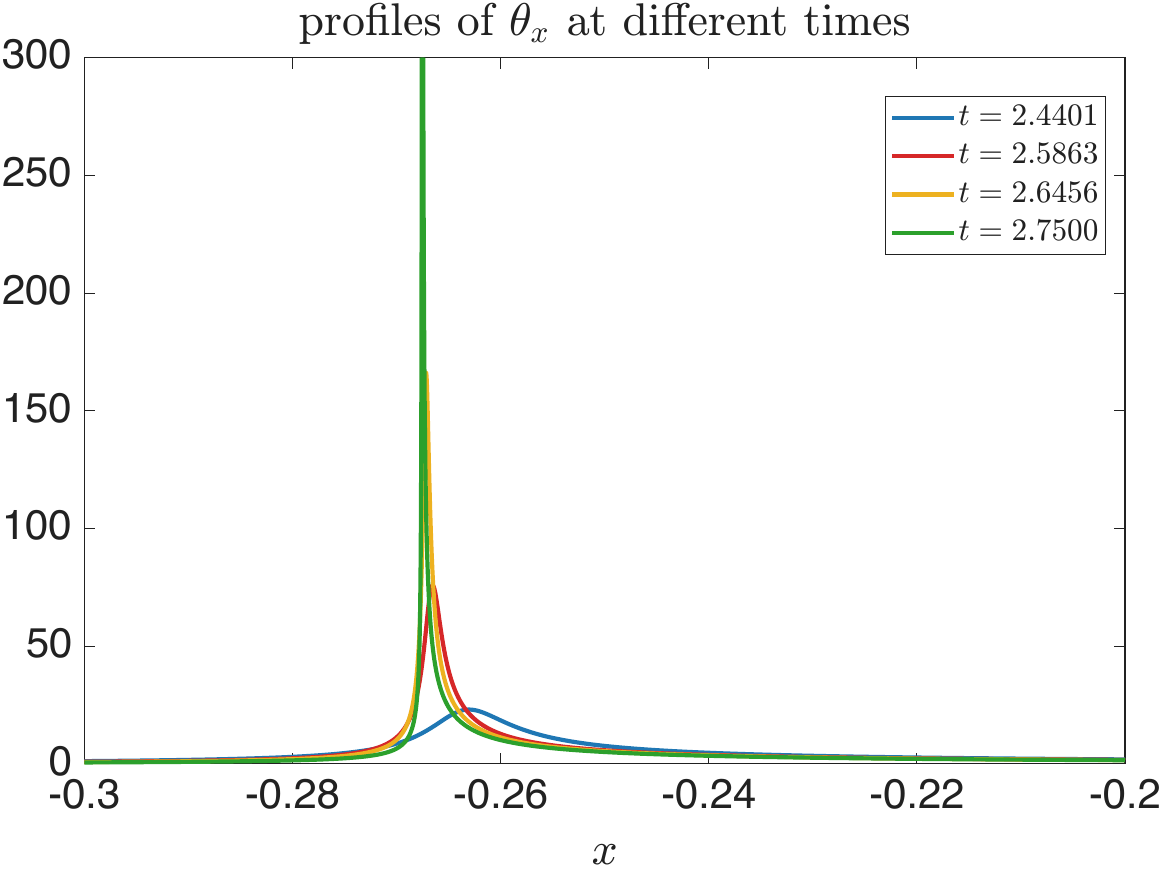}
    \caption[Evolution of spatial profiles of $\omega$ and $\theta_x$ in Scenario 2.]{Evolution of spatial profiles of $\omega$ (left) and $\theta_x$ (right) in Scenario 2. Here $(x,t)$ represents the physical space variables. As demonstrated, $\omega$ and $\theta_x$ develop finite-time blowups at the physical coordinate $x=r(+\infty)$, where $r$ is defined in \eqref{eqt:HL_scenario2_change_of_variables2}.}  
      \label{fig:HL_scenario2_physicalspace}
\end{figure}

Furthermore, we compare the inner self-similar profiles of $\Omega$ and $\Theta_x$  observed during the Stage 1 blowup of Scenario 1 presented in Section \ref{sec:HL_scenario1} with the limiting profiles obtained from the numerical simulations of \eqref{eqt:dynamic_rescaling_of_HL_scenario2_numerical}. As illustrated in Figure \ref{fig:HL_scenario2_profiles_comparison}, after appropriate rescaling, the inner profiles of Scenario 1 exhibit excellent agreement with those of Scenario 2. This strongly suggests that the modified dynamic rescaling equations \eqref{eqt:dynamic_rescaling_of_HL_scenario2_numerical} indeed capture the intrinsic mechanism governing the Stage 1 self-similar blowup described in Section \ref{sec:HL_scenario1}.

\begin{figure}[!htbp]
\centering
        \includegraphics[width=0.4\textwidth]{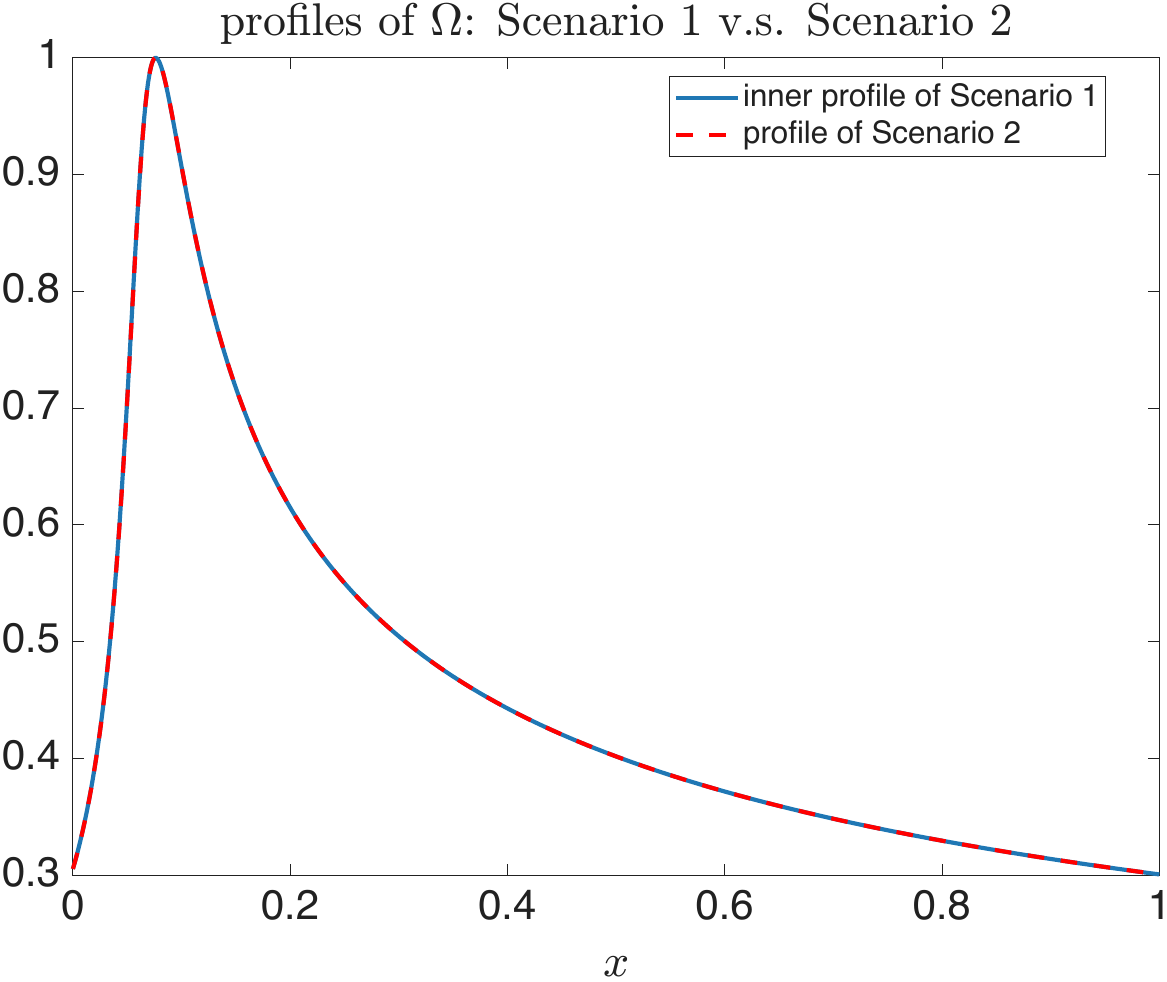}
        \includegraphics[width=0.4\textwidth]{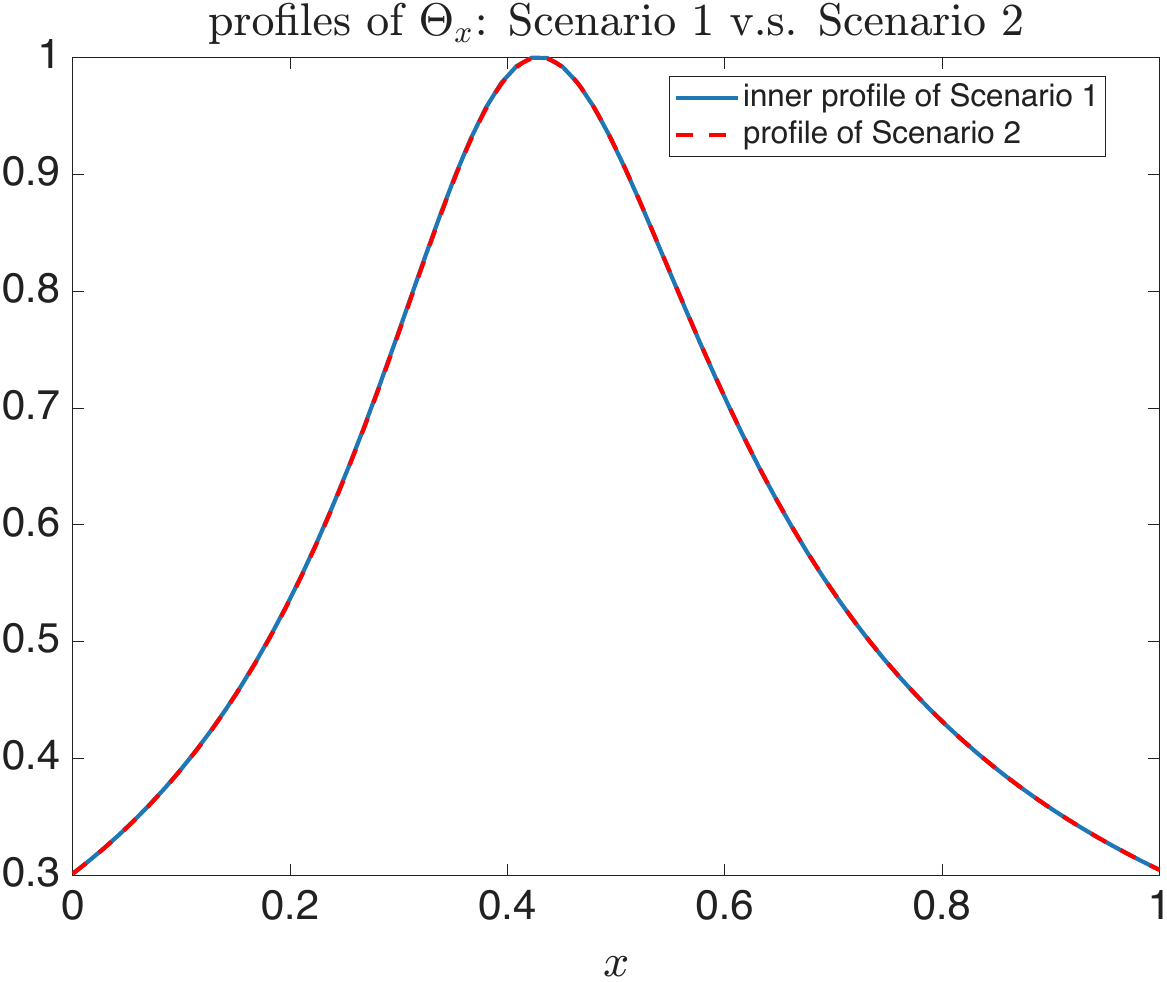}
    \caption[Comparison between the inner profile in the odd symmetry case in Scenario 1 at $t=4.3949$ and the limiting profile in Scenario 2 for the 1D HL model.]{Comparison between the inner profile in the odd symmetry case in Scenario 1 at $t=4.3949$ (blue solid line) and the limiting profile in Scenario 2 (red dashed line) for the 1D HL model. Left figure: $\Omega$; Right figure: $\Theta_x$. The profiles are compared after appropriate rescaling. }  
      \label{fig:HL_scenario2_profiles_comparison}
\end{figure}

\section{Singular Steady Solutions of The HL Model}\label{sec:HL_steady_state}
In this section, we present a complete version of Theorem \ref{thm:HL_weak_steady_state_simplified_version}, which gives singular steady solutions to the HL model \eqref{eqt:1Dhouluo} and its dynamic rescaling counterpart \eqref{eqt:dynamic_rescaling_of_HL}. To start, it is not hard to check the following lemma through direct calculation.
\begin{lemma}
    Let $\bar{\om}:=\mtx{1}_{\{x>0\}}/\sqrt{x}$. It holds that 
    \begin{equation}\label{eqt:definition_of_baru}
    \bar{u}(x):=\frac{1}{\pi}\int_{\R}\bar{\om}(y)\ln \left| \frac{x-y}{y}\right|\idiff y=2\sqrt{-x}\mtx{1}_{\{x<0\}},  \quad \bar{u}_x(x)=\mtx{H}(\bar{\om})(x)=-\frac{\mtx{1}_{\{x<0\}} }{\sqrt{-x}}.
    \end{equation}
\end{lemma}

Before presenting the main theorem, we need a pointwise estimate of the Hilbert transform near the singularity.
\begin{lemma}\label{lem:pointwise_estimate_of_H}
    Let $f\in C^1(\R\backslash[-\varepsilon,\varepsilon])$ for any $\varepsilon>0$. Suppose $|f(x)|\leq C|x|^{-1/2}$ for any $x\in \R$ and $|f_x(x)|\leq C|x|^{-3/2}$ for any $x\in [-\varepsilon_0,\varepsilon_0]$, where $\varepsilon_0>0$ is a given constant. Then for any $x\in [-\varepsilon_0/2,\varepsilon_0/2]$, it holds that 
    \[|\mtx{H}(f)(x)|\leq \frac{C_1}{\sqrt{|x|}},\]
    where $C_1$ is a constant that only depends on $C$.
\end{lemma}
The proof of Lemma \ref{lem:pointwise_estimate_of_H} will be deferred to the end of this section. We are now ready to present the main theorem of this section. Note that the steady state equations of the HL model are given by 
\begin{equation}\label{eqt:steady_state_equation}
    \begin{aligned}
    &u\om_x=\theta_x,\\
    &u\theta_x=0,\\
    &u_x=\mtx{H}(\om), \quad u(0)=0.
    \end{aligned}
\end{equation}
In what follows, we are going to show that     
\[(\bar{\omega},\bar{\theta})=\left(\frac{\mtx{1}_{\{x>0\}}}{\sqrt{x}},\frac{\pi \mtx{1}_{\{x>0\}}}{2}\right)\] 
is a singular steady state of the HL model in the weak sense (i.e. solves \eqref{eqt:steady_state_equation} weakly), by demonstrating that 
\[\bar{\om}\mtx{H}(\bar{\om})=-\frac{\pi}{2}\delta(x),\]
where $\delta$ denotes the Dirac delta distribution centered at $x=0$. Since $\bar{\om}$ and $\mtx{H}(\bar{\om})$ are both singular at $x=0$, the above equation shall be understood in the sense of distributions. However, there is no general theory that guarantees the well-definedness of the product of two distributions.
To overcome this difficulty, we approximate $\bar{\om}$ by a family of smooth functions and establish the weak convergence through careful estimates.

\begin{theorem}\label{thm:HL_steady_state}
Suppose $\{\om_s\}_{s>0}$ is a family of smooth functions that satisfies:
    \begin{itemize}
        \item There exist $p\in(2,+\infty)$ and $q\in(1,2)$ such that for any $\varepsilon>0$, $\om_s\to \bar{\om}$ in $L^p(\R\backslash [-\varepsilon,\varepsilon])$ and $\om_s\to \bar{\om}$ in $L^q([-1,1])$ as $s\to +\infty$.
        \item  There exist an absolute constant $C$ and $\varepsilon_0>0$ such that $|\om_s(x)|\leq C|x|^{-1/2}$ for any $x\in \R$ and $|\partial_x\om_s(x)|\leq C|x|^{-3/2}$ for any $x\in [-\varepsilon_0,\varepsilon_0]$.
    \end{itemize}
     Then for any $\phi\in C_c^{\infty}(\R)$, it holds that 
    \begin{equation}\label{eqt:HL_steady_state_equation}
        \lim_{s\to +\infty} \int_{\R} \om_s \mtx{H} (\om_s) \phi \idiff x =-\frac{\pi}{2}\phi(0)=-\frac{\pi}{2} \la \delta,\phi \ra. 
    \end{equation}
    As a result, 
    \[(\bar{\omega},\bar{\theta})=\left(\frac{\mtx{1}_{\{x>0\}}}{\sqrt{x}},\frac{\pi \mtx{1}_{\{x>0\}}}{2}\right)\]
    solves \eqref{eqt:steady_state_equation} weakly in the sense that 
\begin{equation}\label{eqt:HL_steady_state_equation1}
    \lim_{s\to +\infty} \la u_s\partial_x\om_s,\phi \ra =\frac{\pi}{2}\phi(0)=\la \partial_x\bar{\theta},\phi \ra,
    \end{equation}
 \begin{equation}\label{eqt:HL_steady_state_equation2}
      \lim_{s\to +\infty} \la u_s\partial_x\bar{\theta},\phi \ra =0.
    \end{equation}
    Here $\la f,g \ra:= \int_{\R} f(x)g(x)\idiff x$, and $u_s$ is given by
    \[\partial_xu_s=\mtx{H}(\om_s),\quad  u_s(x)=\frac{1}{\pi}\int_{\R} \om_s(y)\ln\left|\frac{x-y}{y}  \right|\idiff y.\]
\end{theorem}

\begin{remark}
In the statement and the proof of Theorem \ref{thm:HL_steady_state}, we use the subscript $s$ in $\om_s$ to indicate its dependence on the parameter $s$. To avoid ambiguity, we use $\partial_x\om$ (instead of $\om_x$) to denote the partial derivative in space. The same notations also apply to the statement and the proof of Theorem \ref{thm:steady_state_of_HLDS}.
\end{remark}

\begin{remark}
    In fact, one can directly check that $(\bar{\omega},\bar{\theta})$ solves \eqref{eqt:steady_state_equation} in the strong sense at any $x\ne 0$. The weak formulation is employed here primarily to characterize the singular behavior of the solution at $x=0$.
\end{remark}

\begin{remark}
    Let us demonstrate the existence of an approximating family $\{\om_s\}_{s>0}$ that satisfies the assumptions in Theorem \ref{thm:HL_steady_state}. For example, one can choose a family of smooth functions $\{\chi_s\}_{s>0}$ such that $\chi_s(x)=1$ for $x\geq 2/s$, $\chi_s(x)=0$ for $x\leq 1/s$, and $|\partial_x\chi_s(x)|\lesssim s$ for any $x\in \R$, and then define 
    \[\om_s(x):=\frac{\chi_s(x)\mtx{1}_{\{x>0\}}}{\sqrt{x}}.\]
    It is easy to verify that $\{\om_s\}_{s>0}$ satisfies the assumptions in Theorem \ref{thm:HL_steady_state}.
\end{remark}

\begin{proof}[Proof of Theorem \ref{thm:HL_steady_state}]
 We use the Parseval-type identity for the Hilbert transform and the Tricomi's identity to compute that 
\[\int_{\R} \om_s \mtx{H} (\om_s) \phi \idiff x = \int_{\R} \mtx{H}(\om_s \mtx{H} (\om_s)) \mtx{H}( \phi) \idiff x = \int_{\R}  \frac{\mtx{H} (\om_s)^2-\om_s^2}{2} \mtx{H}( \phi) \idiff x.\]
Note that \[\phi(0)=\frac{1}{\pi}\mathrm{P.V.}\int_{\R}\frac{\mtx{H}(\phi)}{x}\idiff x. \]
Our goal is to show   
\[\int_{\R} \frac{\mtx{H} (\om_s)^2-\om_s^2}{2} \mtx{H}(\phi) \idiff x\to -\frac{1}{2} \mathrm{P.V.}\int_{\R}\frac{\mtx{H}(\phi)}{x}\idiff x.\]
For any $\varepsilon>0$, we have 
\begin{align*}
    &\,\ \int_{\R}  \frac{\mtx{H} (\om_s)^2-\om_s^2}{2} \mtx{H}( \phi) \idiff x-\left(-\frac{1}{2} \mathrm{P.V.}\int_{\R}\frac{\mtx{H}(\phi)}{x}\idiff x\right)\\
    &=\frac{-1}{2}\int_{|x|>\varepsilon}\left(\om_s^2-\frac{\mtx{1}_{\{x>0\}}}{x}\right)\mtx{H}(\phi)\idiff x+\frac{1}{2}\int_{|x|>\varepsilon}\left((\mtx{H}\om_s)^2-\frac{-\mtx{1}_{\{x<0\}}}{x}\right)\mtx{H}(\phi)\idiff x\\ & \,\ +\int_{|x|<\varepsilon}  \frac{\mtx{H} (\om_s)^2-\om_s^2}{2} \mtx{H}( \phi) \idiff x+\frac{1}{2}\int_{|x|<\varepsilon}\frac{\mtx{H}(\phi)(x)-\mtx{H}(\phi)(0)}{x}\idiff x\\
    &=:I_1+I_2+I_3+I_4.
\end{align*}
We first estimate $I_1$ as follows:
\begin{equation}\label{eqt:weak_convergence:positive_part}
\begin{aligned}
    |I_1|&=\frac{1}{2}\left|\int_{|x|>\varepsilon}\left(\om_s^2-\bar{\om}^2\right)\mtx{H}(\phi)\idiff x\right|\\ & \leq \frac{1}{2}\|\om_s-\bar{\om}\|_{L^p(\R\backslash[-\varepsilon,\varepsilon])}\|\om_{s}+\bar{\om}\|_{L^p(\R\backslash[-\varepsilon,\varepsilon])} \left\|\mtx{H}(\phi)\right\|_{L^{\frac{p}{p-2}}} \\
    & \lesssim_{\varepsilon,p} \|\om_s-\bar{\om}\|_{L^p(\R\backslash[-\varepsilon,\varepsilon])}.
\end{aligned}
\end{equation}
To control $I_2$, for any $s>0$, we decompose $\om_s$ as $\om_s=\om_s \mtx{1}_{\{|x|<\varepsilon/2\} }+\om_s\mtx{1}_{\{|x|>\varepsilon/2\}}=:\om_{s,1}+\om_{s,2}$ and decompose $\bar{\om}$ similarly. Then for any $x$ such that $|x|>\varepsilon$, 
\begin{align*} 
    |\mtx{H}(\om_{s,1}-\bar{\om}_{1})(x)|&=\frac{1}{\pi}\left|\int_{-\varepsilon/2}^{\varepsilon/2} \frac{\om_{s,1}(y)-\bar{\om}_{1}(y)}{x-y} \idiff y\right|\\ &\leq \| \om_{s,1}-\bar{\om}_{1}\|_{L^q([-\varepsilon/2,\varepsilon/2])} \left\|\frac{1}{x-y} \right\|_{L_y^{\frac{q}{q-1}}([-\varepsilon/2,\varepsilon/2])}\\ &\lesssim_{\varepsilon,q} \frac{\| \om_{s,1}-\bar{\om}_{1}\|_{L^q([-\varepsilon/2,\varepsilon/2])} }{|x|}.
\end{align*}
We thus deduce that
\begin{align*} \|\mtx{H}(\om_s)-\mtx{H}(\bar{\om})\|_{L^p(\R\backslash [-\varepsilon,\varepsilon])}&\leq \|\mtx{H} (\om_{s,1}-\bar{\om}_{1}) \|_{L^p(\R\backslash [-\varepsilon,\varepsilon])}+\|\mtx{H} (\om_{s,2}-\bar{\om}_{2}) \|_{L^p(\R\backslash [-\varepsilon,\varepsilon])}\\
    &\lesssim_{\varepsilon,q} \| \om_{s,1}-\bar{\om}_{1}\|_{L^q([-\varepsilon/2,\varepsilon/2])}\left\|\frac{1}{x}\right\|_{L_x^p(\R\backslash [-\varepsilon,\varepsilon])}+\|\mtx{H} (\om_{s,2}-\bar{\om}_{2}) \|_{L^p}\\
    & \lesssim_{\varepsilon,p,q} \| \om_{s,1}-\bar{\om}_{1}\|_{L^q([-\varepsilon/2,\varepsilon/2])}+\| \om_{s,2}-\bar{\om}_{2} \|_{L^p}\\
    &= \| \om_{s}-\bar{\om}\|_{L^q([-\varepsilon/2,\varepsilon/2])}+\| \om_{s}-\bar{\om} \|_{L^p(\R\backslash [-\varepsilon/2,\varepsilon/2])}.\\
\end{align*}
We have used the fact that Hilbert transform is strong $(p,p)$ in the third inequality. 
As a result, similar to \eqref{eqt:weak_convergence:positive_part}, we have
\[|I_2|=\frac{1}{2}\left|\int_{|x|>\varepsilon}\left((\mtx{H}\om_s)^2-(\mtx{H}\bar{\om})^2\right)\mtx{H}(\phi)\idiff x\right| \lesssim_{\varepsilon,p,q} \| \om_{s}-\bar{\om}\|_{L^q([-\varepsilon/2,\varepsilon/2])}+\| \om_{s}-\bar{\om} \|_{L^p(\R\backslash [-\varepsilon/2,\varepsilon/2])}.
\]
Since $\mtx{H}(\phi)$ is smooth, we write $\mtx{H}(\phi)(x)=\mtx{H}(\phi)(0)+r(x)$, where $|r(x)|\lesssim |x|$. We immediately have $|I_4|\lesssim \varepsilon$. To bound $I_3$, we further decompose it as follows:
    \[I_3 =  \int_{|x|<\varepsilon}  \frac{\mtx{H} (\om_s)^2-\om_s^2}{2} \mtx{H}( \phi)(0) \idiff x+  \int_{|x|<\varepsilon}  \frac{\mtx{H} (\om_s)^2-\om_s^2}{2} r(x) \idiff  x\\
     =:I_{3,1}+I_{3,2}.\]
For any $\varepsilon<\varepsilon_0/2$, we can bound $I_{3,2}$ directly using the uniform bound $|\om_s(x)|\leq C|x|^{-1/2}$, $|\partial_x\om_s(x)|\leq C|x|^{-3/2}$ and Lemma \ref{lem:pointwise_estimate_of_H}:
\[|I_{3,2}|\lesssim \int_{|x|<\varepsilon}  (\om_s^2+\mtx{H}(\om_s)^2)|x|\idiff x \lesssim \int_{|x|<\varepsilon} 1 \idiff x\lesssim \varepsilon.\]
To estimate $I_{3,1}$, we use the Parseval-type identity for the Hilbert transform and the Tricomi's identity again to deduce
\begin{align*}
I_{3,1}&= \mtx{H}( \phi)(0) \int_{\R} \frac{\mtx{H} (\om_s)^2-\om_s^2}{2}\mtx{1}_{\{|x|<\varepsilon\}}\idiff{x}\\
 &= \mtx{H}( \phi)(0) \int_{\R} \mtx{H}\left(\frac{\mtx{H} (\om_s)^2-\om_s^2}{2}\right) \mtx{H}\left(\mtx{1}_{\{|x|<\varepsilon\}}\right)\idiff{x}\\
 &= -\mtx{H}( \phi)(0) \int_{\R} \om_s\mtx{H}(\om_s)h_{\varepsilon}(x)\idiff{x},
\end{align*}
where $h_\varepsilon(x):=\ln(|(x+\varepsilon)/(x-\varepsilon)|)/\pi$. It follows that
\begin{align*}
|I_{3,1}|&\lesssim  \left|\int_{\R} \om_s\mtx{H}(\om_s)h_{\varepsilon}(x)\idiff{x}\right|\\
& \leq \int_{\R} \left(\left|(\om_s-\bar{\om})\mtx{H}(\bar{\om})h_{\varepsilon}(x)\right|+ \left|\bar{\om}\mtx{H}(\om_s-\bar{\om})h_{\varepsilon}(x)\right|+\left|(\om_s-\bar{\om})\mtx{H}(\om_s-\bar{\om})h_{\varepsilon}(x)\right|\right) \idiff{x}\\
&\leq  \int_{\R} \left|(\om_s-\bar{\om})\frac{h_{\varepsilon}(x)}{\sqrt{|x|}}\right|\idiff{x}+\int_{\R}\left| \mtx{H}(\om_s-\bar{\om})\frac{h_{\varepsilon}(x)}{\sqrt{|x|}}\right| (1+|\om_s-\bar{\om}|\sqrt{|x|} )\idiff{x}\\
&\lesssim_{p,q} \|\om_s-\bar{\om}\|_{L^p(\R\backslash[-1,1])}\left\|\frac{h_{\varepsilon}(x)}{\sqrt{|x|}}\right\|_{L^{\frac{p}{p-1}}}+ \|\om_s-\bar{\om}\|_{L^q([-1,1])}\left\|\frac{h_{\varepsilon}(x)}{\sqrt{|x|}}\right\|_{L^{\frac{q}{q-1}}}.
\end{align*}
We have used the identity $\bar{\om}\mtx{H}(\bar{\om})h_{\varepsilon}=0$ in the second inequality. The fourth inequality follows from the uniform bound $|\om_s|\sqrt{|x|}\leq C$ together with the fact that $h_{\varepsilon}/{\sqrt{|x|}}\in L^r(\R)$ for any $r>1$.
In summary, we have proved that for any $\varepsilon\in(0,\varepsilon_0/2]$,
\[\left|\int_{\R}\om_s\mtx{H}(\om_s)\phi \idiff x-\left(-\frac{\pi}{2}\phi(0)\right)\right|
   \leq C_{\varepsilon,p,q}\left(\| \om_{s}-\bar{\om}\|_{L^q([-1,1])}+\| \om_{s}-\bar{\om} \|_{L^p(\R\backslash [-\varepsilon/2,\varepsilon/2])}\right)+C_1\varepsilon,\]
where $C_{\varepsilon,p,q}$ is a constant that depends on $\varepsilon,p,q$, and $C_1$ is an absolute constant. As a result, for any $\varepsilon\in(0,\varepsilon_0/2]$, there exists $S>0$ such that for any $s>S$, 
  \[ \left|\int_{\R} \om_s\mtx{H}(\om_s)\phi \idiff x-\left(-\frac{\pi}{2}\phi(0)\right)\right|\leq (C_1+1)\varepsilon,
\]
which implies \eqref{eqt:HL_steady_state_equation}. Note that \eqref{eqt:HL_steady_state_equation2} holds trivially since $u_s(0)=0$. To prove \eqref{eqt:HL_steady_state_equation1}, we use integration by parts to compute that
    \[\int_{\R} u_s\partial_x \om_s\phi \idiff{x}= -\int_{\R} \partial_x u_s \om_s \phi\idiff{x}- \int_{\R}u_s \om_s\partial_x\phi \idiff{x} =:I_1+I_2.\]
As shown before, we have $I_1\to \pi \phi(0)/2$ as $s\to +\infty$. Suppose that the support of $\phi$ is contained in $[-L,L]$ for some $L>0$. Recall that $\bar{u}$ is given by \eqref{eqt:definition_of_baru}. Note that $\bar{u}(x)\bar{\omega}(x)= 0$ holds strongly for all $x\in\R$. For any $x\in [-L,L]$,
\[|u_s(x)-\bar{u}(x)|\leq \left|\int_0^x |\mtx{H}\om_s-\mtx{H}\bar{\om}| \idiff{y}\right|\lesssim_{L,p,q} \|\om_s-\bar{\om}\|_{L^p(\R\backslash[-1,1])}+\|\om_s-\bar{\om}\|_{L^q([-1,1])},\]
and therefore $\|u_s\|_{L^{\infty}([-L,L])}\lesssim_{L,p,q}1$. We therefore have
\begin{align*} 
    |I_2| & = \left|\int_{\R}u_s (\om_s-\bar{\om})\partial_x\phi \idiff{x}+ \int_{\R}(u_s-\bar{u}) \bar{\om}\partial_x\phi \idiff{x}\right|\\
    & \leq \left(\|u_s\|_{L^{\frac{q}{q-1}}([-L,L])}\|\om_s-\bar{\om}\|_{L^q([-L,L])}+\|u_s-\bar{u}\|_{L^{\frac{q}{q-1}}([-L,L])}\|\bar{\om}\|_{L^q([-L,L])}\right) \|\partial_x\phi\|_{L^{\infty}}\\
    & \lesssim_L \|u_s\|_{L^{\infty}([-L,L])}\|\om_s-\bar{\om}\|_{L^q([-L,L])}+\|u_s-\bar{u}\|_{L^{\infty}([-L,L])}\|\bar{\om}\|_{L^q([-L,L])}\\
    & \lesssim_{L,p,q} \|\om_s-\bar{\om}\|_{L^p(\R\backslash[-1,1])}+\|\om_s-\bar{\om}\|_{L^q([-1,1])}
\xrightarrow{s \to +\infty} 0.
\end{align*}
The proof is thus completed.
\end{proof}

Next, we turn to study singular steady states of the dynamic rescaling equations \eqref{eqt:dynamic_rescaling_of_HL} in the weak sense. Note that the corresponding steady state equations are given by
\begin{equation}\label{eqt:dynamic_rescaling_steady_state_equation}
    \begin{aligned}
    &(U+c_lX)\Omega_X=c_{\om}\Omega+\Theta_X,\\
    &(U+c_lX)\Theta_X=(c_l+2c_{\om})\Theta,\\
    &U_X=\mtx{H}(\Omega), \quad U(0)=0.
    \end{aligned}
\end{equation}
In fact, one can establish results analogous to Theorem \ref{thm:HL_steady_state} for the dynamic rescaling equations \eqref{eqt:dynamic_rescaling_of_HL}, and the presence of $c_l$ and $c_{\om}$ in \eqref{eqt:dynamic_rescaling_steady_state_equation} will not introduce significant new difficulties as it only moves the singularity of the solutions from the origin to a location away from the origin. Specifically, we have the following theorem. 

\begin{theorem}\label{thm:steady_state_of_HLDS}
The tuple
\[(\bar{\Omega},\bar{\Theta},\bar{U},\bar{c}_l,\bar{c}_{\om})=\left(\frac{\mtx{1}_{\{X>1\}}}{\sqrt{X-1}},\frac{\pi \mtx{1}_{\{X>1\}}}{2},2\sqrt{1-X}\mtx{1}_{\{X<1\}}-2,2,-1\right)\] solves \eqref{eqt:dynamic_rescaling_steady_state_equation} in the weak sense.
More specifically, suppose $\{\Omega_s\}_{s>0}$ is a family of smooth functions that satisfies:
    \begin{itemize}
        \item There exist $p\in(2,+\infty)$ and $q\in(1,2)$ such that for any $\varepsilon>0$, $\Omega_s\to \bar{\Omega}$ in $L^p(\R\backslash [1-\varepsilon,1+\varepsilon])$ and $\Omega_s\to \bar{\Omega}$ in $L^q([0,2])$ as $s\to +\infty$.
        \item  There exist an absolute constant $C$ and $\varepsilon_0>0$ such that $|\Omega_s(X)|\leq C|X-1|^{-1/2}$ for any $X\in \R$ and $|\partial_X\Omega_s(X)|\leq C|X-1|^{-3/2}$ for any $X\in [1-\varepsilon_0,1+\varepsilon_0]$.
    \end{itemize}
     Then for any $\phi\in C_c^{\infty}(\R)$,
\begin{equation}\label{eqt:HL_dssteady_state_equation1}
    \lim_{s\to +\infty} \la(U_s+\bar{c}_lX)\partial_X\Omega_s-\bar{c}_{\om}\Omega_s,\phi\ra =\frac{\pi}{2}\phi(1)=\la \partial_X\bar{\Theta},\phi \ra,
\end{equation}
\begin{equation}\label{eqt:HL_dssteady_state_equation2}
     \lim_{s\to +\infty} \la(U_s+\bar{c}_lX)\partial_X\bar{\Theta},\phi \ra =0= \la(\bar{c}_l+2\bar{c}_{\om})\bar{\Theta},\phi\ra,
\end{equation}
    where 
    \[\partial_XU_s=\mtx{H}(\Omega_s),\quad  U_s(X)=\frac{1}{\pi}\int_{\R} \Omega_s(Y)\ln\left|\frac{X-Y}{Y}  \right|\idiff Y.\]
\end{theorem}
\begin{proof}
It is not hard to show 
\[\lim_{s\to +\infty}U_s(1)=\bar{U}(1)=\frac{1}{\pi}\int_{\R} \bar{\Omega}(Y)\ln\left|\frac{1-Y}{Y}  \right|\idiff Y=-2 .\]
 Notice that $\partial_X\bar{\Theta}(X)= \pi \delta(X-1)/2$, therefore 
\[ \la(U_s+\bar{c}_lX)\partial_X\bar{\Theta}(X),\phi \ra=\frac{\pi}{2}(U_s(1)+2)\phi(1)\to 0 \,\  \text{as }s\to +\infty,\]
which implies \eqref{eqt:HL_dssteady_state_equation2}. On the other hand, direct computation shows that
\begin{align*}
    & \int_{\R} ((U_s+\bar{c}_lX)\partial_X \Omega_s-\bar{c}_{\om}\Omega_s)\phi \idiff X \\
    &= \int_{\R}\left((U_s+2X)\partial_X \Omega_s+\Omega_s\right)\phi \idiff X \\
    &= \int_{\R}(U_s+2)\partial_X \Omega_s\phi \idiff X +  \int_{\R} \left((2X-2)\partial_X\Omega_s+\Omega_s\right)\phi\idiff X\\
    &= -\int_{\R}\partial_X U_s\Omega_s\phi \idiff X -\int_{\R}(U_s+2)\Omega_s\partial_X\phi \idiff X +\int_{\R} \left((2X-2)\partial_X\Omega_s+\Omega_s\right)\phi\idiff X\\
    & =: I_1+I_2+I_3,
\end{align*}
where we have used integration by parts in the third equality. Using Theorem \ref{thm:HL_steady_state}, we know that $I_1\to \pi\phi(1)/2$ as $s\to +\infty$. Note that $(\bar{U}(X)+2)\bar{\Omega}(X)=0$ holds strongly for all $X\in \R$. Therefore, it is not hard to see that $I_2\to 0$ as $s\to +\infty$ using similar arguments as in the proof of Theorem \ref{thm:HL_steady_state}. We further control $I_3$ as follows:
\begin{align*}
|I_3|& = \left|\int_{\R} \left((2X-2)\partial_X\left(\Omega_s-\bar{\Omega}\right)+\Omega_s-\bar{\Omega}\right)\phi\idiff X\right|\\
     & =\left| \int_{\R} \left(\Omega_s-\bar{\Omega}\right)\left(\partial_X((2-2X)\phi)+\phi\right)\idiff X\right|\\
     &\lesssim \left\|\Omega_s-\bar{\Omega} \right\|_{L^p(\R\backslash[0,2])}+\left\|\Omega_s-\bar{\Omega} \right\|_{L^q([0,2])}\xrightarrow{s \to +\infty} 0,
\end{align*}
which implies \eqref{eqt:HL_dssteady_state_equation1}. This concludes the proof.
\end{proof}

\begin{remark}\label{rem:HL_steady_state_numerical_relation}
In Theorem \ref{thm:steady_state_of_HLDS}, $\{\Omega_s\}_{s>0}$ is a family of smooth approximations to the singular steady state, where the subscript $s$ in $\Omega_s$ does not necessarily represent time. We introduce this family to explain the meaning of ``a singular steady state in the weak sense''. However, we remark that the numerically computed solutions $\Omega(t)$ in Section \ref{sec:HL_scenario1}, viewed as a family parameterized by time $t$, cannot be directly identified as an approximating family: Due to the presence of regularization by numerical viscosity (see Appendix \ref{sec:HL_numerical method}), the numerical solutions $\Omega(t)$ cannot actually converge to some singular steady state and hence do not satisfy the assumptions of Theorem \ref{thm:steady_state_of_HLDS}. 
Nevertheless, our numerical results in Section \ref{sec:HL_scenario1} suggest the convergence of solutions of \eqref{eqt:numerical_dynamic_rescaling_of_HLscenario1} to a singular steady state in the vanishing viscosity sense. We also do not exclude the possibility that the solution $\Omega(t)$ in the actual vanishing viscosity limit satisfies the assumptions of Theorem \ref{thm:steady_state_of_HLDS} as a family of functions parameterized by time $t$.
\end{remark}

Finally, we conclude this section by proving Lemma \ref{lem:pointwise_estimate_of_H}.
\begin{proof}[Proof of Lemma \ref{lem:pointwise_estimate_of_H}]
    Without loss of generality we assume $x>0$. By the definition of the Hilbert transform, we have
\[\mtx{H}(f)(x)=\frac{1}{\pi}\int_{-\infty}^0 \frac{f(y)}{x-y}\idiff{y}+\frac{1}{\pi}\int_{0}^{2x} \frac{f(y)-f(x)}{x-y}\idiff{y}+\frac{1}{\pi}\int_{2x}^{+\infty} \frac{f(y)}{x-y}\idiff{y}=:I_1+I_2+I_3.\]
We first bound $I_1$ as follows:
\[|I_1|\leq C\int_{0}^{+\infty} \frac{\idiff y}{\sqrt{y}(x+y)}=C\int_{0}^{+\infty} \frac{\idiff {xt}}{\sqrt{tx}(x+tx)}= \frac{C}{\sqrt{x}}\int_{0}^{+\infty} \frac{\idiff {t}}{\sqrt{t}(1+t)}.\]
We can  control $I_3$ similarly:
\[|I_3|\leq C\int_{2x}^{+\infty} \frac{\idiff y}{\sqrt{y}(y-x)}= \frac{C}{\sqrt{x}}\int_{2}^{+\infty} \frac{\idiff {t}}{\sqrt{t}(t-1)}.\]
Note that for any $y\in (0,2x)$, \[\left|\frac{f(y)-f(x)}{y-x}\right|\leq \frac{1}{|x-y|}\left|\int_{x}^{y}|f'(t)|\idiff t\right| \leq \frac{C}{|x-y|}\left|\int_{x}^{y}|t|^{-3/2}\idiff t \right| =\frac{2C}{\sqrt{y}\sqrt{x}(\sqrt{y}+\sqrt{x})}\leq \frac{2C}{x\sqrt{y}}.\]
We thus have 
\[|I_2|\leq \int_{0}^{2x} \frac{2C}{x\sqrt{y}}\idiff{y}=\frac{4\sqrt{2}C}{\sqrt{x}}.\]
Combining the above estimates, we obtain the desired result.
\end{proof}

\section{Novel Self-similar Finite-time Blowup of the 2D Boussinesq Equations}\label{sec:dynamic_rescaling_of_BS}
In this section, we report novel self-similar finite-time blowup scenarios for the 2D Boussinesq equations in the half-space $\mtx{x}=(x_1,x_2)\in\R \times \R_+$,
\begin{equation}\label{eqt:Boussinesq_v2}
    \begin{aligned}
    &\om_{t} + \mtx{u}\cdot \nabla \om=\theta_{x_1},\\
    &\theta_t+\mtx{u}\cdot \nabla \theta=0,\\
    &\mtx{u} = \nabla^{\perp}(-\Delta)^{-1}\omega,
    \end{aligned}
\end{equation}
where $(-\Delta)^{-1}$ denotes the inverse of the negative Laplacian operator on the half-space subject to zero Dirichlet boundary conditions, and $\nabla^{\perp}=(-\partial_{x_2}, \partial_{x_1})^{\top}$. The novel blowup scenarios of \eqref{eqt:Boussinesq_v2} are very similar to those of the 1D HL model reported in previous sections. Our numerical simulation shows that solutions to the 2D Boussinesq equations that initially satisfy certain derivative degeneracy condition can also develop asymptotically self-similar finite-time blowups with singular self-similar profiles. Moreover, this blowup phenomenon also exhibits a two-stage feature: the solution first undergoes a local $L^{\infty}$ blowup at some time $\tilde{T}$, then continues in the weak sense beyond $\tilde{T}$ and develops a local $L^p$ blowup at a later time $T>\tilde{T}$ for some $p>0$. We also conduct a further numerical investigation into the Stage 1 blowup discussed earlier and show that it corresponds to an asymptotically self-similar blowup of \eqref{eqt:Boussinesq_v2} with regular profiles.

\subsection{Scenario 1: novel asymptotically self-similar blowups with singular profiles.}\label{subsec:BS_scenario1}
Before we present the numerical results regarding the potential two-stage self-similar blowup of \eqref{eqt:Boussinesq_v2}, let us first introduce the dynamic rescaling formulation of \eqref{eqt:Boussinesq_v2}. We consider the following change of variables
\begin{equation}\label{eqt:dynamic_rescaling_of_BS_change_of_variables1}
\begin{aligned}
    &\omega(\mtx{x},t)=C_{\omega}(\tau)^{-1} \Omega\left(C_l(\tau)\mtx{x},\tau(t)\right),\\
    &\theta(\mtx{x},t)=C_{\theta}(\tau)^{-1}\Theta\left(C_l(\tau)\mtx{x},\tau(t)\right),\\
    &\mtx{u}(\mtx{x},t)=C_{\omega}(\tau)^{-1}C_{l}(\tau)^{-1}\mtx{U}\left(C_l(\tau)\mtx{x},\tau(t)\right)
\end{aligned}
\end{equation}
with the time-dependent scaling factors defined by
\begin{equation}\label{eqt:dynamic_rescaling_of_BS_change_of_variables2}
\begin{aligned}
    C_{\omega}(\tau)&=\exp\left(\int_0^\tau c_{\omega}(s)\idiff s\right),
    \quad C_{\theta}(\tau)=\exp\left(\int_0^\tau c_{\theta}(s) \idiff s\right),\\
    C_{l}(\tau)&=\exp\left(\int_0^\tau c_l(s)\idiff s\right),
    \quad t(\tau)=\int_0^{\tau}C_{\omega}(s)\idiff s.
\end{aligned}
\end{equation}
For consistency with the scaling properties of the Boussinesq equations, we enforce the relation
\[ c_{\theta} = c_l + 2c_{\omega}. \]
Substituting these ansatzes into \eqref{eqt:Boussinesq_v2} yields the following equivalent form in the rescaled variables $(\mtx{X},\tau)$:
\begin{equation}\label{eqt:dynamic_rescaling_of_BS}
    \begin{aligned}
    &\Omega_{\tau} + (\mtx{U}+c_l\mtx{X})\cdot\nabla\Omega = c_{\omega}\Omega + \Theta_{X_1},\\
    &\Theta_{\tau} + (\mtx{U}+c_l\mtx{X})\cdot\nabla\Theta = (c_l+2c_{\omega})\Theta,\\
    &\mtx{U} = \nabla^{\perp}(-\Delta)^{-1}\Omega,
    \end{aligned}
\end{equation}
where $\mtx{X} = C_l(\tau)\mtx{x}$.
We refer to \eqref{eqt:dynamic_rescaling_of_BS} as the dynamic rescaling formulation of the 2D Boussinesq equations. In particular, the steady state equation to \eqref{eqt:dynamic_rescaling_of_BS} is given by
\begin{equation}\label{eqt:BS_steady_state}
    \begin{aligned}
    &(\mtx{U}+c_l\mtx{X})\cdot\nabla\Omega = c_{\omega}\Omega + \Theta_{X_1},\\
    &(\mtx{U}+c_l\mtx{X})\cdot\nabla\Theta = (c_l+2c_{\omega})\Theta,\\
    &\mtx{U} = \nabla^{\perp}(-\Delta)^{-1}\Omega.
    \end{aligned}
\end{equation}
The reader can easily verify that if $(\bar{\Omega},\bar{\Theta},\bar c_l,\bar c_\omega)$ is a solution to \eqref{eqt:BS_steady_state}, then it corresponds to exact self-similar finite-time blowups of \eqref{eqt:Boussinesq_v2} that take the form 
\[
\omega(\mtx{x},t)=(T-t)^{\lambda}
\tilde{\Omega}\left(\frac{\mtx{x}}{(T-t)^{\gamma}}\right),
\quad
\theta(\mtx{x},t)=(T-t)^{\mu}
\tilde{\Theta}\left(\frac{\mtx{x}}{(T-t)^{\gamma}}\right),
\]
where $\tilde{\Omega}$ and $\tilde{\Theta}$ are rescaled versions of $\bar{\Omega}$ and $\bar{\Theta}$, $\gamma=-\bar{c}_l/\bar{c}_\omega$, $\lambda=-1$ and $\mu=\gamma-2$.

Note that if $(\bar{\Omega}(\mtx{X}),\bar{\Theta}(\mtx{X}),\bar{c}_l,\bar{c}_{\om})$ is a steady state solution of \eqref{eqt:dynamic_rescaling_of_BS}, then
\begin{equation}\label{eqt:scaling_invariance_BS}
(\alpha \bar{\Omega}(\beta \mtx{X}),\alpha^2 \bar{\Theta}(\beta \mtx{X})/\beta,\alpha \bar{c}_l,\alpha \bar{c}_{\om})
\end{equation} is also a steady state solution for any $\alpha\in \R$, $\beta>0$, which indicates that this equation admits two degrees of freedom. Hence, in order to uniquely identify the profile up to rescaling, we need to impose two normalization conditions to fix these two degrees of freedom.

Moreover, if under some normalization conditions the solution tuple $(\Omega,\Theta,c_l,c_{\omega})$ converges to a steady state $(\bar{\Omega},\bar{\Theta},\bar{c_l},\bar{c}_{\om})$ of \eqref{eqt:dynamic_rescaling_of_BS} as $\tau\to\infty$, then the solution would blow up in an asymptotically self-similar manner. We demonstrate this more precisely by the following Lemma.

\begin{lemma}\label{lem:BS_asymptoticallyblowup}
    Suppose \eqref{eqt:dynamic_rescaling_of_BS} converges to some
    steady state $(\bar{\Omega},\bar{\Theta},\bar{c}_l,\bar{c}_{\omega})$
    with $\bar{c}_{\omega}<0$. Furthermore, assume the parameters
    $c_l$ and $c_{\omega}$ converge sufficiently fast in the sense that
   \begin{equation}\label{eqt:convergence_of_clcomega_2D}
    \int_{0}^{+\infty}|c_l(\tau)-\bar{c}_l|\idiff \tau<\infty, \quad \int_{0}^{+\infty}|c_{\omega}(\tau)-\bar{c}_{\omega}|\idiff \tau<\infty .
   \end{equation}
    Then the solution $\omega(\mtx{x},t)$ to the original system
    \eqref{eqt:Boussinesq_v2} develops a self-similar finite-time blowup
    asymptotically. More specifically, there exists a constant $C_0>0$
    such that for any $\mtx{x}\in \R^2_+$ near which $\Omega$
    converges to $\bar{\Omega}$ uniformly, it holds that
    \[\lim_{t\to T}(T-t)\omega\left((T-t)^\gamma \mtx{x}/C_0,t\right)=-\frac{1}{\bar{c}_{\omega}}\bar{\Omega}(\mtx{x}),\] 
where \[T=\int_0^{+\infty}C_{\omega}(s)\idiff s, \quad \gamma= -\frac{\bar{c}_{l}}{\bar{c}_{\omega}}.\]
\end{lemma}
The proof of this lemma is identical to that of the 1D case, and we omit it here for brevity.

\begin{remark}\label{rem:BS_critical_Lp_exponent}
    The conclusion of Lemma \ref{lem:BS_asymptoticallyblowup} describes an asymptotically self-similar blowup of $\om$ at some finite time $T$ with the self-similar profile $\bar{\Omega}$. 
    A scaling analysis similar to that in Remark \ref{rem:HL_critical_Lp_exponent} yields that the critical $L^p$ blowup exponent for $\om$ is $2\gamma$ in the 2D case. 
\end{remark}

With Lemma \ref{lem:BS_asymptoticallyblowup}, one can transform the problem of studying the asymptotically self-similar blowup of the original equations into establishing the asymptotic stability of the dynamic rescaling equations around some steady state. Following this general strategy, Chen and Hou \cite{chen2022stable} established the stability of \eqref{eqt:dynamic_rescaling_of_BS} around some approximated steady state using a powerful computer-assisted approach, thus finally settling the conjecture on the Hou--Luo scenario.   

Note that in the work of Chen and Hou \cite{chen2022stable}, the initial data considered are odd with respect to $x_1$ and satisfy a non-degeneracy condition in the $x_1$-partial derivatives at $x_1=0$. In contrast, the remainder of this subsection is devoted to studying the self-similar blowup scenarios of the 2D Boussinesq equations starting from odd initial data that exhibit degeneracy in the $x_1$-partial derivatives at $x_1=0$.  To the best of our knowledge, the only existing effort in this direction is a preliminary numerical study in the thesis of Liu \cite{liu2017spatial}, which suggested that in the degenerate case, the solutions to \eqref{eqt:Boussinesq_v2} may develop two-scale self-similar blowups analogous to \eqref{eqt:multi_scale_blowup}. Departing from Liu's two-scale hypothesis, we conduct a thorough numerical investigation of the 2D Boussinesq equations within the dynamic rescaling framework and report potential two-stage self-similar blowup scenarios of \eqref{eqt:Boussinesq_v2}, which are remarkably similar to those observed in the 1D HL model.

For notational simplicity, we denote the rescaled space-time variables $(\mtx{X},\tau)$ by $(\mtx{x},t)$ in the rest of this subsection. Furthermore, in the practical numerical simulations, we introduce the change of variables $V := \Theta/x_1$ to ensure that the solution variable exhibits better far-field decay properties. Consequently, the rescaled system \eqref{eqt:dynamic_rescaling_of_BS} takes the form
\begin{equation}\label{eqt:dynamic_BS_scenario1}
\begin{aligned}
&\Omega_t + (\mtx{U}+c_l \mtx{x})\cdot\nabla\Omega = c_{\omega}\Omega + x_1V_{x_1} + V,\\
&V_t + (\mtx{U}+c_l \mtx{x})\cdot\nabla V = \big(2c_{\omega}-U_1/x_1\big)V,\\
&\mtx{U}=\nabla^{\perp}(-\Delta)^{-1}\Omega,
\end{aligned}
\end{equation}
where $U_1$ denotes the first component of $\mtx{U}$.

In view of the scaling invariance property \eqref{eqt:scaling_invariance_BS}, we have two degrees of freedom in choosing $c_l$, $c_{\om}$, and we do this by imposing two normalization conditions. More specifically, the normalization conditions are imposed to enforce that
\begin{equation}\label{eqt:numerically_compute_cl_2D}
U_1(1,0,t) + c_l \equiv 0, \quad \partial_{x_1} U_1(0,0,t) \equiv \partial_{x_1} U_1(0,0,0).
\end{equation}

In the following, we present the numerical results of \eqref{eqt:dynamic_BS_scenario1} imposing the above normalization conditions that support the potential two-stage self-similar blowup scenarios of \eqref{eqt:Boussinesq_v2}. The initial data are chosen as
\[ \Omega(x_1,x_2) = \frac{10 x_1^9}{x_1^{10} + x_2^{10} + 16}, \quad V(x_1,x_2) = \frac{10 x_1^{9}}{x_1^{10} + x_2^{10} + 16}, \]
which are odd with respect to $x_1$ and exhibit degeneracy in the $x_1$-partial derivatives at $x_1=0$. Figure \ref{fig:Bs_sc1_mesh_evo_om_bu} illustrates the early-stage evolution of the spatial profile of $\Omega$. As shown, $\Omega$ demonstrates a distinct trend of developing a singularity at $(1,0)$, which partially confirms Liu's early observations \cite{liu2017spatial}.

\begin{figure}[!htbp]
\centering
        \includegraphics[width=0.39\textwidth]{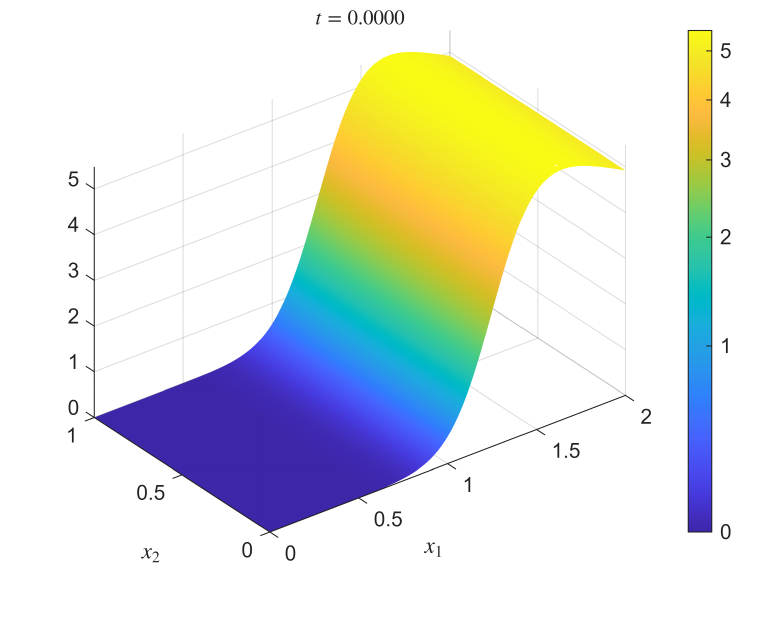}
        \includegraphics[width=0.39\textwidth]{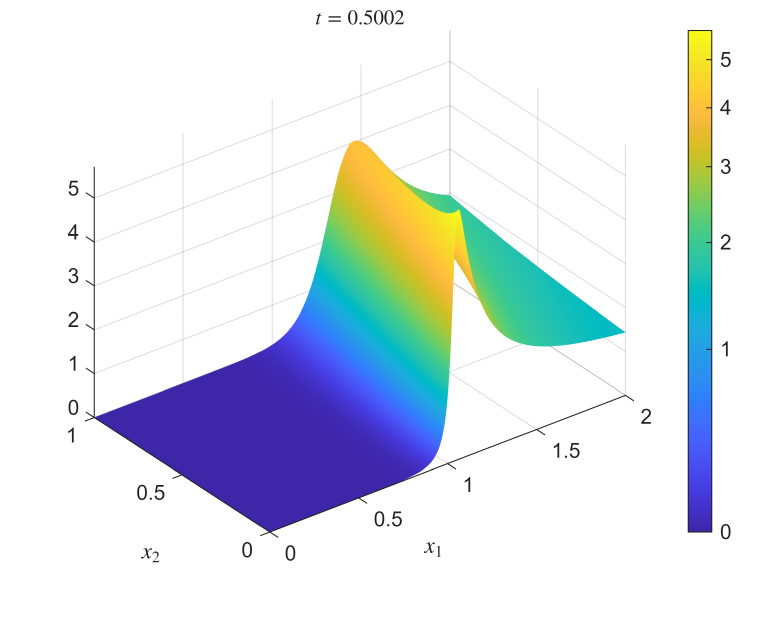}
        \includegraphics[width=0.39\textwidth]{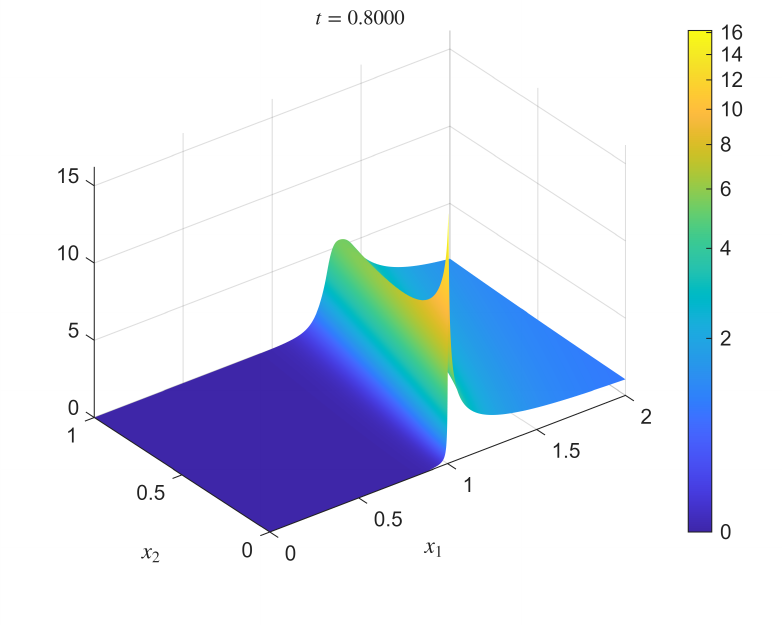}
        \includegraphics[width=0.39\textwidth]{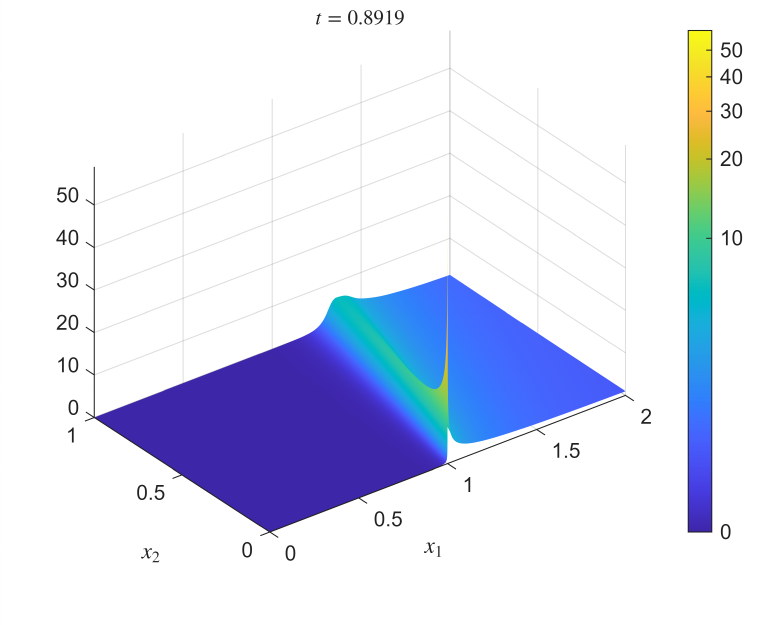}
    \caption[Early evolution of 3D spatial profiles of $\Omega$ (Scenario 1).]{Early evolution of 3D spatial profiles of $\Omega$ (Scenario 1). As time progresses, the profile concentrates rapidly near the focus point
    $(1,0)$.}
    \label{fig:Bs_sc1_mesh_evo_om_bu}
\end{figure}

In further numerical simulations, we employ an adaptive mesh strategy that dynamically clusters grid points near the developing singularity to ensure sufficient resolution. Our results suggest that solutions of \eqref{eqt:dynamic_BS_scenario1} develop asymptotically self-similar finite-time blowups with respect to the dynamic rescaling time. This phenomenon is corroborated by Figures \ref{fig:Bs_sc1_blowup}, \ref{fig:BS_scenario1_boundary_convergence_compare}, \ref{fig:BS_scenario1_contour_convergence_compare}, and \ref{fig:BS_scenario1_surf_convergence_compare}. The top row of Figure \ref{fig:Bs_sc1_blowup} illustrates the rapid growth of $\| \Omega \|_{L^\infty}$ and $\|\Theta_{x_1}\|_{L^\infty}$. In particular, in the singular region, $\Theta_{x_1}$ is observed to scale like $\Omega^2$, which implies that $\Omega$ blows up at the rate of $(T_{\text{dr}}-t)^{-1}$, where $T_{\text{dr}}$ denotes the blowup time with respect to the dynamic rescaling equations \eqref{eqt:dynamic_BS_scenario1}. Note that $T_{\text{dr}}$ corresponds to the first blowup time $\tilde T$ (mentioned earlier at the beginning of Section \ref{sec:dynamic_rescaling_of_BS}) with respect to the physical equation \eqref{eqt:Boussinesq_v2}. This blowup rate is further supported by the bottom row of Figure \ref{fig:Bs_sc1_blowup}, which demonstrates the linear dependence of $\|\Omega\|_{L^{\infty}}^{-1}$ and $\|\Theta_{x_1}\|_{L^\infty}^{-1/2}$ on time $t$. Moreover, Figures \ref{fig:BS_scenario1_boundary_convergence_compare}, \ref{fig:BS_scenario1_contour_convergence_compare}, and \ref{fig:BS_scenario1_surf_convergence_compare} display the spatial profiles of $(\Omega,\Theta_{x_1})$ in the singular region. As evident from these figures, the inner profiles of $\Omega$ and $\Theta_{x_1}$, after proper rescaling, are observed to approach regular limiting profiles, providing strong evidence for the existence of a self-similar blowup mechanism. In the context of the original 2D Boussinesq equations \eqref{eqt:Boussinesq_v2}, this corresponds to asymptotically self-similar boundary blowups of $\omega$ and $\theta_{x_1}$ occurring at a location distinct from the origin. We refer to this phenomenon as the Stage 1 blowup. We will study this phenomenon in more detail in the next subsection.

\begin{figure}[!htbp]
\centering
        \includegraphics[width=0.38\textwidth]{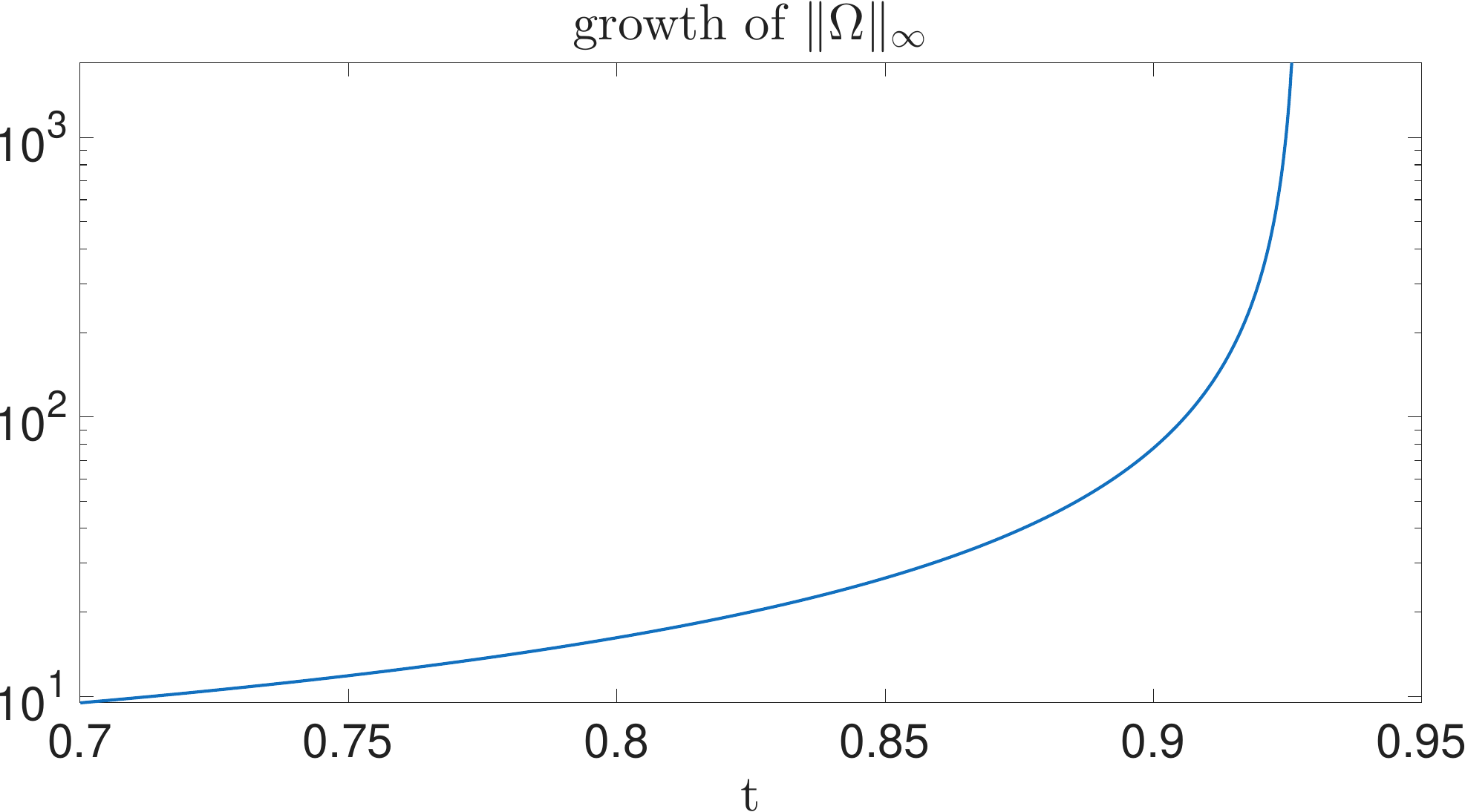}
        \includegraphics[width=0.38\textwidth]{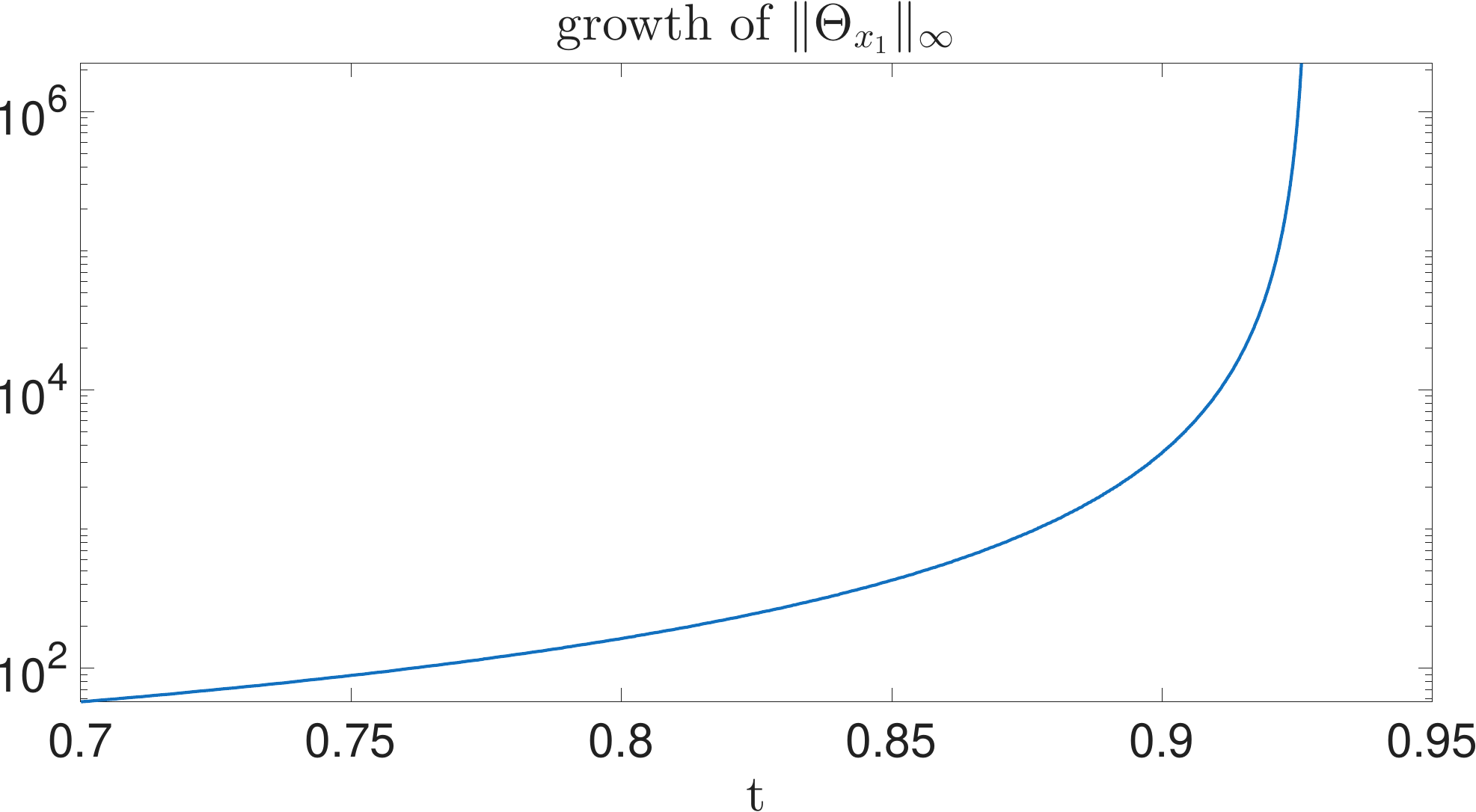}
        \includegraphics[width=0.38\textwidth]{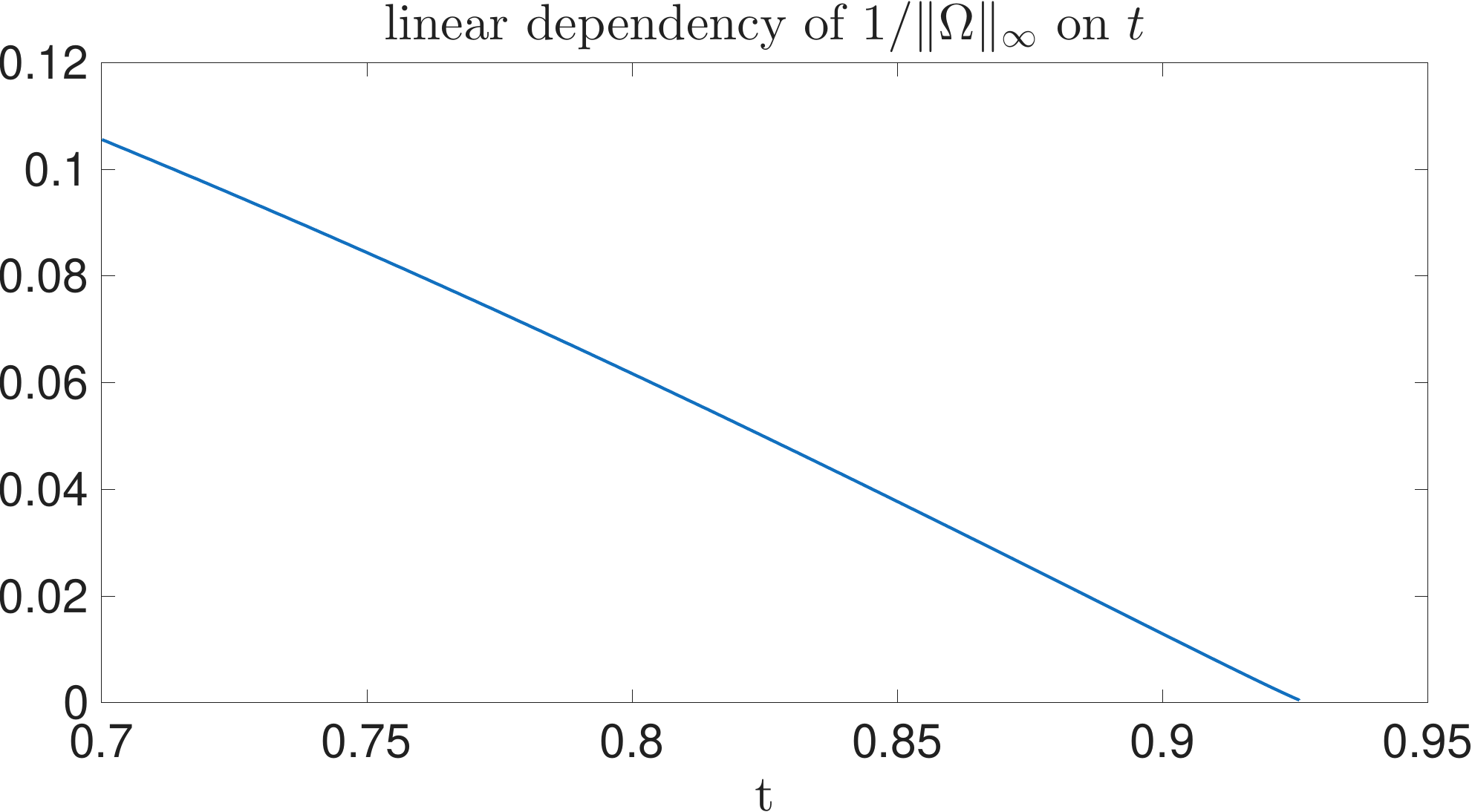}
        \includegraphics[width=0.38\textwidth]{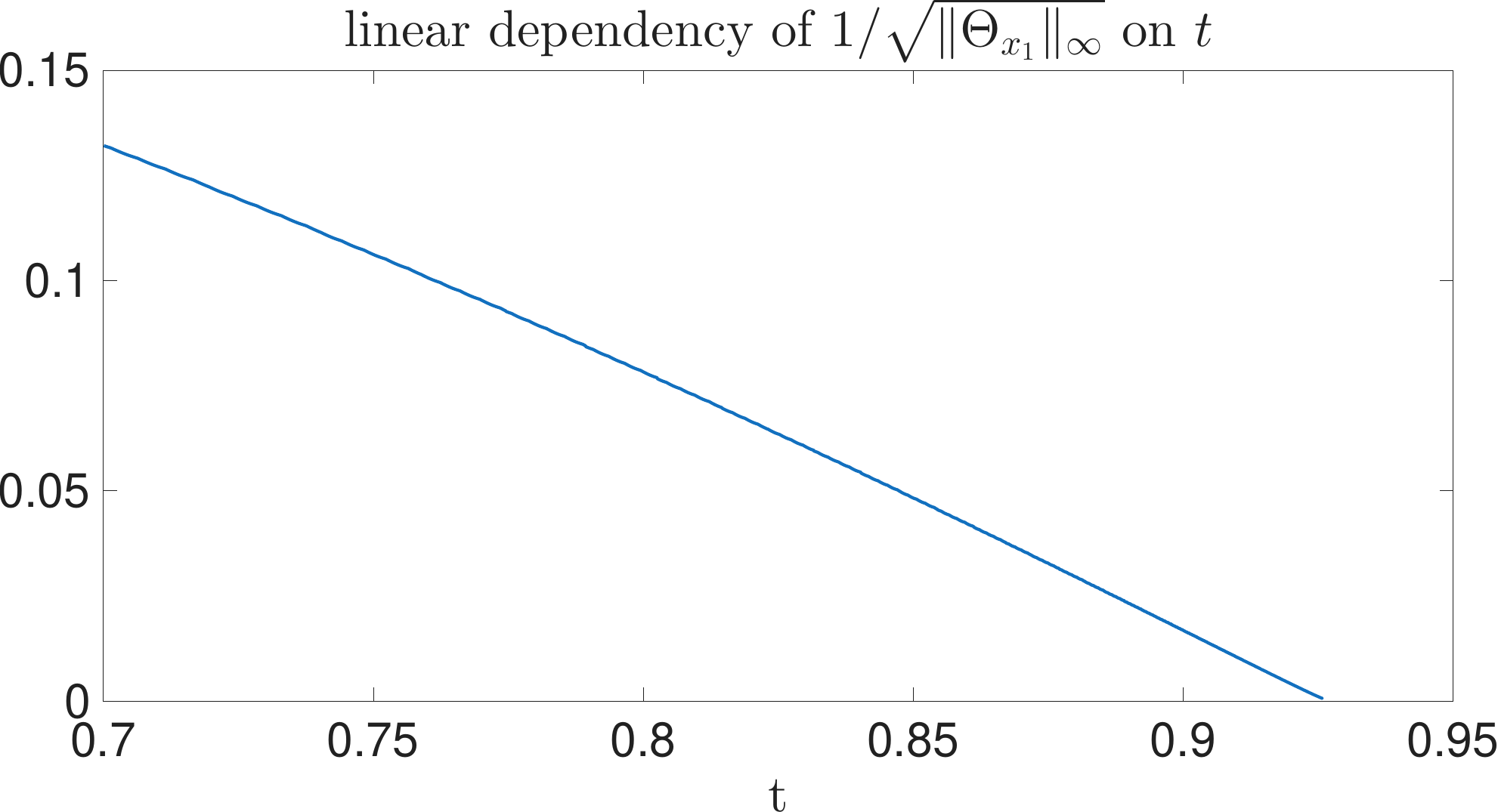}
    \caption[Time evolution of $\|\Omega\|_{L^{\infty}}$ and $\|\Theta_{x_1}\|_{L^{\infty}}$ (Scenario 1, computed on an adaptive mesh).]{Time evolution of $\|\Omega\|_{L^{\infty}}$ and $\|\Theta_{x_1}\|_{L^{\infty}}$ (Scenario 1, computed on an adaptive mesh).
    The top row shows the rapid growth of $\|\Omega\|_{L^{\infty}}$ and
    $\|\Theta_{x_1}\|_{L^{\infty}}$.
    The bottom row shows that $\|\Omega\|_{L^{\infty}}^{-1}$ and
    $\|\Theta_{x_1}\|_{L^{\infty}}^{-1/2}$ are approximately linear in $t$,
    which is consistent with the usual self-similar blowup rate.} 
    \label{fig:Bs_sc1_blowup}
\end{figure}

\begin{figure}[!htbp]
\centering
         \includegraphics[width=0.3\textwidth]{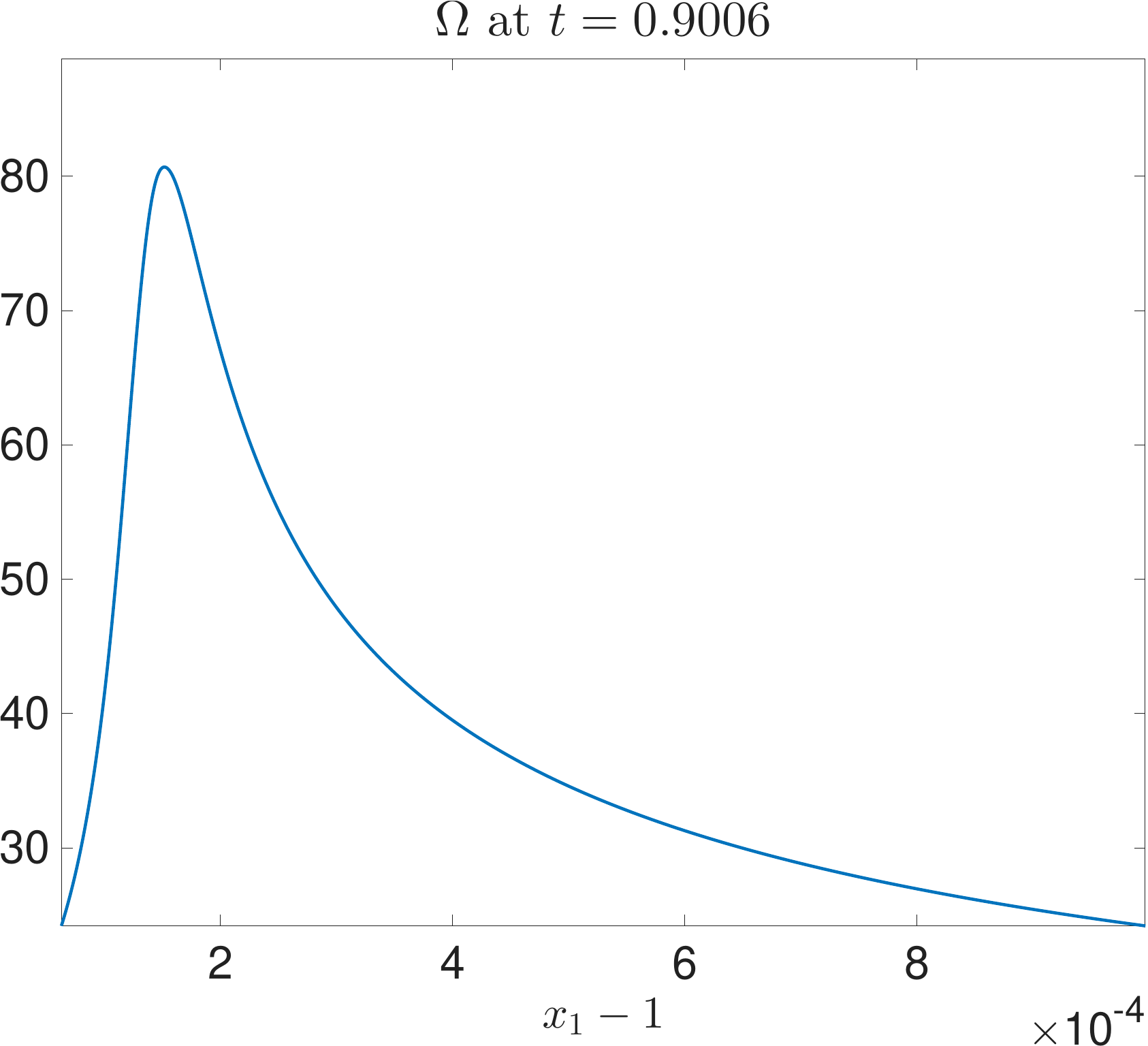}
         \includegraphics[width=0.3\textwidth]{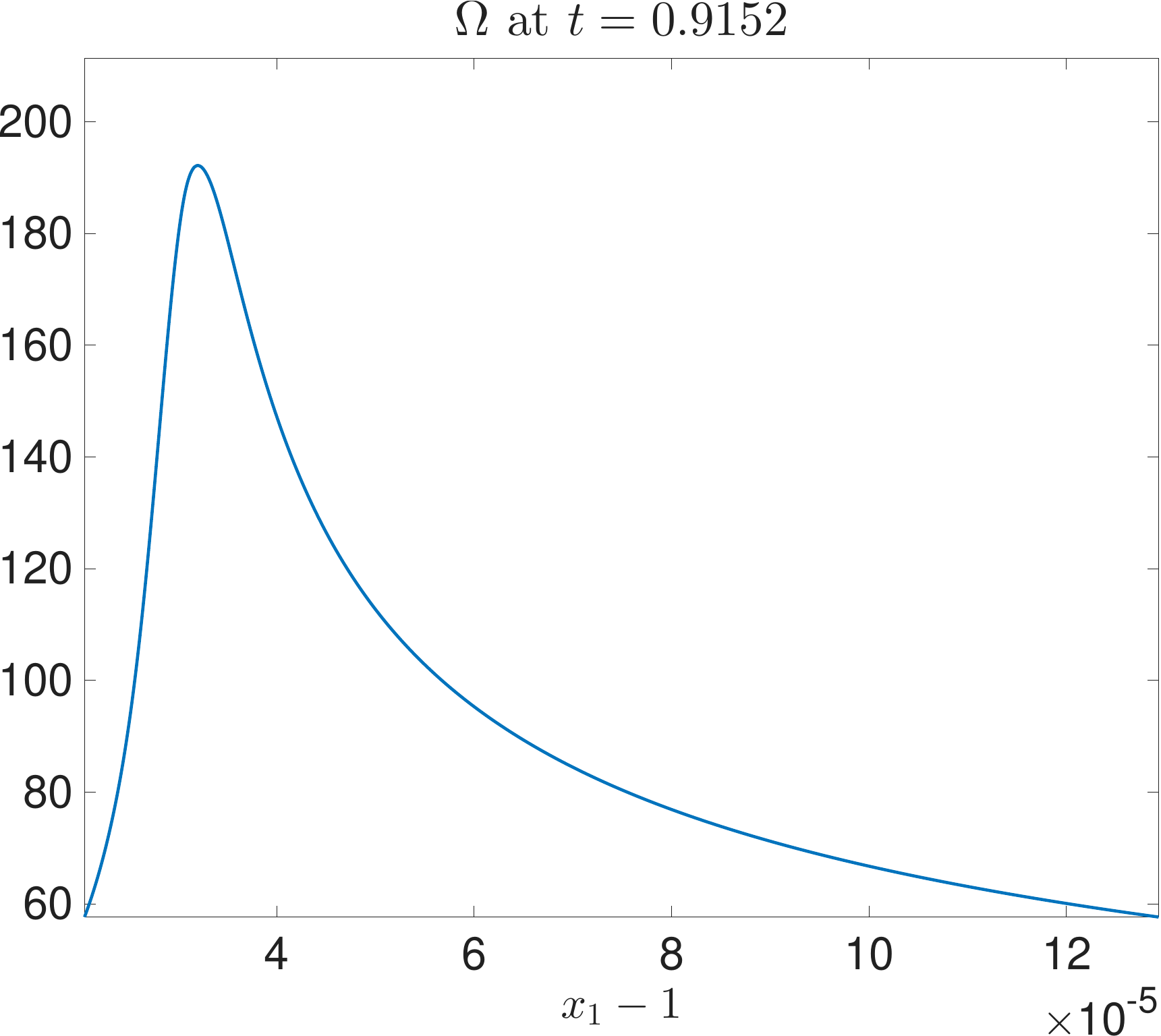}
         \includegraphics[width=0.3\textwidth]{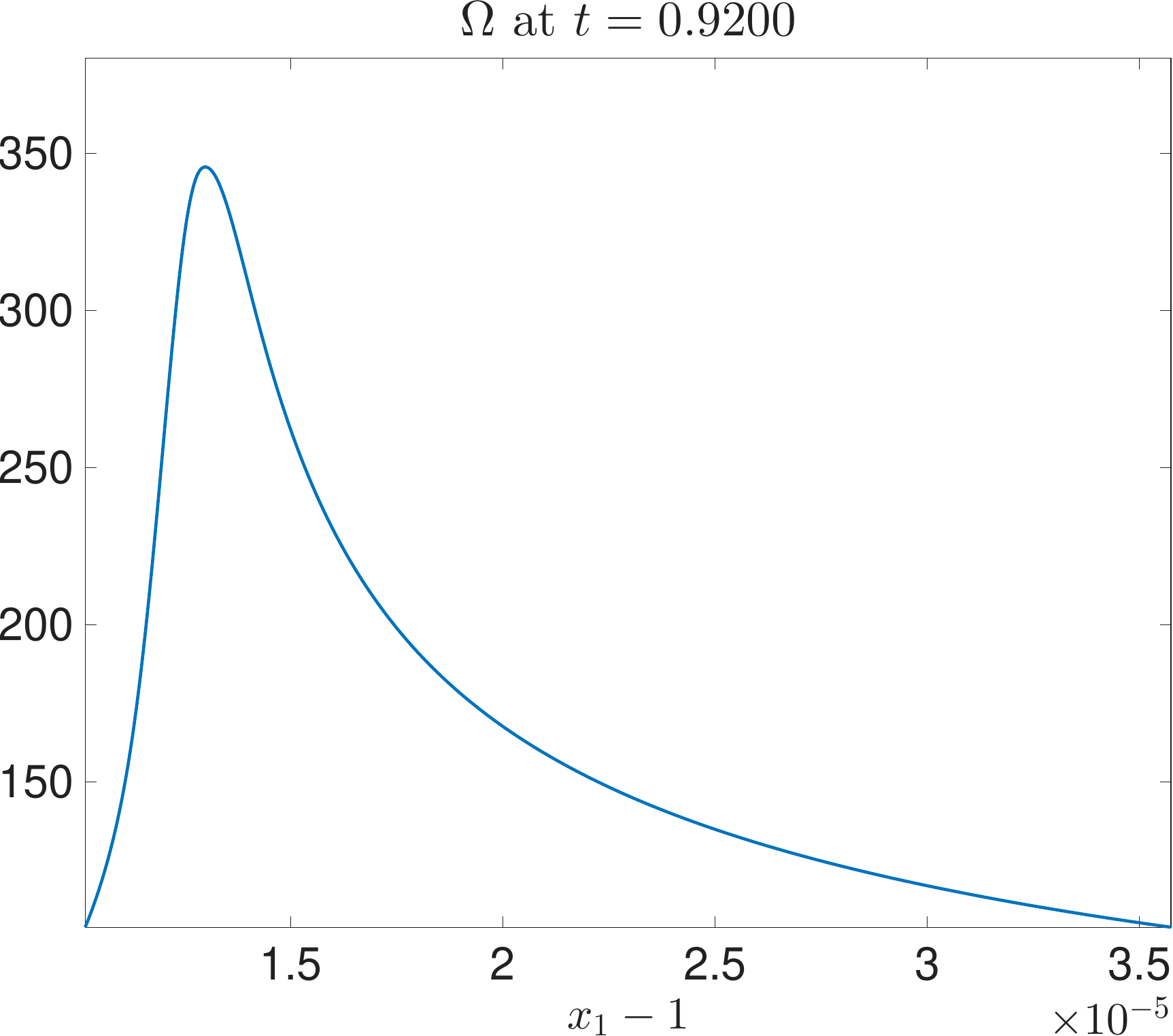}
         \includegraphics[width=0.3\textwidth]{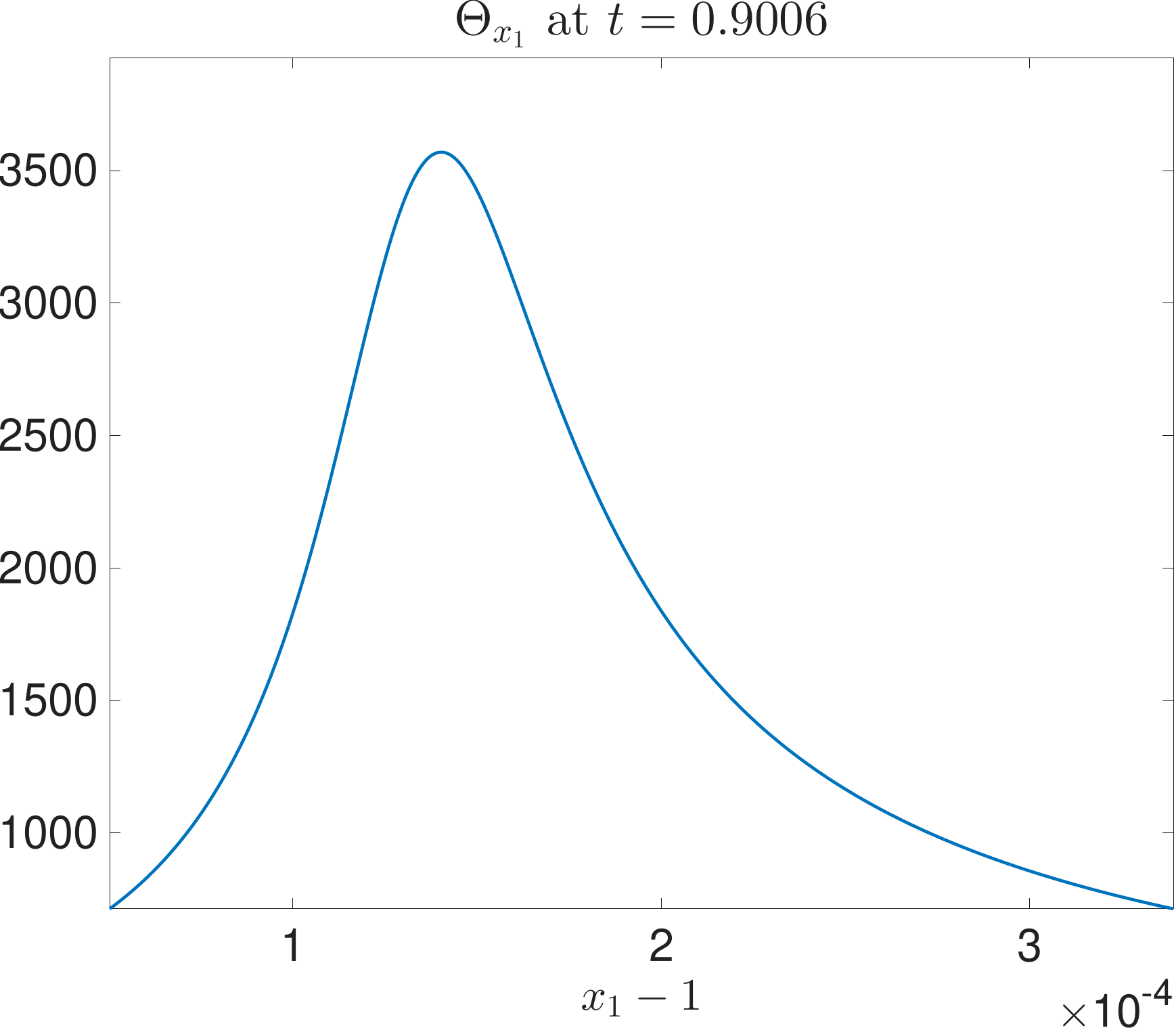}
         \includegraphics[width=0.3\textwidth]{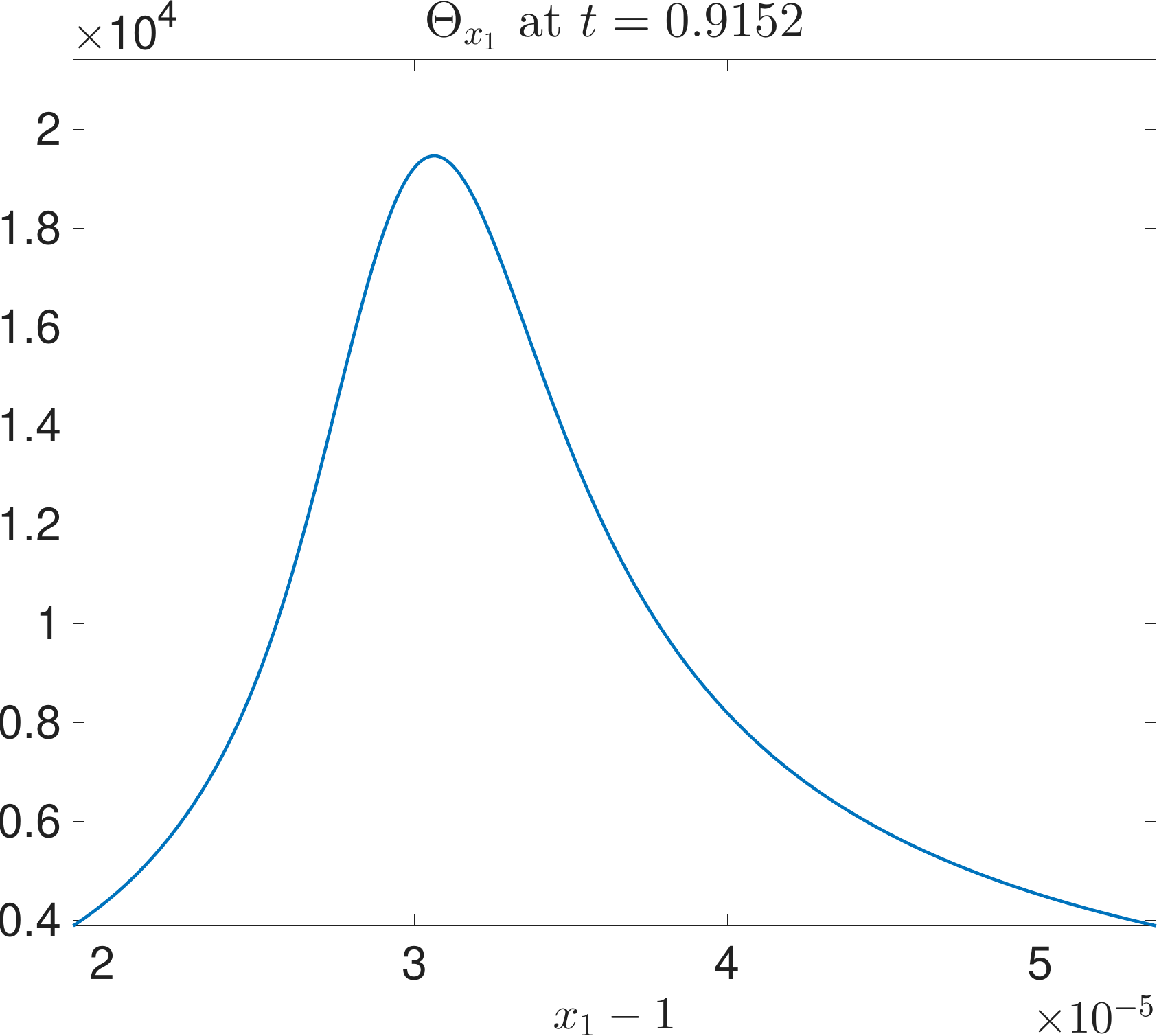}
         \includegraphics[width=0.3\textwidth]{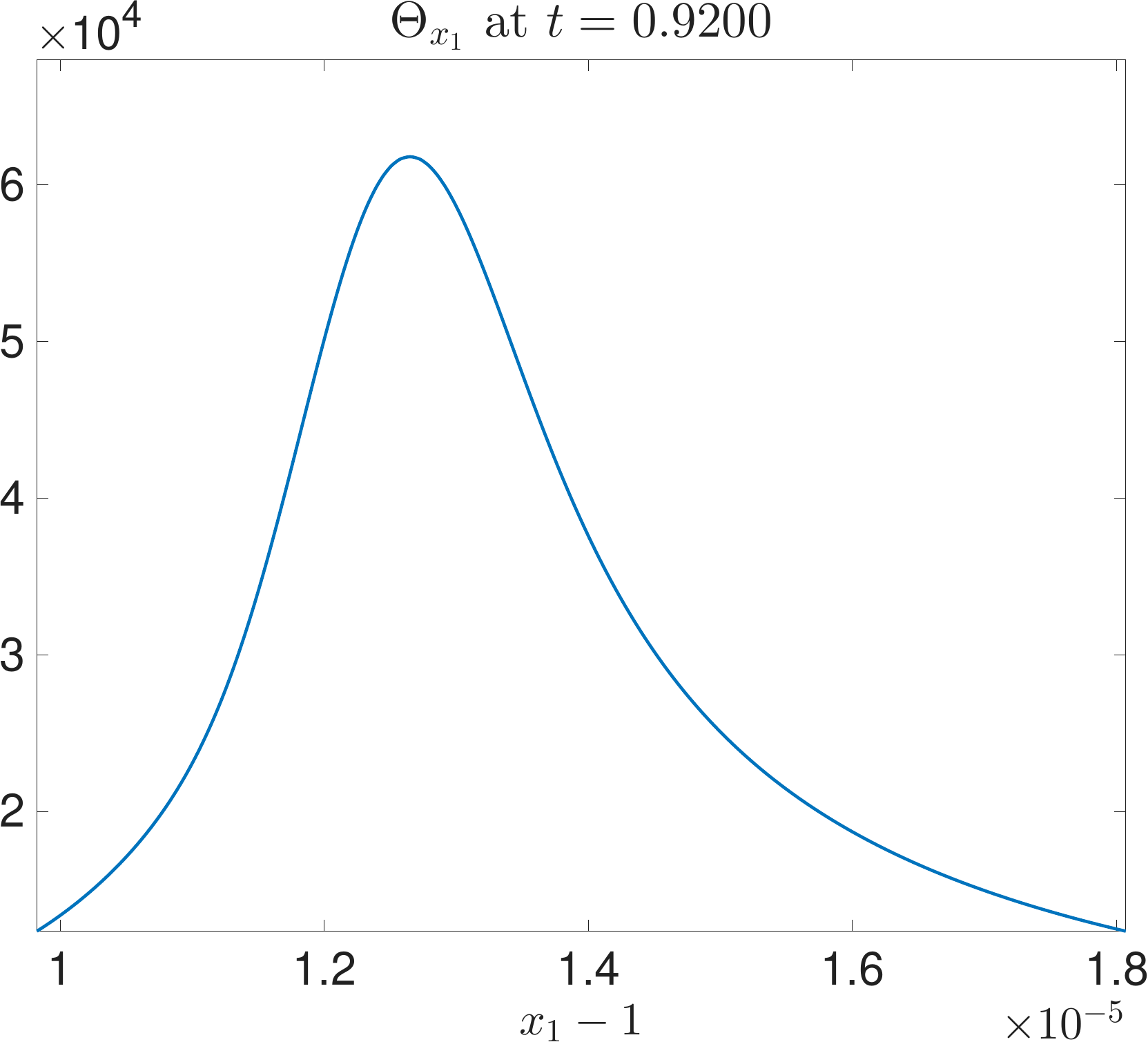}
    \caption[Inner boundary profiles of $\Omega$ and $\Theta_{x_1}$ in Scenario 1.]{Inner boundary profiles of $\Omega$ (top row) and $\Theta_{x_1}$ (bottom row) in Scenario 1. After proper rescaling, the inner profiles eventually stabilize at regular profiles as $t$ approaches the blowup time.}
    \label{fig:BS_scenario1_boundary_convergence_compare}
\end{figure}

\begin{figure}[!htbp]
\centering
         \includegraphics[width=0.32\textwidth]{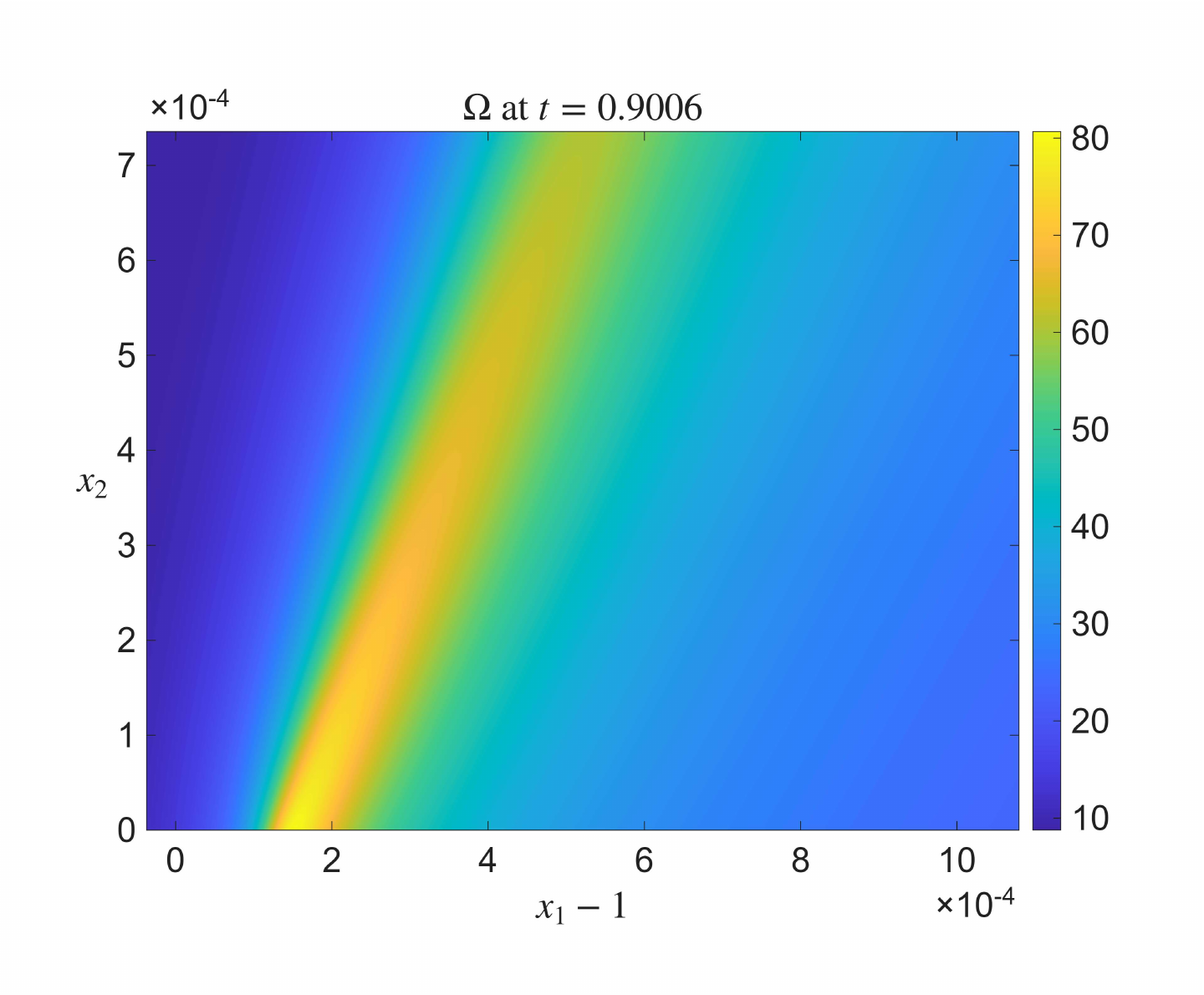}
         \includegraphics[width=0.32\textwidth]{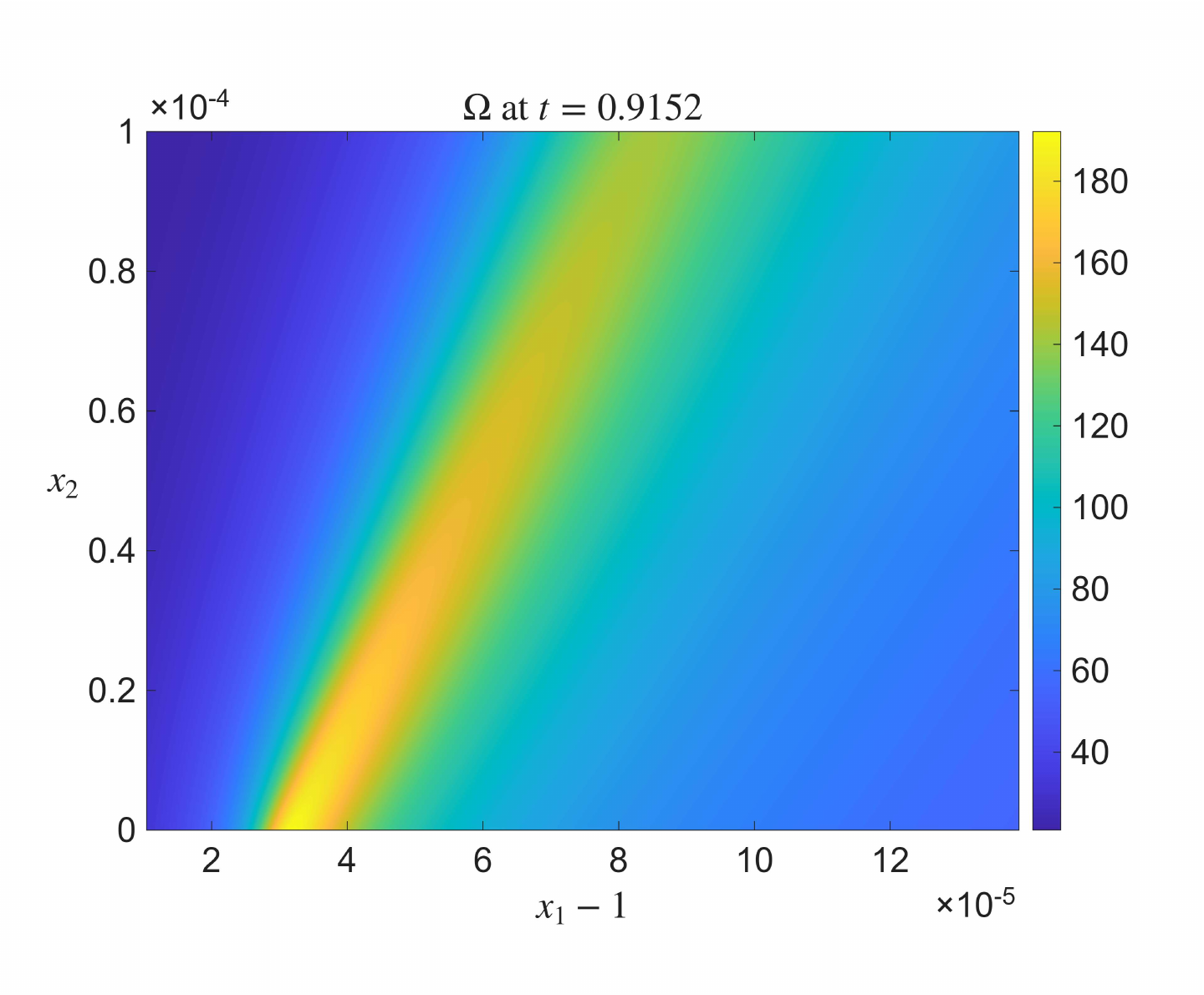}
         \includegraphics[width=0.32\textwidth]{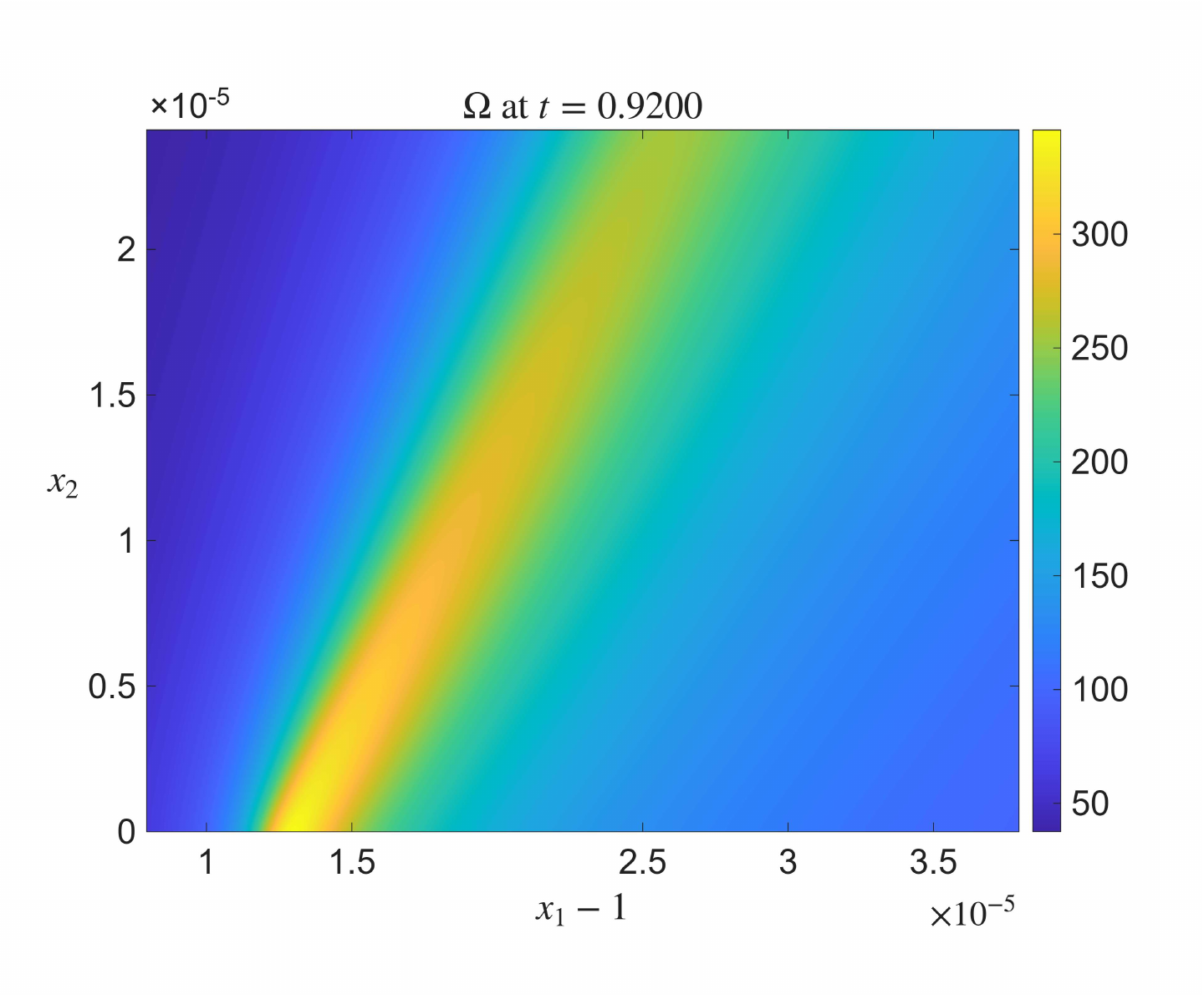}
         \includegraphics[width=0.32\textwidth]{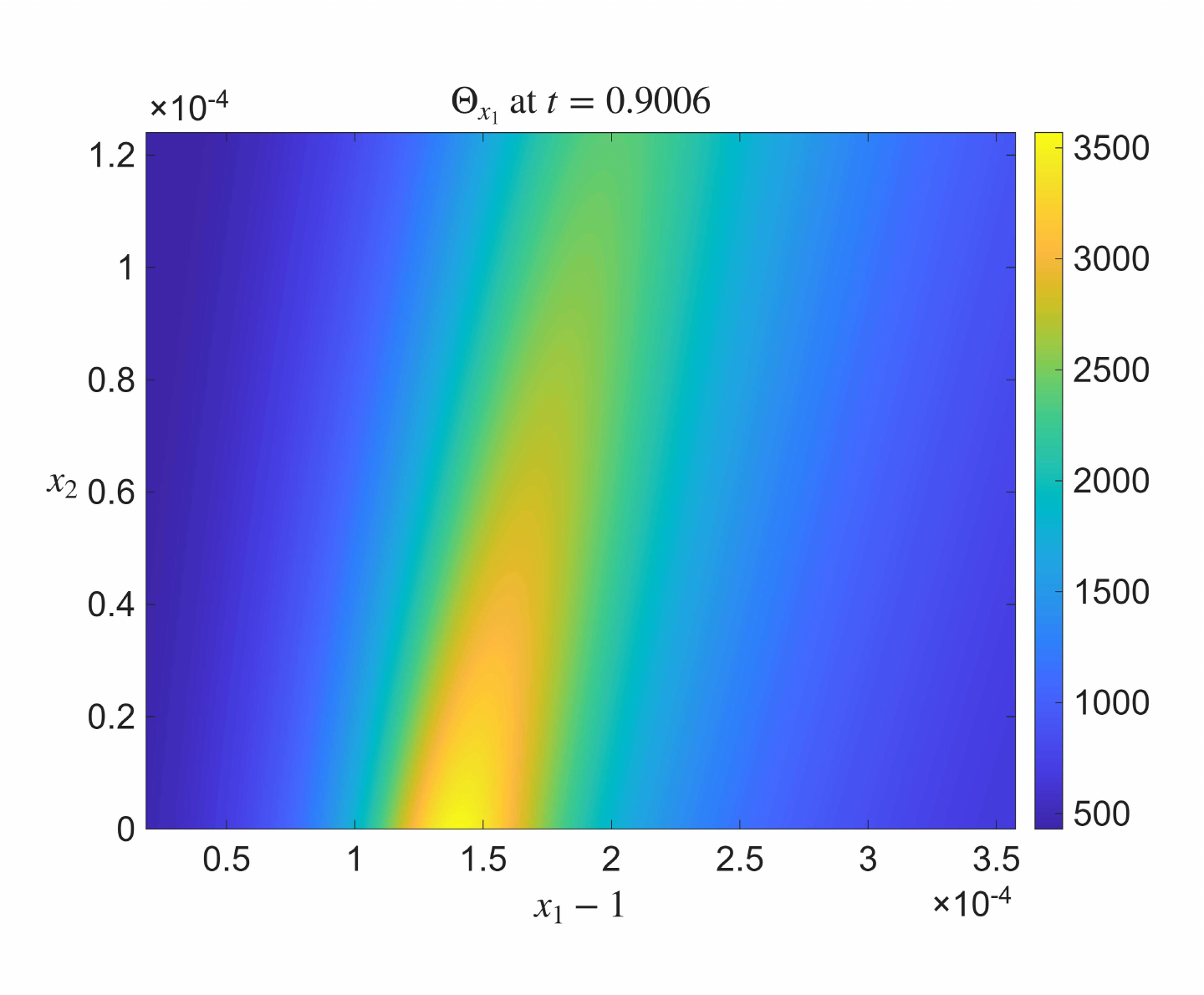}
         \includegraphics[width=0.32\textwidth]{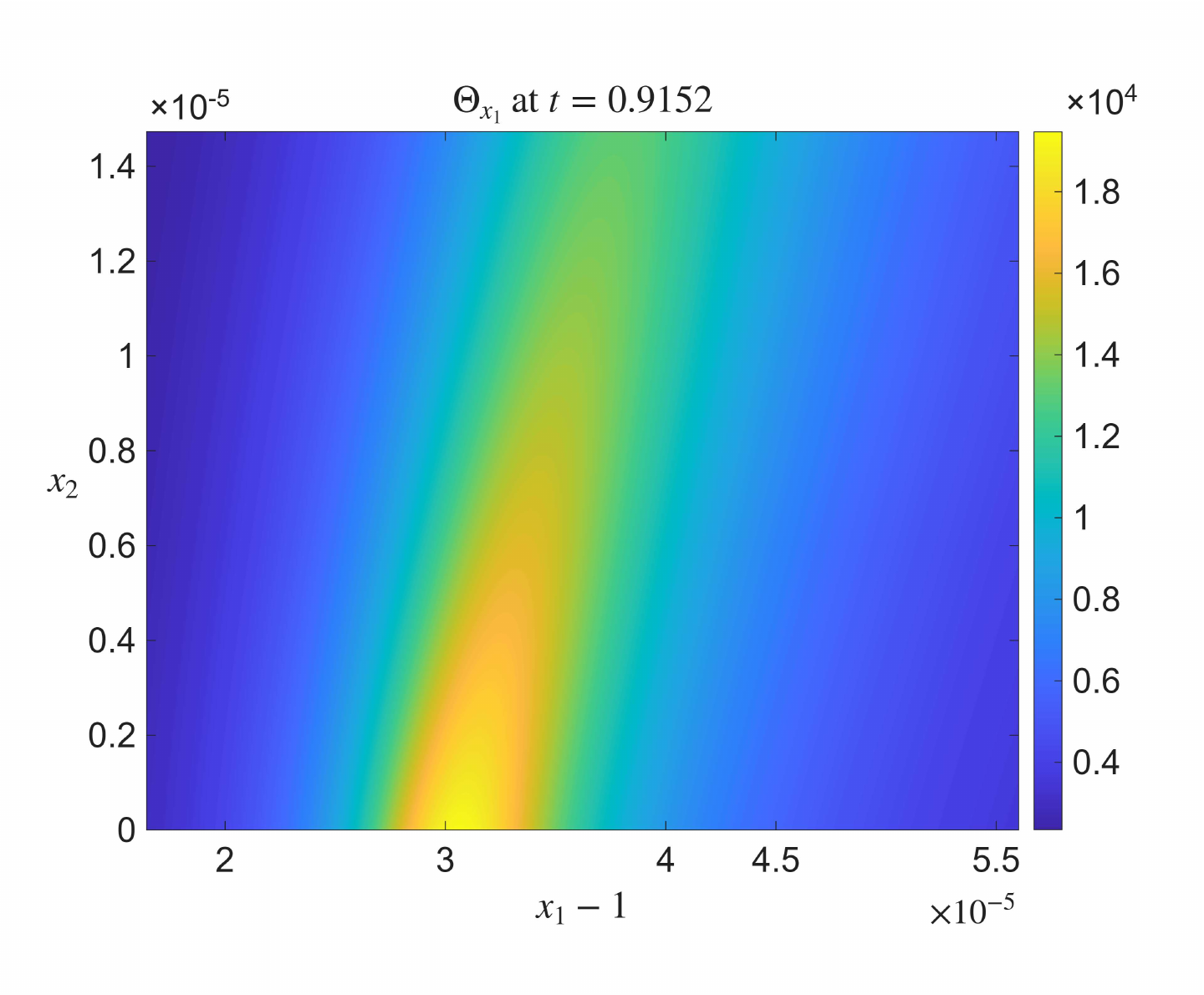}
         \includegraphics[width=0.32\textwidth]{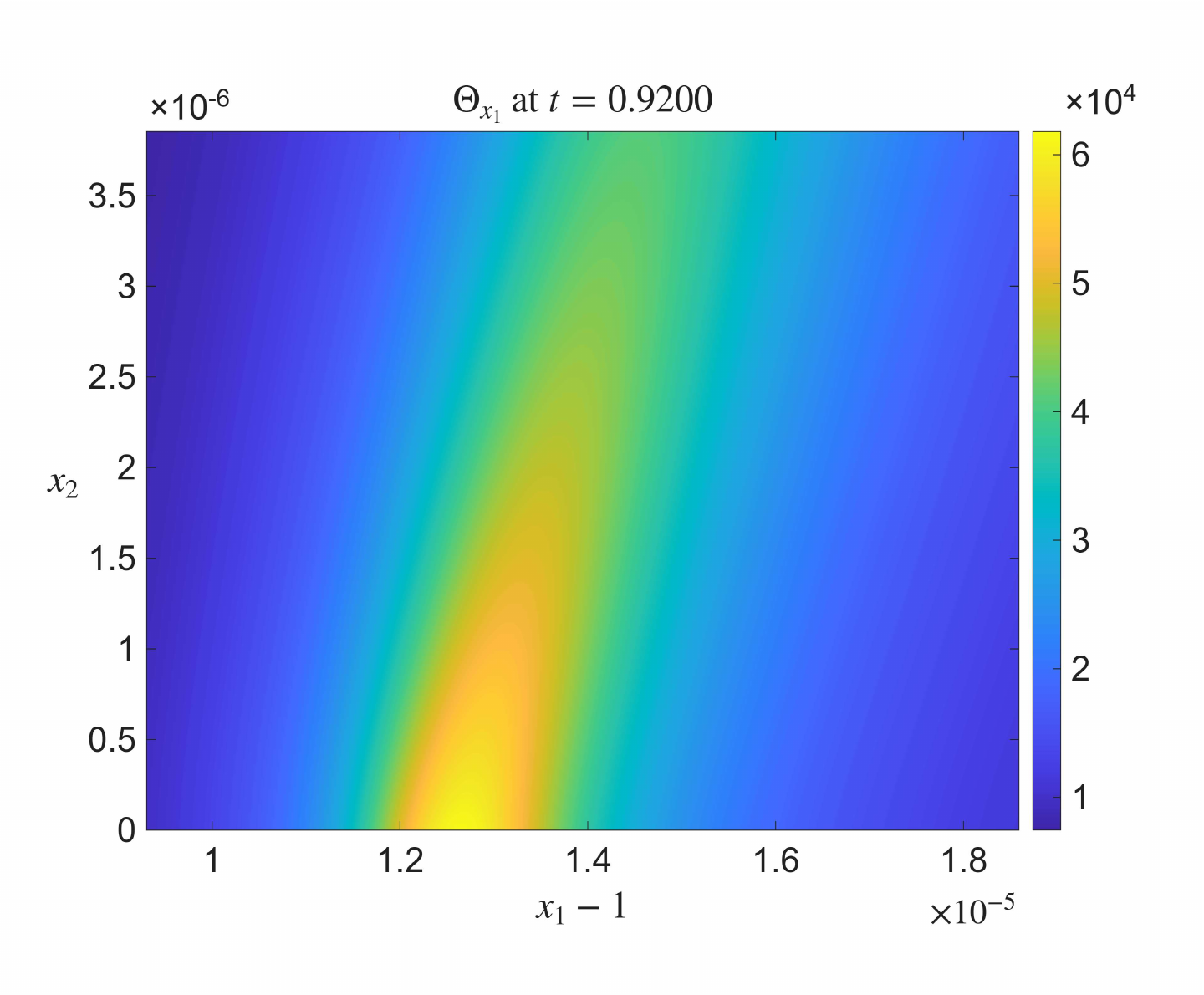}
    \caption[Inner contour profiles of $\Omega$ and $\Theta_{x_1}$ in Scenario 1.]{Inner contour profiles of $\Omega$ (top row) and $\Theta_{x_1}$ (bottom row) in Scenario 1. After proper rescaling, the inner profiles eventually stabilize at regular profiles as $t$ approaches the blowup time.}
    \label{fig:BS_scenario1_contour_convergence_compare}
\end{figure}

\begin{figure}[!htbp]
\centering
         \includegraphics[width=0.32\textwidth]{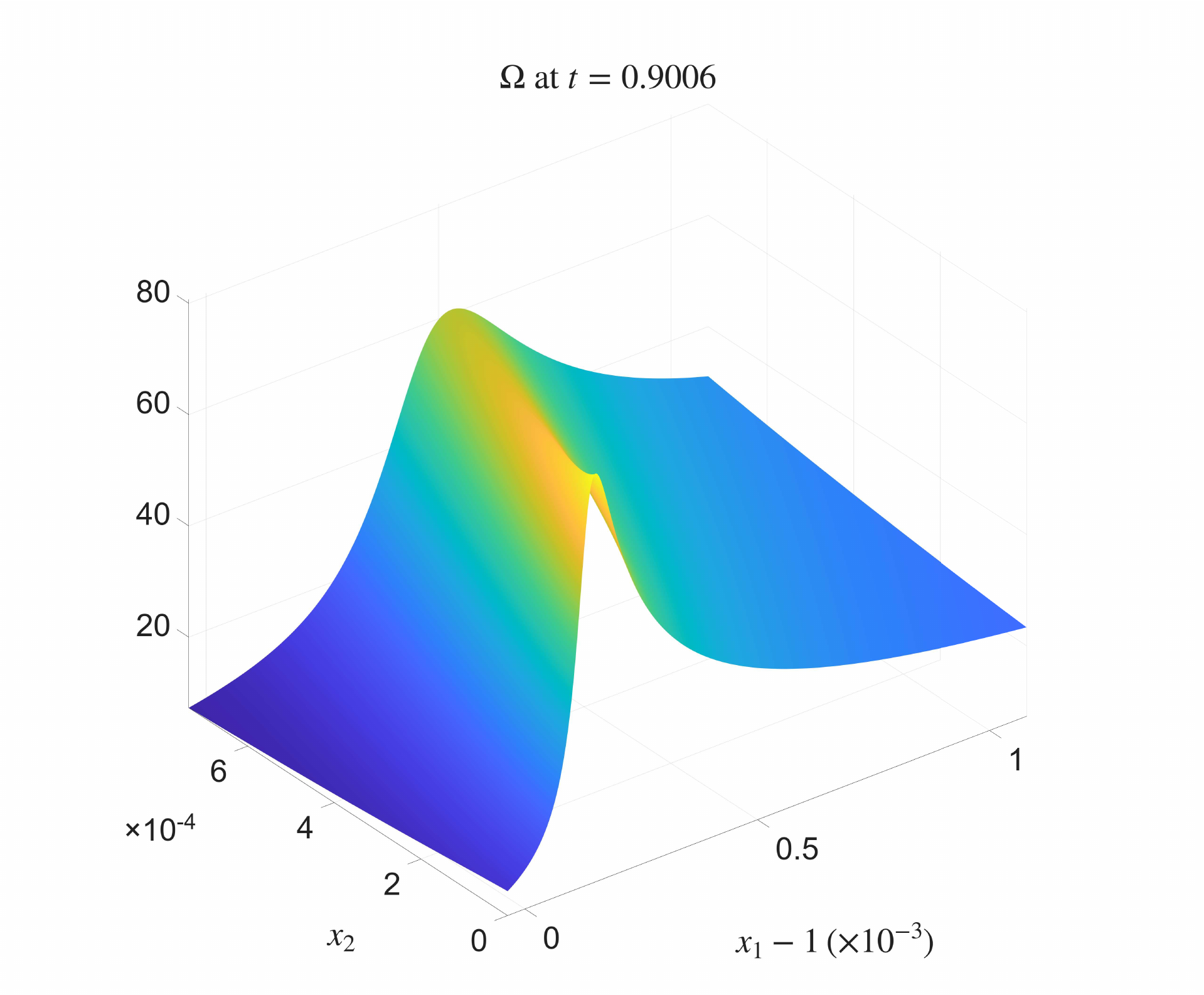}
         \includegraphics[width=0.32\textwidth]{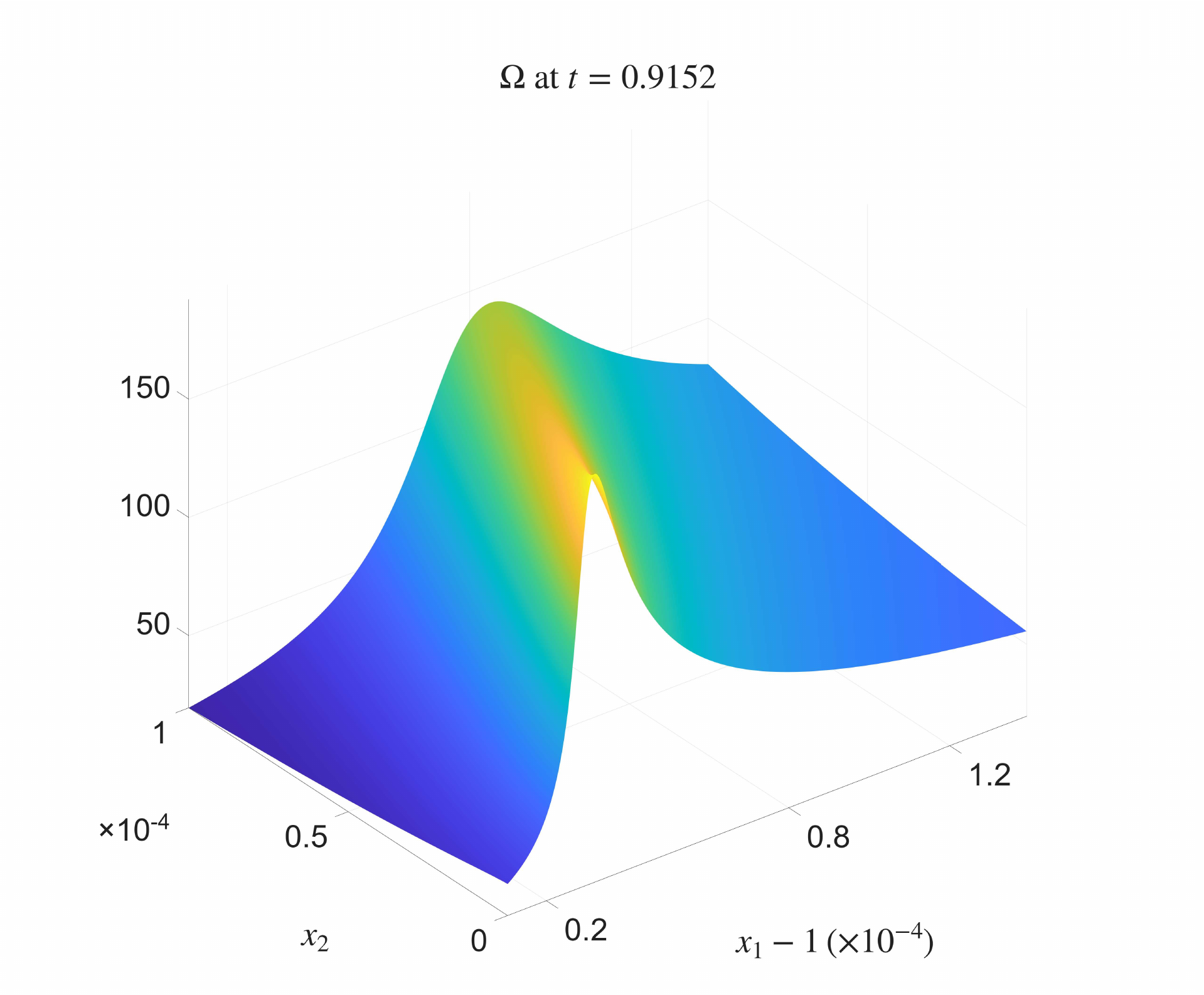}
         \includegraphics[width=0.32\textwidth]{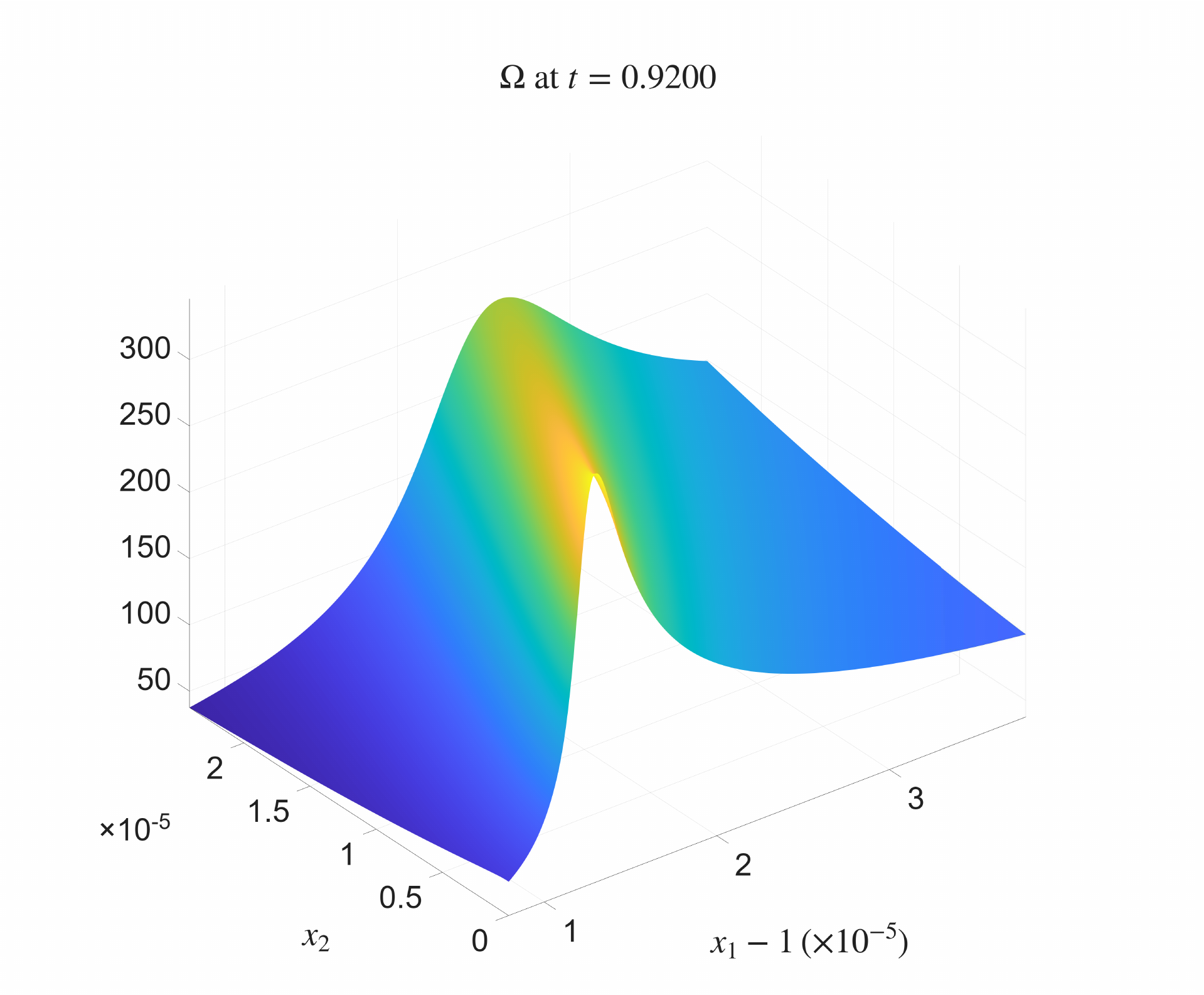}
       
          \includegraphics[width=0.32\textwidth]{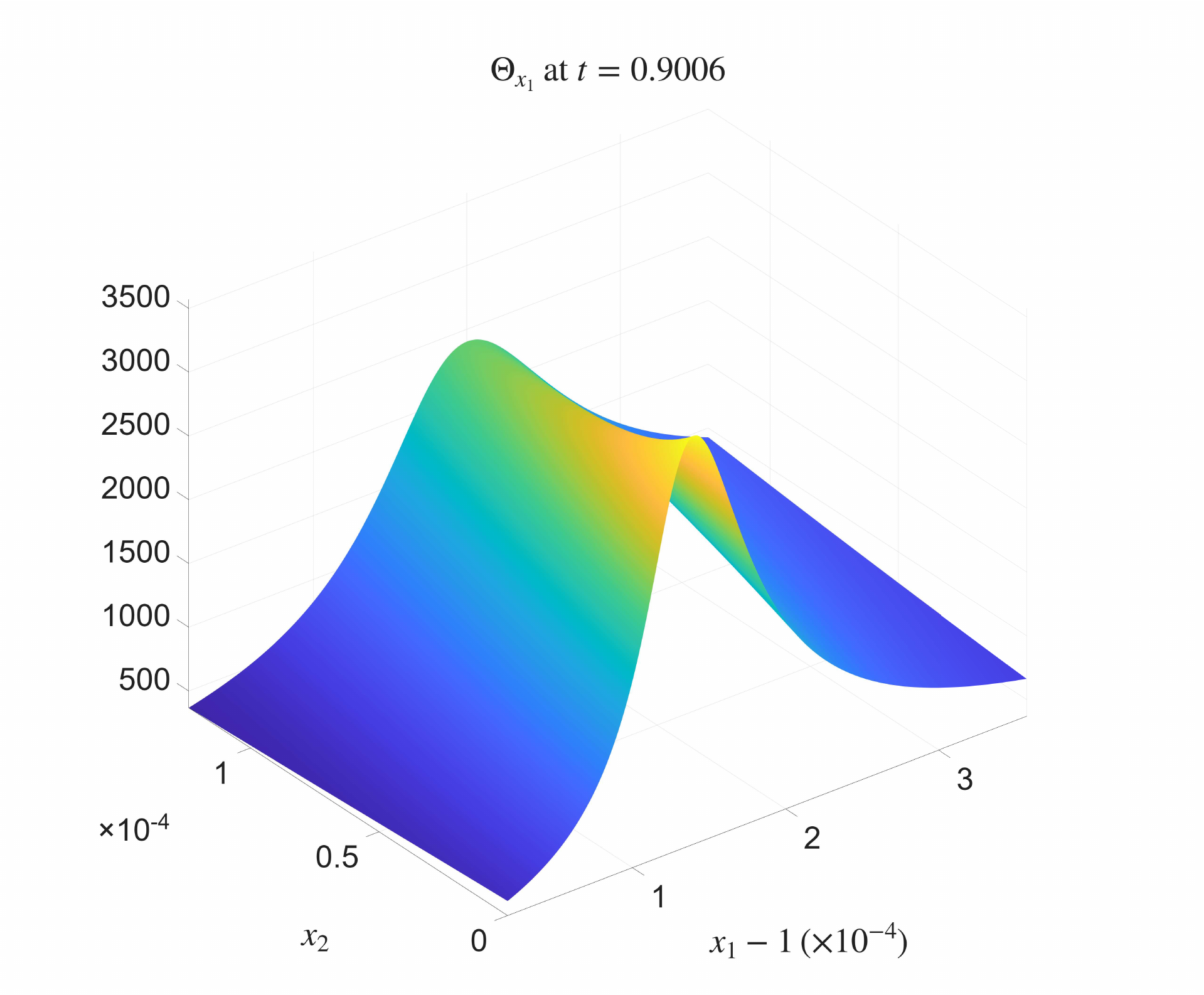}
         \includegraphics[width=0.32\textwidth]{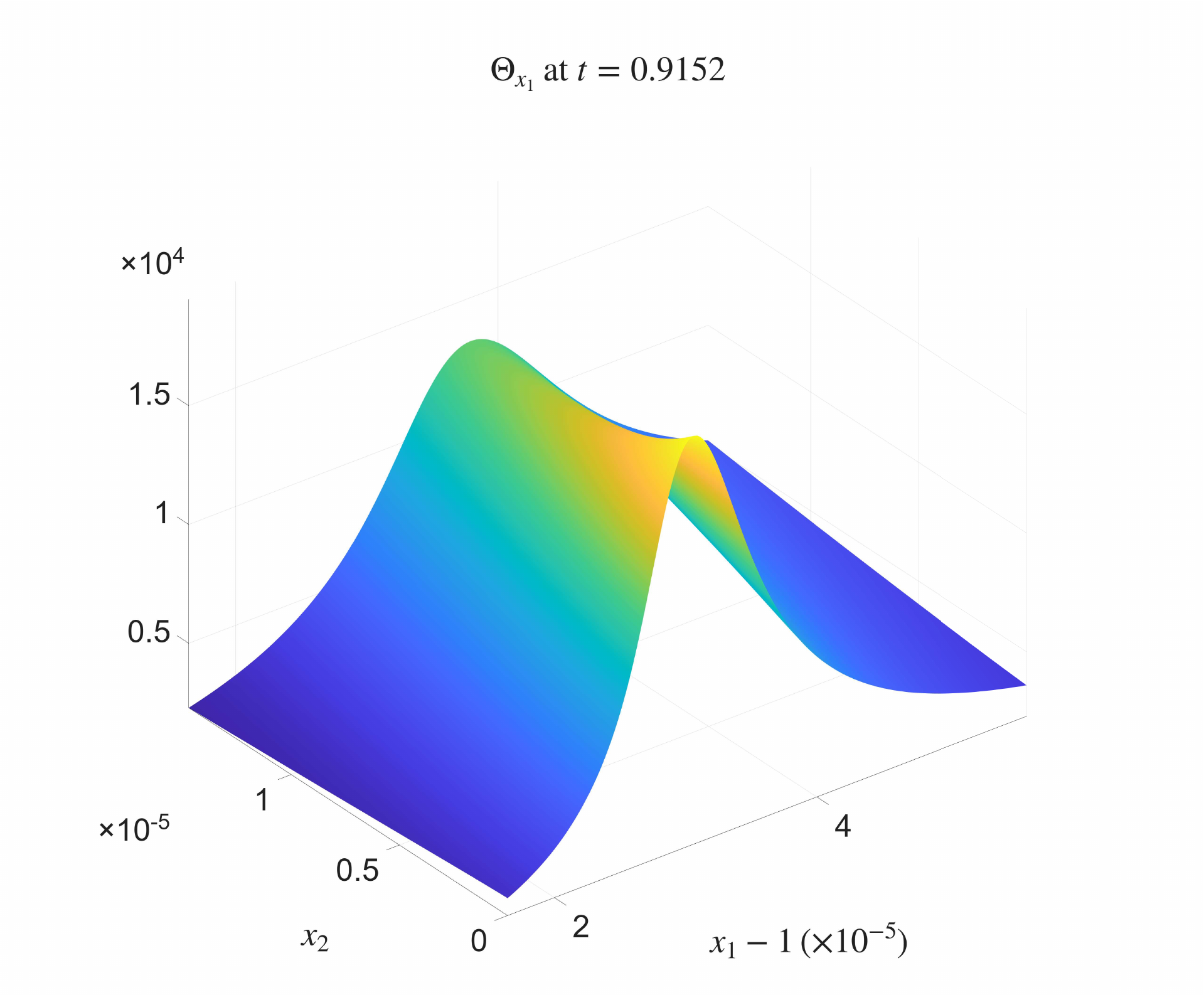}
         \includegraphics[width=0.32\textwidth]{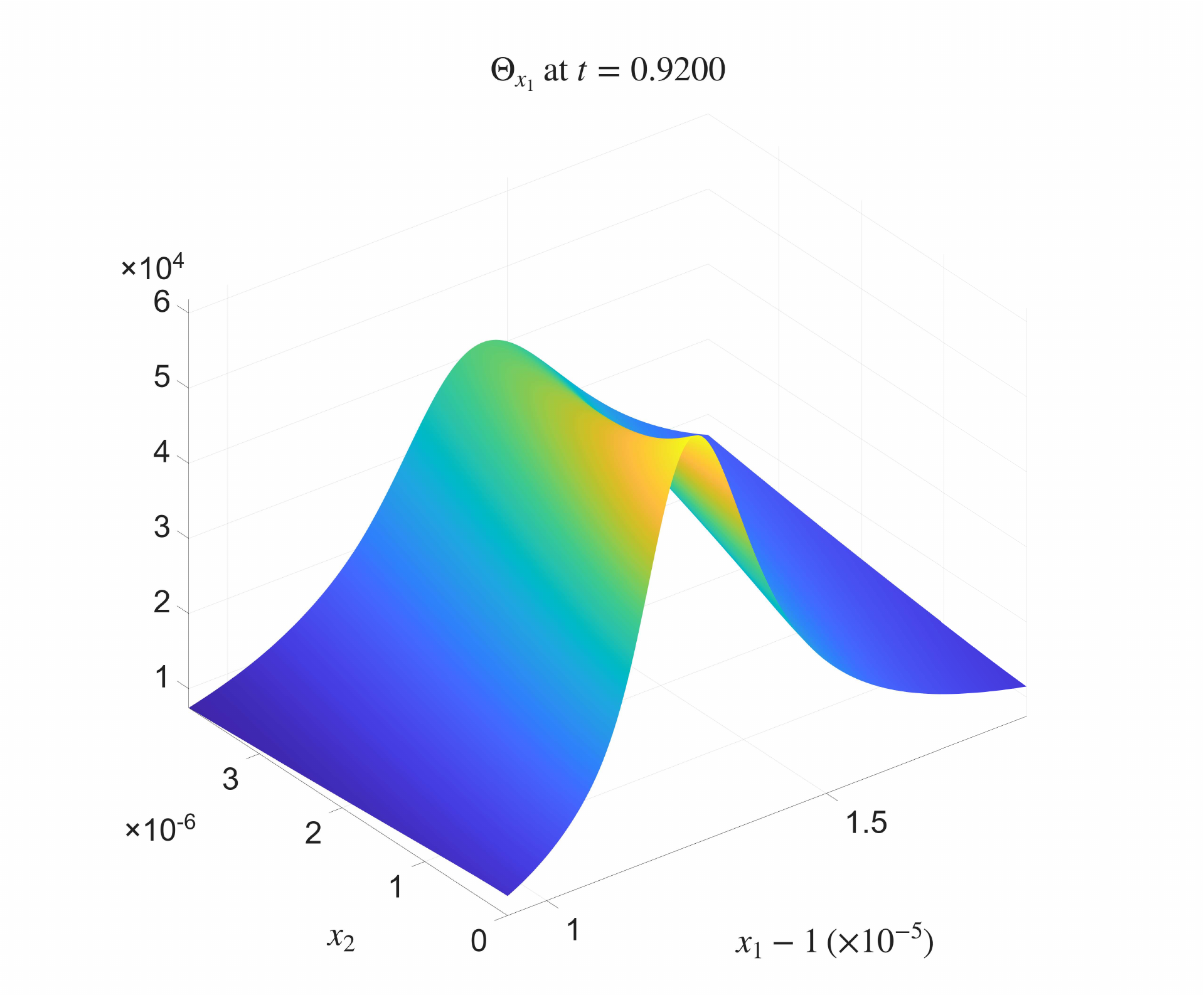}
    \caption[Inner mesh profiles of $\Omega$ and $\Theta_{x_1}$ in Scenario 1.]{Inner mesh profiles of $\Omega$ (top row) and $\Theta_{x_1}$ (bottom row) in Scenario 1. After proper rescaling, the inner profiles eventually stabilize at regular profiles as $t$ approaches the blowup time.}
    \label{fig:BS_scenario1_surf_convergence_compare}
\end{figure}

Motivated by the numerical evidence of the 1D HL model, we are also interested in the long-time behavior of the weak solution to \eqref{eqt:dynamic_rescaling_of_BS} beyond the Stage 1 blowup. Again, to capture the post-blow-up dynamics, we perform simulations on a fixed computational mesh that is locally refined near $(1,0)$. During the computations, we employ suitable numerical regularization. This approach can be interpreted as approximating the weak solution to \eqref{eqt:dynamic_rescaling_of_BS} via the method of vanishing viscosity. We remark that, compared to the 1D HL model, the numerical simulations of the 2D Boussinesq equations are substantially more challenging for two main reasons. First, the computational cost is significantly higher in the 2D case, which makes it exceedingly difficult to perform long-time simulations. Second, unlike the HL model, we are currently unable to obtain a steady singular profile of \eqref{eqt:dynamic_rescaling_of_BS} with explicit expressions that acts as a reference solution to validate our numerical scheme.

Despite these challenges, we can still obtain numerical evidence suggesting that the weak solution of \eqref{eqt:dynamic_rescaling_of_BS} exhibits a strong tendency to approach singular profiles analogous to the 1D case. Figure \ref{fig:Bs_sc1_time_evolution} tracks the evolution of the scaling parameters $c_l/c_{\omega}$ and $c_{\theta}=c_l+2c_{\om}$. Interestingly, the ratio $c_l/c_{\omega}$ eventually stabilizes at a value remarkably close to $-2$ ($c_l/c_{\omega} \approx -2.000048$), which is consistent with the 1D case. Figures \ref{fig:Bs_sc1_mesh_evo_om}, \ref{fig:Bs_sc1_mesh_evo_th}, \ref{fig:Bs_sc1_con_evo_om}, and \ref{fig:Bs_sc1_con_evo_th} illustrate the evolution of the spatial profiles of $\Omega$ and $\Theta$ from different perspectives. As suggested in Figures \ref{fig:Bs_sc1_mesh_evo_om} and \ref{fig:Bs_sc1_con_evo_om}, $\Omega$ exhibits a distinct tendency to settle to a limiting profile that is singular at $(1,0)$ as time progresses. Meanwhile, Figures \ref{fig:Bs_sc1_mesh_evo_th} and \ref{fig:Bs_sc1_con_evo_th} reveal that $\Theta$ develops a sharp transition layer along a right-leaning ray emanating from $(1,0)$. Notably, when restricted to the boundary $x_2=0$, this transition manifests as a strict Heaviside-type jump across $(1,0)$, as illustrated in Figure \ref{fig:Bs_sc1_bd_evo}. Based on these observations, we conjecture that the weak solution to \eqref{eqt:dynamic_rescaling_of_BS} eventually converges to a steady state $(\bar{\Omega}, \bar{\Theta}, \bar{c}_{l}, \bar{c}_{\omega})$, where $\bar{\Omega}$ is singular at $(1,0)$ and $\bar{c}_{l}/\bar{c}_{\omega} = -2$. This implies a local $L^p$ blowup of $\om$ around the origin for some $p>0$. We refer to this subsequent phenomenon as the Stage 2 blowup.

\begin{figure}[!htbp]
\centering
        \includegraphics[width=0.4\textwidth]{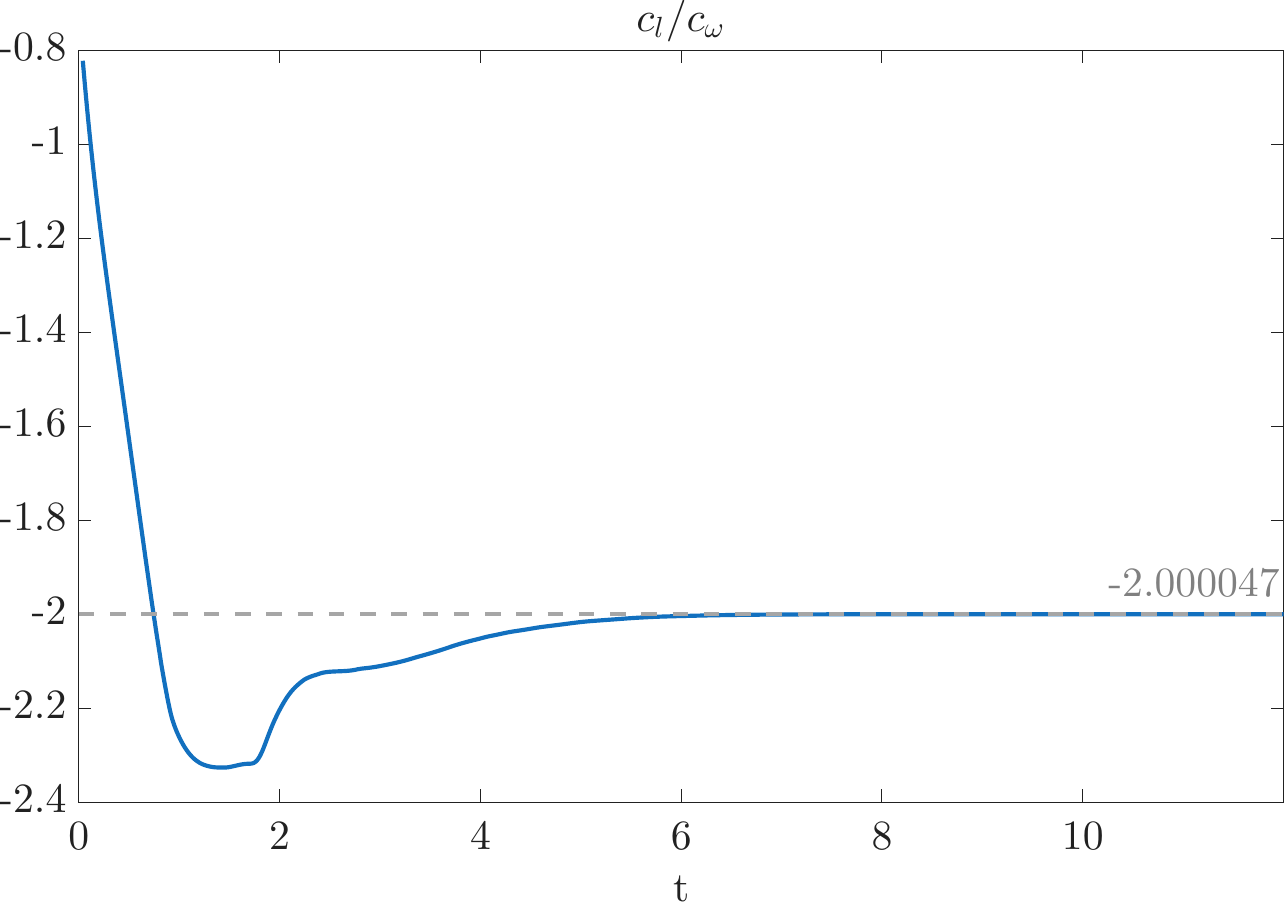}
        \includegraphics[width=0.4\textwidth]{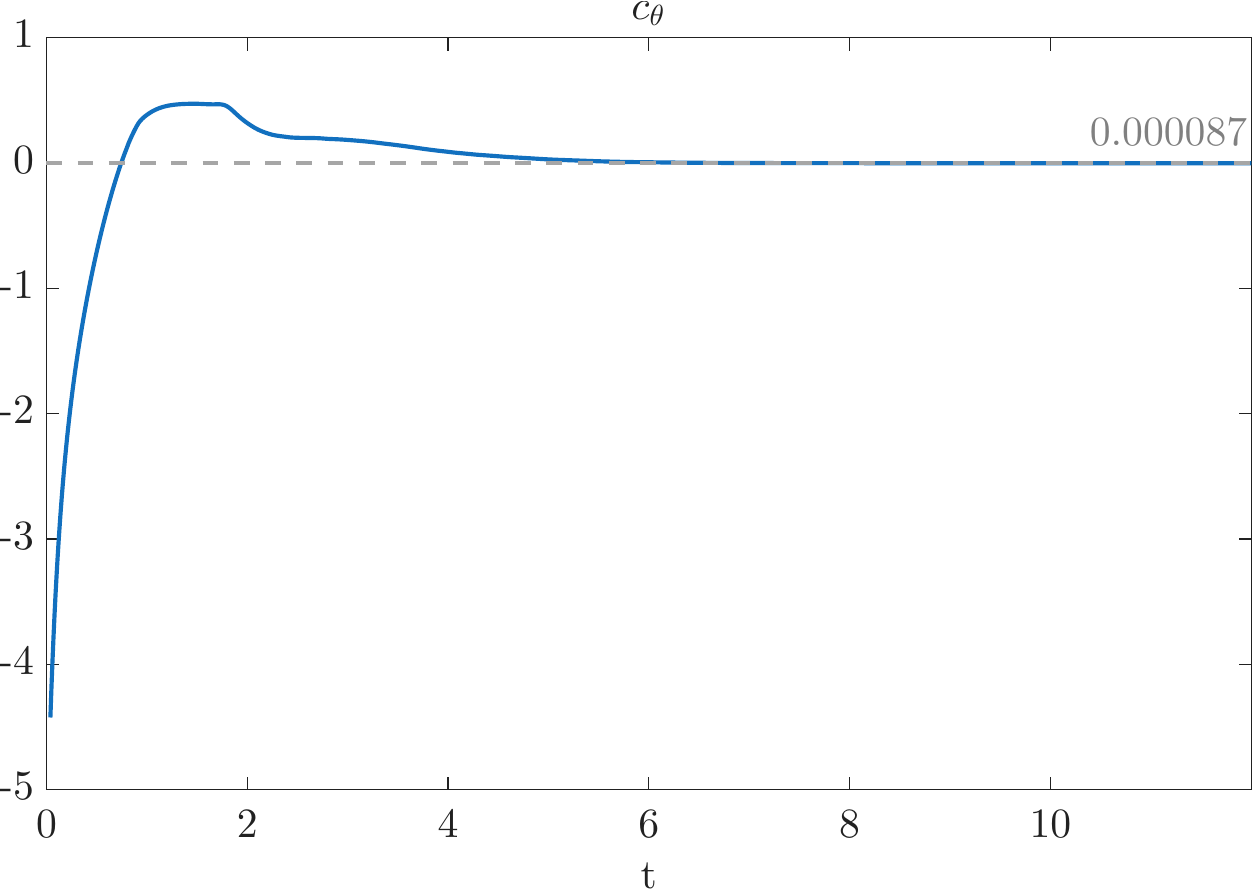}
    \caption[Evolution of $c_l/c_{\om}$ (left figure) and $c_{\theta}=c_l+2c_{\om} $ (Scenario 1, computed on a fixed mesh).]{Evolution of $c_l/c_{\om}$ (left figure) and $c_{\theta} = c_l+2c_{\om}$ (right figure) computed on a fixed mesh in Scenario 1. The ratio $c_l/c_{\omega}$ eventually stabilizes near $-2$, which is consistent with the 1D case.}
    \label{fig:Bs_sc1_time_evolution}
\end{figure}

\begin{figure}[!htbp]
\centering
        \includegraphics[width=0.4\textwidth]{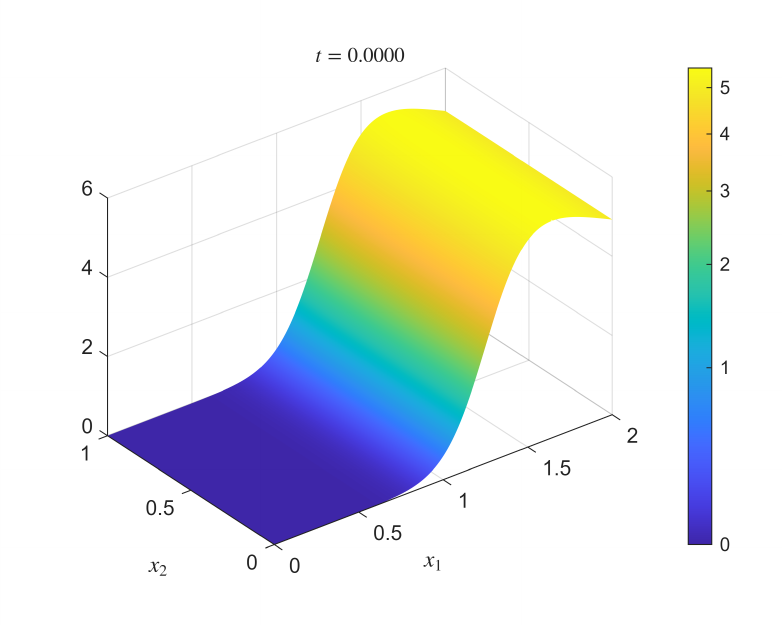}
        \includegraphics[width=0.4\textwidth]{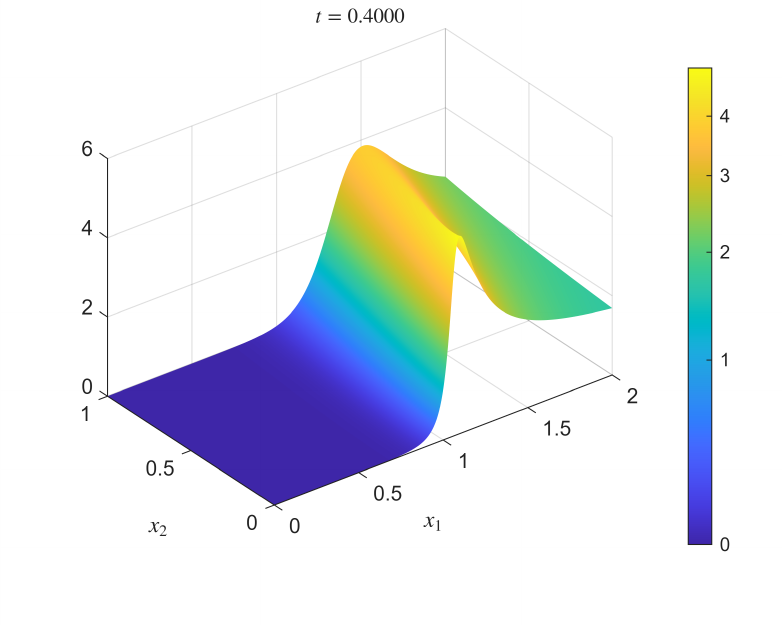}
        \includegraphics[width=0.4\textwidth]{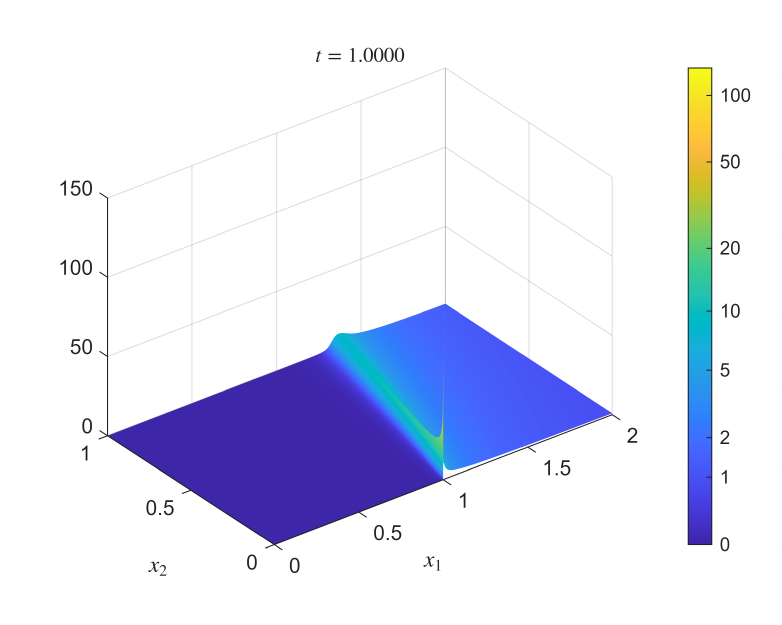}
        \includegraphics[width=0.4\textwidth]{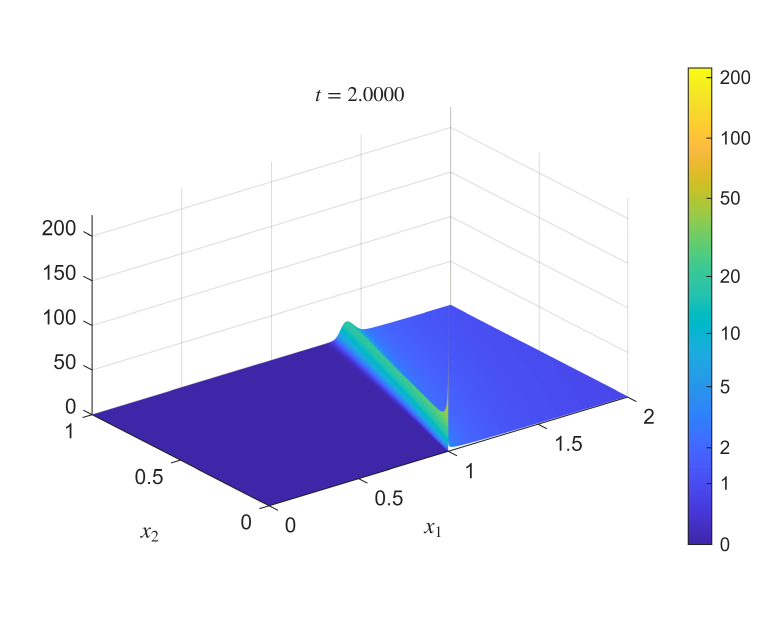}
\caption[Evolution of 3D spatial profiles of $\Omega$ computed on a fixed mesh in Scenario 1.]{Evolution of 3D spatial profiles of $\Omega$ computed on a fixed mesh in Scenario 1. The visualizations highlight the emergence of sharp gradients near the singularity, alongside a progressive concentration of the solution's support.}
    \label{fig:Bs_sc1_mesh_evo_om}
\end{figure}

\begin{figure}[!htbp]
\centering
        \includegraphics[width=0.4\textwidth]{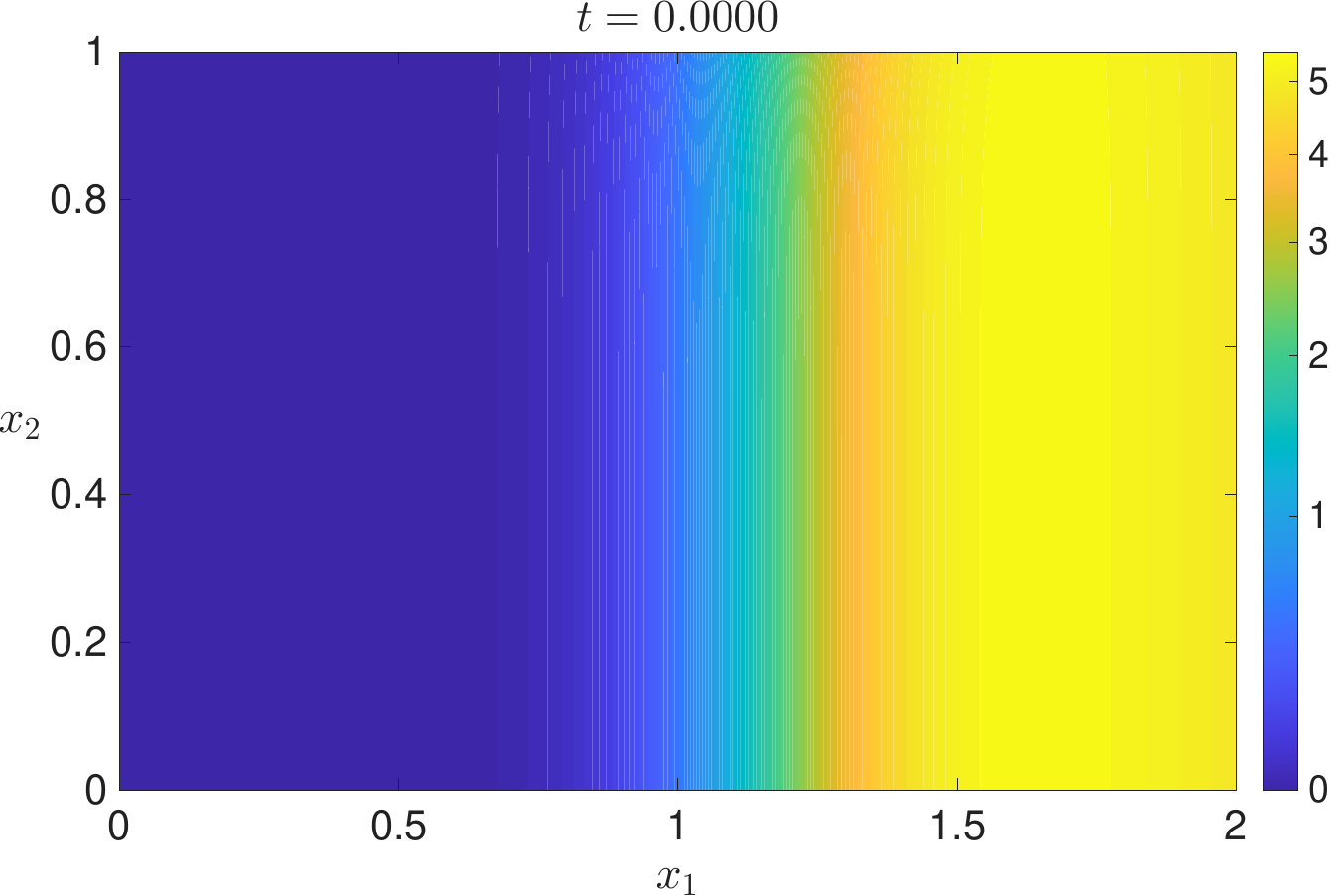}\quad\
        \includegraphics[width=0.4\textwidth]{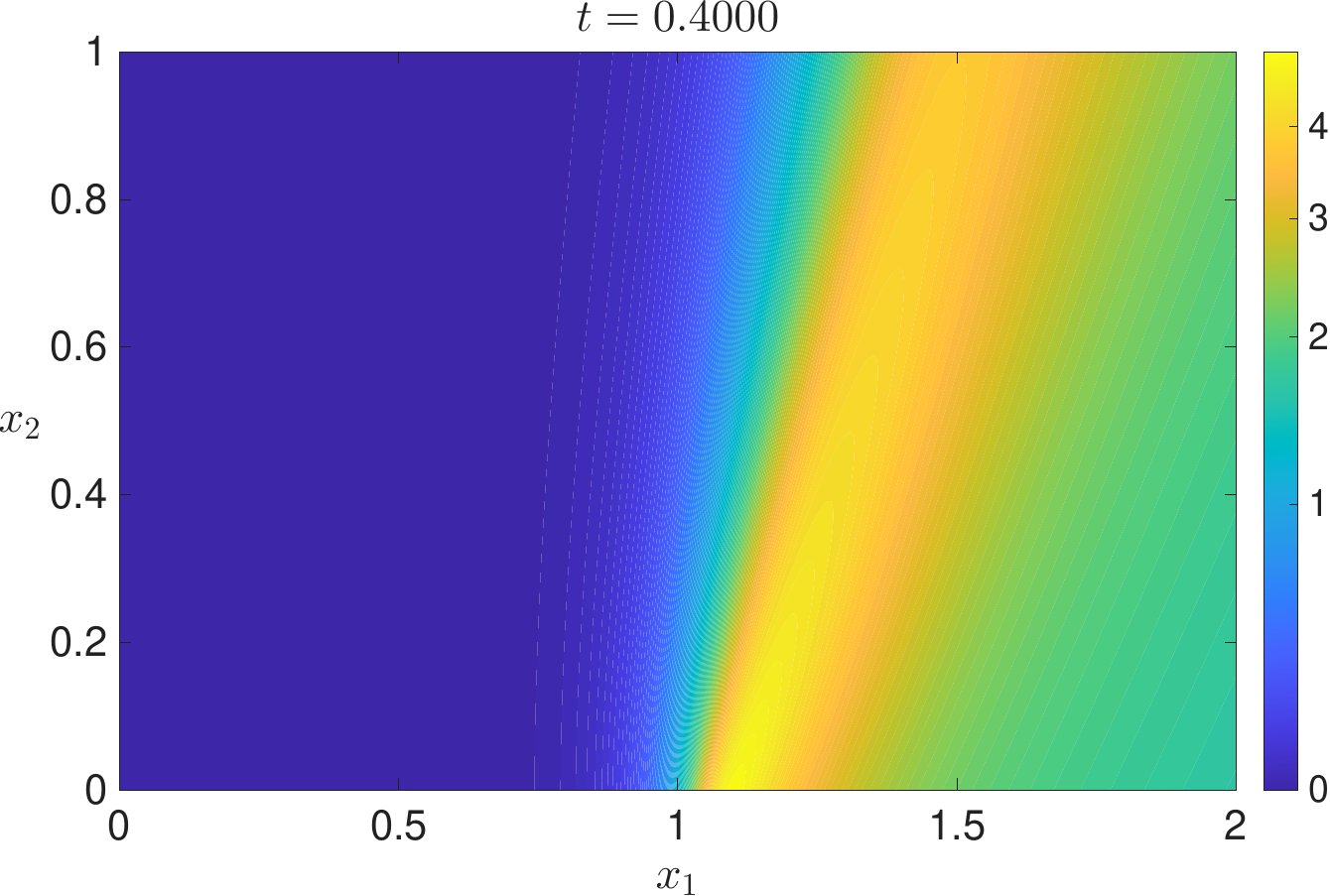}\qquad\\
        \quad\includegraphics[width=0.42\textwidth]{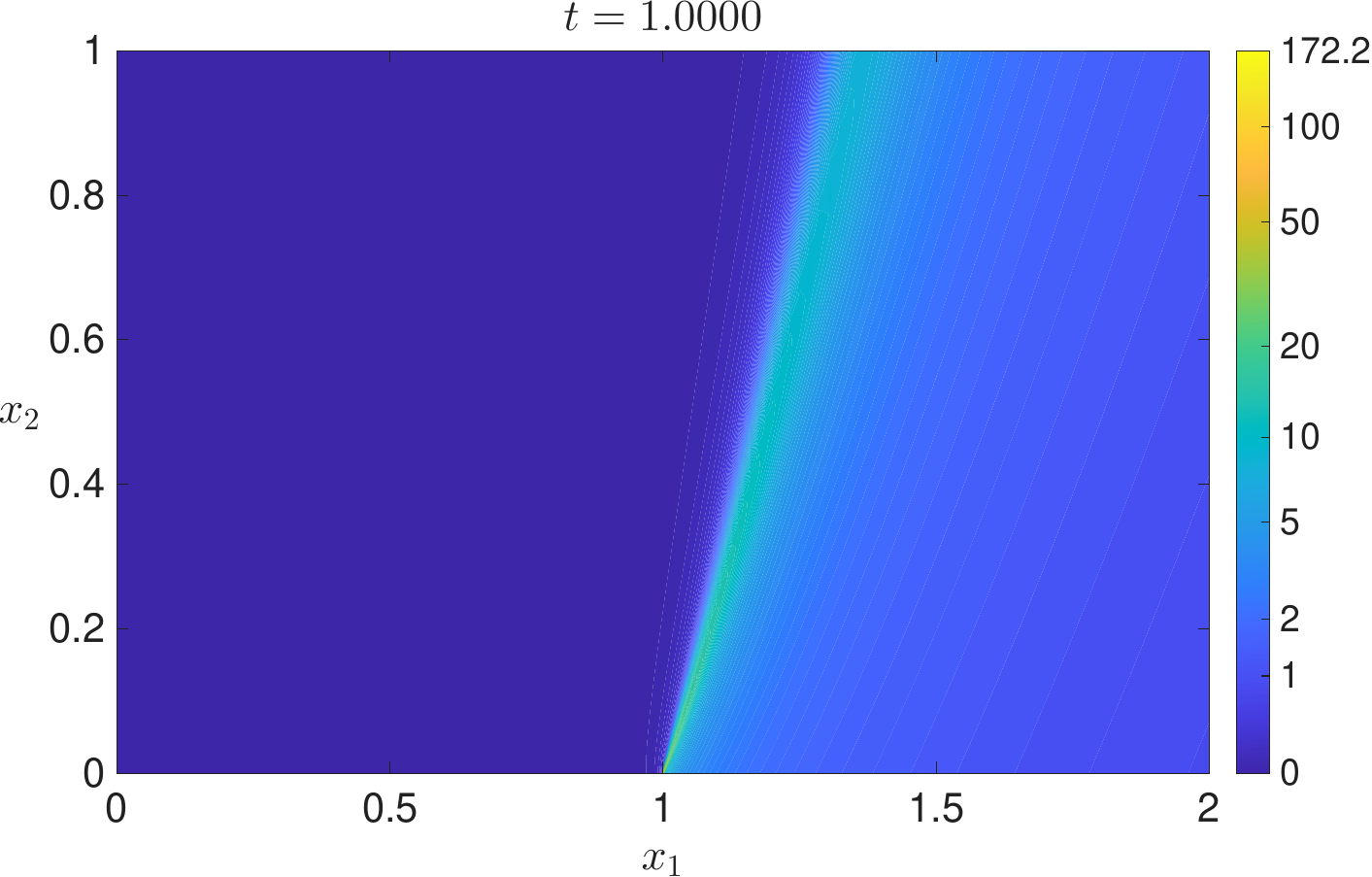}
        \includegraphics[width=0.415\textwidth]{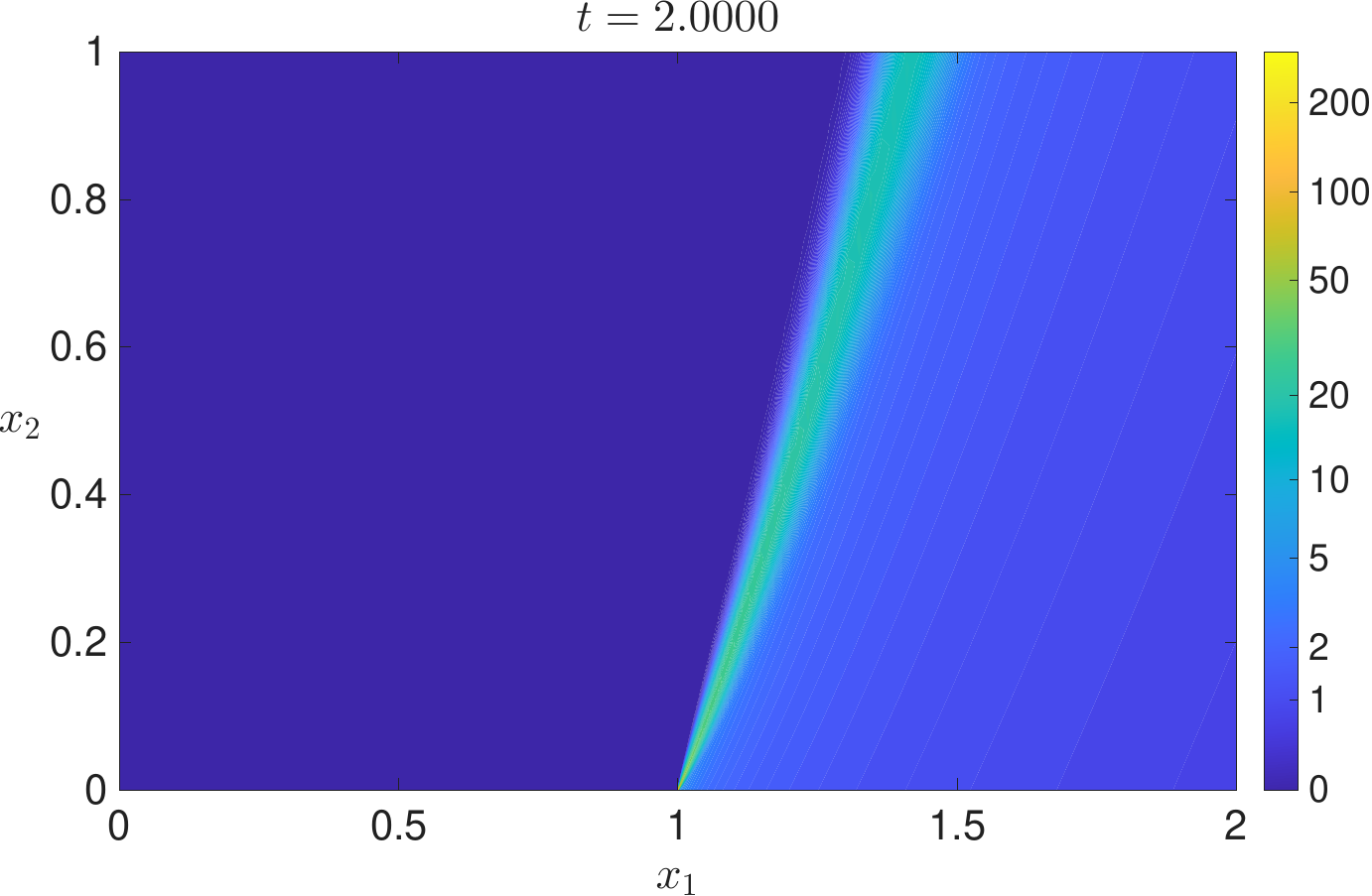}
    \caption[Evolution of contour plots of $\Omega$ computed on a fixed mesh in Scenario 1.]{Evolution of contour plots of $\Omega$ computed on a fixed mesh in Scenario 1. The support of $\Omega$ is observed to concentrate into a right-leaning bounded triangular region, with the strongest concentration near $(1,0)$.}
    \label{fig:Bs_sc1_con_evo_om}
\end{figure}

\begin{figure}[!htbp]
\centering
        \includegraphics[width=0.4\textwidth]{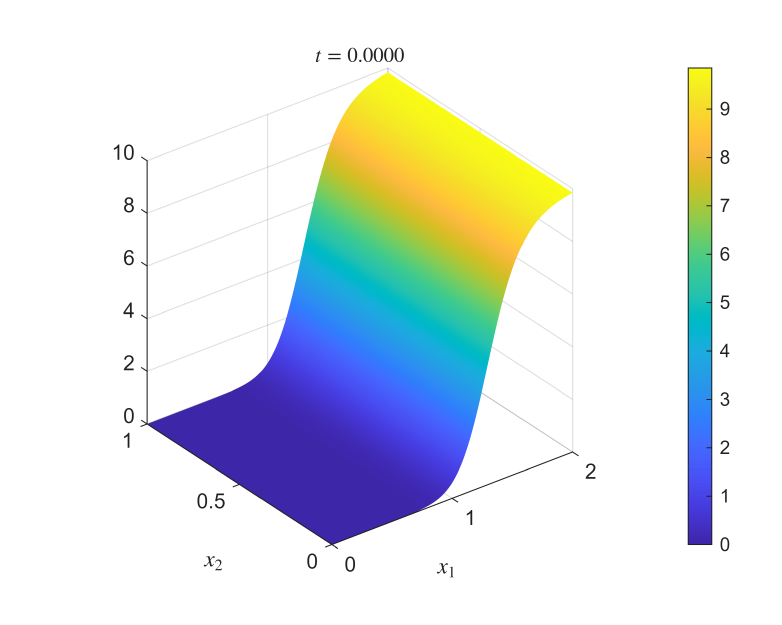}   
        \includegraphics[width=0.4\textwidth]{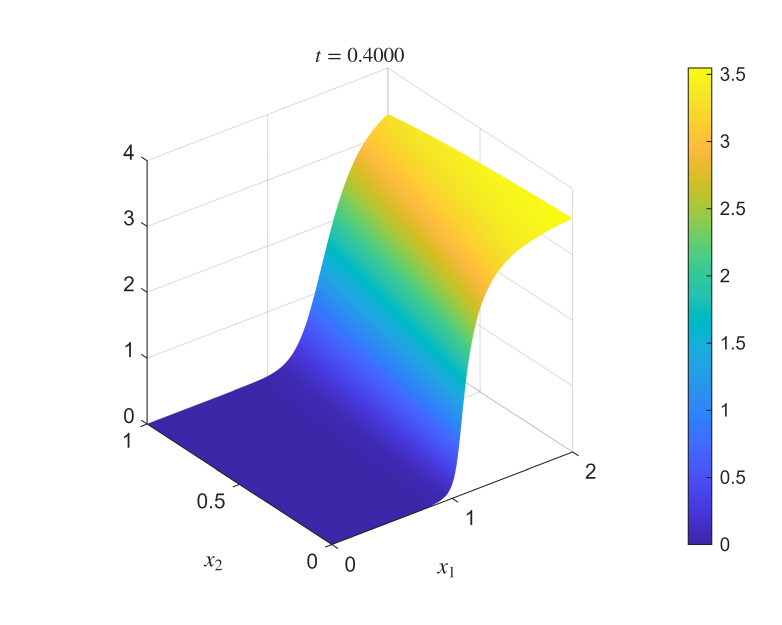}    
        \includegraphics[width=0.4\textwidth]{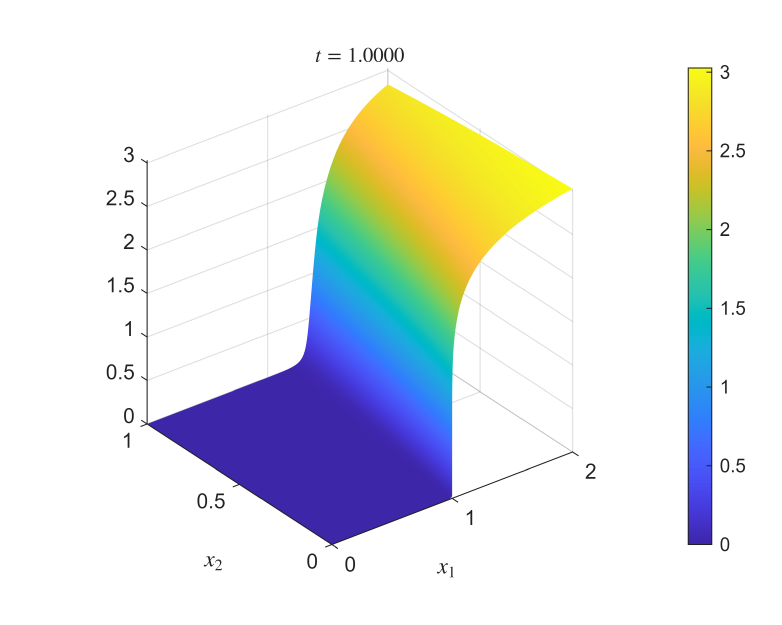}    
        \includegraphics[width=0.4\textwidth]{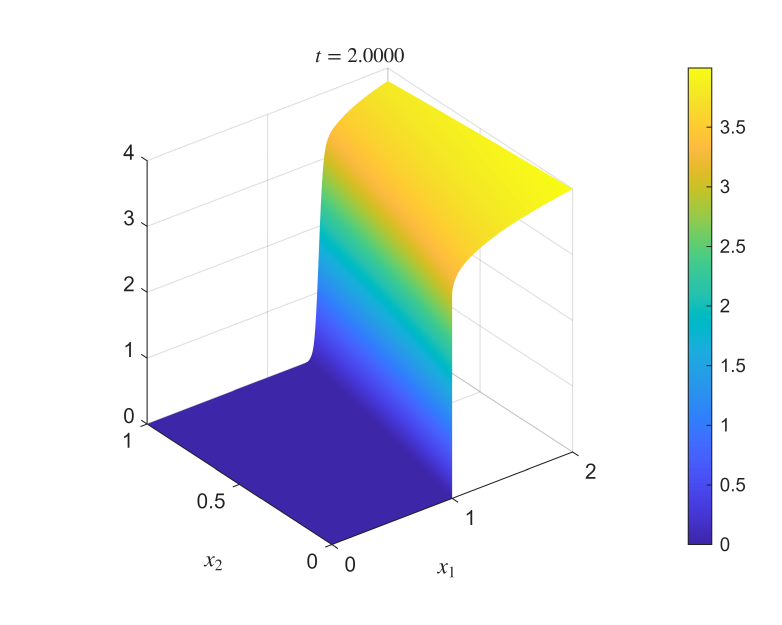}
    \caption[Evolution of 3D spatial profiles of $\Theta$ computed on a fixed mesh in Scenario 1. ]{Evolution of 3D spatial profiles of $\Theta$ computed on a fixed mesh in Scenario 1. The steep transition layer is well-resolved throughout the simulation.}
    \label{fig:Bs_sc1_mesh_evo_th}
\end{figure}

\begin{figure}[!htbp]
\centering
        \includegraphics[width=0.4\textwidth]{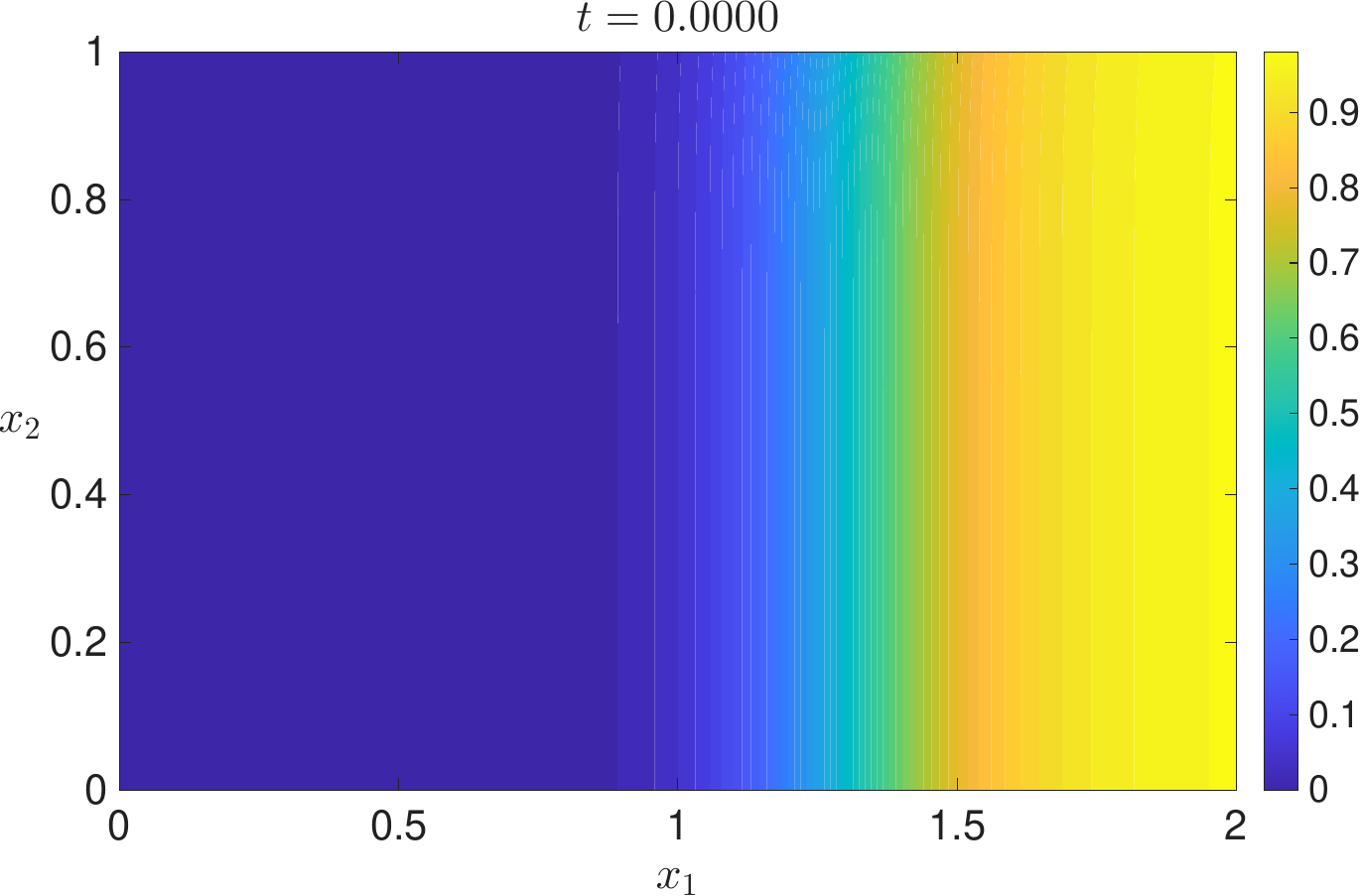}
        \includegraphics[width=0.4\textwidth]{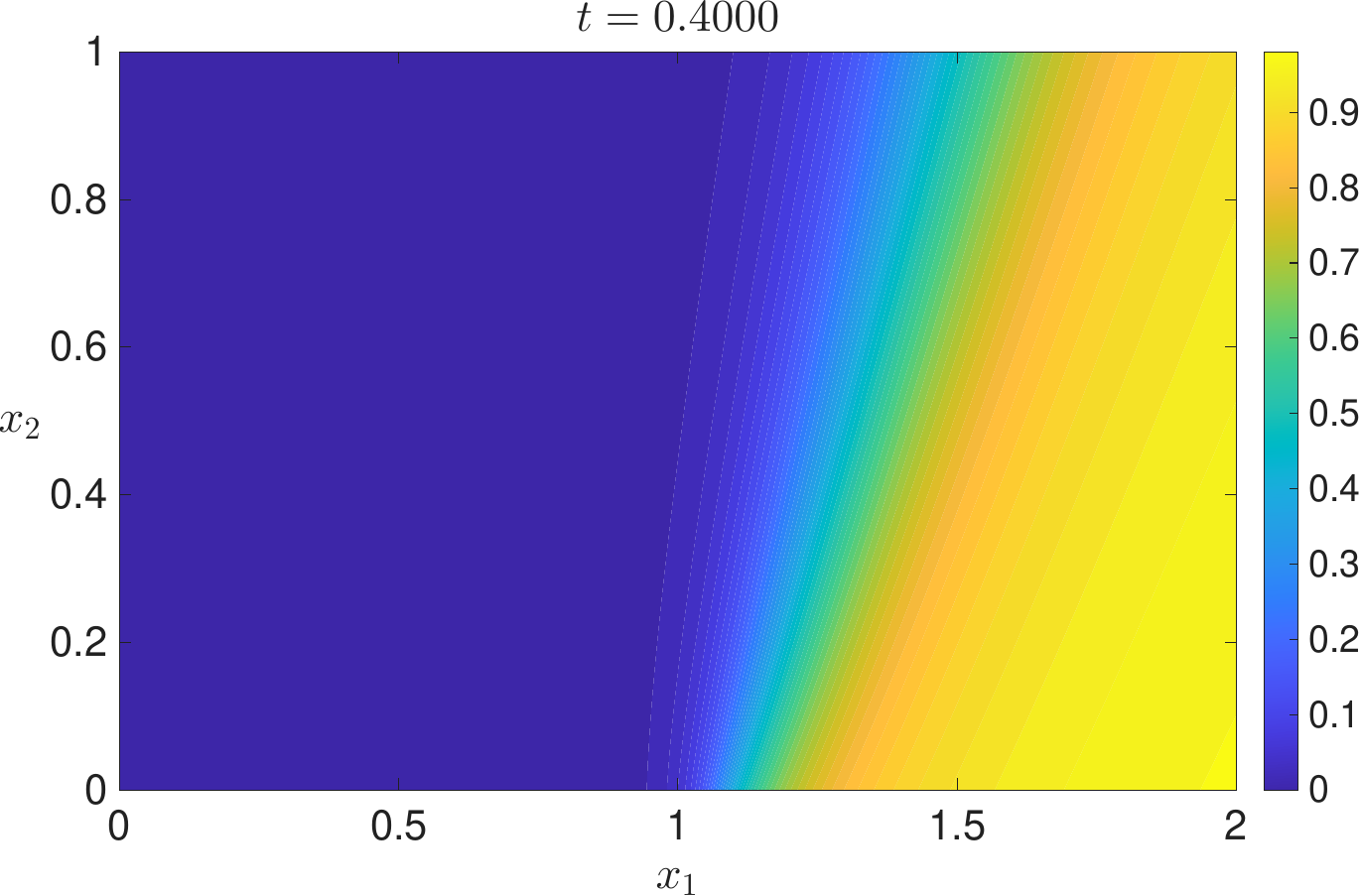}
        \includegraphics[width=0.4\textwidth]{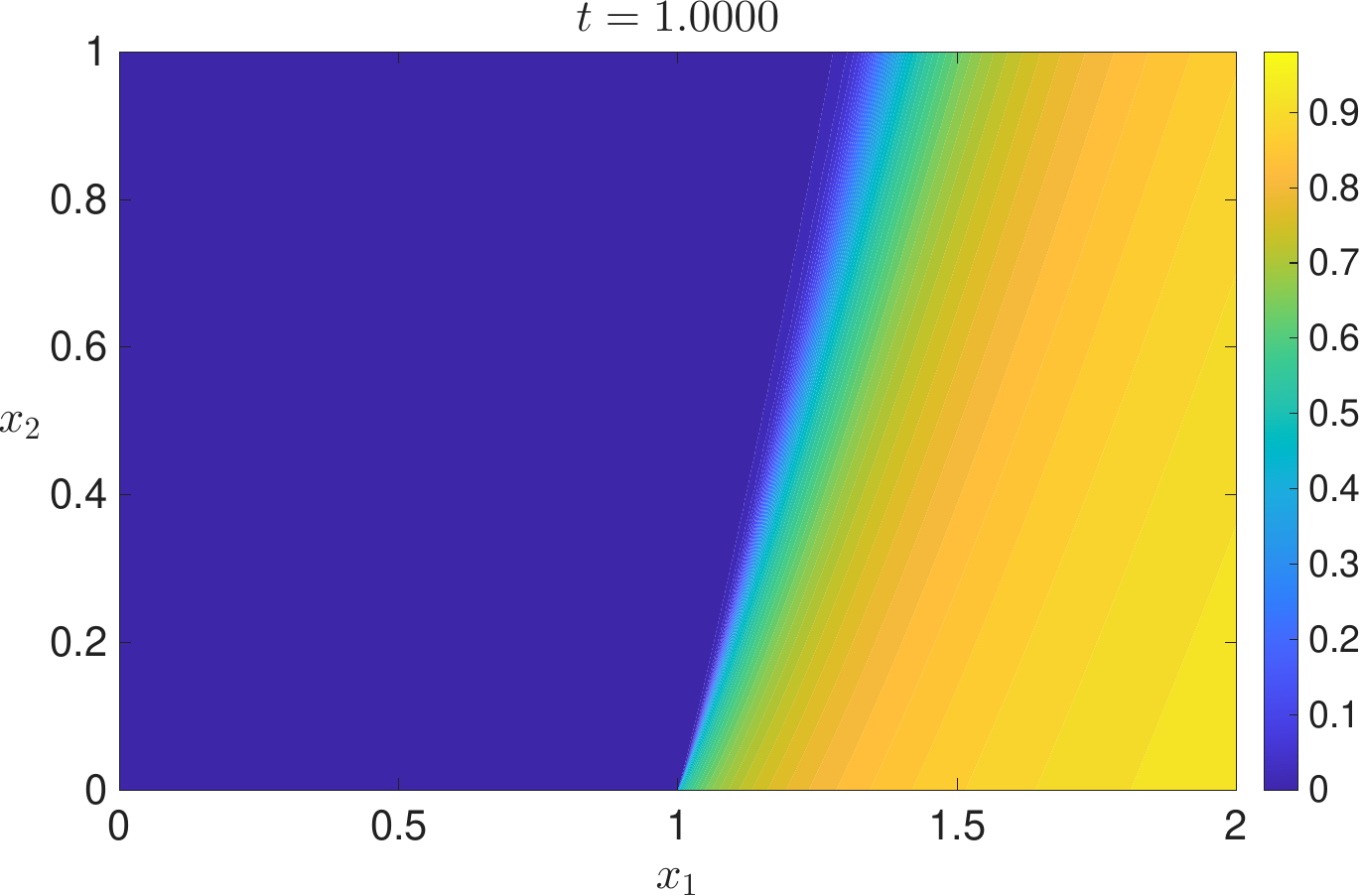}
        \includegraphics[width=0.4\textwidth]{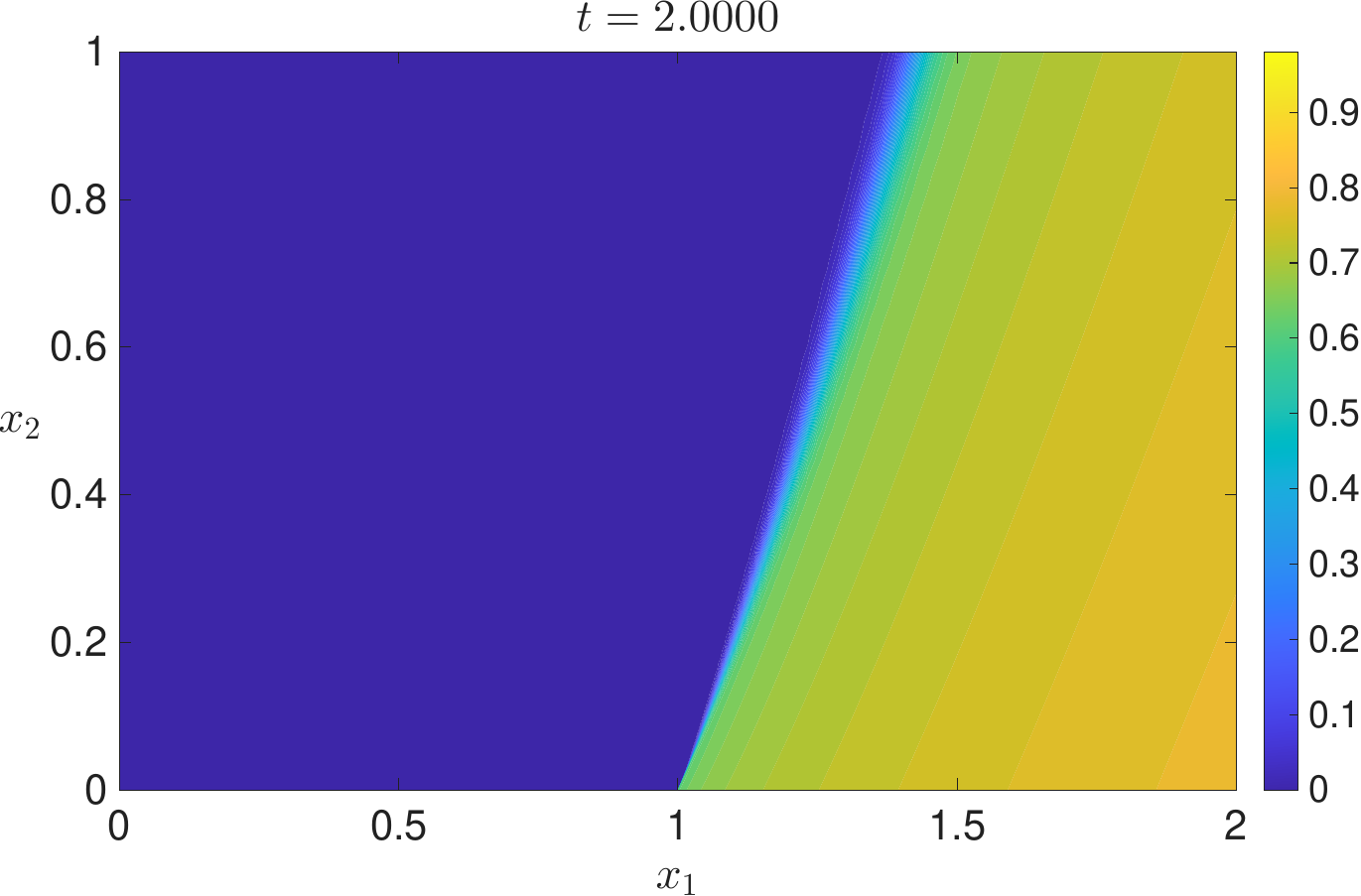}
    \caption[Evolution of contour plots of $\Theta$ computed on a fixed mesh in Scenario 1.]{Evolution of contour plots of normalized $\Theta$ computed on a fixed mesh in Scenario 1.
    As shown, the transition layer becomes increasingly steep and approaches a
    shock-like structure.}
    \label{fig:Bs_sc1_con_evo_th}
\end{figure}

\begin{figure}[!htbp]
\centering
    \begin{subfigure}[b]{0.45\textwidth}
        \includegraphics[width=\textwidth]{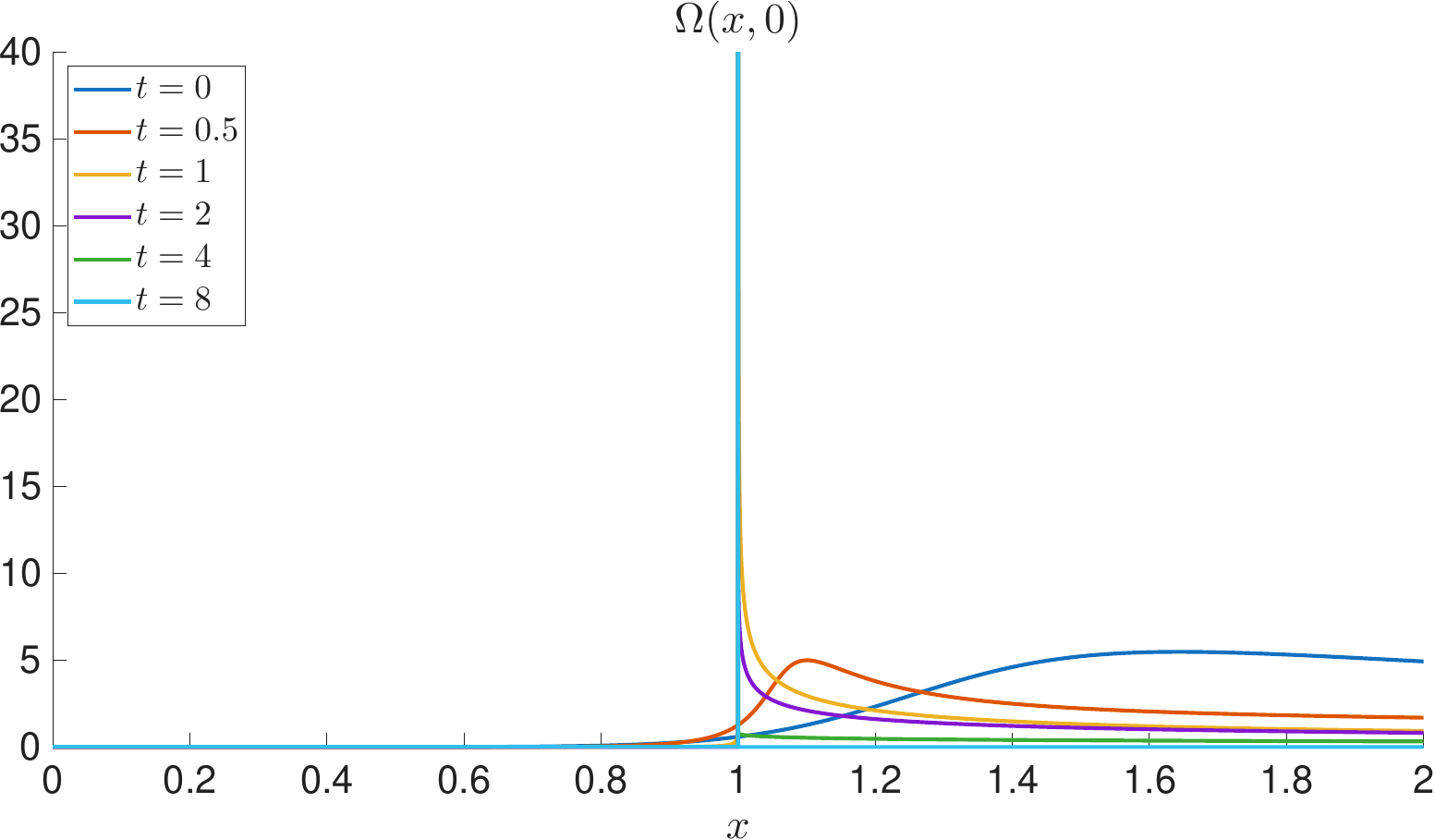}
        
    \end{subfigure}
    \begin{subfigure}[b]{0.45\textwidth}
        \includegraphics[width=\textwidth]{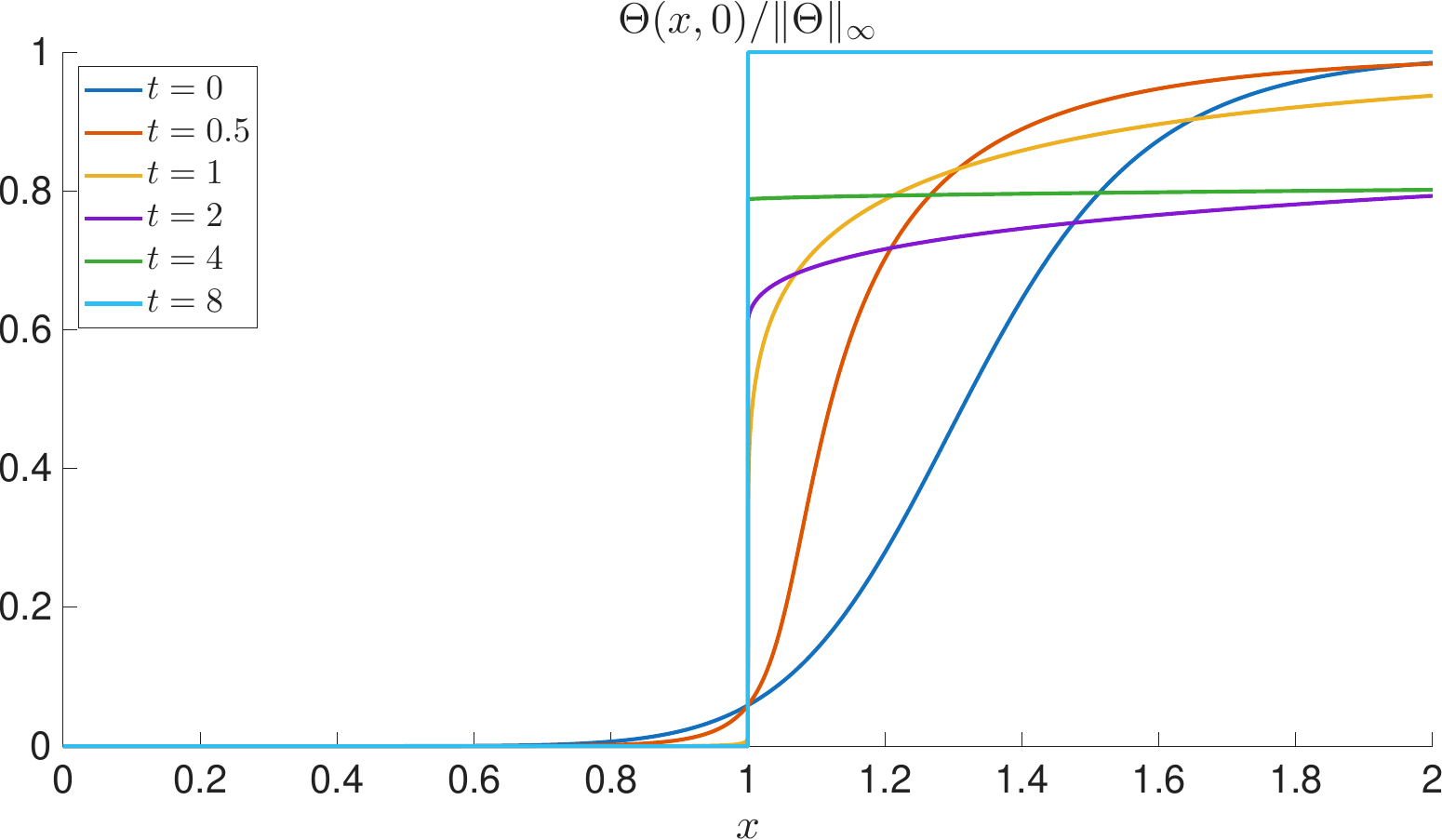}
        
    \end{subfigure}
    \caption[Evolution of cross-sectional profiles of the solutions along $x_2=0$ computed on a fixed mesh in Scenario 1.]{Evolution of cross-sectional profiles of the solutions along $x_2=0$ computed on a fixed mesh in Scenario 1. As illustrated, on the boundary, $\Omega$ exhibits a trend of evolving into a limiting profile that is singular at $x_1=1$, while $\Theta$ approaches a Heaviside-type jump.}  
     \label{fig:Bs_sc1_bd_evo}
\end{figure}

Furthermore, as discussed in Remark \ref{rem:BS_critical_Lp_exponent}, the numerically observed fact that $c_l/c_{\omega}$ eventually stabilizes around $-2$ suggests that the critical $L^p$ blowup exponent for $\om$ is $4$. To visualize this, Figure \ref{fig:Bs_sc1_Lp} plots the local $L^p$ norms of $\om$ as functions of the rescaled time $t$ for $p=3.9,4,4.1$. We remark that these norms are reconstructed from $\Omega$ through the change of variables \eqref{eqt:dynamic_rescaling_of_BS_change_of_variables1} and \eqref{eqt:dynamic_rescaling_of_BS_change_of_variables2}. To obtain the local $L^p$ norm of $\om$ on a fixed domain, one needs to compute the local $L^p$ norm of $\Omega$ on a domain that expands exponentially in the rescaled time. Since $\Omega$ is computed numerically on the finite domain $[0,R_1]\times[0,R_2]$ with $R_1=R_2\approx10^6$, the integration domain cannot expand indefinitely. Once the integration domain reaches the computational boundary, it remains fixed in the rescaled coordinates. Therefore, Figure \ref{fig:Bs_sc1_Lp} actually plots the local $L^p$ norms of $\om$ on a time-dependent domain $E_t$, which is initially fixed and subsequently shrinks. Specifically, $E_t$ is given by 
\[E_t=[0,\min\{1,R_1/C_l(t)\}]\times[0,\min\{1,R_2/C_l(t)\}].\]
Recall that $t$ now denotes the rescaled time $\tau$ in \eqref{eqt:dynamic_rescaling_of_BS_change_of_variables2}.

As shown in Figure \ref{fig:Bs_sc1_Lp}, the computed local $L^{4.1}$ norm of $\om$ grows noticeably, whereas the local $L^{3.9}$ and $L^4$ norms remain bounded over the simulated time interval. These observations provide further numerical evidence that the critical $L^p$ blowup exponent for $\om$ is $4$. We also remark that the observed boundedness of the $L^4(E_t)$ norm does not rule out local $L^4$ blowup of $\om$ on a fixed physical region. First, once the rescaled integration domain of $\Omega$ stops expanding, the computation misses the additional far-field contribution needed to recover the norm on a fixed physical region. Second, even if the computational domain for $\Omega$ were allowed to be unbounded, the asymptotic $L^4$ unintegrability of $\Omega$ at infinity only leads to a logarithm-type growth of the local $L^4$ norm of $\om$, which can be difficult to identify numerically. Finally, although the development of a singularity of $\Omega$ near $(1,0)$ could also contribute to a local integral blowup, it may be locally $L^4$ integrable or smoothed out by numerical regularization.

\begin{figure}[!htbp]
\centering
        \includegraphics[width=0.6\textwidth]{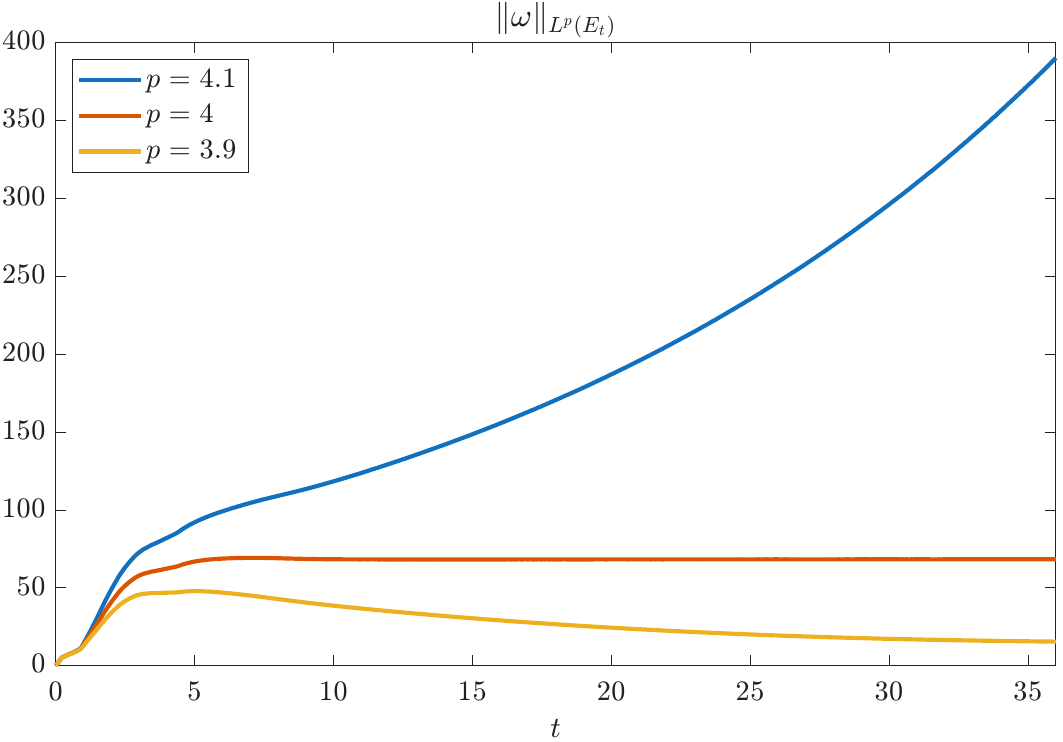}
    \caption[Evolution of $L^p$ norm]{Plots of $\|\omega\|_{L^{p}(E_t)}$ as functions of the rescaled time $t$ for $p =4.1,4,3.9$ in Scenario 1. As shown, the computed local $L^{4.1}$ norm of $\om$ grows noticeably, whereas the local $L^{3.9}$ and $L^4$ norms remain bounded over the simulated time interval. }
    \label{fig:Bs_sc1_Lp}
\end{figure}

Finally, let us conclude this subsection by summarizing the potential two-stage self-similar blowup scenarios of \eqref{eqt:Boussinesq_v2} as follows:
\begin{itemize}
    \item \textbf{Stage 1:} Starting from degenerate initial data, the rescaled variables $\Omega$ and $\Theta_{X_1}$ develop finite-time self-similar blowups with regular profiles at $(X_1,X_2)=(1,0)$. These profiles do not possess any symmetry with respect to the blowup point. Correspondingly, in the physical space, $\omega$ and $\theta_{x_1}$ develop finite-time self-similar singularities at a boundary point away from the origin with the same self-similar profiles. This corresponds to local $L^\infty$ blowups of $\omega$ and $\theta_{x_1}$ at the first blowup time.
\item \textbf{Stage 2:} After the Stage 1 blowup, $\Omega$ and $\Theta_{X_1}$ continue in the weak sense and eventually converge to singular profiles. Correspondingly, $\omega$ and $\theta_{x_1}$ develop finite-time self-similar blowups at the origin with singular profiles. Moreover, as discussed before, our numerical results suggest that $\omega$ develops a local $L^p$ blowup around the origin for $p\geq 4$.
\end{itemize}

\subsection{Scenario 2: novel self-similar finite-time blowups with positive regular profiles.}

Motivated by the Stage 1 blowup observed in Scenario 1, in this subsection we investigate this phenomenon using a modified dynamic rescaling formulation analogous to the 1D case. More specifically, we introduce a time-dependent spatial shift $r(\tau)$ in the following change of variables:
\begin{align*}
    &\omega(\mtx{x},t)=C_{\om}(\tau)^{-1}\Omega\big(C_l(\tau)(\mtx{x}-r(\tau)\mtx{e}_1),\tau(t)\big),\\ &\theta(\mtx{x},t)=C_{\theta}(\tau)^{-1}\Theta\big(C_l(\tau)(\mtx{x}-r(\tau)\mtx{e}_1),\tau(t)\big),\\ &\mtx{u}(\mtx{x},t)=C_{\om}(\tau)^{-1}C_{l}(\tau)^{-1}\mtx{U}\big(C_l(\tau)(\mtx{x}-r(\tau)\mtx{e}_1),\tau(t)\big)
\end{align*}
with
\begin{align*}
    C_{\om}(\tau)&=\exp\left(\int_0^\tau c_{\om}(s)\idiff s\right) ,\quad  C_{\theta}(\tau)=\exp\left(\int_0^\tau c_{\theta}(s) \idiff s\right),\quad  C_{l}(\tau)=\exp\left(\int_0^\tau c_l(s)\idiff s\right),\\ r(\tau)&=-\int_{0}^{\tau}c_r(s)C_l(s)^{-1}\idiff s, \quad  t(\tau)=\int_0^{\tau}C_{\om}(s)\idiff s, \quad c_\theta=c_l+2c_{\om},
\end{align*}
where $\mtx{e}_1=(1,0)$. Then \eqref{eqt:Boussinesq_v2} is reformulated as
\begin{equation}\label{eqt:dynamic_BS_scenario2_v1}
\begin{aligned}
&\Omega_{\tau} + (\mtx{U}+c_l \mtx{X} + c_r \mtx{e}_1)\cdot\nabla\Omega
= c_{\omega}\Omega + \Theta_{X_1},\\
&\Theta_{\tau} + (\mtx{U}+c_l \mtx{X} + c_r \mtx{e}_1)\cdot\nabla\Theta
= (c_l+2c_{\omega})\Theta,\\
&\mtx{U}=\nabla^{\perp}(-\Delta)^{-1}\Omega,
\end{aligned}
\end{equation}
where $\mtx{X} = C_l(\tau)\big(\mtx{x}-r(\tau)\mtx{e}_1\big)$. Similar to the 1D case, it is not hard to see that the asymptotic stability of a steady state of \eqref{eqt:dynamic_BS_scenario2_v1} implies the asymptotically self-similar blowup of the original equations \eqref{eqt:Boussinesq_v2} at the physical coordinate $\mtx{x}=(r(+\infty),0)$. For notational simplicity, we still denote the rescaled variables $(\mtx{X},\tau)$ by $(\mtx{x},t)$ and rewrite \eqref{eqt:dynamic_BS_scenario2_v1} as 
\begin{equation}\label{eqt:dynamic_BS_scenario2}
\begin{aligned}
&\Omega_{t} + (\mtx{U}+c_l \mtx{x} + c_r \mtx{e}_1)\cdot\nabla\Omega
= c_{\omega}\Omega + \Theta_{x_1},\\
&\Theta_{t} + (\mtx{U}+c_l \mtx{x} + c_r \mtx{e}_1)\cdot\nabla\Theta
= (c_l+2c_{\omega})\Theta,\\
&\mtx{U}=\nabla^{\perp}(-\Delta)^{-1}\Omega.
\end{aligned}
\end{equation}

To conduct numerical simulations of \eqref{eqt:dynamic_BS_scenario2}, we need to determine the values of $c_l$, $c_{\omega}$ and $c_r$ by imposing 3 normalization conditions. Specifically, we choose $c_l$, $c_{\omega}$ and $c_r$ to enforce that 
\[
\partial_t\Omega(\mtx{0},t)=0,\quad
\partial_t\Omega_{x_1}(\mtx{0},t)=0,\quad
\partial_t\Omega_{x_1x_1}(\mtx{0},t)=0.
\]
Accordingly, $c_l$, $c_{\omega}$, and $c_r$ are obtained by solving the following linear system:
\begin{equation}\label{eqt:normalization__constant_scenario2}
\begin{cases}
\Omega_{x_1}(\mtx{0})\big(c_r+U_1(\mtx{0})\big) - \Omega(\mtx{0})c_{\omega} = \Theta_{x_1}(\mtx{0}),\\[3pt]
\begin{aligned}
&\Omega_{x_1x_1}(\mtx{0})\big(c_r+U_1(\mtx{0})\big) - \Omega_{x_1}(\mtx{0})c_{\omega} + \Omega_{x_1}(\mtx{0})c_l\\
&\qquad= \Theta_{x_1x_1}(\mtx{0}) - U_{1,x_1}(\mtx{0})\Omega_{x_1}(\mtx{0}),
\end{aligned}\\[3pt]
\begin{aligned}
&\Omega_{x_1x_1x_1}(\mtx{0})\big(c_r+U_1(\mtx{0})\big) - \Omega_{x_1x_1}(\mtx{0})c_{\omega}
+ 2\Omega_{x_1x_1}(\mtx{0})c_l\\
&\qquad= \Theta_{x_1x_1x_1}(\mtx{0}) - \big(U_{1,x_1x_1}(\mtx{0})\Omega_{x_1}(\mtx{0})
+ 2U_{1,x_1}(\mtx{0})\Omega_{x_1x_1}(\mtx{0})\big).
\end{aligned}
\end{cases}
\end{equation}

Unlike the 1D case, we cannot obtain a self-contained system for $\Omega$ and $\Theta_{x_1}$ by taking the $x_1$-derivative of the second equation in \eqref{eqt:dynamic_BS_scenario2}. We therefore numerically solve \eqref{eqt:dynamic_BS_scenario2} directly by imposing the above normalization conditions. Note that if $(\Omega(\mtx{x},t),\Theta(\mtx{x},t))$ solves \eqref{eqt:dynamic_BS_scenario2} with $(c_l,c_{\om},c_r)$ given by \eqref{eqt:normalization__constant_scenario2}, then so does $(\Omega(\mtx{x},t),\Theta(\mtx{x},t)+C \exp(\int^t_0 c_l(s)+2c_{\omega}(s)\,ds))$ for any constant $C$. This is equivalent to the translation invariance of the original Boussinesq equations: if $(\omega(\mtx{x},t),\theta(\mtx{x},t))$ solves \eqref{eqt:Boussinesq_v2}, then $(\omega(\mtx{x},t),\theta(\mtx{x},t)+C)$ also solves \eqref{eqt:Boussinesq_v2} for any constant $C$. Hence, to guarantee that $\Theta(\mtx{x},t)$ does not diverge to infinity, one must determine the constant $C$. This is straightforward in the non-shifted version of the dynamic rescaling equations \eqref{eqt:dynamic_BS_scenario1}, since if we set $\Theta(\mtx{0},0)=0$, the condition $\Theta(\mtx{0},t)=0$ is naturally preserved by the equations, which guarantees that $\Theta$ does not diverge to infinity. However, this does not hold true for the shifted case. Therefore, in the following numerical results regarding the dynamics of $\Theta$, the observed convergence should be understood modulo a time-dependent constant.

In the numerical simulations, we specify the initial data as
\[ \Omega(x_1,x_2) = \exp\left(\frac{-(x_1-3)^2 - x_2^2}{18}\right), \quad \Theta(x_1,x_2) = \frac{1}{\exp\left(-\frac{x_1-3}{4+x_2}\right)+1}.\]
The computation is terminated when $\|\Omega_t\|_{L^{\infty}} < 10^{-6}$, where the $L^\infty$-norm is evaluated over all computational grid points. Figure \ref{fig:Bs_sc2_time_evolution} tracks the time evolution of the residuals $\|\Omega_t\|_{L^2}$ and $\|\Omega_t\|_{L^{\infty}}$. As demonstrated, the stopping criterion is eventually satisfied, suggesting that the solution of \eqref{eqt:dynamic_BS_scenario2} has stabilized at a regular steady state.

\begin{figure}[!htbp]
\centering
        \includegraphics[width=0.45\textwidth]{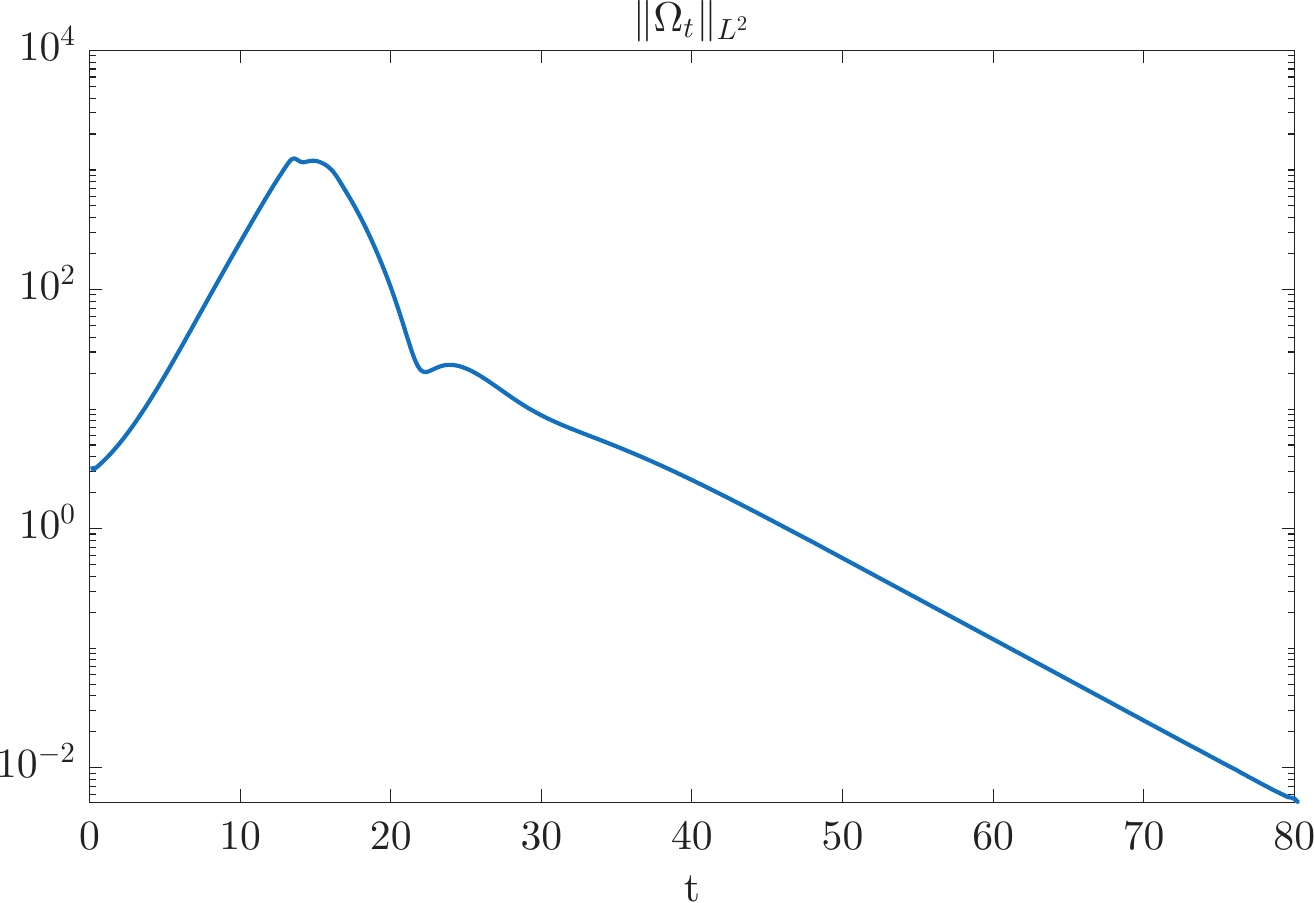}
        \includegraphics[width=0.45\textwidth]{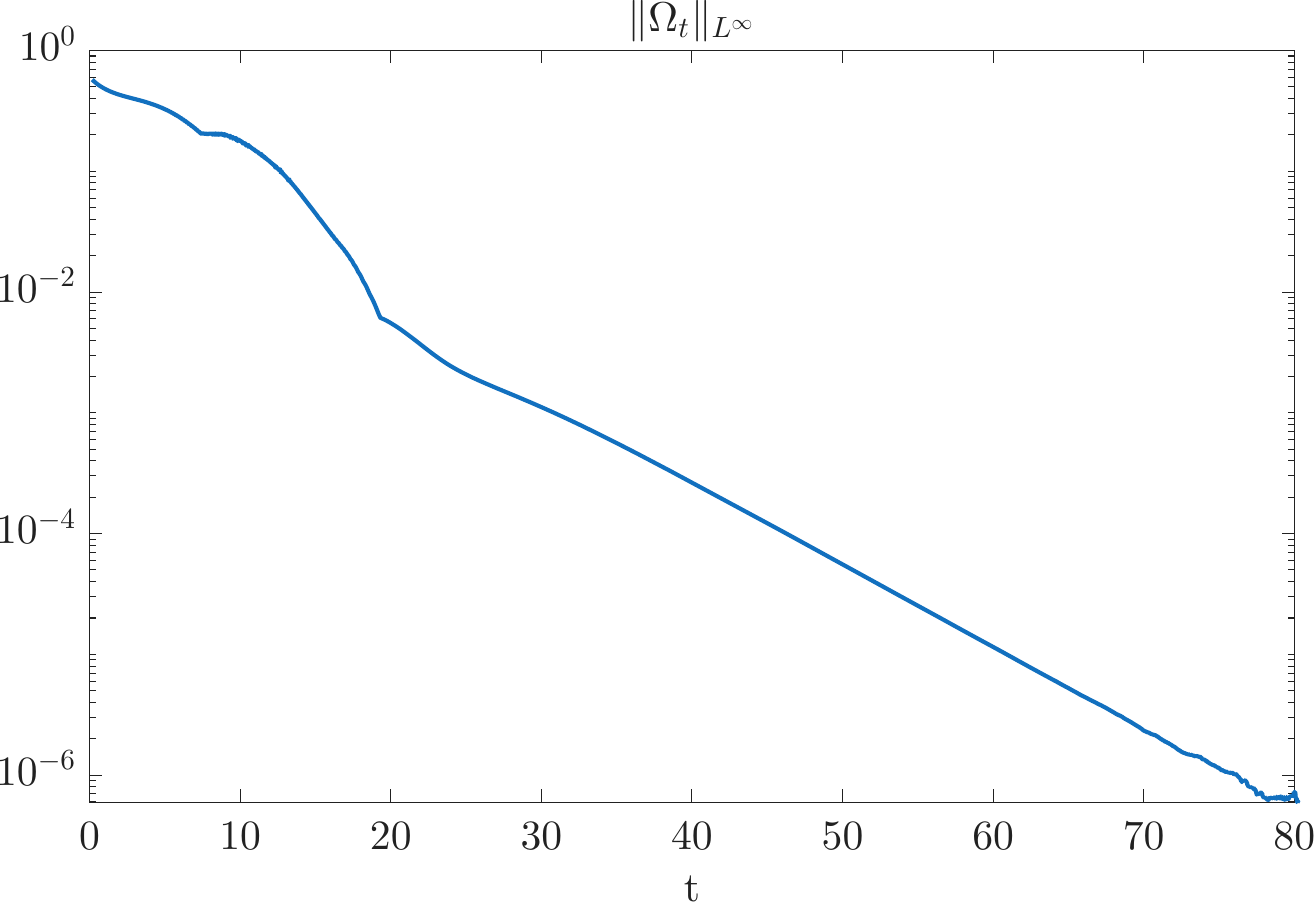}
    \caption[Decay of the residuals $\|\Omega_t\|_{L^2}$ and 
    $\|\Omega_t\|_{L^{\infty}}$ (Scenario 2).]{Decay of the residuals $\|\Omega_t\|_{L^2}$ (left figure) and 
    $\|\Omega_t\|_{L^{\infty}}$ (right figure) in Scenario 2. Both residuals decay rapidly and then level
    off at a small magnitude, which indicates that $\Omega$ eventually stabilizes at a regular profile.}
    \label{fig:Bs_sc2_time_evolution}
\end{figure}

The evolution of the spatial profiles of $\Omega$ and $\Theta$ over time is further illustrated in Figures \ref{fig:Bs_sc2_mesh_evo_om} and \ref{fig:Bs_sc2_mesh_evo_th}, respectively. As observed, $\Omega$ converges to a non-symmetric regular profile that remains strictly positive throughout the computational domain. Meanwhile, $\Theta$ also settles to a regular profile modulo a time-dependent constant. Figure \ref{fig:BS_c} tracks the convergence of $c_r$ and the ratio $c_l/c_{\om}$. The computed limiting value of $c_l/c_{\om}$ is $-2.4489$, which is remarkably close to the 1D case.

\begin{figure}[!htbp]
\centering
        \includegraphics[width=0.4\textwidth]{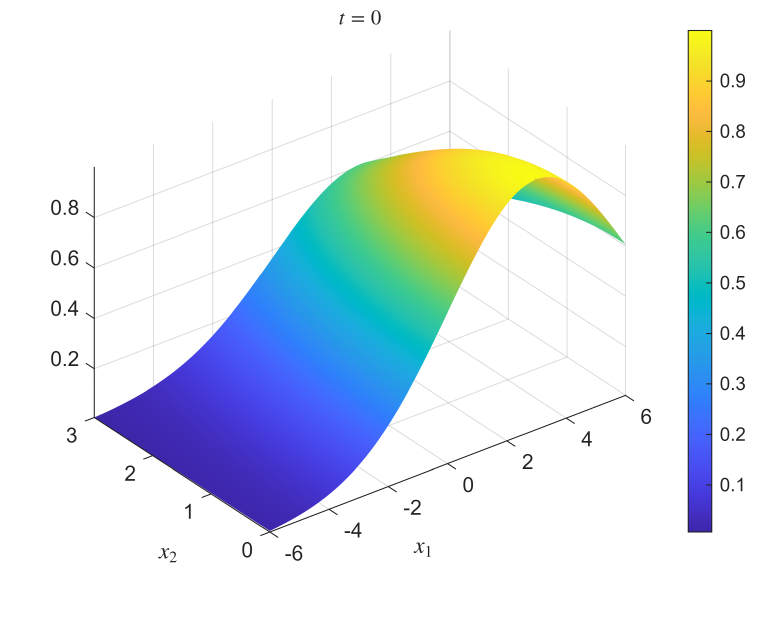}
        \includegraphics[width=0.4\textwidth]{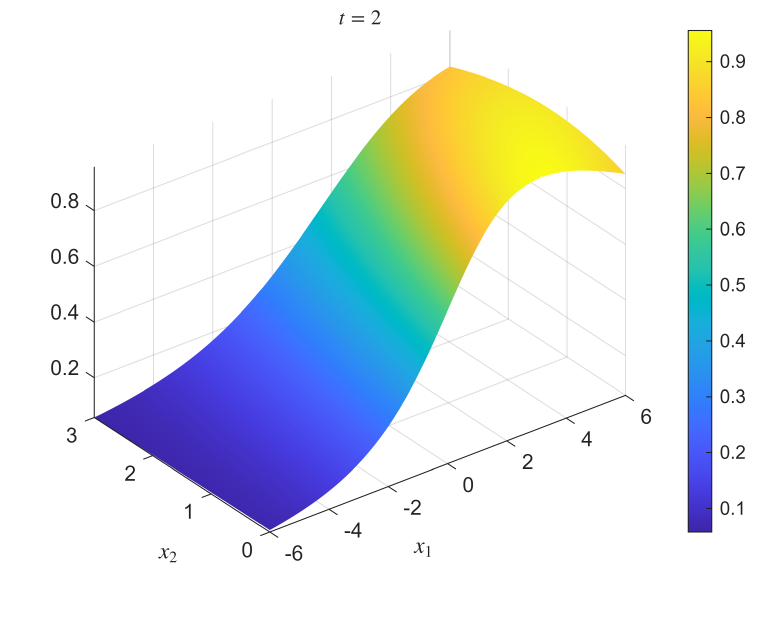}
        \includegraphics[width=0.4\textwidth]{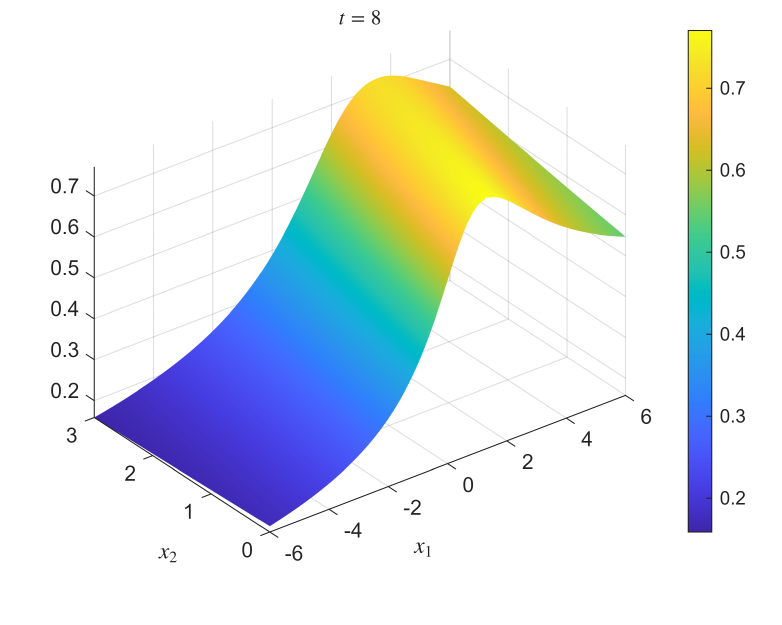}
        \includegraphics[width=0.4\textwidth]{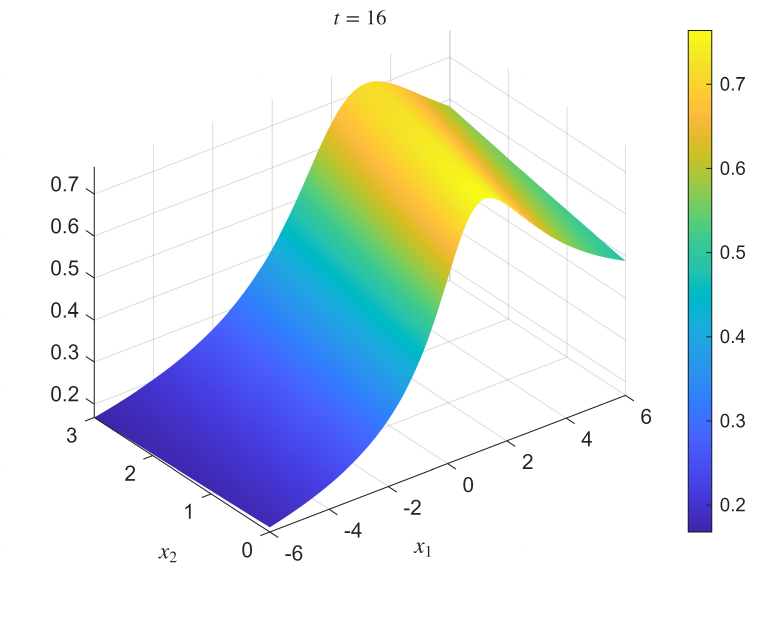}
    \caption[Evolution of 3D spatial profiles of $\Omega$ in Scenario 2.]{Evolution of 3D spatial profiles of $\Omega$ in Scenario 2. As demonstrated, $\Omega$ evolves into a regular limiting profile that remains strictly positive throughout the computational domain.}
    \label{fig:Bs_sc2_mesh_evo_om}
\end{figure}
\begin{figure}[!htbp]
\centering
        \includegraphics[width=0.4\textwidth]{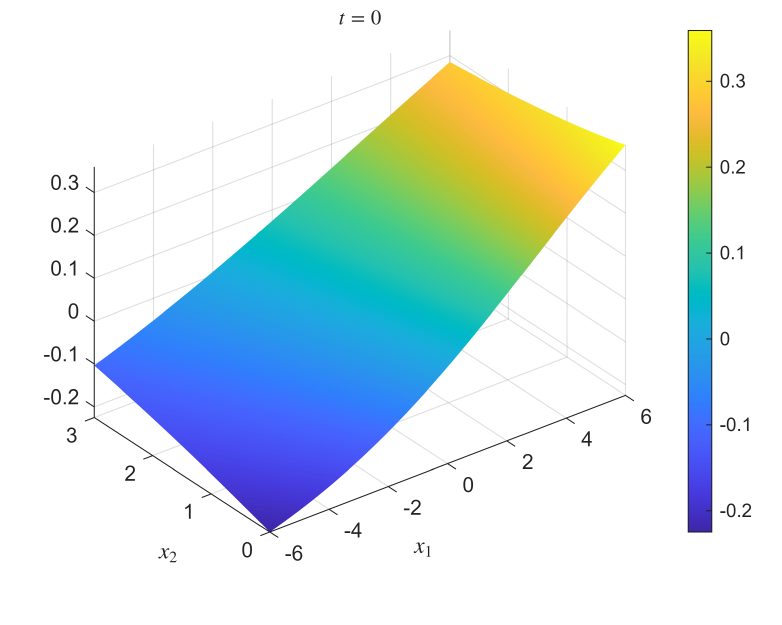}
        \includegraphics[width=0.4\textwidth]{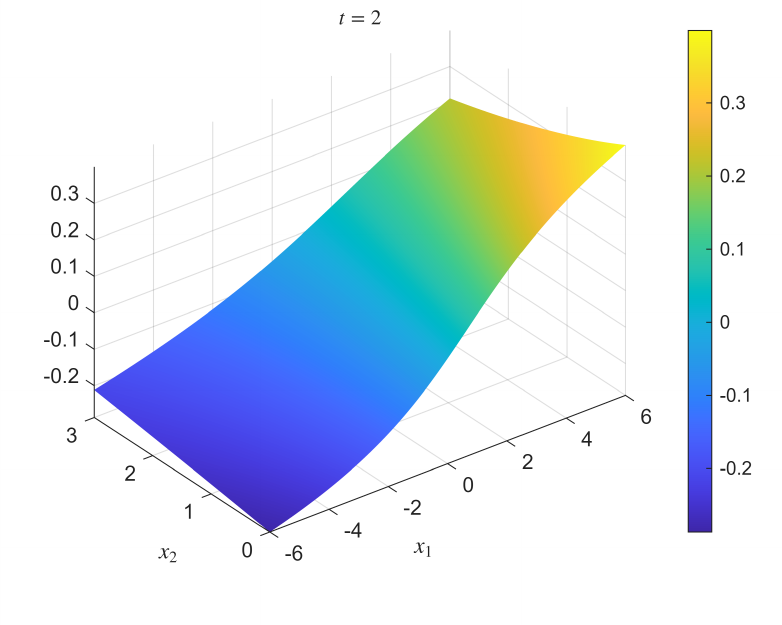}
        \includegraphics[width=0.4\textwidth]{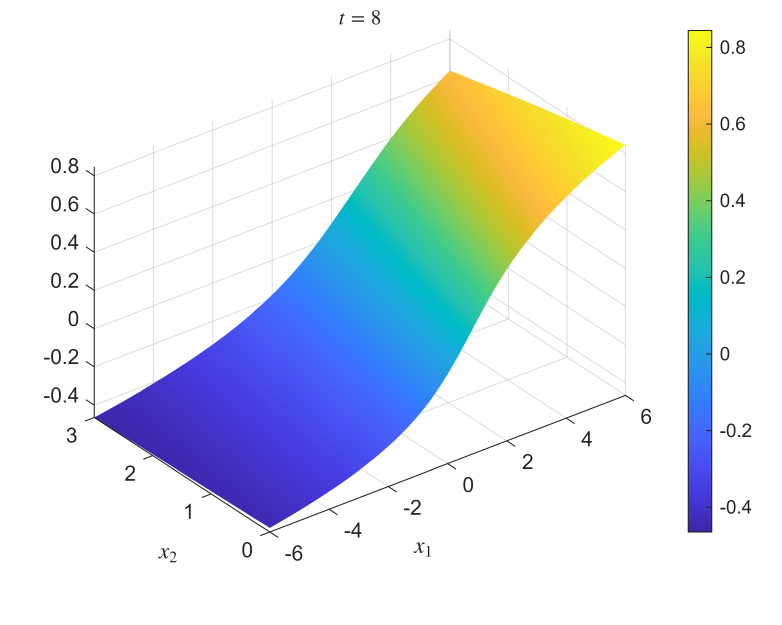}
        \includegraphics[width=0.4\textwidth]{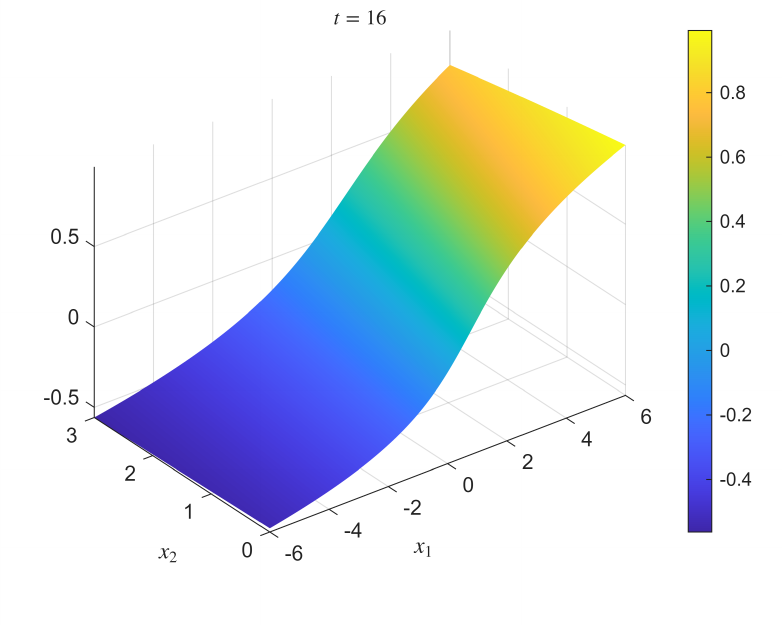}
    \caption[Evolution of 3D spatial profiles of $\Theta$ in Scenario 2.]{Evolution of 3D spatial profiles of $\Theta$ in Scenario 2. Modulo a time-dependent constant, $\Theta$ is observed to approach a regular limiting profile.}
    \label{fig:Bs_sc2_mesh_evo_th}
\end{figure}

\begin{figure}[!htbp]
\centering
        \includegraphics[width=0.45\textwidth]{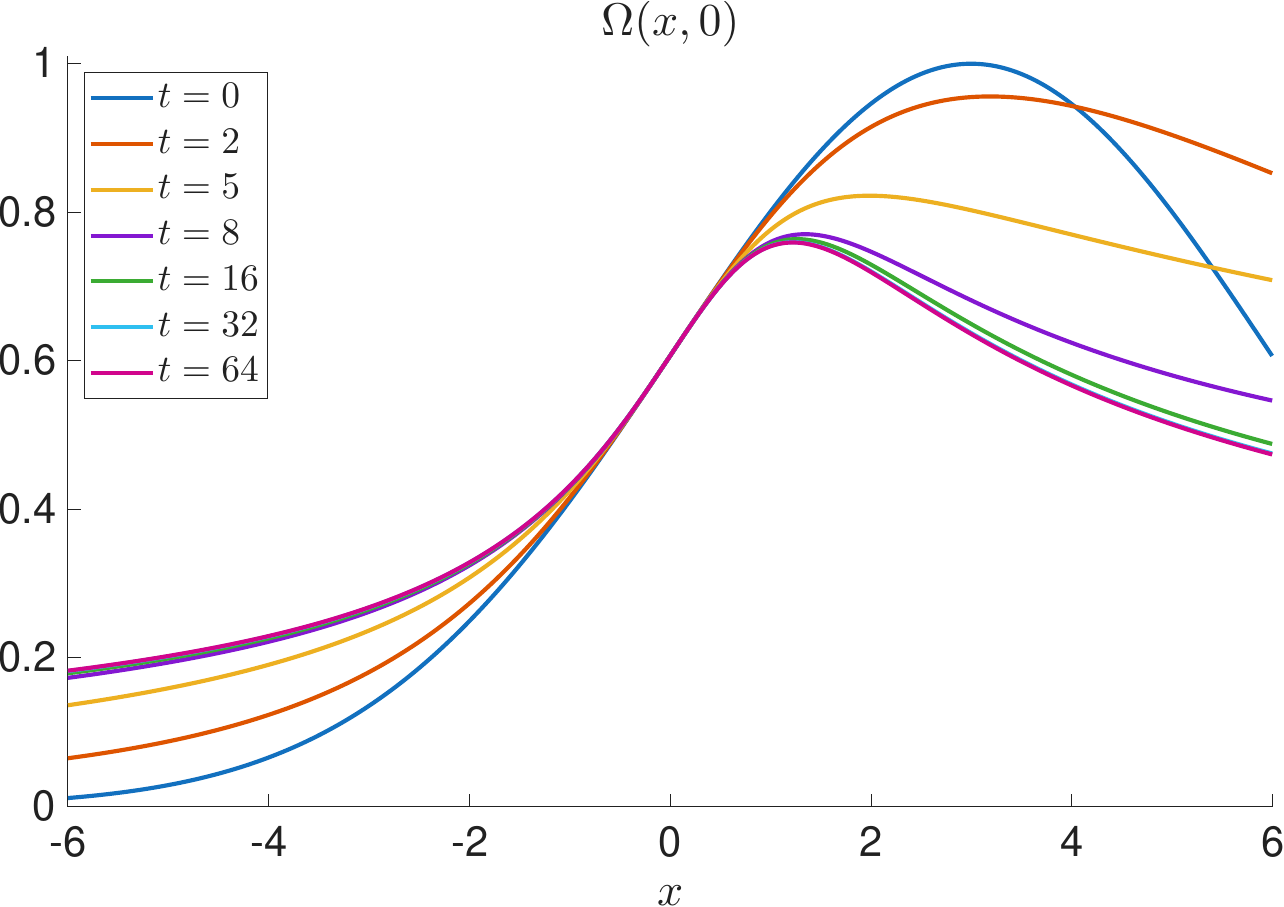}
        \includegraphics[width=0.45\textwidth]{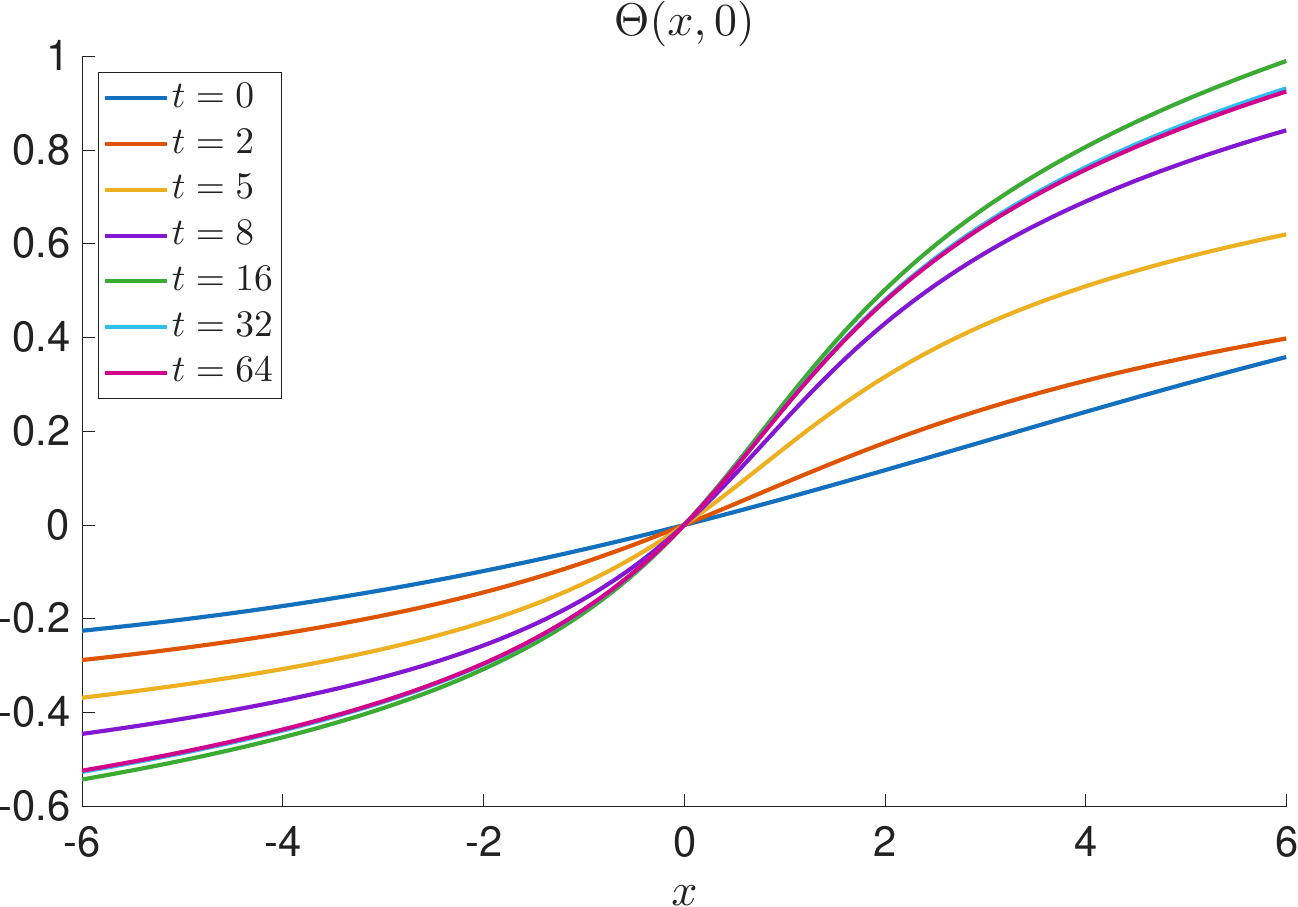}
    \caption[Boundary profiles in Scenario 2.]{Evolution of cross-sectional profiles of the solutions along $x_2=0$ in Scenario 2. As demonstrated, both $\Omega$ (left figure) and $\Theta$ (right figure) approach regular limiting profiles.}  
     \label{fig:Bs_sc2_bd_evo}
\end{figure}

\begin{figure}[!htbp]
\centering
        \includegraphics[width=0.45\textwidth]{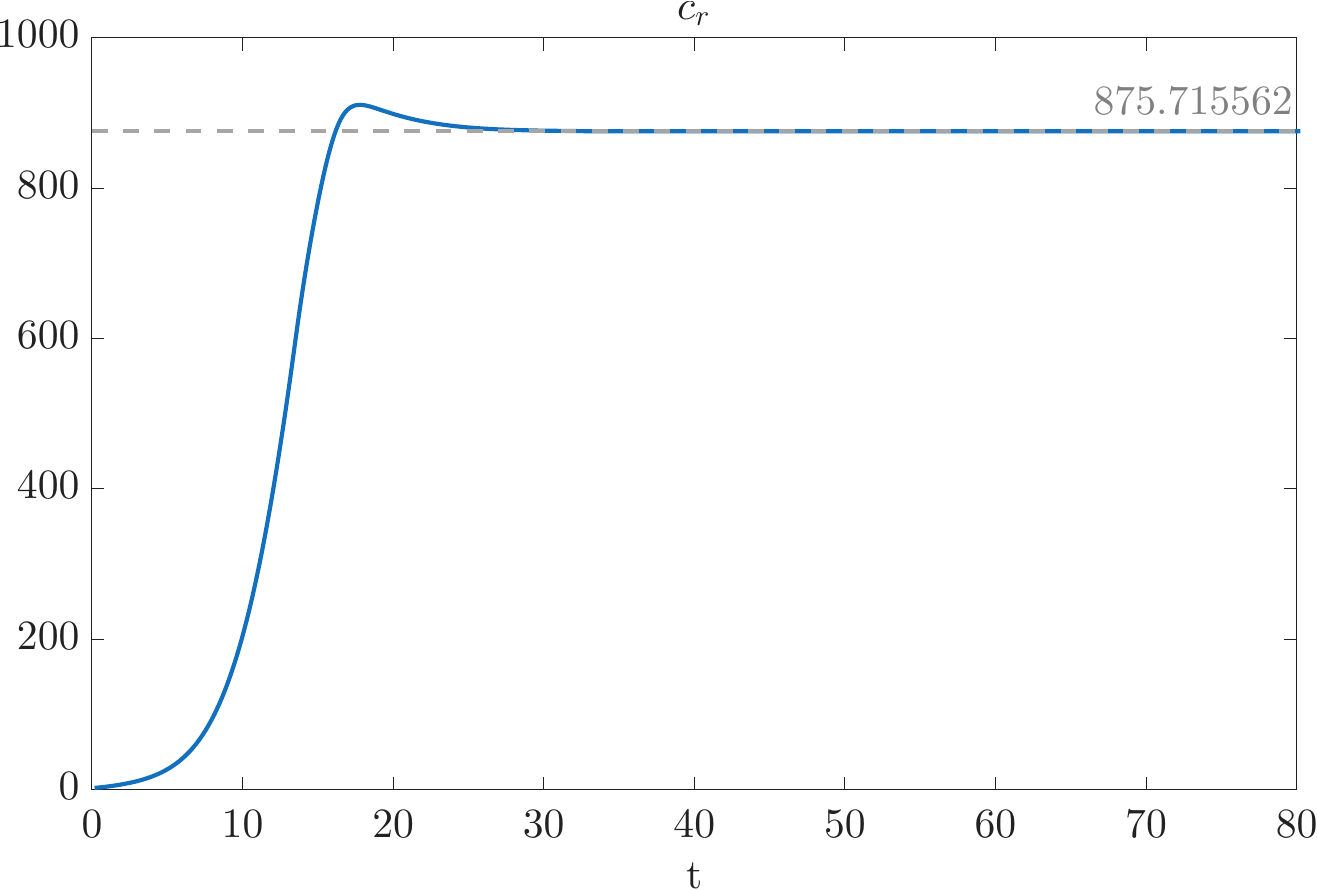}
        \includegraphics[width=0.45\textwidth]{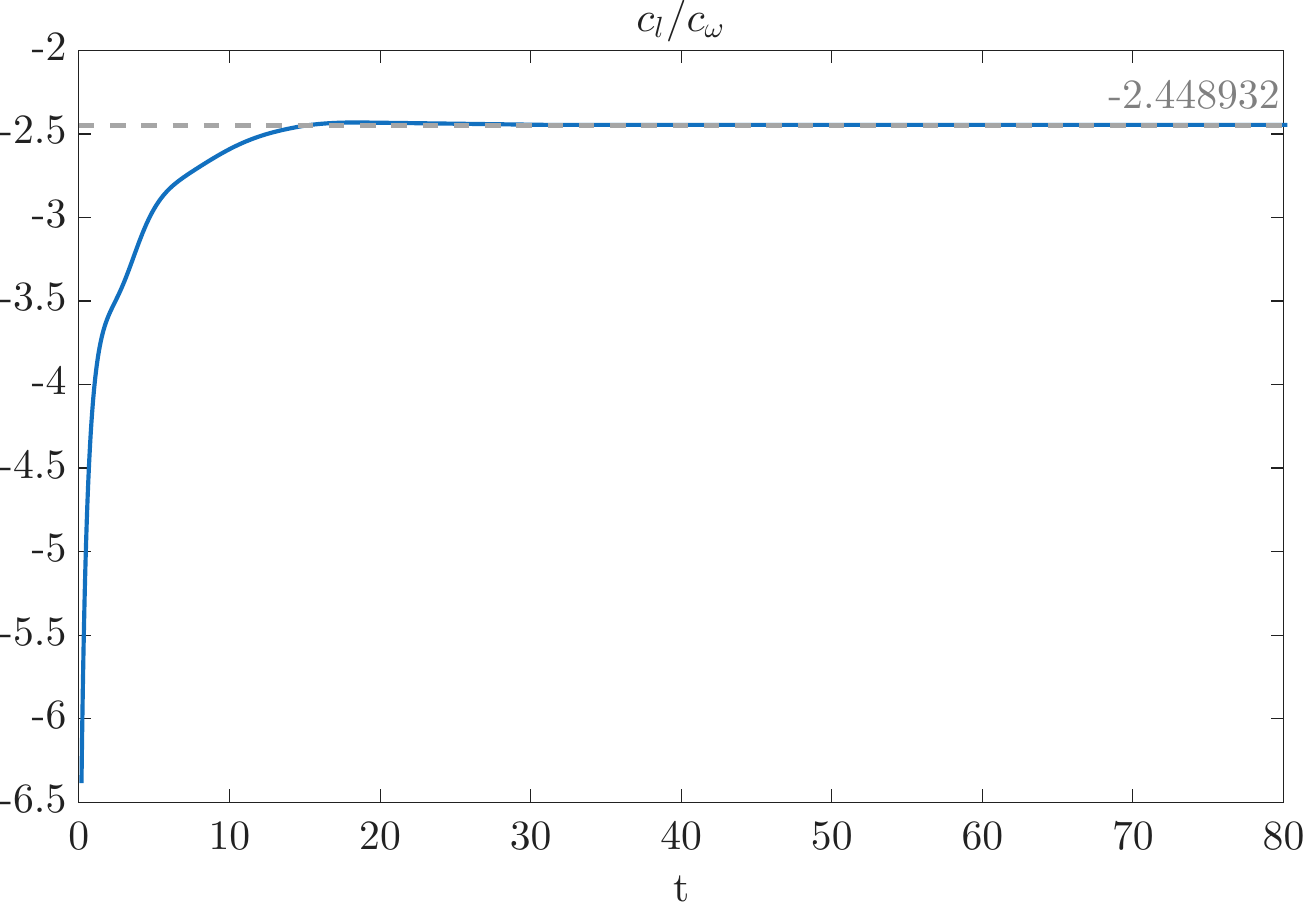}
        \caption[Time evolution of $c_r$ and $c_l/c_{\omega}$ in Scenario 2.]{Time evolution of $c_r$ (left figure) and $c_l/c_{\omega}$ (right figure) in Scenario 2.} 
     \label{fig:BS_c}
\end{figure}

Furthermore, we compare the inner self-similar profiles of $\Omega$ and $\Theta_{x_1}$ observed during the Stage 1 blowup of Scenario 1 presented in Subsection \ref{subsec:BS_scenario1} with the limiting profiles obtained from the numerical simulations of \eqref{eqt:dynamic_BS_scenario2}. Figure \ref{fig:BS_scenario2_profiles_comparison} illustrates a comparison of the cross-sectional profiles between Scenario 1 and Scenario 2. As observed, after appropriate rescaling, the inner profiles in Scenario 1 match closely with the limiting profiles in Scenario 2.
In addition to comparing the boundary cross-sectional profiles, we also present comparisons of the 3D spatial profiles and contour plots in Figures \ref{fig:BS_scenario2_2d_mesh_comparison} and \ref{fig:BS_scenario2_2d_contour_comparison}, respectively. Together, these observations strongly suggest that the modified dynamic rescaling equations \eqref{eqt:dynamic_BS_scenario2} accurately capture the intrinsic mechanism governing the Stage 1 self-similar blow-up described in Subsection \ref{subsec:BS_scenario1}.

\begin{figure}[!htbp]
\centering
        \includegraphics[width=0.45\textwidth]{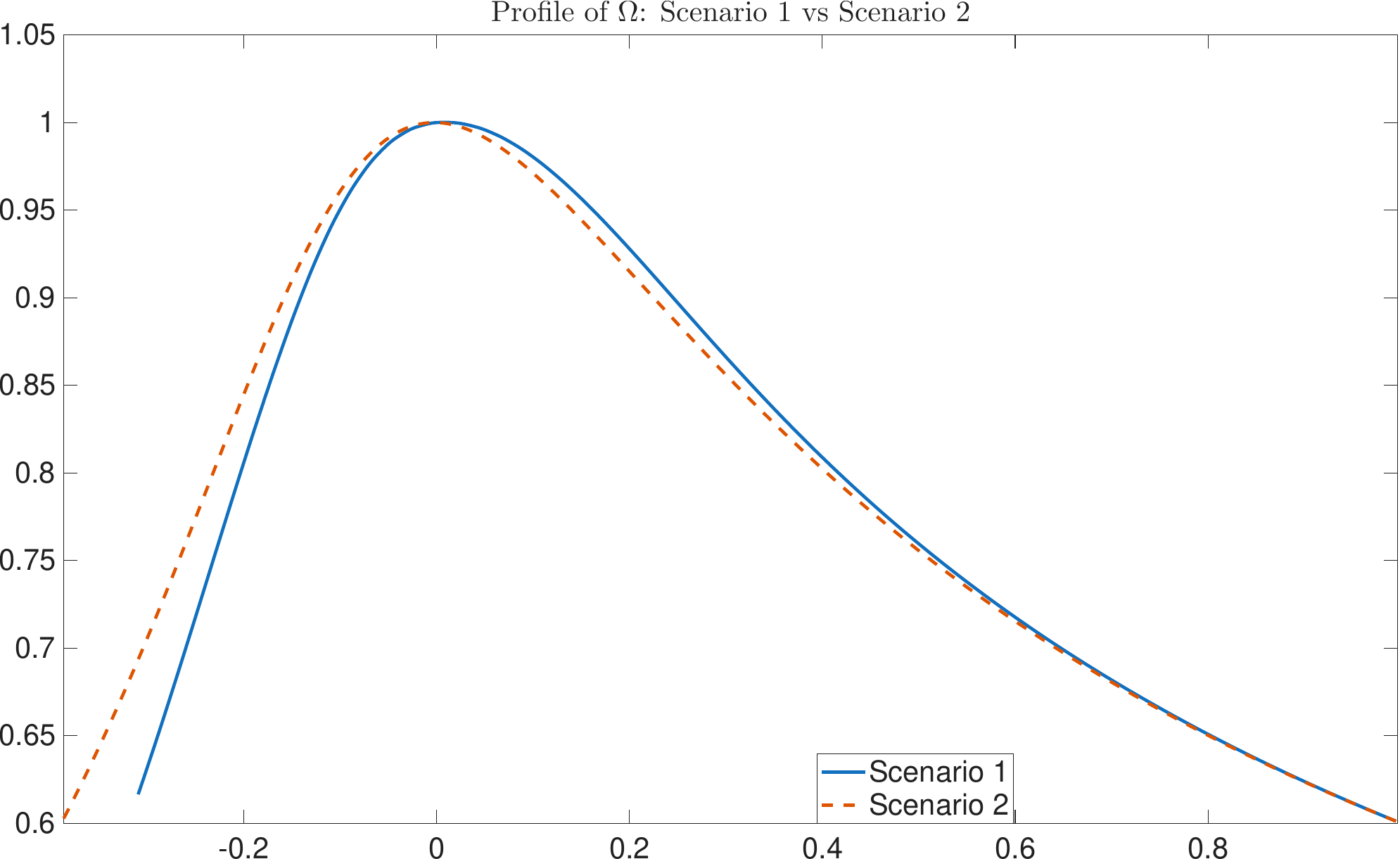}
        \includegraphics[width=0.45\textwidth]{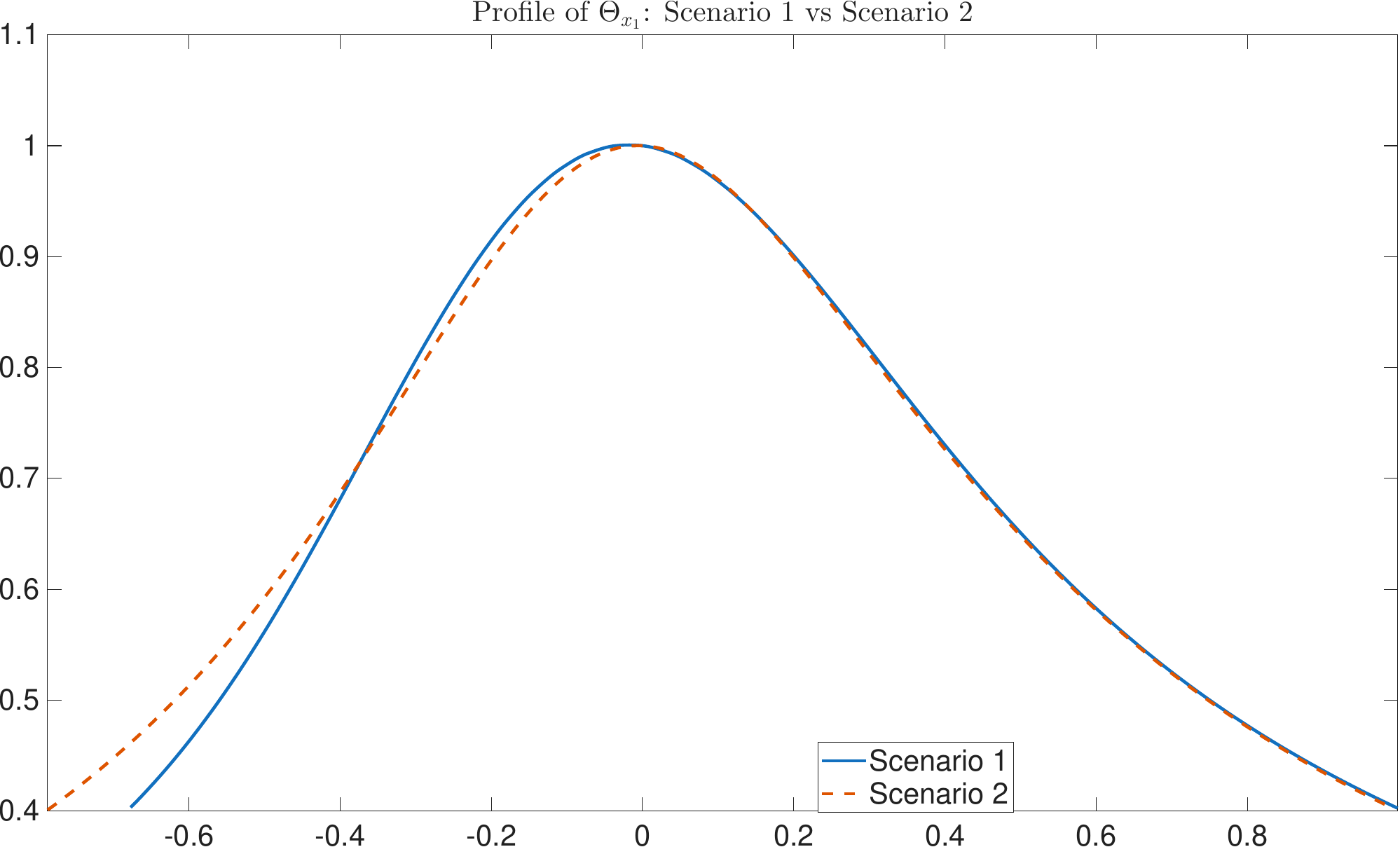}
    \caption[Comparison between the inner profile in Scenario 1 at $t = 0.92$ (blue solid line) and the limiting profile in Scenario 2 (red dashed line) for the 2D Boussinesq equations.]{Comparison between the inner profile in Scenario 1 at $t = 0.92$ (blue solid line) and the limiting profile in Scenario 2 (red dashed line) for the 2D Boussinesq equations. Left figure: profile of $\Omega$; Right figure: profile of $\Theta_{x_1}$. The profiles are compared after appropriate rescaling.}
     \label{fig:BS_scenario2_profiles_comparison}
\end{figure}

\begin{figure}[!htbp]
\centering
       \includegraphics[width=0.4\textwidth]{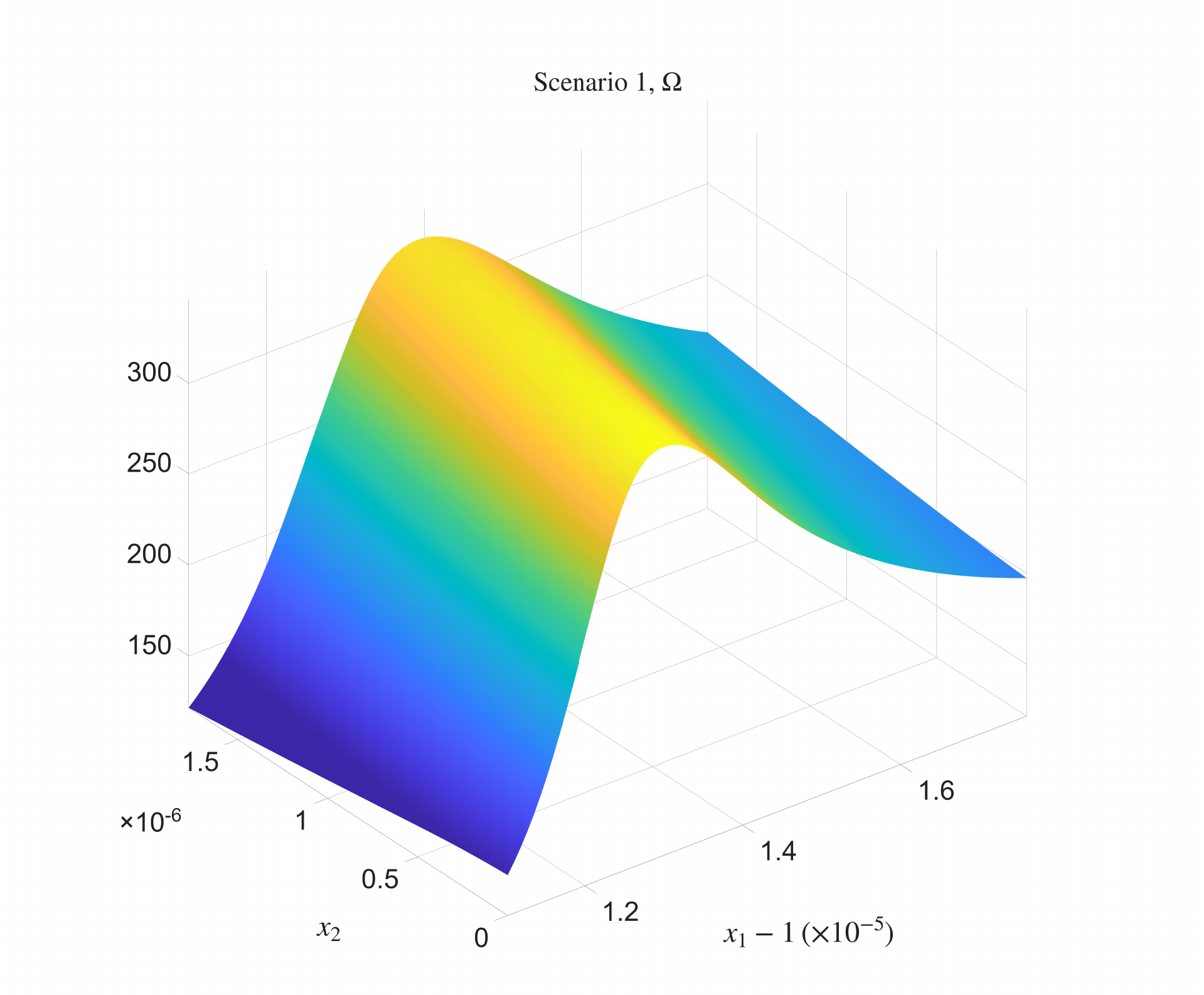}
       \includegraphics[width=0.4\textwidth]{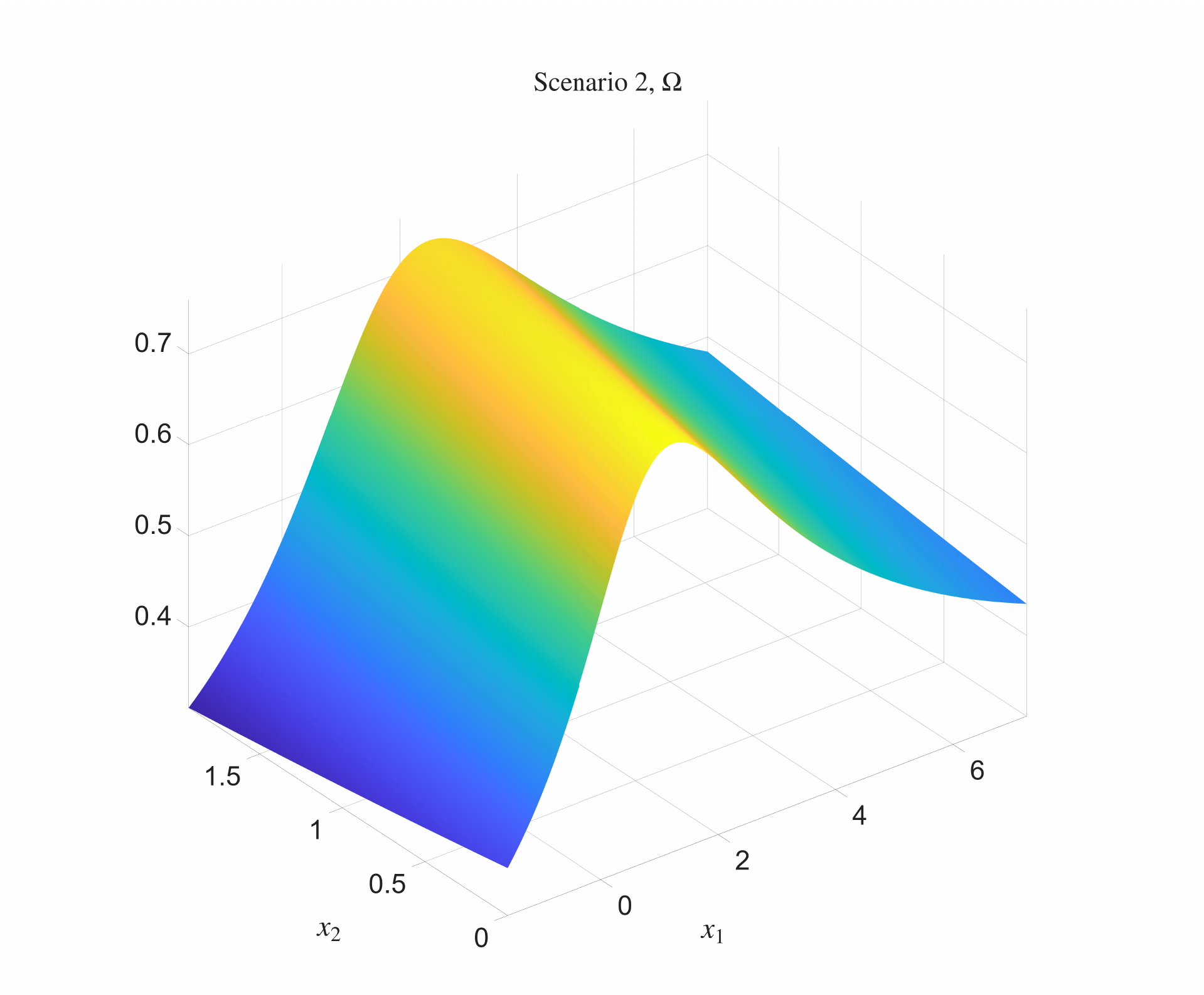}
       \includegraphics[width=0.4\textwidth]{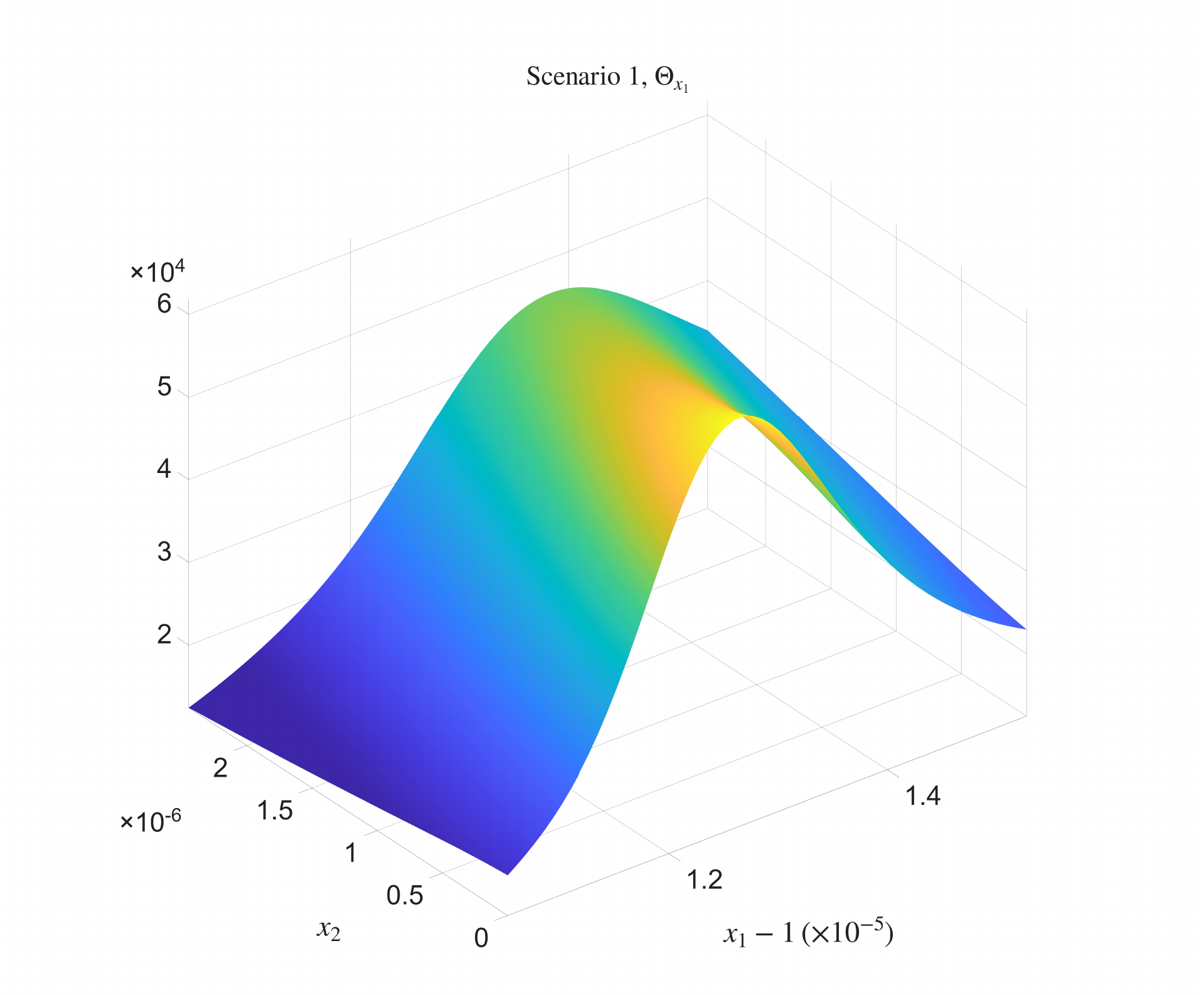}
       \includegraphics[width=0.4\textwidth]{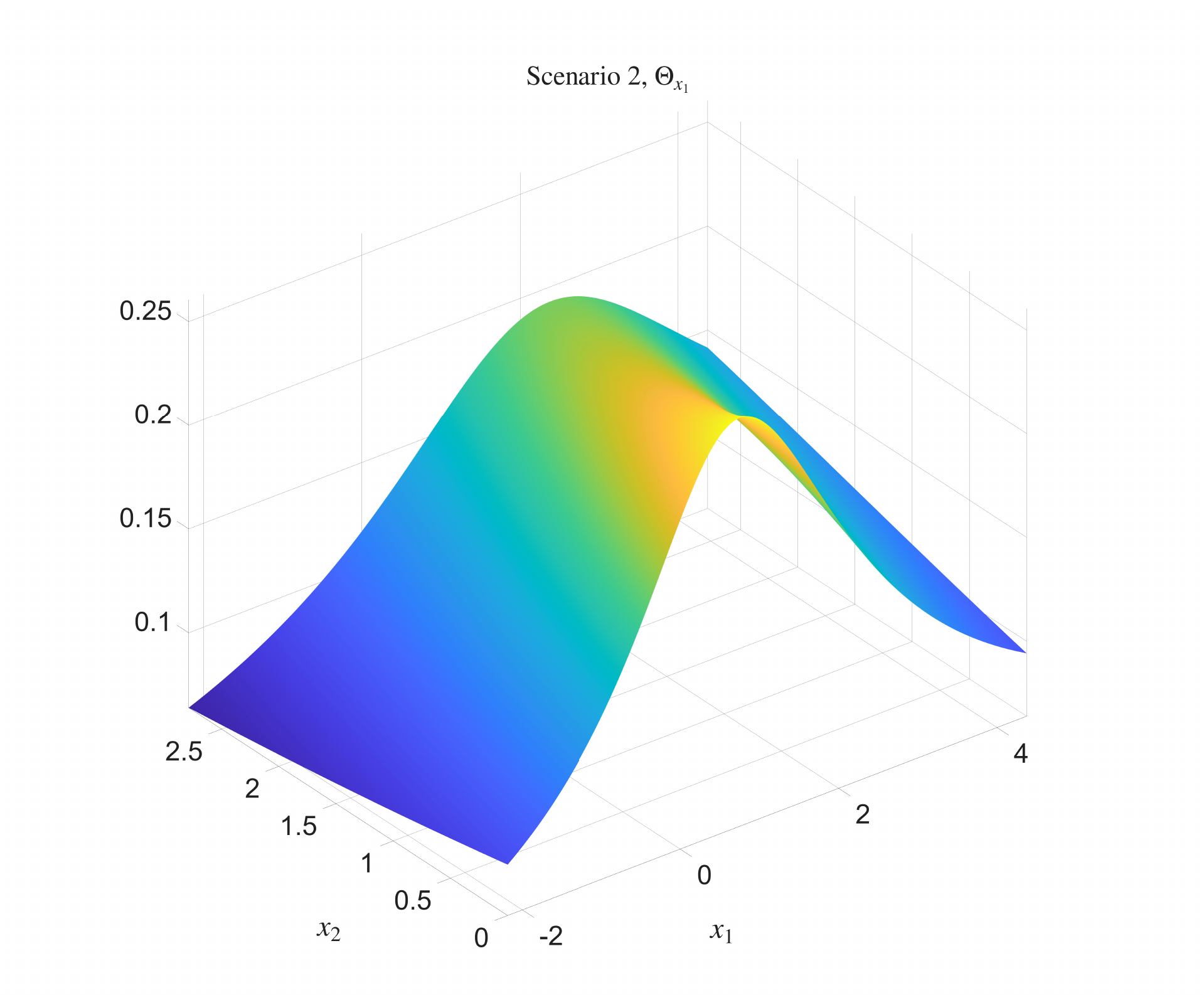}
    \caption[Comparison between Scenario 1 and Scenario 2 for the 2D Boussinesq equations.]{Comparison between Scenario 1 and Scenario 2 for the 2D Boussinesq equations. Top left: inner profile of $\Omega $ in Scenario 1 at $t = 0.92$; Top right: limiting profile of $\Omega$ in Scenario 2. Bottom left: inner profile of $\Theta_{x_1}$ in Scenario 1 at $t = 0.92$; Bottom right: limiting profile of $\Theta_{x_1}$ in Scenario 2.}
    \label{fig:BS_scenario2_2d_mesh_comparison}
\end{figure}
\begin{figure}[!htbp]
\centering
       \includegraphics[width=0.4\textwidth]{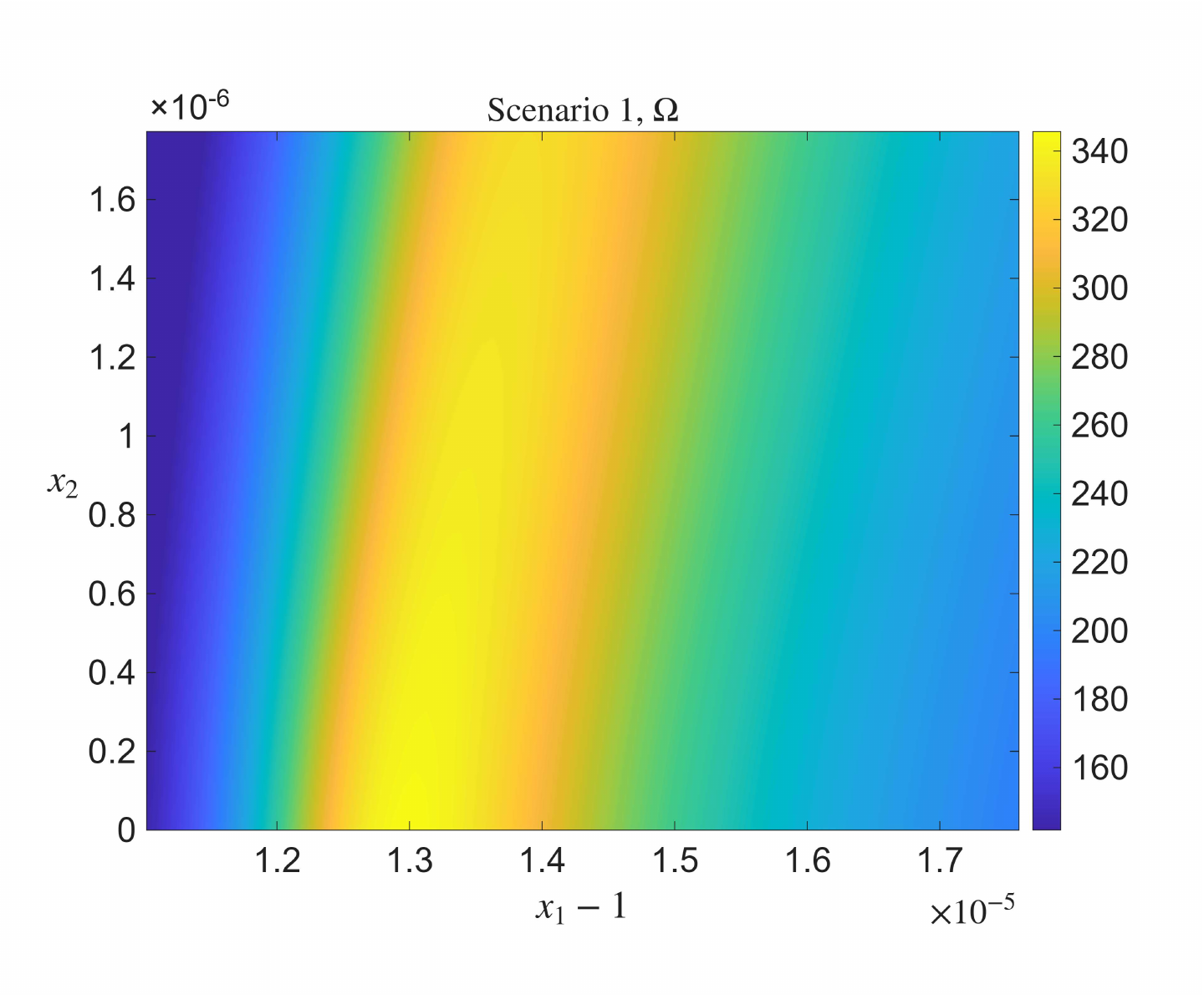}
       \includegraphics[width=0.4\textwidth]{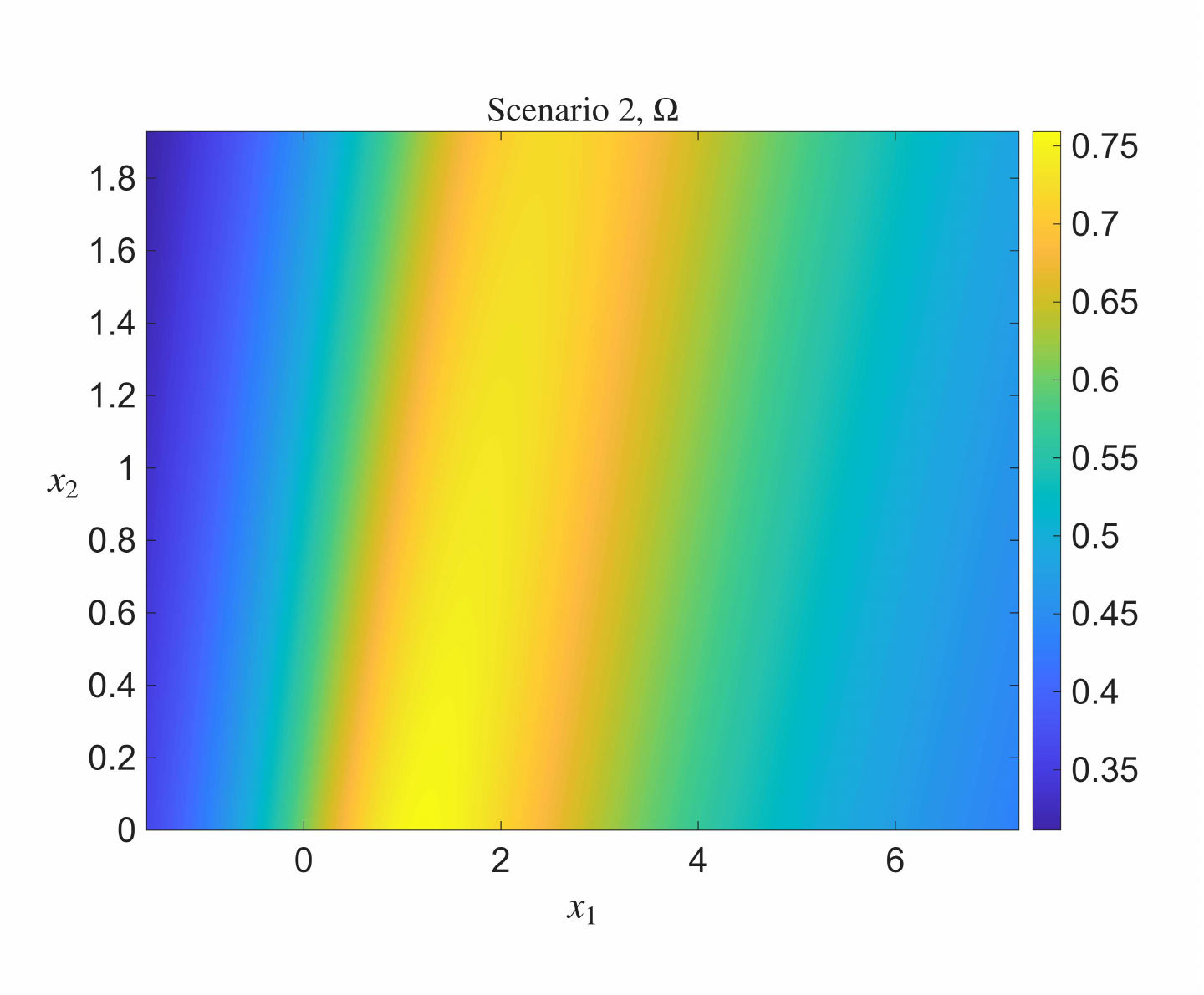}
       \includegraphics[width=0.4\textwidth]{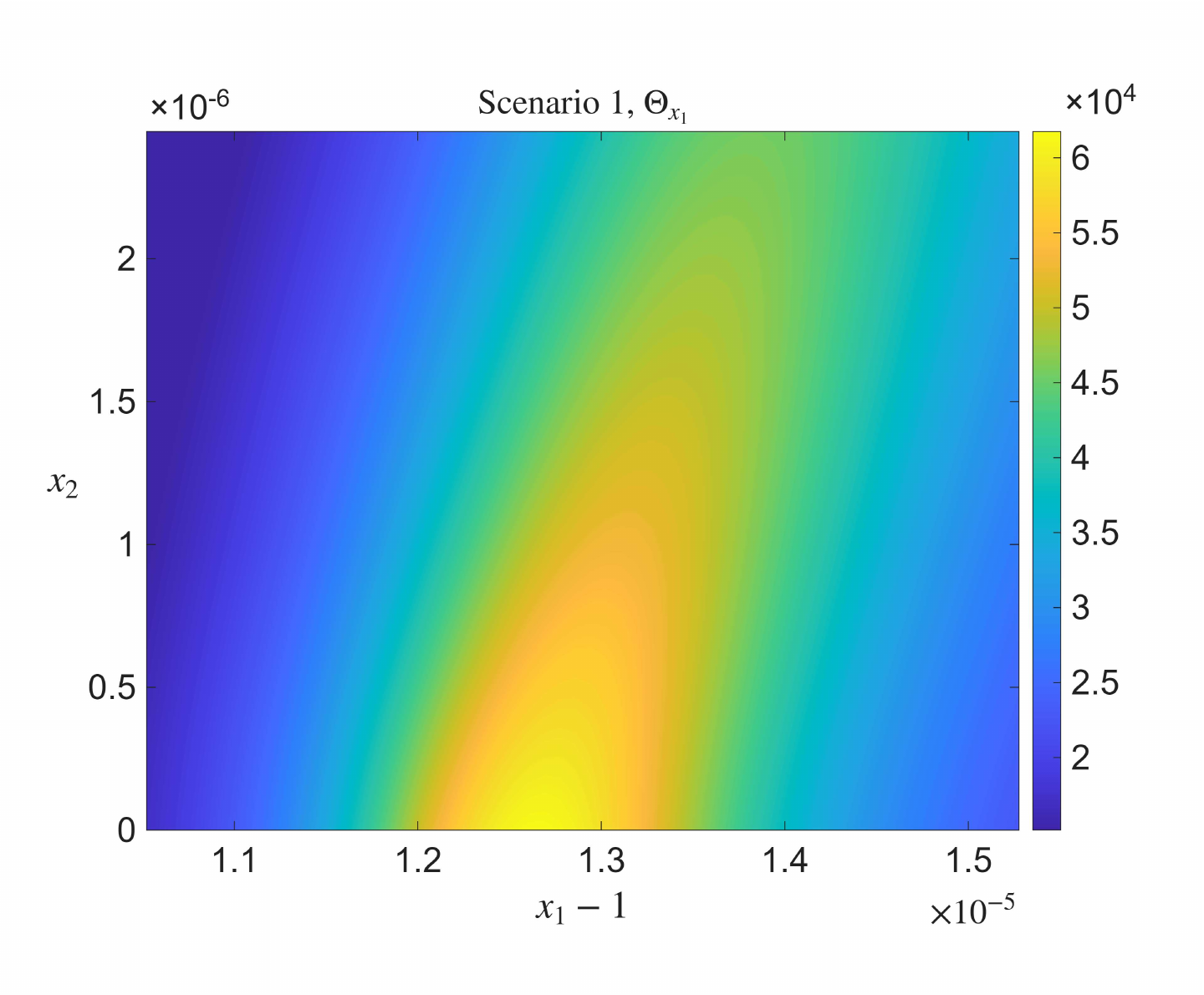}
       \includegraphics[width=0.4\textwidth]{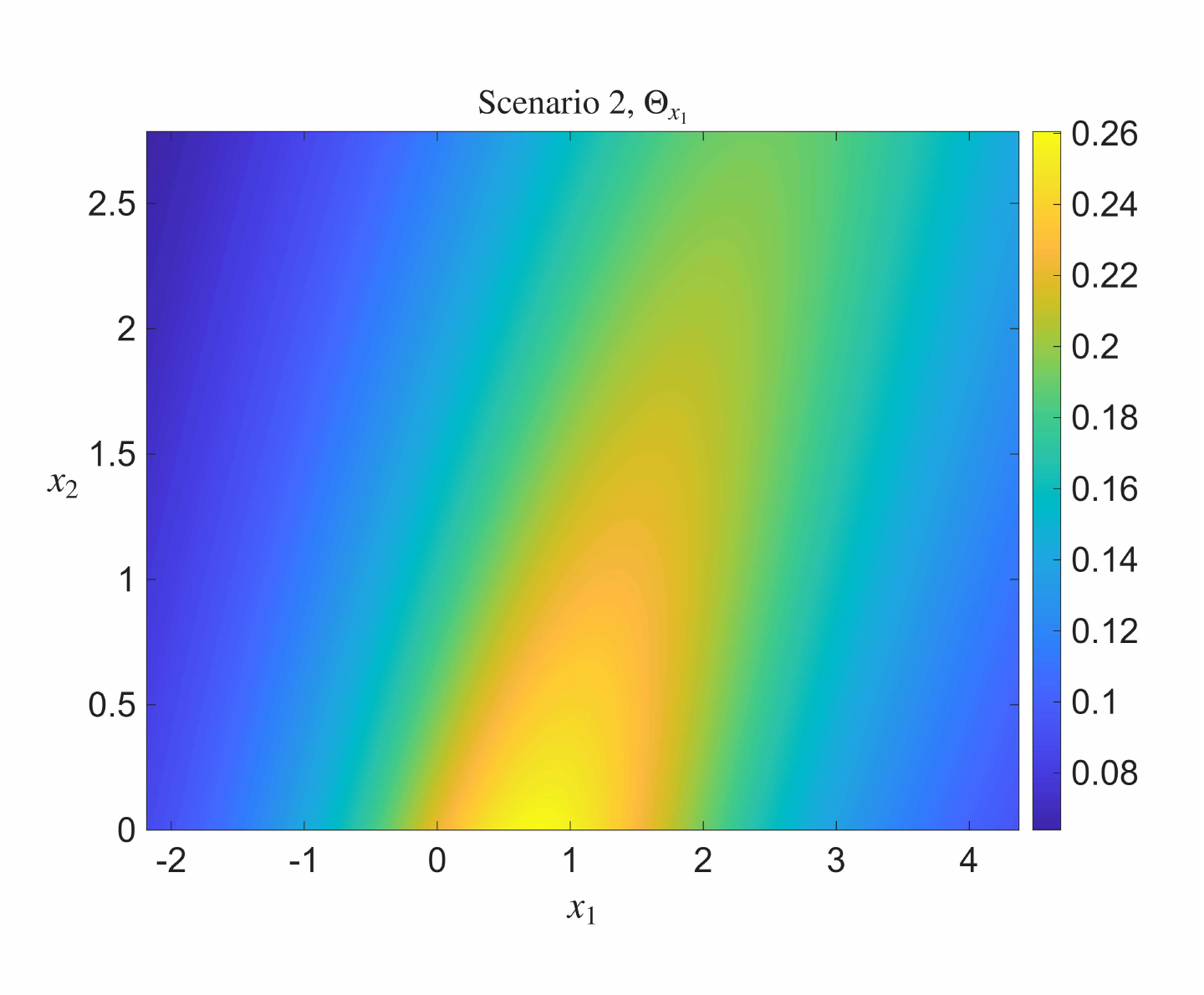}
    \caption[Comparison between Scenario 1 and Scenario 2 for the 2D Boussinesq equations.]{Comparison between Scenario 1 and Scenario 2 for the 2D Boussinesq equations. Top left: inner profile of $\Omega $ in Scenario 1 at $t = 0.92$; Top right: limiting profile of $\Omega$ in Scenario 2. Bottom left: inner profile of $\Theta_{x_1}$ in Scenario 1 at $t = 0.92$; Bottom right: limiting profile of $\Theta_{x_1}$ in Scenario 2.}
    \label{fig:BS_scenario2_2d_contour_comparison}
\end{figure}

\appendix
\section{Translation invariance of the HL model}\label{sec:HL_translation_invariance}
In this section we complete the proof of Lemma \ref{lem:translation_invariance}. 
\begin{proof}[Proof of Lemma \ref{lem:translation_invariance}]
    Suppose that there exists a time-dependent quantity $c(t)$ such that $\om_2(x,t)=\om_1(x-c(t),t)$, $\theta_2(x,t)=\theta_1(x-c(t),t)$. Then we have 
\begin{align*}
u_2(x,t)&=\frac{1}{\pi}\int_{\R}\om_2(y,t)\ln\left|\frac{x-y}{x_2(t)-y}\right|\idiff y\\
        &=\frac{1}{\pi}\int_{\R}\om_1(y-c(t),t)\ln\left|\frac{x-c(t)-(y-c(t))}{x_2(t)-c(t)-(y-c(t))}\right|\idiff y\\
        &= \frac{1}{\pi}\int_{\R}\om_1(y,t)\left(\ln\left|\frac{x-c(t)-y}{x_1(t)-y}\right|-\ln\left|\frac{x_2(t)-c(t)-y}{x_1(t)-y}\right|\right)\idiff y\\
        &=u_1(x-c(t),t)-u_1(x_2(t)-c(t),t).
\end{align*}
Substituting the above expressions into the original HL model \eqref{eqt:1Dhouluo}, we find that $(\om_2,\theta_2,u_2)$ solves the same equation as $(\om_1,\theta_1,u_1)$ if and only if $c'(t)=-u_1(x_2(t)-c(t),t)$. Therefore, the conclusion follows immediately from the uniqueness of the solutions to the Cauchy problem of \eqref{eqt:1Dhouluo} (for a brief discussion on the local well-posedness of \eqref{eqt:1Dhouluo}, see \cite{choi2017finite}).
\end{proof}

\section{Numerical methods for the HL model}\label{sec:HL_numerical method}
In this appendix, we describe the numerical methods that we use to compute the dynamic rescaling equations of the Hou--Luo model \eqref{eqt:1Dhouluo}.
\subsection{Computational mesh}
In both scenarios, the solutions are computed on a finite domain $D$, where $D=[0,L]$ in Scenario 1 and $D=[-L,L]$ in Scenario 2, with $L \approx 10^{12}$ chosen to be sufficiently large. We construct computational meshes in $D$ with grid points $\{x_i\}_{i=0}^N$ satisfying $x_{i-1} < x_i$ and a mesh stretching ratio bounded by $\max \{(x_i-x_{i-1})/(x_{i+1}-x_i),(x_{i+1}-x_{i})/(x_{i}-x_{i-1})\}<1.01$. In what follows, we use $\Omega_i^k$ and $V_i^k$ to denote the numerical solutions at the discrete time step $t_k$ and the spatial grid point $x_i$, respectively.   

In Scenario 1, to accurately resolve the singularity of the solutions, the mesh is dynamically refined near $x=1$. Let $y_1 < y_2$ be the spatial points satisfying $\Omega(y_1) = \Omega(y_2) = 0.3 \|\Omega\|_{L^{\infty}}$. In the adaptive mesh setting, we record the value $w = y_2 - y_1$ each time the mesh is updated. The mesh is refined whenever $w$ decreases below half of its previously recorded value, ensuring that the interval $[y_1, y_2]$ contains 80 grid points in the newly generated mesh. When computing the numerical solutions using a fixed mesh, the finest grid spacing near $x=1$ is $2\times 10^{-8}$.

\subsection{Computing the Hilbert transform}
Given the grid point values $\{\Omega_i^k\}_{i=0}^N$, we use a standard cubic spline interpolation to obtain the whole function $\Omega^k$. The spline boundary conditions specify the first derivative at each endpoint. These derivatives are approximated using the function values at the three grid points closest to each endpoint. We then use the kernel integral to obtain the values of $U^k$ and $U_x^k$ at grid points:
\begin{itemize}
    \item For Scenario 1:
\begin{equation}\label{eqt:numerically_compute_Hilbert_transform}
    \begin{aligned}
    U_i^k&=\frac{1}{\pi}\int_{0}^{L}\Omega^k(y)\ln\left| \frac{x_i-y}{y}\right| \idiff y,\quad  U_{x,i}^k=\frac{1}{\pi}\mathrm{P.V.}\int_{0}^{L}\frac{\Omega^k(y)}{x_i-y} \idiff y, \quad \,\  \text{if } \operatorname{supp}(\Omega^k)\subset [0,+\infty), \\
U_i^k&=\frac{1}{\pi}\int_{0}^{L}\Omega^k(y)\ln\left| \frac{x_i-y}{x_i+y}\right| \idiff y,\quad  U_{x,i}^k=\frac{1}{\pi}\mathrm{P.V.}\int_{0}^{L}\frac{2y\Omega^k(y)}{x_i^2-y^2} \idiff y, \,\ \text{if } \Omega^k \text{ is odd}.  
    \end{aligned}
\end{equation}
\item For Scenario 2:
\begin{equation}\label{eqt:numerically_compute_Hilbert_transform1}
    U_i^k =\frac{1}{\pi}\int_{-L}^{L}\Omega^k(y)\ln\left| \frac{x_i-y}{y}\right| \idiff y,\quad  U_{x,i}^k=\frac{1}{\pi}\mathrm{P.V.}\int_{-L}^{L}\frac{\Omega^k(y)}{x_i-y} \idiff y.
\end{equation}
\end{itemize}
In particular, for each $x_i$, the contributions to the above integrals from ${[x_{i-m},x_{i+m}]}$ are computed explicitly using the cubic spline expression of $\Omega^k$. The remaining contributions are approximated using a piecewise 8-point Legendre-Gauss quadrature, in order to avoid large round-off error.

\subsection{Overall algorithm}
Let us describe the overall algorithm for numerically solving equations \eqref{eqt:numerical_dynamic_rescaling_of_HLscenario1} and \eqref{eqt:dynamic_rescaling_of_HL_scenario2_numerical}. The algorithm consists of the following steps:
\begin{enumerate}
    \item The initial data are chosen as follows.
 \begin{itemize}
    \item For the odd symmetry case in Scenario 1,
    \[\Omega^0=\frac{\lambda_1 x^7}{x^{8}+5},\quad V^0=\frac{x^5}{x^6+10}.\] 
    \item For the one-sided case in Scenario 1,
   \[\Omega^0=\frac{\lambda_2\exp(-1/x) x^7}{x^{8}+1}, \quad V^0=\frac{\exp(-1/x)x^5}{x^6+10}. \]
   \item For Scenario 2,
    \[\Omega^0=\frac{1}{(x-1)^2+1},\quad V^0=\frac{1}{2(x-1)^2+2}.\]
 \end{itemize}
   Here $\lambda_1, \lambda_2$ are constants that guarantee $\mtx{H}(\Omega ^0)(0)=-1$.
    \item The spatial derivatives of $\Omega^k$ and $V^k$ at grid points are obtained using their cubic spline expressions. The grid point values of $U^k$ and $U_x^k$ are computed according to \eqref{eqt:numerically_compute_Hilbert_transform} or \eqref{eqt:numerically_compute_Hilbert_transform1}. The normalization constants are then determined via \eqref{eqt:HLscenario1_normlization_1}, \eqref{eqt:HLscenario1_normlization_2} or \eqref{eqt:HLscenario2_normalization}.
    \item The solutions $\Omega^k$ and $V^k$ are advanced in time using a 2-stage, 2nd-order explicit Runge-Kutta method with adaptive time stepping. The discrete time step size $\Delta t_k=t_{k+1}-t_{k}$ is given by
    \[\Delta t_k=\left\{\begin{array}{ll} 
        C_1\min_{0\leq i\leq N-1} (x_{i+1}-x_i)/|U_i^k+c_l^kx_i|& \text{for Scenario 1}, \\ C_2\min_{0\leq i\leq N-1} (x_{i+1}-x_i)/|U_i^k+c_l^kx_i+c_r^k|& \text{for Scenario 2} \end{array} \right. \]
     respecting the CFL stability condition, where $C_1, C_2<1$ are suitable constants.
    \item For Scenario 1, a numerical regularization procedure (which will be discussed in the next subsection) is applied after each time step. Furthermore, in the adaptive mesh setting, the mesh is updated when the criterion described in the previous subsection is satisfied. For Scenario 2, no numerical regularization is applied and the mesh is fixed throughout the simulation.
\end{enumerate}

\subsection{Numerical regularization}
In Scenario 1, we apply a suitable numerical regularization scheme for two main reasons:
\begin{enumerate}
    \item The solution may develop oscillations near the singularity, which can cause numerical instability. 
    \item When using a fixed mesh, the rapid growth of the solution's $L^{\infty}$ norm and the fine grid spacing near the singularity lead to extremely small time steps due to the CFL stability condition, which makes it computationally expensive to simulate the solution for a long time.
\end{enumerate}
To suppress the potential oscillations and control the growth of the solution's $L^{\infty}$ norm, we apply the following numerical regularization scheme to $\Omega$ and $V$ at the end of each time step:

\begin{align*} \Omega_i^k &\leftarrow \frac{x_{i+1}-x_i}{x_{i+1}-x_{i-1}}\alpha_i^k\Omega_{i-1}^k+(1-\alpha_i^k)\Omega_i^k+\frac{x_{i}-x_{i-1}}{x_{i+1}-x_{i-1}}\alpha_i^k\Omega_{i+1}^k, \\ V_i^k &\leftarrow \frac{x_{i+1}-x_i}{x_{i+1}-x_{i-1}}\alpha_i^k V_{i-1}^k+(1-\alpha_i^k)V_i^k+\frac{x_{i}-x_{i-1}}{x_{i+1}-x_{i-1}}\alpha_i^kV_{i+1}^k, \end{align*} where $\alpha_i^k$ is an adaptive coefficient depending on the minimum mesh spacing $h_{\min}$ and the discrete time step $\Delta t_k$. Specifically, let $\beta_i^k=50 \Delta t_k$ and 
\[\gamma_i^k=\left\{ \begin{array}{ll}
    1, &\text{if }x_i<1+10000h_{\min},\\
 \displaystyle  \frac{1+20000h_{\min}-x_i}{10000h_{\min}},&\text{if } x_i\in[1+10000h_{\min},1+20000h_{\min}], \\ 
0, & \text{if } x_i>1+20000h_{\min}.
\end{array}\right.\]
We then set $\alpha_i^k = \beta_i^k + \gamma_i^k/10$ in the fixed mesh setting and $\alpha_i^k = \beta_i^k$ in the adaptive mesh setting.

To illustrate the robustness of our results with respect to the numerical regularization, we repeat the computations using a fixed mesh with 8777 grid points on $[0,L]$, where $L\approx10^7$ and $h_{\min}\approx2\times10^{-7}$. We use a coarser mesh than in Section~\ref{sec:HL_scenario1} to reduce computational time. For $m=1,2,4,8$, we set the regularization coefficient to
\[
\alpha_i^{k,m}=\min\{\beta_i^k,0.2\}+\frac{m\gamma_i^k}{10},
\]
and compute the solutions up to the terminal time $t=30$ starting from the same odd-symmetric initial data as in Section~\ref{sec:HL_scenario1}. We remark that larger values of $m$ correspond to stronger numerical regularization.

Figure \ref{fig:HL_regularization_robustness} compares the terminal profiles of $\Omega$ and $\Theta$ for $m=1,2,4,8$. The top row displays the solution profiles on a larger spatial scale. As shown, the profiles for different values of $m$ are very close away from the singular region near $x=1$. The bottom row zooms in on the solution profiles near $x=1$. As the regularization strength decreases, the peak of $\Omega$ becomes higher and narrower, while $\Theta$ develops a sharper transition across a thinner layer. These observations are consistent with convergence to singular profiles in the vanishing viscosity limit and support the robustness of our numerical results.

\begin{figure}[!htbp]
    \centering
    \includegraphics[height=0.37\textwidth]{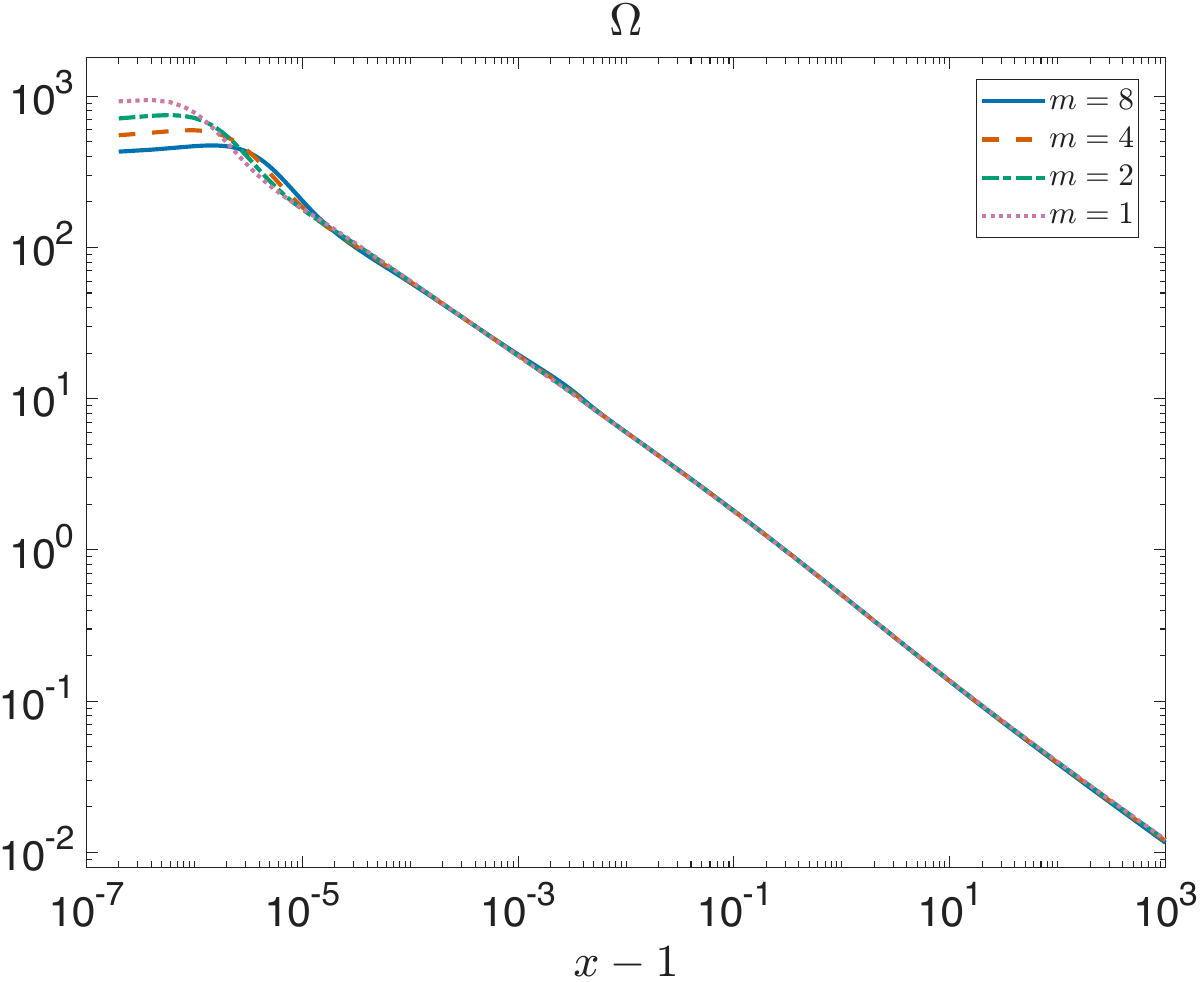}
    \hspace{0.025\textwidth}
    \includegraphics[height=0.37\textwidth]{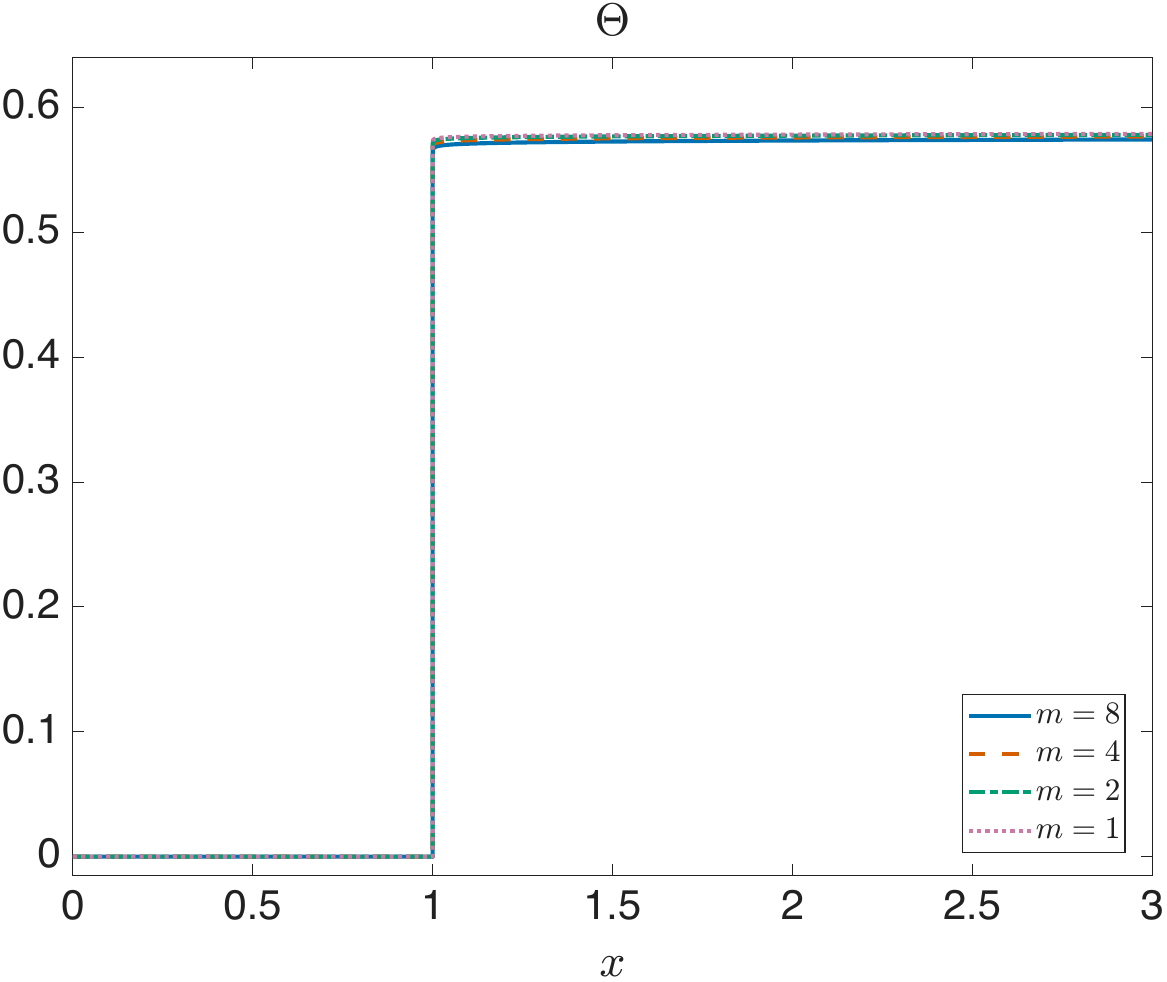}
    \par\medskip
    \includegraphics[height=0.37\textwidth]{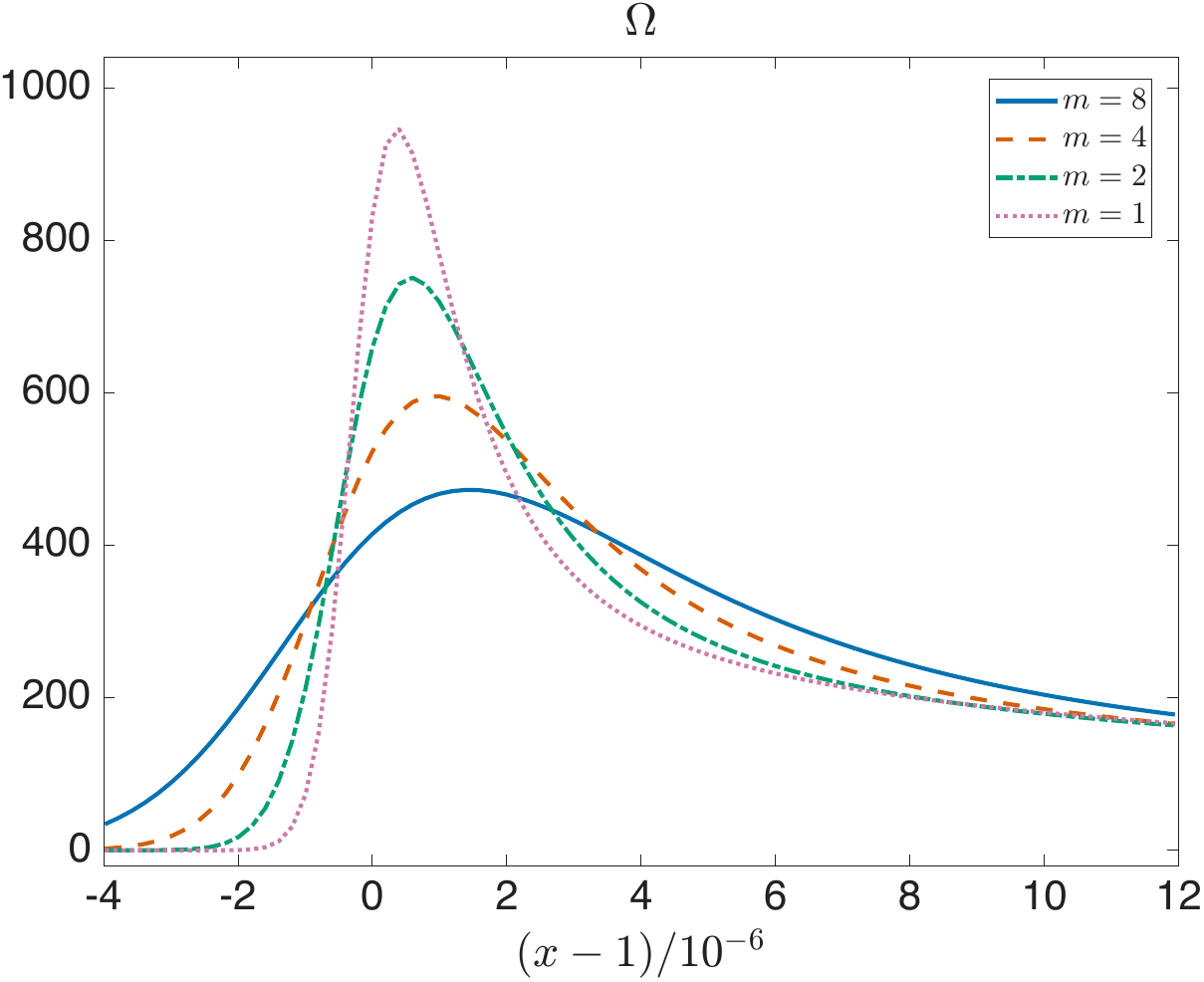}
    \hspace{0.025\textwidth}
    \includegraphics[height=0.37\textwidth]{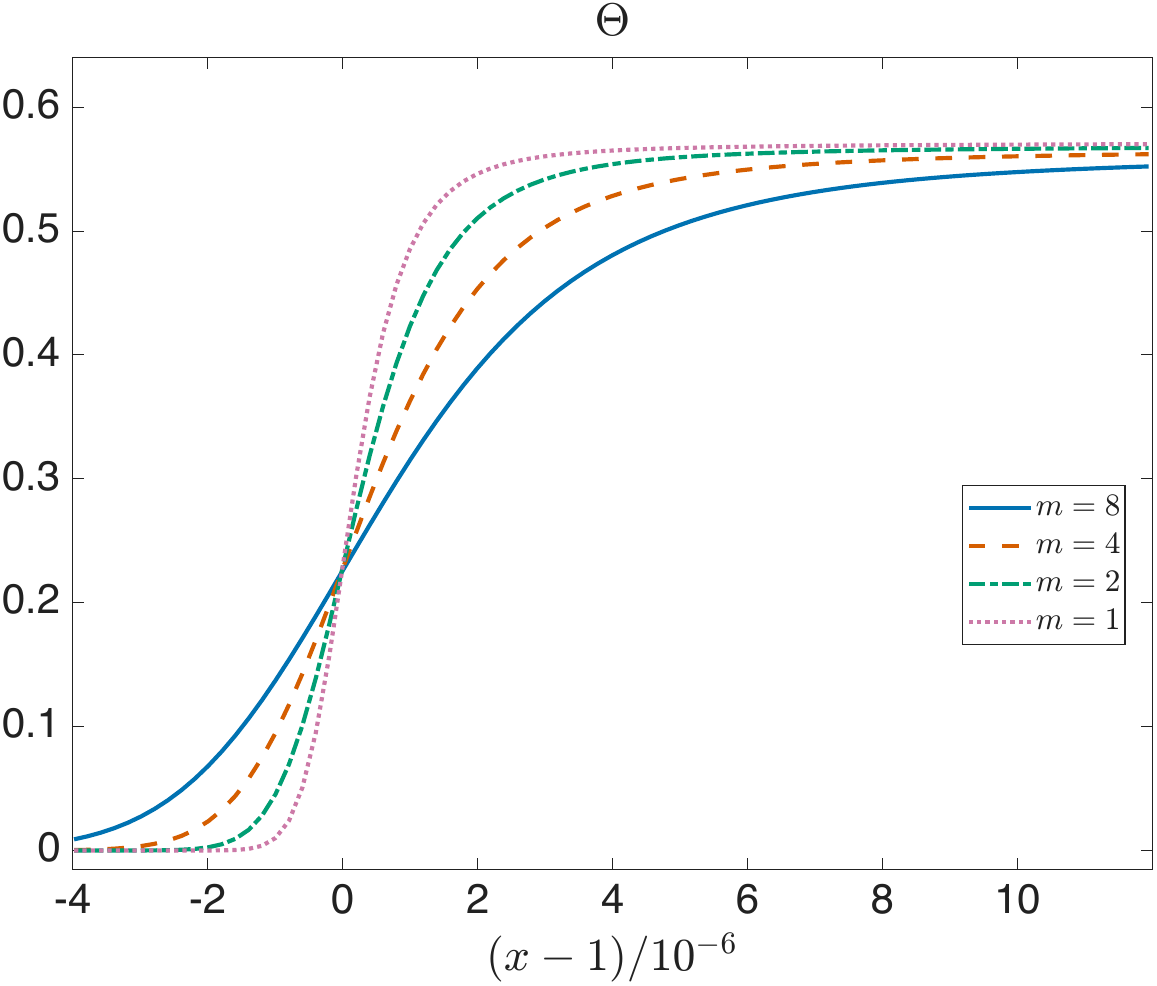}
    \caption[Dependence of the terminal profiles on numerical regularization.]{Profiles of $(\Omega,\Theta)$ for $m=1,2,4,8$ at $t=30$. Here larger values of $m$ correspond to stronger numerical regularization. Top left: log--log plot of $\Omega$ against $x-1$ for $x>1$; top right: profiles of $\Theta$ on $0\leq x\leq3$; bottom left: profiles of $\Omega$ near $x=1$; bottom right: profiles of $\Theta$ near $x=1$.}
    \label{fig:HL_regularization_robustness}
\end{figure}

\section{Resolution Study of the HL model}
In this section, we conduct resolution studies for the numerical methods of the HL model by investigating the convergence of the numerical solutions under mesh refinement and demonstrate the accuracy and robustness of our numerical scheme.

\subsection{Scenario 1}
In this subsection, we perform a resolution study on the numerical solutions of \eqref{eqt:numerical_dynamic_rescaling_of_HLscenario1} obtained via the adaptive mesh strategy. Recall that the adaptive mesh method is employed to investigate the Stage 1 blowup of the system. Consequently, as the solution approaches the blowup time, certain variables may develop large $L^{\infty}$ norms. To verify the effectiveness of our scheme, we quantify the relative error of some solution variable $f$ by comparing it to a reference solution $\hat{f}$ computed on a finer mesh at the same time instant. Specifically, for the domain $D$ considered, we interpolate $f$ onto the reference mesh of $\hat{f}$ and then compute the sup-norm relative error between $f$ and $\hat{f}$ on $D$ as 
 \[\mathrm{E}_{f}(D)= \frac{\|f-\hat{f}\|_{L^\infty(D)}}{\|\hat{f}\|_{L^{\infty}(D)}},\]
 where $\|\cdot\|_{L^{\infty}(D)}$ is evaluated over all grid points of the reference mesh that lie in $D$.

Table \ref{tab:HL_resolution_study_scenario1} reports the relative errors of $\Omega$ and $\Theta$ computed on meshes with varying densities in the odd symmetry case. The results demonstrate that the relative error decreases as the mesh is refined, indicating that the numerical solutions approach the high-resolution reference solution. This behavior confirms the reliability and effectiveness of our numerical scheme.

\begin{table}[h!]
  \centering  
    \vspace{2mm} 
    \begin{tabular}{c c c c}     
      \toprule 
      \textbf {Grid points (\#)} & \textbf{$\mathrm{E}_{\Omega}(\R)$}& \textbf{$\mathrm{E}_{\Omega}(\R\backslash{[0.9999,1.0001]})$} & \textbf{$\mathrm{E}_{\Theta}(\R)$}   \\
      \midrule
      1384  & $7.4066\times 10^{-1}$ & $4.1856 \times 10^{-3}$ & $3.7588\times 10^{-2}$ \\
      2903  & $5.1998\times 10^{-1}$ & $1.2187 \times 10^{-3}$ & $8.8915\times 10^{-3}$ \\
      6038  & $2.5881\times 10^{-1}$ & $3.9646 \times 10^{-4}$ & $1.7926\times 10^{-3}$ \\
      7501  & $1.8304\times 10^{-1}$ & $3.2827 \times 10^{-4}$ & $9.3233\times 10^{-4}$ \\
      10128 & $9.6718\times 10^{-2}$ & $2.0130 \times 10^{-4}$& $3.8108\times 10^{-4}$ \\
      12348 & $4.9708\times 10^{-2}$ & $1.1136 \times 10^{-4}$& $1.9012\times 10^{-4}$ \\
      \bottomrule
    \end{tabular}
    \caption{Relative $L^\infty$ errors of $\Omega$ and $\Theta$ at $t=4.39$. The reference solutions $\hat{\Omega}$ and $\hat{\Theta}$ are computed on a fine mesh with $14007$ grid points. Notably, the relative error of $\Omega$ away from the singular region is significantly smaller than that over the whole space.}
    \label{tab:HL_resolution_study_scenario1}
\end{table}

\subsection{Scenario 2}
In this subsection, we perform a resolution study on the numerical solutions of \eqref{eqt:dynamic_rescaling_of_HL_scenario2_numerical} obtained on meshes with different densities. Recall that in this scenario, the numerical solutions eventually settle to regular profiles. Moreover, for all mesh densities considered, the convergence criterion is satisfied at the end of the iteration:
\[|\Omega_t(x_i)|<10^{-6},\quad |\Theta_{x,t}(x_i)|<10^{-6} \]
for any $i=0,1,2,\cdots ,N$. This indicates that the time derivatives of the terminal variables $\Omega$ and $\Theta_x$ are consistently small at the computational grid points. To verify the effectiveness of our scheme globally, we select a set of test points $\{y_i\}_{i=1}^M$ distinct from the computational grid points and compute the maximal residuals of the terminal solution variables $\Omega$ and $\Theta_x$ at these locations:
\[
  \mathrm{Re}_{\Omega} := \sup_{i \in \{1, \dots, M\}} |\Omega_t(y_i)|, \quad \mathrm{Re}_{\Theta_x} := \sup_{i \in \{1, \dots, M\}} |\Theta_{x,t}(y_i)|.
\]
Here, we interpret $\Omega$ and $\Theta_x$ as piecewise $C^2$ functions on $\mathbb{R}$ via the cubic spline interpolation to ensure that $\Omega_t(y_i)$ and $\Theta_{x,t}(y_i)$ are well-defined.

 Table \ref{tab:HL_resolution_study_scenario2} reports the maximal residuals of $\Omega$ and $\Theta_x$ on the test points $\{y_i\}_{i=1}^M$ for varying mesh densities. It is observed that with mesh refinement, the residuals decrease over the whole space $\mathbb{R}$, which corroborates the effectiveness of our numerical scheme.

 \begin{table}[h!]
  \centering  
    \vspace{2mm} 
    \begin{tabular}{c c c}     
      \toprule 
      \textbf {Grid points (\#)} & \textbf{$\mathrm{Re}_{\Omega}$} & \textbf{$\mathrm{Re}_{\Theta_x}$} \\
      \midrule
      1221  & $2.6921\times 10^{-5}$    & $2.3809\times 10^{-5}$ \\
      1525  & $2.1642\times 10^{-5}$    & $1.5409\times 10^{-5}$ \\
      2033  & $1.6478\times 10^{-5}$    & $9.7263\times 10^{-6}$ \\
      2441  & $1.3415\times 10^{-5}$    & $7.8620\times 10^{-6}$ \\
      3051  & $1.0212\times 10^{-5}$    & $6.0356\times 10^{-6}$ \\
      4067  & $8.6111\times 10^{-6}$    & $4.9949\times 10^{-6}$ \\
      6101  & $6.0221\times 10^{-6}$    & $3.4735\times 10^{-6}$ \\
      12203 & $3.1075\times 10^{-6}$    & $1.8424\times 10^{-6}$ \\
      \bottomrule
    \end{tabular}
    \caption{Maximal residuals of $\Omega$ and $\Theta_x$ on the test points $\{y_i\}_{i=1}^M$ for varying mesh densities.}
    \label{tab:HL_resolution_study_scenario2}

\end{table}
\section{The numerical methods for 2D Boussinesq equations}\label{sec:2D_numerical_method}

In the two-dimensional setting, computational cost significantly limits the feasible grid resolution. To resolve the evolving profile accurately, we therefore employ high-order numerical schemes.

Our discretization is hybrid on the rectangular domain $D=[-L_1,R_1]\times[0,R_2]$: the elliptic stream-function recovery is handled by a spline-based finite-element method, while the transport terms are discretized by a finite-difference WENO scheme.
\subsection{Adaptive mesh}

We use a mapping from the computational domain $(\rho,\eta)$ to the physical domain $(x_1,x_2)$:
\[
    x_1(\rho) = \int_0^{\rho} \varrho_1(s)\idiff s - L_1,\qquad
    x_2(\eta) = \int_0^{\eta} \varrho_2(w)\idiff w,
\]
where $(x_1,x_2)\in[-L_1,R_1]\times[0,R_2]$ and the mesh densities $\varrho_1,\varrho_2$ are given by
\begin{align*}\label{eqt:2dDensity}
\varrho_1(\rho) &= \alpha_1^{(1)}
    + \alpha_1^{(2)}\,\Gamma\!\left(1 - \frac{\rho}{\beta_1}\right)
    + \alpha_1^{(3)}\,\exp\!\Big(-\pi\,\tfrac{(\rho-1)^2}{\sigma_1^2}\Big)
    + \alpha_1^{(4)}\,\exp\!\Big(-\pi\,\tfrac{\rho^2}{\sigma_3^2}\Big),\\
\varrho_2(\eta) &= \alpha_2^{(1)}
    + \alpha_2^{(2)}\Big[ \Gamma\!\left(1 - \tfrac{\eta}{\beta_2}\right) + \Gamma\!\left(1 + \tfrac{\eta}{\beta_2}\right)\Big] \\
&\quad\; + \alpha_2^{(3)}\Big[\exp\!\Big(-\pi\,\tfrac{(\eta-1)^2}{\sigma_3^2}\Big) + \exp\!\Big(-\pi\,\tfrac{(\eta+1)^2}{\sigma_3^2}\Big)\Big],\\
\Gamma(\cdot) &= \frac{1}{\exp(\cdot)+1}.
\end{align*}
The density function $\varrho_2(\eta)$ is chosen to be even in $\eta$, so the mesh remains nearly symmetric and close to uniform near the boundary $x_2=0$.

The parameter vector $\sigma=(\sigma_1,\sigma_2,\sigma_3)$ controls the width and sharpness of the refinement features. The coefficient vectors $\alpha_1=(\alpha_1^{(1)},\alpha_1^{(2)},\alpha_1^{(3)},\alpha_1^{(4)})$ and $\alpha_2=(\alpha_2^{(1)},\alpha_2^{(2)},\alpha_2^{(3)})$, together with $\beta=(\beta_1,\beta_2)$, are updated dynamically to track the evolving profile. In practice, the mesh first concentrates around the $10\%$ level set and then shifts its finest resolution toward the location of maximal vorticity magnitude.

In our simulations, we fix $\sigma = (1/16,1/8,1/8)$, while $\alpha_1$, $\alpha_2$, and $\beta$ are updated adaptively.

The two scenarios are treated with different symmetry settings.
In Scenario 1, we enforce odd symmetry with respect to $x_1=0$ for vorticity and thus stream function,
that is,
\[
\omega(-x_1,x_2)=-\omega(x_1,x_2),\qquad \psi(-x_1,x_2)=-\psi(x_1,x_2),
\]
and therefore set $L_1=0$ and compute on $[0,R_1]\times[0,R_2]$.
In Scenario 2, no additional symmetry in $x_1$ is imposed; we solve on the full truncated upper-half-plane
$[-L_1,R_1]\times[0,R_2]$, with typical scales
\[
R_1 = R_2\approx 10^6,\qquad L_1\approx 10^5.
\]

\subsection{Velocity recovery} 

The velocity is recovered from the vorticity through the Biot--Savart law. In the half-plane
$\R\times\R_+$ with no-penetration boundary condition $u_2|_{x_2=0}=0$, we introduce the stream
function $\psi$ such that
\[
\mtx{u}=\nabla^\perp\psi=(-\psi_{x_2},\psi_{x_1}),\qquad -\Delta\psi=\omega,\qquad \psi|_{x_2=0}=0.
\]
Let $\bar{\mtx{y}}=(y_1,-y_2)$ denote the reflection of $\mtx{y}=(y_1,y_2)$ across
$x_2=0$, then
\begin{equation}\label{eqt:BS_biot_savart_halfplane_stream}
\psi(\mtx{x})=\frac{1}{2\pi}\int_{\R\times\R_+}\omega(\mtx{y})\ln\frac{|\mtx{x}-\bar{\mtx{y}}|}{|\mtx{x}-\mtx{y}|}  \idiff{\mtx{y}}.
\end{equation}
In Scenario 1, the odd symmetry in $x_1$ is imposed throughout the computation. In particular,
$\omega(0,x_2)=\psi(0,x_2)=0$, which is used in ghost-value construction at $x_1=0$.
Moreover, the Biot--Savart kernels are reduced to the half-strip $\{y_1>0,y_2>0\}$ by symmetry.
Let
\[
\hat{\mtx{y}}:=(-y_1,y_2),\qquad \bar{\mtx{y}}:=(y_1,-y_2).
\]
Then for $x_1\ge 0$, we use
\begin{equation}\label{eqt:BS_biot_savart_halfplane_stream_sc1}
\psi(\mtx{x})=\frac{1}{2\pi}\int_{\R_+\times\R_+}
\omega(\mtx{y})\ln\frac{|\mtx{x}-\bar{\mtx{y}}|\,|\mtx{x}-\hat{\mtx{y}}|}{|\mtx{x}-\mtx{y}|\,|\mtx{x}+\mtx{y}|} \idiff{\mtx{y}}.
\end{equation}

In Scenario 2, this symmetry constraint is removed, and velocity recovery is carried out on the full
truncated domain stated above.

Numerically, we do not evaluate \eqref{eqt:BS_biot_savart_halfplane_stream} and \eqref{eqt:BS_biot_savart_halfplane_stream_sc1} directly at all grid
points. Instead, we adopt a two-step elliptic recovery:
\begin{enumerate}
    \item Use \eqref{eqt:BS_biot_savart_halfplane_stream_sc1} in Scenario 1 (or
    \eqref{eqt:BS_biot_savart_halfplane_stream} in Scenario 2) to compute Dirichlet data of $\psi$ on
    the outer boundary of the truncated box $D$.
    \item Solve the Poisson problem $-\Delta\psi=\omega$ in $D$ with the above boundary data by a
    conforming spline finite-element method.
\end{enumerate}
We use tensor-product 4th-order B-splines centered at the grid points. This yields a globally smooth
($C^2$) approximation of $\psi$, after which the velocity is recovered by differentiation,
$\mtx{u}=(-\psi_{x_2},\psi_{x_1})$. Compared with direct Biot--Savart quadrature in the full domain,
this hybrid strategy is more efficient on adaptive meshes and remains reliable near the boundary and
near singular regions.

\subsection{Time-stepping scheme}

The system is advanced in time by an explicit third-order strong-stability-preserving Runge-Kutta method (SSPRK3). To suppress spurious oscillations near steep gradients, we use the fifth-order weighted essentially non-oscillatory (WENO5) reconstruction to approximate upwind derivatives.

The WENO5 reconstruction is applied dimension-by-dimension in both the $x_1$- and $x_2$-directions. For definiteness, we present the formula in the $x_1$-direction for each fixed $x_2$; the $x_2$-direction is completely analogous. The interface values $\omega_{i+1/2}^{\pm}$ are reconstructed by a convex combination of third-order polynomials on three substencils. For the left-biased reconstruction $\omega_{i+1/2}^-$, the substencils are $S_0=\{x_{1,i-2},x_{1,i-1},x_{1,i}\}$, $S_1=\{x_{1,i-1},x_{1,i},x_{1,i+1}\}$, and $S_2=\{x_{1,i},x_{1,i+1},x_{1,i+2}\}$.

\[
\omega_{i+1/2}^- = \sum_{k=0}^2 \omega_k q_k,
\]
where the third-order approximations $q_k$ are
\begin{align*}
q_0 &= \frac{1}{3}\omega_{i-2} - \frac{7}{6}\omega_{i-1} + \frac{11}{6}\omega_i, \\
q_1 &= -\frac{1}{6}\omega_{i-1} + \frac{5}{6}\omega_i + \frac{1}{3}\omega_{i+1}, \\
q_2 &= \frac{1}{3}\omega_i + \frac{5}{6}\omega_{i+1} - \frac{1}{6}\omega_{i+2}.
\end{align*}
The nonlinear weights $\omega_k$ are defined as
\[
\omega_k = \frac{\alpha_k}{\sum_{j=0}^2 \alpha_j}, \quad \alpha_k = \frac{d_k}{(\epsilon + \beta_k)^2},
\]
where $d_0=0.1$, $d_1=0.6$, and $d_2=0.3$ are the ideal weights, and $\epsilon$ is a small regularization parameter to prevent division by zero. We use $\epsilon=10^{-6}$ in Scenario 1 and $\epsilon=10^{-10}$ in Scenario 2 to better resolve smooth limiting profiles. The smoothness indicators $\beta_k$ are
\begin{align*}
\beta_0 &= \frac{13}{12}(\omega_{i-2}-2\omega_{i-1}+\omega_i)^2 + \frac{1}{4}(\omega_{i-2}-4\omega_{i-1}+3\omega_i)^2, \\
\beta_1 &= \frac{13}{12}(\omega_{i-1}-2\omega_i+\omega_{i+1})^2 + \frac{1}{4}(\omega_{i-1}-\omega_{i+1})^2, \\
\beta_2 &= \frac{13}{12}(\omega_i-2\omega_{i+1}+\omega_{i+2})^2 + \frac{1}{4}(3\omega_i-4\omega_{i+1}+\omega_{i+2})^2.
\end{align*}
The right-biased reconstruction $\omega_{i+1/2}^+$ is obtained by symmetry. The convective derivative is then approximated by an upwind flux determined by the local sign of the transport velocity. Near the boundaries, ghost values are filled consistent with the half-plane formulation and with the parity of the even/odd extensions across $x_2=0$. At the far-field boundaries, we use flow splitting with  prescribed exterior data on inflow segments and extrapolation on outflow segments.

The time step is selected according to the CFL condition with a constant $0.4$.

\subsection{Resolution study}

In this subsection, we conduct resolution studies for the numerical methods of the 2D Boussinesq equations by investigating the convergence of numerical solutions under mesh refinement, and demonstrate the accuracy and reliability of our numerical scheme.

\subsubsection{Scenario 1}

In this subsection, we perform a resolution study at a fixed target time $t=0.2$ by comparing adjacent resolutions on dyadic meshes. For each pair $(N,2N)$, we interpolate the coarse-grid solution $(\Omega_N, \Theta_N)$ onto the fine grid as $(\Omega_{N\to 2N},\Theta_{N\to 2N})$ and compare it to the fine grid solution $(\Omega_{2N},\Theta_{2N})$. Specifically we investigate the grid-point value of their differences
\[
\Delta \Omega_{N\to 2N}=\Omega_{N\to 2N} -\Omega_{2N},\qquad
\Delta \Theta_{N\to 2N}=\Theta_{N\to 2N} -\Theta_{2N}.
\]  

Table \ref{tab:BS_resolution_study_scenario1} reports $\sup(|\Delta \Omega_{N\to 2N}|)$ and $\sup(|\Delta \Theta_{N\to 2N}|)$ for $N=64,128,256,512$. Both errors decrease monotonically under mesh refinement, supporting the numerical reliability of the Scenario 1 computation.

\begin{table}[h!]
    \centering
        \vspace{2mm}
        \begin{tabular}{c c c}
            \toprule
            \textbf{Coarse $\to$ Fine} & \textbf{$\sup(|\Delta \Omega|)$} & \textbf{$\sup(|\Delta \Theta|)$} \\
            \midrule
            $64\to 128$   & $1.3991\times 10^{-1}$ & $1.1706\times 10^{0}$ \\
            $128\to 256$  & $4.9028\times 10^{-2}$ & $7.6318\times 10^{-1}$ \\
            $256\to 512$  & $2.2386\times 10^{-2}$ & $4.2334\times 10^{-1}$ \\
            $512\to 1024$ & $3.9207\times 10^{-3}$ & $6.3977\times 10^{-2}$ \\
            \bottomrule
        \end{tabular}
         \caption{$L^{\infty}$ errors from adjacent-resolution comparison at $t=0.2$ (Scenario 1).}
        \label{tab:BS_resolution_study_scenario1}
\end{table}

\subsubsection{Scenario 2}
In this subsection, we perform a resolution study on the numerical solutions obtained on meshes with different densities. Recall that in this scenario, the numerical solutions eventually settle to regular profiles. 
 For all mesh densities considered, the convergence criterion $\|\widehat{\Omega}_t\|_{\infty}<10^{-6}$ is satisfied at the end of the iteration.
Here $\widehat{\Omega}$ refers to the numerical solution of $\Omega$.

 Same as in the 1D case, to verify the effectiveness of our scheme globally, we evaluate how close the computed terminal solution is to a discrete steady state by re-applying the full discrete operator on a finer mesh. More precisely, for a solution computed on a mesh of size $N$, we first lift $\Omega$ to a mesh of size $2N$ via high-order interpolation, then re-evaluate the full discrete operators on the fine grid, including the Poisson solve, velocity recovery, and WENO flux reconstruction, to obtain the instantaneous residuals $\Omega_t$ and $(\Theta_{x_1})_t$. Meanwhile, the scaling factors $c_l$ and $c_\omega$ are re-calculated on the fine grid so that the normalization conditions remain satisfied, ensuring that the measured residual reflects only the discretization error of the putative steady state. We then select a set of test points $\{(x_i,y_i)\}_{i=1}^M$, distinct from the computational grid points, and define
\[
    \mathrm{Re}_{\Omega} := \sup_{i \in \{1, \dots, M\}} |\Omega_t(x_i,y_i)|,
    \qquad
    \mathrm{Re}_{\Theta_{x_1}} := \sup_{i \in \{1, \dots, M\}} |(\Theta_{x_1})_t(x_i,y_i)|.
\]

Table \ref{tab:BS_resolution_study_scenario2} reports the $L^\infty$ residuals of $\Omega_t$ and $(\Theta_{x_1})_t$ on different mesh levels. The residuals decrease monotonically under refinement, supporting the numerical reliability of the Scenario 2 computation.

\begin{table}[h!]
    \centering
    \vspace{2mm}
    \begin{tabular}{c c c c c}
	        \textbf{Mesh size $N$} & \textbf{$\mathrm{Re}_{\Omega}$} & \textbf{$\mathrm{Re}_{\Theta_{x_1}}$}  \\
        \midrule
        $128$ & $1.1385\times 10^{-1}$ & $1.4508\times 10^{-2}$\\
        $256$ & $3.8487\times 10^{-2}$ & $3.2856\times 10^{-3}$  \\
        $512$ & $4.3836\times 10^{-3}$ & $5.1663\times 10^{-4}$  \\
        \bottomrule
    \end{tabular}
        \caption{$L^{\infty}$ residuals of $\Omega_t$ and $(\Theta_{x_1})_t$ (Scenario 2).}
    \label{tab:BS_resolution_study_scenario2}
\end{table}

Overall, the above observations indicate that the errors and residuals decrease as the mesh is refined, which corroborates the accuracy and reliability of the 2D numerical scheme.

\vspace{2mm}

\subsection*{Acknowledgement} The authors are supported by the National Key R\&D Program of China under the grant 2021YFA1001500 and the National Natural Science Foundation of China under the grant NSFC No. 12288101.

\bibliographystyle{myalpha}
\bibliography{reference}

\end{document}